\documentclass[11pt]{article}
\usepackage{fullpage,graphicx,psfrag,amsmath,amsfonts,verbatim}
\usepackage{indentfirst}
\usepackage[small,bf]{caption}
\usepackage{color}
\usepackage{ragged2e} 
\usepackage{booktabs,makecell, multirow, tabularx}
\usepackage{enumerate}
\usepackage{amsthm}
\usepackage{amsfonts}
\usepackage{amssymb}
\usepackage{graphicx}
\usepackage{setspace}
\usepackage{hyperref}
\usepackage{tikz}
\usepackage[numbers]{natbib} 
\allowdisplaybreaks[4]
\usepackage{array}

\def\m{\mkern 1mu}
\def\mm{\mkern 2mu}
\def\mn{\mkern- 1mu}
\def\mmn{\mkern- 2mu}

\newcommand{\argmin}{\mathop{\rm argmin}}

\newtheorem{thm}{Theorem}[section]
\newtheorem{lem}{Lemma}[section]
\newtheorem{cor}{Corollary}[section]
\newtheorem{prop}{Proposition}[section]

\newtheorem{defn}{Definition}[section]

\newtheorem{remark}{Remark}[section]

\def\approxcorrect{\checkmark\kern-1.1ex\raisebox{.89ex}{$\times$}}

\usepackage{amsmath,amsfonts,bm}

\def\eqref#1{equation~\ref{#1}}

\def\1{\bm{1}}

\DeclareMathAlphabet{\mathsfit}{\encodingdefault}{\sfdefault}{m}{sl}
\SetMathAlphabet{\mathsfit}{bold}{\encodingdefault}{\sfdefault}{bx}{n}

\title{DRM Revisited: A Complete Error Analysis}
\author{
  Yuling Jiao$^{1,2}$,
  Ruoxuan Li$^{1}$,
  Peiying Wu$^{1}$,
  Jerry Zhijian Yang$^{1,2}$\thanks{Corresponding author.},
  and Pingwen Zhang$^{1,2,3}$ \\
  \\
  {\small $^1$School of Mathematics and Statistics, Wuhan University, Wuhan 430072, China} \\
  {\small $^2$Hubei Key Laboratory of Computational Science, Wuhan University, Wuhan 430072, China} \\
  {\small $^3$School of Mathematical Sciences, Peking University, Beijing 100871, China}
}
\date{}

\begin{document}
\onehalfspacing
\maketitle

\begin{abstract}
 In this work, we address a foundational question in the theoretical analysis of the Deep Ritz Method (DRM) under the over-parameteriztion regime: Given a target precision level, how can one determine the appropriate number of training samples, the key architectural parameters of the neural networks, the step size for the projected gradient descent optimization procedure, and the requisite number of iterations, such that the output of the gradient descent process closely approximates the true solution of the underlying partial differential equation to the specified precision ?

\end{abstract}

\noindent%

\section{Introduction}
\label{intro}
Classical numerical methods, such as finite element methods \cite{brenner2007mathematical,ciarlet2002finite}, face difficulties  when solving high-dimensional PDEs.  The success of deep learning methods in  high-dimensional data analysis has led to the development of  promising approaches for solving high-dimensional PDEs using deep neural networks, which have attracted much attention  \citep{Cosmin2019Artificial,Justin2018DGM,DeepXDE,raissi2019physics,Weinan2017The,Yaohua2020weak,Berner2020Numerically,han2018solving}. Due to the excellent approximation power of deep neural networks, several numerical schemes have been proposed for solving PDEs, including physics-informed neural networks (PINNs)  \citep{raissi2019physics}, weak adversarial networks (WAN) \citep{Yaohua2020weak} and the deep Ritz method (DRM) \citep{Weinan2017The}. PINNs is based on residual minimization, while WAN is inspired by Galerkin method. Based on classical Ritz method, the deep Ritz method is proposed to solve variational problems corresponding to a class of PDEs, which has become one of the most renowned approaches in the field of elliptic equations.

The success of deep learning methods in solving high dimensional PDEs has propelled the advancement of its theoretical research. It is now widely recognized that, as a non-parametric estimation method, error analysis in deep learning for PDEs includes approximation error, statistical error (also called generalization error), and optimization error \cite{grohs2022mathematical,telgarsky2021deep,weinan2020machine,bach2023learning}. To date, the existing convergence analysis for these deep solvers   has predominantly focused on characterizing the trade-offs between the approximation error and the statistical error \citep{weinan2019priori,hong2021rademacher,lu2021priori,hutzenthaler2020overcoming,shin2020convergence,lanthaler2022error,muller2021error,mishra2021enhancing,kutyniok2022theoretical,son2021sobolev,wang2022and,weinan2020comparative,jiao2022rate,duan2021convergence,lu2021machine,mishra2022estimates,ji2024deep,yang2024deeper,hu2023solving1, hu2023solving2, dai2023solving}. Meanwhile,  these results are conducted in scenarios where the number of neural network parameters is smaller than the number of training samples. However, in practical applications, over-parameterized networks, where the number of parameters far exceeds the number of samples, are more commonly used since empirical evidence suggests that over-parameterization  makes  the training computationally more efficient. Moreover, recent theoretical studies have indicated  that  the training loss will  converges to zero linearly  if one properly initialized the (stochastic) gradient descent  specialized   in over-parameterized regimes, even though the optimization problem  is highly non-convex \cite{jacot2018neural,allen2019convergence,du2019gradient,zou2019improved,liu2022loss,chizat2019lazy}.

The fundamental drivers behind the empirical success of over-parameterized deep learning models continue to elude full understanding, especially when  simultaneously accounting for the complex interplay between approximation, generalization, and optimization \citep{belkin2021fit,bartlett2021deep,berner2022modern}.
Extensive research efforts have been dedicated to elucidating the role of over-parameterization in linear and kernel models, particularly from the perspective of the double descent phenomenon
\citep{belkin2018understand,belkin2019reconciling,hastie2022surprises,
belkin2019does,liang2020just,nakkiran2020optimal,bartlett2020benign,
tsigler2023benign,belkin2021fit,bartlett2021deep,tsigler2023benign}.
However, a crucial gap remains in providing a comprehensive error analysis that jointly accounts for all three key error components: approximation, generalization, and optimization. This challenge persists even for the empirical risk minimization estimator in over-parameterized deep learning settings, which has been shown to potentially yield inconsistent results \cite{kohler2021over}.


\subsection{Contributions}
\begin{itemize}
\item 
In this work, we have established the first comprehensive error analysis for the deep
Ritz method in the over-parameterized setting. This analysis jointly accounts for all three key error components: approximation error, statistical (generalization) error, and optimization error. 
\item Technically, we have derived a novel error decomposition, where the optimization error term we employ is distinct from and tighter than those used in prior literature. This error decomposition is of independent theoretical interest and holds value for the analysis of other deep learning tasks.
\item 
Unlike previous analyses of optimization error, a key feature of our main results is that they do not require the entire training dynamics to remain confined within an infinitesimally small neighborhood of the initial parameter values. This reduces the gap between theory and practical training.
\end{itemize}

\subsection{Organizations}
The paper is organized as follows. In Section \ref{pre}, we first introduce the notation, the parallel neural network architecture $\mathcal{PNN}$, and the projected gradient descent (PGD) algorithm used for optimization.  Then, following these preliminaries, we present the main theorem of this paper.
In Section \ref{proof}, we present the proof of the main theorem, divided into five subsections covering a novel error decomposition method, approximation error bounds, statistical error estimates, optimization error control, and the combination of all the separate analysis.
In Section \ref{rw}, we discuss related work in detail and highlight our contributions. Finally, in Section \ref{conclu}, we provide a summary of the paper and outline our planned future work. All proof details are provided in the appendix.


\section{Main Result}
\label{pre}
In this section, we present the main theoretical result of this paper, which establishes the first comprehensive error analysis for the deep Ritz method in the over-parameterized setting. To achieve this, we first need some groundwork: In Sections \ref{nota} and \ref{subsec:topo}, we introduce necessary notations and the parallel neural network class $\mathcal{PNN}$ used in this paper. In Section \ref{drm}, we review the deep Ritz method. In Section \ref{pgd}, we provide a detailed exposition of the employed optimization algorithm: the projected gradient descent (PGD) algorithm. Finally, in Section \ref{main}, we formally propose our main result.
\subsection{Notation} \label{nota}
In this section, we provide all the notations needed in this paper. We use bold-faced letters to denote vectors and capital letters to denote matrices or fixed parameters. Unless otherwise specified, $C$ represents a constant, and $C(a,b)$ or $C_i(a,b)$ represents functions that only depend on $a$ and $b$. For two positive functions $f(x)$ and $g(x)$, the asymptotic notation $f(x) = \mathcal{O}(g(x))$ denotes $f(x) \leq C g(x)$ for some constant $C > 0$. The notation $\widetilde{\mathcal{O}}(\cdot)$ is used to ignore logarithmic terms. 

Let $\mathbb{N}$ denotes the set of natural numbers. We define $\mathbb{N}^{+}:=\{x\in \mathbb{N} \, | \,  x>0\}$. If $x\in \mathbb{R}$, $\lfloor x \rfloor := \max\{k\in \mathbb{N}: k\leq x\}$ denotes the largest integer strictly smaller than $x$ and $\lceil x \rceil := \min\{k\in \mathbb{N}: k\geq x\}$ denotes the smallest integer strictly larger than $x$. If $N \in \mathbb{N}^{+}$, we define $[N]:=\{1,2,\ldots,N\}$ to be set of all positive integers less than or equal to $N$. We use the usual \textit{multi-index} notation, i.e. for $\boldsymbol{\alpha}\in \mathbb{N}^{d}$ we write $\|\boldsymbol{\alpha}\|_1:=\alpha_1+\ldots+\alpha_d$ and $\boldsymbol{\alpha}! := \alpha_1 !\cdot \ldots \cdot \alpha_d !$. 

Let $\Omega \subset \mathbb{R}^d$ be an open set. For a function $f: \Omega \to \mathbb{R}$, we denote its (weak or classical) derivative of order $\boldsymbol{\alpha}$ by
$$
D^{\boldsymbol{\alpha}}f :=\frac{\partial^{\|\boldsymbol{\alpha}\|_1} f}{{\partial x_1}^{\alpha_1}{\partial x_2}^{\alpha_2}\cdots{\partial x_d}^{\alpha_d}}\, .
$$

For $s \in \mathbb{N} \cup {\infty}$, we denote by $C^s(\Omega)$ the set of $s$-times continuously differentiable functions on $\Omega$. Additionally, if $\overline{\Omega}$ is compact, we set, for $f \in C^s(\Omega)$,
$$
\|f\|_{C^s(\overline{\Omega})}:= \max_{0\leq \|\boldsymbol{\alpha}\|_{1}\leq s} \sup_{x\in \Omega} |D^{\boldsymbol{\alpha}}f(x)|\, .
$$
For any $s \in \mathbb{N}$ and $1\leq p < \infty$, we define the \textit{Sobolev space} $W^{s,p}(\Omega)$ by
$$
W^{s,p}(\Omega):=\{f\in L^p(\Omega): D^{\boldsymbol{\alpha}}f\in L^p(\Omega), \ \forall \boldsymbol{\alpha}\in \mathbb{N}^{d} \text \ with \  \|\boldsymbol{\alpha}\|_1 \leq s\}\, .
$$
In particular, when $p=2$, we define $H^s(\Omega):=W^{s,2}(\Omega)$ for any $s\in\mathbb{N}$. Moreover, for any $f\in W^{s,p}(\Omega)$ with $1\leq p <\infty$, we define the Sobolev norm by
$$
\|f\|_{W^{s,p}(\Omega)}:=\Bigg(\sum_{0\leq \|\boldsymbol{\alpha}\|_1\leq s} \|D^{\boldsymbol{\alpha}}f\|^p_{L^p(\Omega)}\Bigg)^{1/p}\, .
$$
When $p=\infty$, we have
$$
\|f\|_{W^{s,\infty}(\Omega)} := \max_{0\leq \|\boldsymbol{\alpha}\|_1\leq s}\|D^{\boldsymbol{\alpha}}f\|_{L^{\infty}(\Omega)}\, . 
$$

\subsection{Topology of the deep networks} \label{subsec:topo}
Let $L, d, N_{0}, \ldots, N_{L} \in \mathbb{N}$. We consider the function $\phi_{\boldsymbol{\theta}}: \mathbb{R}^{d} \rightarrow \mathbb{R}$ that can be parameterized by a $\rho$-activated neural network of the form
\begin{equation}
\label{eq1}
\begin{aligned}
\phi_{0}(\boldsymbol{x}) &=\boldsymbol{x}\, , \\
\phi_{\ell}(\boldsymbol{x}) &=\rho (\boldsymbol{A}_{\ell-1} \phi_{\ell-1}(\boldsymbol{x}) + \boldsymbol{b}_{\ell-1}), \quad \ell=1, \ldots, L-1 \, , \\
\phi_{L}(\boldsymbol{x}) &=\boldsymbol{A}_{L-1} \phi_{L-1}(\boldsymbol{x}) + \boldsymbol{b}_{L-1}\, ,
\end{aligned}
\end{equation}
where $\boldsymbol{A}_{\ell} =(a_{i,j}^{(\ell)}) \in \mathbb{R}^{N_{\ell+1} \times N_{\ell}}$, $\boldsymbol{b}_{\ell} = (b_{i}^{(\ell)}) \in \mathbb{R}^{N_{\ell+1}}$ with $N_{0}=d$ and $N_{L}=1$. The number $W:=\max \{N_{1}, \ldots, N_{L}\}$ is called the width of the network, and $L$ is called the depth of the network. For convenience, we denote $\mathfrak{n}_{\ell}$, $\ell = 1, \ldots, L$, as the number of nonzero weights in the first $\ell$ layers of the network, with $\mathfrak{n}_{L} \leq \mathfrak{D}(W, L, d)$. Here, $\mathfrak{D}(W, L, d)$ is defined as
\begin{align} \label{eq:align_dim}
\mathfrak{D}(W, L, d):=(W+1)[(L-2)W+d+1] \, ,
\end{align}
Meanwhile, $W$ is generally greater than $d$ in the following context. Therefore, we will also use the following estimate $\mathfrak{n}_{L} \leq W(W+1)L$ without loss of generality.

Denote $\boldsymbol{\theta} = (a_{\scriptscriptstyle {1,1}}^{\scriptscriptstyle (0)}, \ldots, a_{\scriptscriptstyle N_L, N_{L-1}}^{\scriptscriptstyle (L-1)}, b_{\scriptscriptstyle 1}^{\scriptscriptstyle (0)}, \ldots, b_{\scriptscriptstyle N_L}^{\scriptscriptstyle (L-1)}) $ as the weight vector of the neural network and $\Theta$ as the set of all weight vectors $\boldsymbol{\theta}$. When the activation function $\rho$ is clear, we use the notation $\mathcal{NN}(W, L, B_{\boldsymbol{\theta}})$ to refer to the collection of functions $\phi_{\boldsymbol{\theta}}$ implemented by a $\rho$-activated neural network with width $W$, depth $L$, and the weight vector $\boldsymbol{\theta}$ satisfying $\|\boldsymbol{\theta}\|_{\infty} \leq B_{\boldsymbol{\theta}}$. 

Note that for any weight vector $\boldsymbol{\theta} = (a_{\scriptscriptstyle {1,1}}^{\scriptscriptstyle (0)}, \ldots, a_{\scriptscriptstyle N_L, N_{L-1}}^{\scriptscriptstyle (L-1)}, b_{\scriptscriptstyle 1}^{\scriptscriptstyle (0)}, \ldots, b_{\scriptscriptstyle N_L}^{\scriptscriptstyle (L-1)})$, we can always elevate it to a $\mathfrak{D}(W, L, d)\m$-dimensional vector 
$$
\boldsymbol{\theta}^{\prime} = (a_{\scriptscriptstyle {1,1}}^{\scriptscriptstyle (0)}, \ldots, a_{\scriptscriptstyle W, d}^{\scriptscriptstyle (0)}, a_{\scriptscriptstyle {1,1}}^{\scriptscriptstyle (1)}, \ldots, a_{\scriptscriptstyle W, W}^{\scriptscriptstyle (L-2)}, a_{\scriptscriptstyle {1,1}}^{\scriptscriptstyle (L-1)}, \ldots, a_{\scriptscriptstyle 1, W}^{\scriptscriptstyle (L-1)}, b_{\scriptscriptstyle 1}^{\scriptscriptstyle (0)}, \ldots, b_{\scriptscriptstyle 1}^{\scriptscriptstyle (L-1)})
$$
by padding zeros, while both $\phi_{\boldsymbol{\theta}}
$ 
and $\phi_{\boldsymbol{\theta}^{\prime}}$ represents the same neural network in $\mathcal{NN}(W, L, B_{\boldsymbol{\theta}})$. Therefore, for any $\phi_{\boldsymbol{\theta}_1}, \phi_{\boldsymbol{\theta}_2} \in \mathcal{NN}(W, L, B_{\boldsymbol{\theta}})$, we can align $\boldsymbol{\theta}_1$ and $\boldsymbol{\theta}_2$ to two $\mathfrak{D}(W, L, d)\m$-dimensional vectors, and then add, subtract, or compare them.

In addition, we introduce a Parallel Neural Network class denoted as $\mathcal{PNN}(\mathfrak{m}, M, \{W, L, B_{\boldsymbol{\theta}}\})$, which represents a linear combination of $\mathfrak{m}$ sub-networks $\mathcal{NN}(W^k, L^k, B_{\boldsymbol{\theta}}^k)$ for $k=1, \ldots, \mathfrak{m}$. Specifically, for any $\Phi_{\mathfrak{m}, \boldsymbol{\theta}}(\boldsymbol{x}) \in \mathcal{PNN}(\mathfrak{m}, M, \{W, L, B_{\boldsymbol{\theta}}\})$, it can be expressed as
$$
\Phi_{\mathfrak{m}, \boldsymbol{\theta}}(\boldsymbol{x}) = \sum_{k=1}^{\mathfrak{m}} c_k \phi^{k}_{\boldsymbol{\theta}}(\boldsymbol{x})\, , \quad c_k \in \mathbb{R}\, ,
$$
where $\phi^{k}_{\boldsymbol{\theta}}(\boldsymbol{x}) \in \mathcal{NN}(W^k, L^k, B_{\boldsymbol{\theta}}^k)$ and
$$
\max_{k=1, \ldots, \mathfrak{m}} W^k \leq W \, , \quad \max_{k=1, \ldots, \mathfrak{m}} L^k \leq L \, , \quad \max_{k=1, \ldots, \mathfrak{m}} B_{\boldsymbol{\theta}}^k \leq B_{\boldsymbol{\theta}} \, , \quad \sum_{k=1}^{\mathfrak{m}} |c_k| \leq M\, .
$$

\begin{figure}
\centering
\includegraphics[width=0.9 \linewidth]{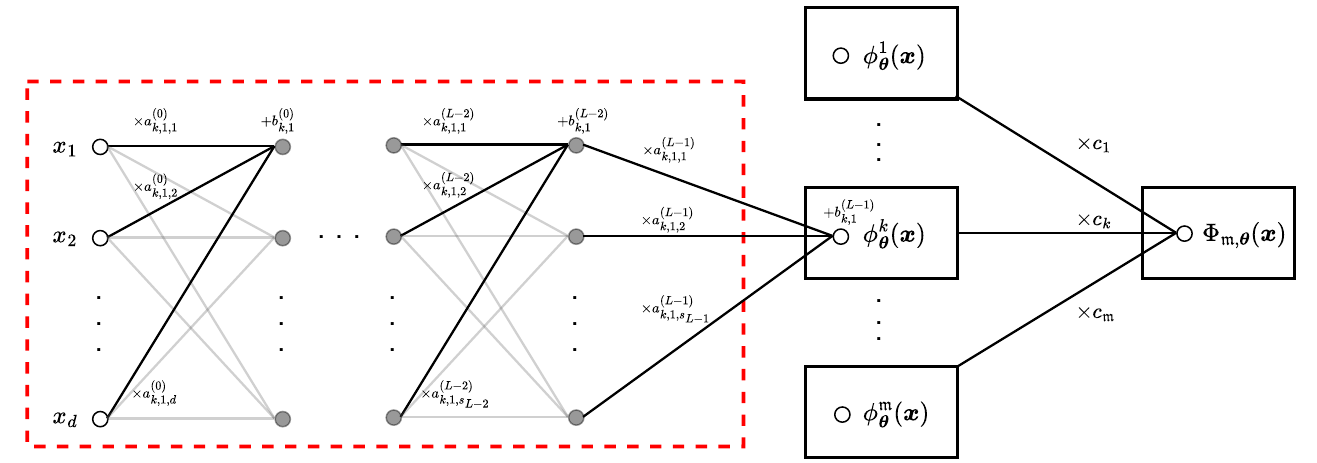}
\caption{This figure illustrates the structure of the Parallel Neural Network. The structure within the red box represents the sub-neural networks, which are fully connected networks, where the dark-colored nodes signify activation functions.}
\end{figure}

Define $\boldsymbol{\theta}_{\rm in}^{\mathfrak{m}} := (\boldsymbol{\theta}_1, \ldots, \boldsymbol{\theta}_{\mathfrak{m}})$ where $\boldsymbol{\theta}_{k} =  (a_{\scriptscriptstyle {k,1,1}}^{\scriptscriptstyle (0)}, \ldots, a_{\scriptscriptstyle k, N_L, N_{L-1}}^{\scriptscriptstyle (L-1)}, b_{\scriptscriptstyle k,1}^{\scriptscriptstyle (0)}, \ldots, b_{\scriptscriptstyle k, N_L}^{\scriptscriptstyle (L-1)})$. Define $\boldsymbol{\theta}_{\rm out}^{\mathfrak{m}} := (c_1, \ldots, c_{\mathfrak{m}})$. Define $\Theta^{\mathfrak{m}}$ as the set of all weight vectors $\boldsymbol{\theta}_{\rm total}^{\mathfrak{m}} := (\boldsymbol{\theta}_{\rm in}^{\mathfrak{m}}, \boldsymbol{\theta}_{\rm out}^{\mathfrak{m}})$ that parameterize $\Phi_{\mathfrak{m}, \boldsymbol{\theta}}$. Where it does not cause ambiguity, the symbol $\mathcal{PNN}$ will be used both as an abbreviation for some specific $\mathcal{PNN}(\mathfrak{m}, M, \{W, L, B_{\boldsymbol{\theta}}\})$, and to refer to a general parallel neural network class composed of multiple sub-networks. 
The notation specific to this paper is shown in the following table:
\begin{table}[ht]
\caption{Notation Specific to This Paper}
\centering
\def\temptablewidth{0.9\textwidth}
{\rule{\temptablewidth}{1pt}}
\begin{tabular*}{\temptablewidth}{@{\extracolsep{\fill}}cccc}
\hline
&$\mathfrak{m}$& the number of the sub-networks\\
\hline
&$\boldsymbol{\theta}_{k}$&$ (a_{{\scriptscriptstyle k,1,1}}^{\scriptscriptstyle {(0)}}, \ldots, a_{\scriptscriptstyle k, N_{\mkern- 2mu L}, N_{\mkern- 2mu L-1}}^{\scriptscriptstyle (L-1)}, b_{\scriptscriptstyle k,1}^{\scriptscriptstyle (0)}, \ldots, b_{\scriptscriptstyle k, N_{\mkern- 2mu L}}^{\scriptscriptstyle (L-1)})$\\
\hline
&$\boldsymbol{\theta}_{\rm in}^{\mathfrak{m}}$&$(\boldsymbol{\theta}_1, \ldots,\boldsymbol{\theta}_{\mathfrak{m}})$\\
\hline
&$\boldsymbol{\theta}_{\rm out}^{\mathfrak{m}}$&$(c_1, \ldots, c_{\mathfrak{m}})$\\
\hline
&$ \boldsymbol{\theta}_{\rm total}^{\mathfrak{m}}$ &$(\boldsymbol{\theta}_{\rm in}^{\mathfrak{m}}, \boldsymbol{\theta}_{\rm out}^{\mathfrak{m}})$\\ 
\hline
&$M$&$\|\boldsymbol{\theta}_{\rm out}^{\mathfrak{m}}\|_{1}\leq M$\\
\hline
\end{tabular*}
{\rule{\temptablewidth}{1pt}}
\end{table}

\subsection{Deep Ritz method} \label{drm}
Now, we recall the deep Ritz method (DRM) proposed in \cite{Weinan2017The}. 
Let $[0,1]^d$ be the unit hypercube on $\mathbb{R}^d$, $\Omega\subset [0,1]^d$ be a bounded open set and $\partial\Omega$ be the boundary of $\Omega$. Consider the elliptic equation on $\Omega$ equipped with Neumann boundary condition: 
\begin{equation}
\label{eq4}
    -\Delta u +\omega u = h \ \ on \ \,  \Omega \, , \qquad \frac{\partial u}{\partial \boldsymbol{n}} = g \ \ on \ \,  \partial\Omega.
\end{equation}
With the following assumptions on the known terms: 
$$
\partial\Omega \in C^{2+\alpha} \, , \quad h\in L^{2}(\Omega)\, , \quad g\in H^{1/2}(\partial\Omega)\, , \quad \omega(x)\in C^{\alpha}(\bar{\Omega}) \, , \quad \omega(x)\geq c_0 >0\, ,
$$
where $0<\alpha<1$,
equation (\ref{eq4}) has a unique weak solution $u_{0} \in H^{2}(\Omega)$ \citep{agmon1959estimates}. Let $B_0 = \max\{\|h\|_{ L^{\infty}(\Omega)}, \|g\|_{L^{\infty}(\partial\Omega)}, \|\omega\|_{ L^{\infty}(\Omega)}\}$ and  define the energy functional $\mathcal{L}$ as follows:
\begin{equation}
\label{eq10}
    \mathcal{L}(u)=\int_{\Omega} \bigg(\frac{1}{2}\|\nabla u\|_2^2+ \frac{1}{2} \omega|u|^2 - hu \bigg) \mathrm{d}x - \int_{\partial\Omega} (g Tu) \, \mathrm{d}s \, ,
\end{equation}
where  $T$ is the trace operator. Proposition \ref{prop4.1} demonstrates that minimizing $\mathcal{L}(u)$ in (\ref{eq10}) is equivalent to reducing the distance between $u$ and $u_0$ in the $H^1$ norm.

\begin{prop}
\label{prop4.1}
For any $u \in H^1(\Omega)$, it holds that
$$
\frac{c_0 \wedge 1}{2}\|u-u_{0}\|_{H^1(\Omega)}^2 \leq \mathcal{L}(u)-\mathcal{L}(u_{0}) \leq \frac{B_0 \vee 1}{2}\|u-u_{0}\|_{H^1(\Omega)}^2\, .
$$
\end{prop}
\begin{proof}
For any $u \in H^1(\Omega)$, set $v=u-u_{0}$, then
\begin{align*}
\mathcal{L} & (u)=\mathcal{L}(u_0+v) \\
= & \int_{\Omega} \Big(\frac{1}{2}\| \nabla(u_0+v)\|_2^2+\frac{1}{2}\omega|u_0+v|^2-h(u_0+v)\Big) \m \mathrm{d}x  - \int_{\partial\Omega}  g(T u_0+T v)\m\mathrm{d}s\\
= & \int_{\Omega}\Big(\frac{1}{2}\|\nabla u_0\|_2^2 +\frac{1}{2}\omega |u_0|^2  - u_0\m h \Big) \mathrm{d}x - \int_{\partial \Omega} g \mm T u_0 \, \mathrm{d}s + \int_{\Omega} \Big(\frac{1}{2}\|\nabla v\|_2^2 + \frac{1}{2} \omega |v|^2 \Big) \mathrm{d}x \\
& +\Big[\int_{\Omega}\nabla u_0 \nabla v \, \mathrm{d}x + \int_{\Omega} u_0\m v \, \mathrm{d}x - \int_{\Omega} h\m v \, \mathrm{d}x - \int_{\partial \Omega} g \mm T v \,\mathrm{d}s\Big] \\
= & \mathcal{L}(u_0)+ \int_{\Omega} \Big(\frac{1}{2}\|\nabla v\|_2^2 + \frac{1}{2} \omega |v|^2 \Big) \mathrm{d}x\, ,
\end{align*}
where the last equality is due to the fact that $u_{0}$ is the weak solution of equation (\ref{eq4}). Hence
\begin{align*}
\frac{c_0 \wedge 1}{2}\|v\|_{H^1(\Omega)}^2 \leq \mathcal{L}(u)-\mathcal{L}(u_{0}) & =\int_{\Omega} \Big(\frac{1}{2}\|\nabla v\|_2^2 + \frac{1}{2} \omega |v|^2 \Big) \mathrm{d}x  \leq \frac{\|\omega\|_{L^{\infty}(\Omega)} \vee 1}{2}\|v\|_{H^1(\Omega)}^2\, ,
\end{align*}
that is
$$
\frac{c_0 \wedge 1}{2}\|u-u_{0}\|_{H^1(\Omega)}^2 \leq \mathcal{L}(u)-\mathcal{L}(u_{0}) \leq \frac{B_0 \vee 1}{2}\|u-u_{0}\|_{H^1(\Omega)}^2\, .
$$
\end{proof}

To facilitate the implementation of deep learning algorithms, we use Monte Carlo method to discretize the energy functional $\mathcal{L}$. First, (\ref{eq10}) is rewritten as 
\begin{equation}
\label{eq11}
    \mathcal{L}(u)=|\Omega|\underset{X\sim U(\Omega)}{\mathbb{E}}\bigg[\frac{\| \nabla u(X)\|_2^2}{2}+ \frac{\omega(X)u^2(X)}{2} -u(X)h(X)\bigg]-|\partial\Omega|\underset{Y\sim U(\partial\Omega)}{\mathbb{E}}[Tu(Y)g(Y)] \, ,
\end{equation}
where $U(\Omega)$, $U(\partial\Omega)$ are the uniform distribution on $\Omega$ and $\partial\Omega$. Based on (\ref{eq11}), we introduce the discrete version $\widehat{\mathcal{L}}(u)$:
\begin{equation}
\label{eq12.}
    \widehat{\mathcal{L}}(u)=\frac{|\Omega|}{N_{\rm in}}\sum_{p=1}^{N_{\rm in}}\bigg[\frac{\| \nabla u(X_p)\|_2^2}{2} +\frac{\omega(X_p)u^2(X_p)}{2}- u(X_p)h(X_p)\bigg]-\frac{|\partial\Omega|}{N_b}\sum_{p=1}^{N_b}[u(Y_p)g(Y_p)] \, ,
\end{equation}
where $\{X_p\}_{p=1}^{N_{\rm in}} \sim_{\text{i.i.d.}} U(\Omega) \m$, $ \{Y_p\}_{p=1}^{N_b} \sim_{\text{i.i.d.}} U(\partial\Omega)$. Then, we select a deep neural network class $\mathcal{F}_{\boldsymbol{\theta}}$, within which we will minimize $\widehat{\mathcal{L}}(u_{\boldsymbol{\theta}})$ for $u_{\boldsymbol{\theta}} \in$ $\mathcal{F}_{\boldsymbol{\theta}}$. 

In this paper, our choice is the $\mathcal{PNN}$ parallel neural network $u_{\mathfrak{m}, \boldsymbol{\theta}}=\sum_{k=1}^{\mathfrak{m}}c_k \phi_{\boldsymbol{\theta}}^k$ introduced in Section \ref{subsec:topo}, which comprises $\mathfrak{m}$ sub-networks. For convenience, we set the width of all sub-networks in $u_{\mathfrak{m}, \boldsymbol{\theta}}$ to 
$W$, and the depth of all sub-networks in 
$u_{\mathfrak{m}, \boldsymbol{\theta}}$ to 
$L$. Now, (\ref{eq12.}) turns into:
\begin{align}
\label{eq12}
    \widehat{\mathcal{L}}(u_{\mathfrak{m}, \boldsymbol{\theta}})= \, & \frac{|\Omega|}{N_{\rm in}}\sum_{p=1}^{N_{\rm in}}\bigg[\frac{\| \nabla u_{\mathfrak{m}, \boldsymbol{\theta}}(X_p)\|_2^2}{2} +\frac{\omega(X_p)u_{\mathfrak{m}, \boldsymbol{\theta}}^2(X_p)}{2}- u_{\mathfrak{m}, \boldsymbol{\theta}}(X_p)h(X_p)\bigg] \notag \\
    &-\frac{|\partial\Omega|}{N_b}\sum_{p=1}^{N_b}[u_{\mathfrak{m}, \boldsymbol{\theta}}(Y_p)g(Y_p)] \, .
\end{align}

\subsection{Projected gradient descent} \label{pgd}

Specifically, we use the projected gradient descent (PGD) algorithm to minimize $\widehat{\mathcal{L}}(u_{\mathfrak{m}, \boldsymbol{\theta}})$ in (\ref{eq12}), which is an iterative optimization method suitable for constrained optimization problems.

As shown in Section \ref{subsec:topo}, the weights of $u_{\mathfrak{m}, \boldsymbol{\theta}}$ are $\boldsymbol{\theta}^{\mathfrak{m}}_{ \rm total}=(\boldsymbol{\theta}^{\mathfrak{m}}_{ \rm in}, \boldsymbol{\theta}^{\mathfrak{m}}_{ \rm out})$. Since the Monte Carlo samples $\{X_p\}_{p=1}^{N_{\rm in}}$, $\{Y_p\}_{p=1}^{N_{b}}$ are fixed during the optimization process, $\widehat{\mathcal{L}}(u_{\mathfrak{m}, \boldsymbol{\theta}})$ becomes a function solely dependent on $\boldsymbol{\theta}_{\rm total}^{\mathfrak{m}}$. Hence, we denote it as $\widehat{F}(\boldsymbol{\theta}_{\rm total}^{\mathfrak{m}})=\widehat{F}(\boldsymbol{\theta}_{\rm in}^{\mathfrak{m}}, \boldsymbol{\theta}_{\rm out}^{\mathfrak{m}})$. Now, we formally introduce the PGD algorithm used in this paper, which consists of the following three steps:
\vskip 3mm
\noindent \textbf{Initialization.} We start with an initial guess \( (\boldsymbol{\theta}_{\rm total}^{\mathfrak{m}})^{\scriptscriptstyle [0]} = (\boldsymbol{\theta}_{\rm in}^{\mathfrak{m}}, \boldsymbol{\theta}_{\rm out}^{\mathfrak{m}})^{\scriptscriptstyle [0]} \) as follows:
\begin{itemize}
    \item[(i)] For the linear coefficients $\boldsymbol{\theta}_{\rm out}^{\mathfrak{m}}$, set
    \begin{align} \label{eq:out0}     
(\boldsymbol{\theta}_{\rm out}^{\mathfrak{m}})^{\scriptscriptstyle [0]} = \boldsymbol{0}\, , \quad  i.e. \quad (c_{k})^{\scriptscriptstyle [0]}=0 \quad (k=1, \ldots, \mathfrak{m})\, .
    \end{align}
\item[(ii)] For the sub-network parameters
$\boldsymbol{\theta}_{\rm in}^{\mathfrak{m}}$, initialize each element in $(\boldsymbol{\theta}_{\rm in}^{ \mathfrak{m}})^{\scriptscriptstyle [0]}$ to follow the same uniform distribution $\mathcal{U}[-B, B]$ independently, i.e.,    
\begin{align} \label{eq:init}
\big(a_{k, i, j}^{(\ell)}\big)^{\scriptscriptstyle [0]} \sim_{\text{i.i.d.}} \mathcal{U}[-B, B]\, , \quad \big(b_{k, i}^{(\ell)}\big)^{\scriptscriptstyle [0]} \sim_{\text{i.i.d.}} \mathcal{U}[-B, B] \, ,
\end{align}
where $k = 1,\ldots,\mathfrak{m}$, $\ell = 0,\ldots,{L-1}$. When $\ell=0$,  $i=1,\ldots W$, $j=1,\ldots,d$; When $\ell=1,\ldots,{L-2}$, $i,j=1,\ldots,{W}$; When $\ell=L-1$, $i=1$, $j=1,\ldots,W$.
\end{itemize}

\noindent  \textbf{Constraint set.} Then, we choose $\eta,\zeta>0$, and determine the constraint set  as follows:
\begin{itemize}
    \item[(i)] Let $A_{\eta}(\omega)$ be the (random) set of all weight vectors $\boldsymbol{\theta}^{\mathfrak{m}}_{\rm in}$ which satisfy
\begin{align} \label{eq:eta}
\big\|\boldsymbol{\theta}^{\mathfrak{m}}_{\rm in}-(\boldsymbol{\theta}^{\mathfrak{m}}_{\rm in})^{\scriptscriptstyle [0]}\big\|_2 \leq \eta \, .    
\end{align}

\item[(ii)] Let $B_{\mkern- 1mu \zeta}$ be the set of all weight vectors $\boldsymbol{\theta}^{\mathfrak{m}}_{\rm out}$ which satisfy
\begin{align} \label{eq:M}
\big\|\boldsymbol{\theta}^{\mathfrak{m}}_{\rm out}\big\|_1=\sum_{k=1}^{\mathfrak{m}}|c_k| \leq \zeta \, .
\end{align}
\end{itemize}

\noindent  \textbf{Iterative Update.} Finally, let $T \in \mathbb{N}^{+}$, $\lambda>0$. For each iteration $t = 0, \ldots, T-1$, do:
\begin{itemize}
       \item[(i)] Compute the gradient of the objective function at the current point: 
       \[
       \boldsymbol{g}^{[t]} = \nabla_{\boldsymbol{\theta}^{\mathfrak{m}}_{ \rm total}} \widehat{F}\big((\boldsymbol{\theta}^{\mathfrak{m}}_{\rm in})^{\scriptscriptstyle[t]}, (\boldsymbol{\theta}^{\mathfrak{m}}_{\rm out})^{\scriptscriptstyle[t]}\big) \, .
       \]
       \item[(ii)] Update the weight vector by  first performing a gradient descent step with a step size of $\lambda$ and then projecting the result onto the feasible set:
       \begin{align} \label{eq:pgd}
\big(\boldsymbol{\theta}^{\mathfrak{m}}_{\rm in}, \boldsymbol{\theta}^{\mathfrak{m}}_{\rm out}\big)^{\scriptscriptstyle[t+1]}=\operatorname{Proj}_{A_{\eta} \times B_{\mkern- 1mu \zeta}} \! \big\{(\boldsymbol{\theta}^{\mathfrak{m}}_{\rm in}, \boldsymbol{\theta}^{\mathfrak{m}}_{\rm out})^{\scriptscriptstyle[t]}-\lambda \, \boldsymbol{g}^{[t]} \, \big\} \, ,
\end{align}
       where \( \operatorname{Proj}_{\mathcal{C}} \) denotes the projection operator onto the set \( \mathcal{C} \) .
   \end{itemize}
\begin{remark}
    The projection onto the $\ell_2$ ball $A_\eta$ can be expressed in closed form, while the projection onto the $\ell_1$ ball $B_\zeta$ can be implemented exactly with a linear dependence on the dimension \cite{duchi2008efficient}.
\end{remark}
 
In the following, we will use $\mathcal{A}$ to represent the PGD algorithm, and use $u_{\mathcal{A}}$ to denote the output of $\mathcal{A}$ which serves as the final solution. It is evident that $u_{\mathcal{A}}$ is exactly $u_{\mathfrak{m}, \boldsymbol{\theta}}$ parameterized with $(\boldsymbol{\theta}^{\mathfrak{m}}_{\rm in}, \boldsymbol{\theta}^{\mathfrak{m}}_{\rm out})^{\scriptscriptstyle[T]}$.

\subsection{Main result} \label{main}
After the aforementioned preparation, we now formally state the main theorem of this work, which provids a comprehensive end-to-end error analysis for solving elliptic equations via the deep Ritz method under the over-parameterized setting.
\begin{thm}\label{mainth}
When applying the deep Ritz method to solve (\ref{eq4}), we utilize the $\mathcal{PNN}$ architecture, which comprises $\mathfrak{m}$ sub-networks with  width  $W$ and depth $L$. We initialize the network parameters through (\ref{eq:out0}) and (\ref{eq:init}), setting the linear coefficients connecting the sub-networks to $\boldsymbol{0}$, and making each sub-network weight following the uniform distribution $\mathcal{U}[-B, B]$ independently. Let $\eta$ and $\zeta$ be the projection radius described in (\ref{eq:eta}) and (\ref{eq:M}), respectively. 
Let $N_{\rm in} = N_{b} = N_{s}$ be the sample size of the Monte Carlo method in (\ref{eq12}). Now, let $u_{\mathcal{A}}$ be the output of the PGD algorithm in (\ref{eq:pgd}) with iteration steps $T$ and step size $\lambda$. For any $0< \epsilon \ll 1$, set
\begin{alignat*}{3}
{\mathfrak{m}}& = \lceil C \cdot \epsilon^{-C_1(\mu, d, \beta, \beta_0, n)} \rceil \,,  &\qquad  W &= 2^{\lceil \log_2 (d+1) \rceil+1}\, , &\qquad  L&= \lceil \log_2 (d+1) \rceil+2 \,   , \\
\qquad B & = C \cdot \epsilon^{-2-\frac{2d}{n-\mu-1}} \, , &\eta &= \epsilon^{-\beta}\, , &\qquad   \zeta & = C \cdot \epsilon^{-\frac{3d}{2(n-\mu-1)}}\,,\\
T& = C \cdot \epsilon^{-C_2(\mu, d, \beta, \beta_0, n)}\, , &\qquad  \lambda &= C \cdot \epsilon^{C_2(\mu, d, \beta, \beta_0, n)} \, , &\qquad  N_{s} &= \lceil C \cdot \epsilon^{-C_3(\mu, d, \beta_0, n)} \rceil \, .
\end{alignat*}
Suppose that $u_0\in H^{n}(\Omega)$ for $n \geq 2$ is the target solution of the elliptic partial differential equation (\ref{eq4}). Then, with probability at least $1- 2\cdot \epsilon^{C_3(\mu, d, \beta_0, n)}$, the total error
$$\|u_{\mathcal{A}}-u_{0}\|_{H^1(\Omega)}^2 \leq C \m \epsilon \log^{1/2} (C\m\epsilon^{-1}) = \tilde{\mathcal{O}}(\epsilon)\, ,$$
 where
\begin{align*}
& \beta_0 = \max\{\beta, \,  2+2 \m d \m (n-\mu-1)^{-1}\}\, , \\
& C_1(\mu, d, \beta, \beta_0, n) = \frac{C_0\cdot d^3\log^2(d+1)}{n-\mu-1} + 11\beta_0\log(d+1)+33\beta_0 + 3\beta \, ,\\
& C_2(\mu, d, \beta, \beta_0, n) = \frac{C_0^{\prime}\cdot d^3\log^2(d+1)}{n-\mu-1} + 15\beta_0\log(d+1)+45\beta_0 + 3\beta\, ,\\
& C_3(\mu, d, \beta_0, n) = 4\beta_0\log(d+1)+\frac{6d}{n-\mu-1}+12\beta_0 + 2\, .
\end{align*}
Meanwhile, $C$ denotes a universal constant which is defined place by place and only depends on $\Omega, W, L, d, n$; $B_0$, $C_0$ and $C_0^{\prime}$ are positive constants; $0<\mu<1$ is an arbitrarily small positive number.
\end{thm}
\begin{proof}
See Appendix \ref{proof of main result}.
\end{proof}

\begin{remark}
The assumption $u_0\in H^{n}(\Omega)$ for $n > 2$ can be achieved by increasing the regularity of the coefficient $\omega$ and the right-hand side functions $h, g$ in equation (\ref{eq4}). For instance, with $\partial \Omega \in C^{\m n}$, such assumption would be realized if we get $h \in H^{n-2}(\Omega)$, $g \in H^{n-3/2}(\partial\Omega)$ and $ \omega \in C^{\m n-2}(\bar{\Omega})$. See \cite{agmon1959estimates} for proof.  
\end{remark}
\begin{remark}
It can be observed that our analytical results do not necessitate the neural network parameters to have close initial values during the iterative algorithm, which is a restrictive technique requirement commonly employed in previous analyses of optimization error \citep{jacot2018neural,allen2019convergence,du2019gradient,zou2019improved,liu2022loss,chizat2019lazy,nguyen2021proof,lu2020mean,mahankali2024beyond}. However, when the projection radius of the sub-network weights, $\eta$, becomes too large, exceeding their initialization range ($\beta_0=\beta$), the algorithm will require more Monte Carlo sampling points, a higher level of over-parameterization, a smaller iteration step size, and more iterations to achieve the same precision. This theoretical finding emphasizes the significance of appropriately applying gradient clipping during the optimization process.
\end{remark}

\section{Proofs}
\label{proof}
To prove our main result Theorem \ref{mainth}, we need to show that the algorithm output \( u_{\mathcal{A}} \) can approximate the target solution \( u_0 \) of the elliptic partial differential equation (\ref{eq4}) with a specified precision \(\epsilon\). We achieve this objective through the following five steps. 

\textbf{Step 1: Decomposing the total error.}
   The total error between \( u_{\mathcal{A}} \) and \( u_0 \) can be decomposed into three main components: approximation error,  statistical error and a novel and tighter optimization error,   details can be found in Section \ref{errde}.

\textbf{Step 2: Constructing a parallel neural network with explicit weight bound and weight norm  to  approximate  in Sobolev space.}  Building on methods from \cite{yarotsky2017error}, \cite{jiao2023approximation}, and \cite{guhring2021approximation}, we derive an approximation error bound for neural networks in the
Sobolev space.  This bound is achieved by explicitly constructing  \( \tanh \)-activated  neural networks that approximate local Taylor polynomials.
 Notably, the constructed network has a parallel architecture, meaning the final neural network is a linear combination of many structurally similar fully connected sub-networks,  see  Section \ref{subsec:topo} for detail. 
 By construction, we explicitly control the weight bound and weight norm of the deep network. This is critical for the generalization error analysis in the over-parameterized setting, as demonstrated in Step 4. Furthermore, we utilize this constructed over-parameterized network to define and analyze a novel and tighter optimization error term in Step 3.

\textbf{Step 3: Utilizing the property of over-parameterization to analyze the new optimization error.}
We choose the projection gradient descent (PGD) algorithm for optimization.
By definition, the new optimization error is further bounded by sum of the initialization error and iteration error, as shown in  (\ref{eq:opt decomp}). 
With the help of  over-parameterization, the initialization value  of the neural network will likely capture enough information about the best approximation network function constructed in Step 2, and thus we can control the initialization error. Additionally, the iteration error of the PGD algorithm will be effectively controlled choosing the  total iteration steps large enough and 
step size smaller enough. 
The key feature of our analysis of the new optimization error term is that the radius of the projection regions can diverge at a certain rate, thereby avoiding the stagnation of training dilemma encountered in previous optimization error analyses.


\textbf{Step 4: Obtaining an upper bound on the statistical error for over-parameterized deep neural network class.}
The PGD optimization process in Step 3 will place the  output $u_{\mathcal{A}}$ within an over-parameterized neural network class. In this case, we cannot directly utilize tools from empirical process theory \cite{VanJhon, van2000empirical, gine2021mathematical} to bound the statistical error, as it would yield an upper bound that becomes uncontrollably large in the over-parameterized setting.
Thanks to the explicit  upper bounds on the  weight constructed in  constraints in Step 1   and the projection in Step 3,  
 we derive size independent statistical error by bounding the Rademacher complexity of the parallel structured neural network class directly  through definition and induction. Thus, By setting the Monte Carlo sample sizes for the boundary and interior points  properly 
 we ensure the statistical error stays within acceptable bounds, see Section \ref{sta} for detail.

\textbf{Step 5: Synthesizing the error analysis from each component.} By synthesizing the analysis of approximation, optimization, and generalization error from Step 2 to Step 4, we could control the total error between \( u_{\mathcal{A}} \) and \( u_0 \) 
within the desired precision under appropriate parameter settings, thus proving our main result Theorem \ref{mainth}.

\subsection{New error decomposition}
\label{errde}

To conduct an end-to-end error analysis between $u_{\mathcal{A}}$ and $u_0$, we propose the following error decomposition theorem. A crucial factor for making such analysis feasible is the introduction of a novel `optimization error', which is denoted as $$\mathcal{E}_{\text {opt }}^{-}:=\widehat{\mathcal{L}}(u_{\mathcal{A}})-\widehat{\mathcal{L}}(\bar{u})\, ,$$ where $\bar{u}$ is the best approximation element in some parallel neural network class $\mathcal{PNN}^{\m \prime}$, defined by (without loss of generality) 
\begin{equation}\label{bapp}
\bar{u}\in\argmin_{u\in \mathcal{PNN}^{\m \prime}}\|u-u_0\|^{2}_{H^1(\Omega)}\, .
\end{equation}
\begin{thm}
\label{thm_err_decom}
Let $u_{\mathcal{A}}$ be the PGD algorithm output when using DRM to solve (\ref{eq4}) and $\bar{u} \in \mathcal{PNN}^{\m \prime}$ defined in (\ref{bapp}), the $H^1$ distance between $u_{\mathcal{A}}$ and the true solution $u_0$ can be decomposed into   
\begin{equation*}
\begin{split}
   ||u_{\mathcal{A}}-u_{0}||^2_{H^1{(\Omega)}} \leq \frac{2}{c_0 \wedge 1}  \bigg\{\underbrace{\Big[\frac{B_0 \vee 1}{2}\|\bar{u}-u_{0}\|_{H^1(\Omega)}^2\Big]}_{\mathcal{E}_{app}} + \underbrace{\Big[\widehat{\mathcal{L}}(u_{\mathcal{A}})-\widehat{\mathcal{L}}(\bar{u})\Big] }_{\mathcal{E}^{-}_{opt}} + \underbrace{2\sup_{u\in\mathcal{PNN}}|\mathcal{L}(u)-\widehat{\mathcal{L}}(u)| }_{\mathcal{E}_{sta}}\bigg\} \, .
\end{split}
\end{equation*}
\end{thm}
\begin{proof}
By Proposition \ref{prop4.1} we have,
\begin{align*}
   &\|u_{\mathcal{A}}-u_{0}\|^2_{H^1{(\Omega)}} \\
   & \quad \leq \frac{2}{c_0 \wedge 1}  \bigg\{ \Big[{\mathcal{L}}(u_{\mathcal{A}})-\widehat{\mathcal{L}}(u_{\mathcal{A}})\Big]+  \Big[\widehat{\mathcal{L}}(u_{\mathcal{A}})-\widehat{\mathcal{L}}(\bar{u}) \Big] + \Big[\widehat{\mathcal{L}}(\bar{u})- \mathcal{L}(\bar{u}) \Big] + \Big[\mathcal{L}(\bar{u}) - \mathcal{L}(u_{0}) \Big] \bigg\}\\
   & \quad \leq \frac{2}{c_0 \wedge 1}  \bigg\{ \underbrace{\Big[\mathcal{L}(\bar{u})-\mathcal{L}(u_{0})\Big]}_{\mathcal{E}_{app}} +\underbrace{\Big[\widehat{\mathcal{L}}(u_{\mathcal{A}})-\widehat{\mathcal{L}}(\bar{u})\Big] }_{\mathcal{E}^{-}_{opt}} + \underbrace{\Big[2\sup_{u\in\mathcal{PNN}}|\mathcal{L}(u)-\widehat{\mathcal{L}}(u)| \Big]}_{\mathcal{E}_{sta}} \bigg\}\\
   &\quad \leq \frac{2}{c_0 \wedge 1}  \bigg\{\underbrace{\Big[\frac{B_0 \vee 1}{2}\|\bar{u}-u_{0}\|_{H^1(\Omega)}^2\Big]}_{\mathcal{E}_{app}} + \underbrace{\Big[\widehat{\mathcal{L}}(u_{\mathcal{A}})-\widehat{\mathcal{L}}(\bar{u})\Big] }_{\mathcal{E}^{-}_{opt}} + \underbrace{2\sup_{u\in\mathcal{PNN}}|\mathcal{L}(u)-\widehat{\mathcal{L}}(u)| }_{\mathcal{E}_{sta}}\bigg\}\, .
\end{align*}   
\end{proof}

\begin{remark}
    In traditional error decomposition, the optimization error is defined as $\mathcal{E}_{opt}=\widehat{\mathcal{L}}(u_{\mathcal{A}})-\widehat{\mathcal{L}}(\hat{u})$, where $\hat{u}$ denotes the ERM estimator of (\ref{eq12}), i.e.,
    \begin{align*}
        \hat{u} = \argmin_{u_{\mathfrak{m}, \boldsymbol{\theta}} \in \mathcal{PNN}} \widehat{\mathcal{L}}(u_{\mathfrak{m}, \boldsymbol{\theta}}) \, .
    \end{align*}
    Since we have
      \begin{align*}
          \mathcal{E}_{opt}^{-}=\widehat{\mathcal{L}}(u_{\mathcal{A}})-\widehat{\mathcal{L}}(\hat{u})+\widehat{\mathcal{L}}(\hat{u})-\widehat{\mathcal{L}}(\bar{u}) \leq \widehat{\mathcal{L}}(u_{\mathcal{A}})-\widehat{\mathcal{L}}(\hat{u}) + 0 =  \mathcal{E}_{opt}\, ,
      \end{align*}
    the newly defined $\mathcal{E}_{opt}^{-}$ is clearly tighter than $\mathcal{E}_{opt}$. Additionally, due to the highly non-convex training objective of deep neural networks, it is challenging to obtain detailed information about $\hat{u}$, making analysis of $\mathcal{E}_{opt}$ difficult. In contrast, the best approximation element  $\bar{u}$ is explicitly constructed  according to the target solution, as shown in  Theorem \ref{thm_err_decom}, so its information can be fully grasped, greatly facilitating the analysis of $\mathcal{E}_{opt}^{-}$. In Section \ref{opt}, we will prove that when the over-parameterization level is sufficiently high, the initialization parameters of neural networks will contain sufficient information about $\bar{u}$ with high probability, and meanwhile, the iteration error of the PGD algorithm can be well controlled. Combining these two points, we can control $\mathcal{E}_{opt}^{-}$ with arbitrary precision, thus achieving a complete end-to-end error analysis between $u_{\mathcal{A}}$ and $u_0$.
\end{remark}

\subsection{Approximation error}
\label{app}
In this section, we provide an upper bound for the approximation error which  characterizes   the ability of the specifically constructed neural network $\bar{u}$ to approximate the true solution $u_0$. The detailed proof can be found in the Appendix \ref{proof of app}. Recall that 
$$
\mathcal{E}_{app} := \|\bar{u}-u_{0}\|_{H^1(\Omega)}^2\, ,
$$
where $\bar{u}$ is  defined in (\ref{bapp}). 
 Following the results in \cite{guhring2021approximation}, we show that given arbitrary accuracy $\epsilon>0$, any $f \in \mathcal{F}_{n, d, p}$
can be $\epsilon$-approximated in weaker Sobolev norms $W^{k,p}$ (with $n,k,p\in \mathbb{N}, n\geq k+1$ and $1\leq p \leq \infty$) within some $tanh$-based parallel neural network class $\mathcal{PNN}^{\m \prime}$. Here, $\mathcal{F}_{n, d, p}$ is defined as
$$\mathcal{F}_{n, d, p}:=\big\{f \in W^{n, p}([0,1]^{d}):\|f\|_{W^{n, p}([0,1]^{d})} \leq 1\big\} \, .$$

\begin{thm}
\label{thm4.1}
Let $n, k, d, \bar{\mathfrak{m}} \in \mathbb{N}$, $n\geq k+1$, $1\leq p\leq \infty$ and $|\boldsymbol{\alpha}|_1 \leq n-1$, $C$ be a positive number and $C(n,d)$ be a polynomial that depends on $n$ and $d$. Let $f \in \mathcal{F}_{n, d, p}$. For some sufficiently small $\epsilon^{*}>0$ and any $0<\epsilon < \epsilon^{*}$, there exists a neural network $\Phi_{\bar{\mathfrak{m}}, \bar{\boldsymbol{\theta}}} \in \mathcal{PNN}(\bar{\mathfrak{m}}, \bar{M}, \{\bar{W}, \bar{L}, B_{\bar{\boldsymbol{\theta}}}\})$ with 
\begin{align*}
&\bar{\mathfrak{m}}=C_1(n, d, k)\epsilon^{-\frac{d}{n-k-\mu k}} \, , \quad \bar{M} = C_2(n, d, k)\epsilon^{-\frac{d(p+1)}{(n-k-\mu k)p}}\, , \quad \bar{W} = 2^{\lceil \log_2 (d+|\boldsymbol{\alpha}|_1) \rceil+1} \, , \\
&~~~~~~~~~~~~~~~~ \bar{L}= \lceil \log_2 (d+|\boldsymbol{\alpha}|_1) \rceil+2 \, , \quad B_{\bar{\boldsymbol{\theta}}} = C_3(n,d,k)\epsilon^{-2-\frac{2d}{n-k-\mu k}} \, ,
\end{align*}
such that
$$
\big\|f-\Phi_{\bar{\mathfrak{m}}, \bar{\boldsymbol{\theta}}}\big\|_{W^{k, p}([0,1]^{d})} \leq \epsilon\, ,
$$
where $0<\mu<1$ is an arbitrarily small positive number.
\end{thm}
\begin{proof}
    See Appendix \ref{app4}.
\end{proof}

Since we have assumed $\Omega \subset [0, 1]^d$ without loss of generality, we need the following extension result.

\begin{prop}
\label{prop4.2}
Let $k\in \mathbb{N}^{+}$, $1\leq p < \infty$. There exists a linear operator $E$ from $W^{k,p}(\Omega)$ to $W^{k,p}_0([0, 1]^d)$ and $Eu=u$ in $\Omega$.
\end{prop}
\begin{proof}
    See Theorem 7.25 in \cite{gilbarg1977elliptic}.
\end{proof}

\begin{cor}
\label{cor4.1}
Given any $u_{0} \in \mathcal{F}_{n, d, 2}$, for some sufficiently small $\epsilon^{*}>0$ and any $0<\epsilon < \epsilon^{*}$, there exists a neural network $\bar{u}=u_{\bar{\mathfrak{m}}, \bar{\boldsymbol{\theta}}}\in \mathcal{PNN}(\bar{\mathfrak{m}}, \bar{M}, \{\bar{W}, \bar{L}, B_{\bar{\boldsymbol{\theta}}}\})$ with
\begin{align*}
&\bar{\mathfrak{m}}=C_1(n, d)\epsilon^{-\frac{d}{n-\mu-1}}\,, \quad \bar{M} = C_2(n, d)\epsilon^{-\frac{3d}{2(n-\mu-1)}}\, , \quad \bar{W} = 2^{\lceil \log_2 (d+1) \rceil+1}\, , \\
&~~~~~~~~~\bar{L}= \lceil \log_2 (d+1) \rceil+2 \, , \quad B_{\bar{\boldsymbol{\theta}}} = C_3(n, d)\epsilon^{-2-\frac{2d}{n-\mu-1}} \, ,
\end{align*}
such that
$$
\|u_{0}-u_{\bar{\mathfrak{m}}, \bar{\boldsymbol{\theta}}}\|_{H^{1}(\Omega)} \leq \epsilon\, ,
$$
where $0<\mu<1$ is an arbitrarily small positive number.
\end{cor}
\begin{proof}
    Plugging in $k = 1$ and $p=2$ into Theorem \ref{thm4.1} and using the fact that $\|u_{0} - u_{\bar{\mathfrak{m}}, \bar{\boldsymbol{\theta}}}\|_{W^{1,2}(\Omega)}\leq \|Eu^* - u_{\bar{\mathfrak{m}}, \bar{\boldsymbol{\theta}}}\|_{W^{1,2}([0,1]^d)}$,  where $E$ is the extension operator in Proposition \ref{prop4.2}, we obtain the desired result.
\end{proof}

\subsection{Optimization error}
\label{opt}
In this section, we provide a complete analysis of the optimization error $\mathcal{E}^{-}_{opt}$.
Recall that for $u_{\mathcal{A}} \in \mathcal{PNN}$ as the output of the PGD algorithm in Section \ref{pgd} and $\bar{u} \in \mathcal{PNN}^{\m \prime}$ as the best approximation element defined in (\ref{bapp}), $\mathcal{E}^{-}_{opt}$ is defined as
$$
\mathcal{E}^{-}_{opt} := \widehat{\mathcal{L}}(u_{\mathcal{A}})-\widehat{\mathcal{L}}(\bar{u})\, .
$$
Specifically, let $\bar{u} = u_{\bar{\mathfrak{m}}, \bar{\boldsymbol{\theta}}} \in \mathcal{PNN}(\bar{\mathfrak{m}}, \bar{M}, \{\bar{W}, \bar{L}, B_{\bar{\boldsymbol{\theta}}}\})$ in Corollary \ref{cor4.1}. Now, $\mathcal{E}^{-}_{opt}$ is  expressed as
\begin{align} \label{eq:opt-}
\mathcal{E}^{-}_{opt} = \widehat{\mathcal{L}}(u_{\mathcal{A}})-\widehat{\mathcal{L}}(u_{\bar{\mathfrak{m}}, \bar{\boldsymbol{\theta}}})\, .
\end{align}

As indicated in (\ref{eq:opt-}), the weights of $u_{\bar{\mathfrak{m}}, \bar{\boldsymbol{\theta}}}\in \mathcal{PNN}(\bar{\mathfrak{m}}, \bar{M}, \{\bar{W}, \bar{L}, B_{\bar{\boldsymbol{\theta}}}\})$ are treated as `target parameters' during optimization process, namely, the sub-network parameters $(\bar{\boldsymbol{\theta}}_1, \ldots, \bar{\boldsymbol{\theta}}_{\bar{\mathfrak{m}}}) $ of $u_{\bar{\mathfrak{m}}, \bar{\boldsymbol{\theta}}}$, and the linear coefficients $(\bar{c}_1, \ldots, \bar{c}_{\bar{\mathfrak{m}}})$. Driven by this, we set the sub-network width $W$ in our implemented $u_{\mathfrak{m}, \boldsymbol{\theta}} \in \mathcal{PNN}$ to $\bar{W}$, the sub-network depth $L$ to $\bar{L}$, and the uniform distribution range $B$ in (\ref{eq:init}) to $B_{\bar{\boldsymbol{\theta}}}$. For random sub-network initialization $(\boldsymbol{\theta}_{\rm in}^{\mathfrak{m}})^{\scriptscriptstyle [0]} = (\boldsymbol{\theta}_1, \ldots, \boldsymbol{\theta}_{\mathfrak{m}})^{\scriptscriptstyle [0]}$ in (\ref{eq:init}) with $B=B_{\bar{\boldsymbol{\theta}}}$, we aim to define an event which contains all the `sufficiently good' initialization with respect to the target $(\bar{\boldsymbol{\theta}}_1, \ldots, \bar{\boldsymbol{\theta}}_{\bar{\mathfrak{m}}})$. For this endeavor, we propose the following definition.

\begin{defn} \label{def:G}
    Let ${G}_{\mathfrak{m}, \bar{\mathfrak{m}}, R, \delta}$ be the event where, for each  $\bar{\boldsymbol{\theta}}_k$, $k=1,\ldots,\bar{\mathfrak{m}}$, there exists at least $R$ sub-network weight vectors $(\boldsymbol{\theta}_{i_{k, v}})^{\scriptscriptstyle [0]}$, $v=1,\ldots,R$ in the random initialization $(\boldsymbol{\theta}_{\rm in}^{\mathfrak{m}})^{\scriptscriptstyle [0]}$,  s.t. 
    $$
    \|(\boldsymbol{\theta}_{i  _{k, v}})^{\scriptscriptstyle [0]}-\bar{\boldsymbol{\theta}}_k\|_{\infty} \leq \delta \, , \quad k=1,\ldots,\bar{\mathfrak{m}} \, , \quad v=1,\ldots,R \, .
    $$
    We also require that $i_{k, v}\neq i_{k^{\prime}, v^{\prime}}$ when either $k \neq k^{\prime}$ or $v \neq v^{\prime}$.
\end{defn}
\begin{remark}
    Simply put, $G_{\mathfrak{m}, \bar{\mathfrak{m}}, R, \delta}$ ensures that for each target $\bar{\boldsymbol{\theta}}_k$, at least $R$ sub-networks have already sufficiently approximated it during parameter initialization phase.
\end{remark}


In the rest of this section, we will always let $\mathfrak{m}=\bar{\mathfrak{m}} \cdot R \cdot Q$, where $R, Q \in \mathbb{N}$. We further formalize the random indices $i_{k,v}$ involved in Definition \ref{def:G}, and introduce a series of integer-valued random variables:
\begin{align} \label{eq:skv}
    s_{k, v}(\omega)\, , \quad k=1,\ldots,\bar{\mathfrak{m}} \, , \quad v=1,\ldots,R \, .
\end{align}

First, arrange the $\mathfrak{m}$ sub-networks of $u_{\mathfrak{m}, \boldsymbol{\theta}}$ in a given order. If $\omega \in G_{\mathfrak{m}, \bar{\mathfrak{m}}, R, \delta}$, when $k=1$, let $s_{1, v} \in [\mathfrak{m}] $ be the index of the $v$-th sub-network satisfying $\|(\boldsymbol{\theta}_{\cdot})^{\scriptscriptstyle[0]} - \bar{\boldsymbol{\theta}}_1\|_{\infty} \leq \delta$; when $k>1$, let $s_{k, v} \in [\mathfrak{m}] \setminus \{s_{l, v}: l<k, v \le R\} $ be the index of the $v$-th sub-network among the rest $\mathfrak{m}-(k-1)R$ sub-networks satisfying $\|(\boldsymbol{\theta}_{\cdot})^{\scriptscriptstyle[0]} - \bar{\boldsymbol{\theta}}_k\|_{\infty} \leq \delta$. If $\omega \notin G_{\mathfrak{m}, \bar{\mathfrak{m}}, R, \delta}$, we simply set all the $s_{k, v}=(k-1)R+v$.

After these, a crucial idea involves defining a set of random `transition parameters', utilizing $s_{k, v}$ and $(\bar{c}_1,\ldots,\bar{c}_{\bar{\mathfrak{m}}})$, to bridge the parameters of $u_{\mathcal{A}}$ and those of the target $u_{\bar{\mathfrak{m}}, \bar{\boldsymbol{\theta}}}$. Specifically, we define the random weight vectors as follows
\begin{align} \label{eq:transpara}
\boldsymbol{\theta}_{\rm total}^{\mathfrak{m}, *} := (\boldsymbol{\theta}_{\rm in}^{\mathfrak{m}, *}, \boldsymbol{\theta}_{\rm out}^{\mathfrak{m}, *}) \, , \quad \text{where} \,\,\, \boldsymbol{\theta}_{\rm in}^{\mathfrak{m}, *} := (\boldsymbol{\theta}_{\rm in}^{\mathfrak{m}})^{\scriptscriptstyle [0]} \, .
\end{align}
For $\boldsymbol{\theta}_{\rm out}^{\mathfrak{m}, *} := (c_1^*, \ldots, c_{\mathfrak{m}}^*)$, when the indices of $c^{*}_i$ coincide with $s_{k, v}$, we set
\begin{align*}
c^{*}_{s_{k, v}} := \frac{\bar{c}_{k}}{R}\, , \quad   k=1, \ldots, \bar{\mathfrak{m}}\, , \quad v=1, \ldots, R \, .
\end{align*}
 Otherwise, we set
$$
c^{*}_{q} := 0 \, , \quad q \notin \{s_{k, v}: k=1, \ldots, \bar{\mathfrak{m}}\, , \  v=1, \ldots, R\} \, .
$$

When $u_{\mathfrak{m}, \boldsymbol{\theta}}$ is parameterized with $\boldsymbol{\theta}^{\mathfrak{m}, *}_{ {\rm total}}$, we denote it as $u^{*}_{\mathfrak{m}}$ and expand it as follows
\begin{align} \label{eq:u*} 
u_{\mathfrak{m}}^{*}(\boldsymbol{x})  = \sum_{s=1}^{\mathfrak{m}} c^{*}_{s}\cdot \big(\phi_{\boldsymbol{\theta}}^{s}\big)^{\scriptscriptstyle [0]}(\boldsymbol{x}) = \sum_{k=1}^{ \bar{\mathfrak{m}}} \sum_{v=1}^{R}c^{*}_{s_{k, v}}\cdot\big(\phi_{\boldsymbol{\theta}}^{s_{k,v}}\big)^{\scriptscriptstyle [0]}(\boldsymbol{x}) = \sum_{k=1}^{ \bar{\mathfrak{m}}}\sum_{v=1}^{R} \frac{\bar{c}_k}{R}\cdot\big(\phi_{\boldsymbol{\theta}}^{s_{k,v}}\big)^{\scriptscriptstyle [0]}(\boldsymbol{x}) \, ,
\end{align}
where $(\phi_{\boldsymbol{\theta}}^{s})^{\scriptscriptstyle [0]}$ represents the $s$-th sub-network $\phi^s_{\boldsymbol{\theta}}$ in $u_{\mathfrak{m}, \boldsymbol{\theta}}$ parameterized with $(\boldsymbol{\theta}_s)^{\scriptscriptstyle [0]}$. To elucidate the relationship between  $u^{*}_{\mathfrak{m}}$ and $u_{\bar{\mathfrak{m}}, \bar{\boldsymbol{\theta}}}$ more clearly, we expand $u_{\bar{\mathfrak{m}}, \bar{\boldsymbol{\theta}}}$ as follows
\begin{align} \label{eq:umbar}
    u_{\bar{\mathfrak{m}}, \bar{\boldsymbol{\theta}}}(\boldsymbol{x})  = \sum_{k=1}^{\bar{\mathfrak{m}}} \bar{c}_k\cdot \phi_{\bar{\boldsymbol{\theta}}}^{k}(\boldsymbol{x}) = \sum_{k=1}^{ \bar{\mathfrak{m}}}\sum_{v=1}^{R} \frac{\bar{c}_k}{R}\cdot \phi_{\bar{\boldsymbol{\theta}}}^{k}(\boldsymbol{x}) \, ,
\end{align}
where $\phi_{\bar{\boldsymbol{\theta}}}^{k}$ denotes the $k$-th sub-network in $u_{\bar{\mathfrak{m}}, \bar{\boldsymbol{\theta}}}$, whose weights are $\bar{\boldsymbol{\theta}}_k$. Hence, if there is only a small difference between  $\bar{\boldsymbol{\theta}}_k$ and $(\boldsymbol{\theta}_{s_{k, v}})^{\scriptscriptstyle[0]}$, we can expect $u^{*}_{\mathfrak{m}}$ and $u_{\bar{\mathfrak{m}}, \bar{\boldsymbol{\theta}}}$ to be also close. As we have known, event $G_{\mathfrak{m}, \bar{\mathfrak{m}}, R, \delta}$ ensures this.


Utilizing $u^{*}_{\mathfrak{m}}$, $\mathcal{E}^{-}_{opt}$ can be further decomposed into the following two terms:  
\begin{align} \label{eq:opt decomp}
\mathcal{E}^{-}_{opt} &= \widehat{\mathcal{L}}(u_{\mathcal{A}})-\widehat{\mathcal{L}}(u_{\bar{\mathfrak{m}}, \bar{\boldsymbol{\theta}}}) = \underbrace{\widehat{\mathcal{L}}(u_{\mathcal{A}})-\widehat{\mathcal{L}}(u_{\mathfrak{m}}^{*})}_{\text{iteration error}} + \underbrace{ \widehat{\mathcal{L}}(u_{\mathfrak{m}}^{*}) -\widehat{\mathcal{L}}(u_{\bar{\mathfrak{m}}, \bar{\boldsymbol{\theta}}})}_{\text{initialization error}}\, . \end{align}

\noindent {\bf Iteration error: }According to Section \ref{pgd}, $u_{\mathcal{A}}$ is exactly $u_{\mathfrak{m}, \boldsymbol{\theta}}$ parameterized with $(\boldsymbol{\theta}^{\mathfrak{m}}_{ \rm total})^{\scriptscriptstyle [T]}$, the final output of iterative equation (\ref{eq:pgd}), and $\widehat{\mathcal{L}}(u_{\mathcal{A}})$ is denoted as $\widehat{F}((\boldsymbol{\theta}^{\mathfrak{m}}_{ \rm in})^{\scriptscriptstyle [T]}, (\boldsymbol{\theta}^{\mathfrak{m}}_{ \rm out})^{\scriptscriptstyle [T]})$.  In the same manner, $\widehat{\mathcal{L}}(u_{\mathfrak{m}}^{*})$ can be expressed as $\widehat{F}((\boldsymbol{\theta}^{\mathfrak{m}}_{ \rm in})^{\scriptscriptstyle [0]}, \boldsymbol{\theta}^{\mathfrak{m}, *}_{ \rm out})$. Hecne, the iteration error describes how $(\boldsymbol{\theta}^{\mathfrak{m}}_{ \rm in}, \boldsymbol{\theta}^{\mathfrak{m}}_{ \rm out})^{\scriptscriptstyle [T]}$ as output of the iterative PGD algorithm are controlled by the transition parameters $\boldsymbol{\theta}^{\mathfrak{m}, *}_{\rm total}=((\boldsymbol{\theta}^{\mathfrak{m}}_{ \rm in})^{\scriptscriptstyle [0]}, \boldsymbol{\theta}^{\mathfrak{m}, *}_{ \rm out})$, which is also the reason for its naming.

Before conducting a detailed analysis, we need to specify the particular $\mathcal{PNN}$ class to which $u_{\mathcal{A}}$ belongs, as this is crucial for the generalization error analysis in the subsequent Section \ref{sta}. Firstly, by (\ref{eq:M}), we have $\|(\boldsymbol{\theta}^{\mathfrak{m}}_{\rm out})^{\scriptscriptstyle [t]}\|_{1} \leq \zeta$. Additionally, since 
$$\|(\boldsymbol{\theta}^{\mathfrak{m}}_{\rm in})^{\scriptscriptstyle [t]}\|_{\infty} \leq \|(\boldsymbol{\theta}^{\mathfrak{m}}_{\rm in})^{\scriptscriptstyle [0]}\|_{\infty}  + \|(\boldsymbol{\theta}^{\mathfrak{m}}_{\rm in})^{\scriptscriptstyle [t]}-(\boldsymbol{\theta}^{\mathfrak{m}}_{\rm in})^{\scriptscriptstyle [0]}\|_{\infty} \leq \|(\boldsymbol{\theta}^{\mathfrak{m}}_{\rm in})^{\scriptscriptstyle [0]}\|_{\infty}  + \|(\boldsymbol{\theta}^{\mathfrak{m}}_{\rm in})^{\scriptscriptstyle [t]}-(\boldsymbol{\theta}^{\mathfrak{m}}_{\rm in})^{\scriptscriptstyle [0]}\|_{2} \, ,
$$
we have $\|(\boldsymbol{\theta}^{\mathfrak{m}}_{\rm in})^{\scriptscriptstyle [t]}\|_{\infty} \leq B_{\bar{\boldsymbol{\theta}}} + \eta$ according to (\ref{eq:eta}). Thus, $u_{\mathcal{A}} \in \mathcal{PNN}(\mathfrak{m}, \zeta, \{\bar{W}, \bar{L}, B_{\bar{\boldsymbol{\theta}}} + \eta\})$.

Meanwhile, it is evident that $$\|\boldsymbol{\theta}_{\rm out}^{\mathfrak{m}, *}\|_1=\sum_{s=1}^{\mathfrak{m}}|c_s^{*}|=\sum_{k=1}^{\bar{\mathfrak{m}}}|\bar{c}_k|=\bar{M} \, , \quad \|\boldsymbol{\theta}_{\rm in}^{\mathfrak{m}, *}\|_{\infty}=\|(\boldsymbol{\theta}^{\mathfrak{m}}_{\rm in})^{\scriptscriptstyle [0]}\|_{\infty} \leq B_{\bar{\boldsymbol{\theta}}} < B_{\bar{\boldsymbol{\theta}}} + \eta \, ,$$
so we have $u^{*}_{\mathfrak{m}} \in \mathcal{PNN}(\mathfrak{m}, \bar{M}, \{\bar{W}, \bar{L}, B_{\bar{\boldsymbol{\theta}}}+ \eta\})$. Proposition \ref{prop:iteration}
demonstrates that by also 
setting $\zeta = \bar{M}$ in (\ref{eq:M}), and by properly selecting the parameter $R$, the number of iterations $T$, and the step size $\lambda$, we can control the upper bound of the iteration error to any desired precision.

\begin{prop} \label{prop:iteration}
    Let $\zeta=\bar{M}$ in (\ref{eq:M}). Then, we get $u_{\mathcal{A}}$, the output of the PGD algorithm through (\ref{eq:pgd}), belonging to $\mathcal{PNN}(\mathfrak{m}, \bar{M}, \{\bar{W}, \bar{L}, B_{\bar{\boldsymbol{\theta}}} + \eta\})$. Also, we run the algorithm with step size $\lambda$ satisfying
    \begin{align*}
        \lambda = T^{-1} \wedge 2\mm C_2^{-1} \m \mm \mathfrak{m}^{-1} \mm \bar{M}^{-2} \mm (B_{\bar{\boldsymbol{\theta}}}+\eta)^{-4\bar{L}} \, ,
    \end{align*}
    where $T$ is the total number of iterations, and $\eta$ is the projection radius of sub-network weights in (\ref{eq:eta}). Then, with $u^{*}_{\mathfrak{m}} \in \mathcal{PNN}(\mathfrak{m}, \bar{M}, \{\bar{W}, \bar{L}, B_{\bar{\boldsymbol{\theta}}}\})$ defined in (\ref{eq:u*}), the iteration error in (\ref{eq:opt decomp}) is bounded by
\begin{align*} \widehat{\mathcal{L}}(u_{\mathcal{A}})-\widehat{\mathcal{L}}(u_{\mathfrak{m}}^{*}) &  \leq  C_3 \cdot \bar{M}\cdot \eta \cdot (B_{\bar{\boldsymbol{\theta}}}+\eta)^{3\bar{L}}  \cdot \|\boldsymbol{\theta}_{\rm out}^{\mathfrak{m}, *}\|_2+\frac{1}{2}\|\boldsymbol{\theta}_{\rm out}^{\mathfrak{m}, *}\|_2^2+\frac{C_1 \cdot \mathfrak{m} \cdot \bar{M}^2 \cdot (B_{\bar{\boldsymbol{\theta}}}+\eta)^{4\bar{L}}}{2\mm T}  \\
& \leq \frac{C_3 \cdot \bar{M}^2\cdot (B_{\bar{\boldsymbol{\theta}}}+\eta)^{3\bar{L}} \cdot \eta}{\sqrt{R}}+\frac{\bar{M}^2}{2 R}+\frac{C_1 \cdot \mathfrak{m} \cdot \bar{M}^2 \cdot (B_{\bar{\boldsymbol{\theta}}}+\eta)^{4\bar{L}}}{2\mm T} \, .
\end{align*}
Here, $C_1, C_2$ and $ C_3$ are universal constants which only depend on $\Omega, \bar{W}, \bar{L}, d$ and $B_0$.
\end{prop}

The proof of Proposition \ref{prop:iteration} can be found in Appendix \ref{iter}.
\begin{remark}
    In Proposition \ref{prop:iteration}, we have utilized the following property of $\|\boldsymbol{\theta}_{\rm out}^{\mathfrak{m}, *}\|_2$:
\begin{align} \label{eq:norm*}
    \|\boldsymbol{\theta}_{\rm out}^{\mathfrak{m}, *}\|_2 = \sqrt{\sum_{s=1}^{\mathfrak{m}}|c_s^{*}|^2} = \frac{1}{\sqrt{R}} \sqrt{\sum_{k=1}^{\bar{\mathfrak{m}}}|\bar{c}_k|^2} \leq \frac{1}{\sqrt{R}} \sum_{k=1}^{\bar{\mathfrak{m}}}|\bar{c}_k| = \frac{1}{\sqrt{R}} \|\boldsymbol{\theta}_{\rm out}^{\mathfrak{m}, *}\|_1  \leq \frac{\bar{M}}{\sqrt{R}} \, .
\end{align}
That is, as $R$, which in some sense controls the over-parameterization degree of $u_{\mathfrak{m}, \boldsymbol{\theta}}$, increases, the upper bound of $\|\boldsymbol{\theta}_{\rm out}^{\mathfrak{m}, *}\|_2$ decays polynomially. As we have seen, this property allows us to control the iteration error to any given precision by letting $R \rightarrow \infty$ with $\bar{M}$ fixed, which underscores the importance of over-parameterization in our analysis.
\end{remark} 
\vskip 5mm
\noindent {\bf Initialization error: }Then, we turn to the initialization error. By (\ref{eq:u*}) and (\ref{eq:umbar}), it would be well controlled if there exists only a slight perturbation between the target weights $\bar{\boldsymbol{\theta}}_k$ and random initialization $(\boldsymbol{\theta}_{s_{k, v}})^{\scriptscriptstyle[0]}$, which also explains its name. More specifically, to precisely bound this term, $\mathbb{P}({G}_{\mathfrak{m}, \bar{\mathfrak{m}}, R, \delta})$ in Definition \ref{def:G} needs to sufficiently approach $1$ as $\delta \rightarrow 0$, which intuitively imposes a requirement on the size of $\mathfrak{m}$.

The following Proposition \ref{prop:pertur} concretizes such intuition, which sates that by adjusting the number of sub-networks in the implemented $u_{\mathfrak{m}, \boldsymbol{\theta}}$, namely, the degree of over-parameterization, we can bound initialization error with arbitrary high probability and precision.
\begin{prop} \label{prop:pertur}
    Choose $u_{\bar{\mathfrak{m}}, \bar{\boldsymbol{\theta}}} \in \mathcal{PNN}(\bar{\mathfrak{m}}, \bar{M}, \{\bar{W}, \bar{L}, B_{\bar{\boldsymbol{\theta}}}\})$ in Corollary \ref{cor4.1}. Let $\delta>0$, $R, Q\in\mathbb{N}$ while $Q$ is sufficiently large. If we set the number of sub-networks  $\mathfrak{m} = \bar{\mathfrak{m}} \cdot R \cdot Q$, then with probability at least 
    \begin{align*}
        1 - \bar{\mathfrak{m}}R\Big[1-\delta^{\bar{W}(\bar{W}+1)\bar{L}}(2B_{\bar{\boldsymbol{\theta}}})^{-\bar{W}(\bar{W}+1)\bar{L}}\Big]^{Q} \, ,
    \end{align*}
    the initialization error in (\ref{eq:opt decomp}) is bounded by
    \begin{align*}
        \widehat{\mathcal{L}}(u_{\mathfrak{m}}^{*}) -\widehat{\mathcal{L}}(u_{\bar{\mathfrak{m}}, \bar{\boldsymbol{\theta}}}) \leq C_4 \cdot \bar{M}^2 \cdot B^ {\raisebox{0.2ex}{$\scriptscriptstyle 3\bar{L}$}}_{\bar{\boldsymbol{\theta}}} \cdot \delta \, ,
    \end{align*}
    with $u^{*}_{\mathfrak{m}} \in  \mathcal{PNN}({\mathfrak{m}}, \bar{M}, \{\bar{W}, \bar{L}, B_{\bar{\boldsymbol{\theta}}}+ \eta\})$ defined in (\ref{eq:u*}). $C_4$ is a universal constant which only depend on $\Omega, \bar{W}, \bar{L}, d$ and $B_0$..
\end{prop}
The proof of Proposition \ref{prop:pertur} can be found in Appendix \ref{init}. Combining Propositions \ref{prop:iteration} and \ref{prop:pertur}, we obtain the following estimate of the optimization error $\mathcal{E}^{-}_{opt}$ in (\ref{eq:opt-}).
\begin{thm} \label{thm:opterror}
   Choose $u_{\bar{\mathfrak{m}}, \bar{\boldsymbol{\theta}}} \in \mathcal{PNN}(\bar{\mathfrak{m}}, \bar{M}, \{\bar{W}, \bar{L}, B_{\bar{\boldsymbol{\theta}}}\})$ in Corollary \ref{cor4.1}.  Let $\delta>0$, $R, Q\in\mathbb{N}$ while $Q$ is sufficiently large. Let $u_{\mathcal{A}} \in  \mathcal{PNN}(\mathfrak{m}, \bar{M}, \{\bar{W}, \bar{L}, B_{\bar{\boldsymbol{\theta}}} + \eta\})$ be the output of the PGD algorithm in (\ref{eq:pgd}) with $\zeta=\bar{M}$ in (\ref{eq:M}) and step size $\lambda$. If we set the number of sub-networks  $\mathfrak{m} = \bar{\mathfrak{m}} \cdot R \cdot Q$, and make $\lambda$ satisfying
    \begin{align*}
            \lambda = T^{-1} \wedge 2\mm C_2^{-1} \m \mm \bar{\mathfrak{m}}^{-1} \mm R^{-1} \mm Q^{-1} \mm \bar{M}^{-2} \mm (B_{\bar{\boldsymbol{\theta}}}+\eta)^{-4\bar{L}} \, ,
    \end{align*}
    where $T$ is the total number of iterations, then with probability at least 
    \begin{align*}
        1 - \bar{\mathfrak{m}}R\Big[1-\delta^{\bar{W}(\bar{W}+1)\bar{L}}(2B_{\bar{\boldsymbol{\theta}}})^{-\bar{W}(\bar{W}+1)\bar{L}}\Big]^{Q} \, ,
    \end{align*}
     the optimization error $\mathcal{E}^{-}_{opt}$ in (\ref{eq:opt-}) is upper bounded by
     \begin{align*}      
     \mathcal{E}^{-}_{opt} \leq \frac{C_3 \cdot \bar{M}^2\cdot (B_{\bar{\boldsymbol{\theta}}}+\eta)^{3\bar{L}} \cdot \eta}{\sqrt{R}}+\frac{\bar{M}^2}{2 R}+\frac{C_1 \cdot \mathfrak{m} \cdot \bar{M}^2 \cdot (B_{\bar{\boldsymbol{\theta}}}+\eta)^{4\bar{L}}}{2\mm T} + C_4 \cdot \bar{M}^2 \cdot B^ {\raisebox{0.2ex}{$\scriptscriptstyle 3\bar{L}$}}_{\bar{\boldsymbol{\theta}}} \cdot \delta \, ,
     \end{align*}
     where $\eta$ is the projection radius of sub-network weights in (\ref{eq:eta}). $C_1, C_2, C_3$ and $ C_4$ are universal constants which only depend on $\Omega, \bar{W}, \bar{L}, d$ and $B_0$.
\end{thm}

\subsection{Statistical error} \label{sta}
In this section, we present the upper bound of the statistical error, while detailed proofs can be found in the Appendix \ref{proof of sta}. Notice that the statistical error
$$
\mathcal{E}_{sta} := \sup_{u\in\mathcal{PNN}}\big|\mathcal{L}(u)-\widehat{\mathcal{L}}(u)\big|
$$
is a random variable, since it is a function of the Monte Carlo sample points $\{X_p\}_{p=1}^{N_{\rm in}}$, $\{Y_p\}_{p=1}^{N_{b}}$. Our task is to control $\mathcal{E}_{sta}$ with high probability. 

\begin{thm}
\label{thm5.2}
Let $\mathcal{PNN} = \mathcal{PNN}(\mathfrak{m}, M, \{W, L, B_{\boldsymbol{\theta}}\})$. Let $N_{\rm in} = N_{b} = N_{s}$ in the Monte Carlo sampling. Let $0<\xi<1$. Then, with probability at least $1-\xi$, it holds that
\begin{align*}
    \mathcal{E}_{sta}&=\sup_{u_{\mathfrak{m}, \boldsymbol{\theta}}\in\mathcal{PNN}}\big|\mathcal{L}(u_{\mathfrak{m}, \boldsymbol{\theta}})-\widehat{\mathcal{L}}(u_{\mathfrak{m}, \boldsymbol{\theta}})\big| \\
    &\leq C(\Omega, B_0, d, W, L)\cdot M^2  B_{\boldsymbol{\theta}}^{2L} N_{s}^{-\frac{1}{2}}  \big(\sqrt{\log(B_{\boldsymbol{\theta}} WLN_{s})}+\sqrt{\log\xi^{-1}}\big) \, ,
\end{align*}
where $C(\Omega, B_0, d, W, L)$ is a universal constant which only depends on $\Omega, B_0, d, W$ and $L$.
\end{thm}
 Theorem \ref{thm5.2} above analyzes the statistical error of a general $\mathcal{PNN}$ class.
It is worth noting that the upper bound achieved in Theorem \ref{thm5.2} is not affected by the number of sub-networks $\mathfrak{m}$, aiding us in managing the statistical error within the over-parameterized setting, where $\mathfrak{m}$ can grow arbitrarily large.

As shown in Theorem \ref{thm:opterror}, in our practical analysis, the PGD algorithm output $u_{\mathcal{A}}$ belongs to $\mathcal{PNN}(\mathfrak{m}, \bar{M}, {\bar{W}, \bar{L}, B_{\bar{\boldsymbol{\theta}}} + \eta})$. Combining this with Theorem \ref{thm5.2}, we obtain more specific upper bounds of the statistical error $\mathcal{E}_{sta}$.

\begin{cor}
\label{col5.2}
Choose $\mathcal{PNN} = \mathcal{PNN}(\mathfrak{m}, \bar{M}, \{\bar{W}, \bar{L}, B_{\bar{\boldsymbol{\theta}}} + \eta\})$. Let $N_{\rm in} = N_{b} = N_{s}$ in the Monte Carlo sampling. Let $0<\xi<1$. Then, with probability at least $1-\xi$, it holds that
\begin{align*}
    \mathcal{E}_{sta}\leq C(\Omega, B_0, d, \bar{W}, \bar{L})\cdot \bar{M}^2   (B_{\bar{\boldsymbol{\theta}}}+\eta)^{2\bar{L}} N_{s}^{-\frac{1}{2}}  \big\{\log^{1/2}[(B_{\bar{\boldsymbol{\theta}}}+\eta) \bar{W} \bar{L} N_{s}]+ \sqrt{\log\xi^{-1}}\big\} \, ,
\end{align*}
where $C(\Omega, B_0, d, W, L)$ is a universal constant which only depends on $\Omega, B_0, d, W$ and $L$.
\end{cor}

\subsection{Proof of main result}
\label{proof of main result}
Based on all the analyses above, we will now present the detailed proof of our main result, Theorem \ref{mainth}. To the best of our knowledge, we are the first to provide a comprehensive error analysis that integrates approximation error, generalization error, and optimization error for using deep over-parameterized networks to solve PDE problems.

Firstly, according to Corollary \ref{cor4.1}, we know that for any $\epsilon>0$, there exists a neural network $$u_{\bar{\mathfrak{m}}, \bar{\boldsymbol{\theta}}}\in \mathcal{PNN}(\bar{\mathfrak{m}}, \bar{M}, \{\bar{W}, \bar{L}, B_{\bar{\boldsymbol{\theta}}}\})$$ with
\begin{align*}
&\bar{\mathfrak{m}}=C(n, d)\epsilon^{-\frac{d}{n-\mu-1}}\, , \quad \bar{M} = C(n, d)\epsilon^{-\frac{3d}{2(n-\mu-1)}} \, , \quad \bar{W} = 2^{\lceil \log_2 (d+1) \rceil+1} \, , \\
&~~~~~~~~~~~\bar{L}= \lceil \log_2 (d+1) \rceil+2 \, , \quad B_{\bar{\boldsymbol{\theta}}} = C(n, d)\epsilon^{-2-\frac{2d}{n-\mu-1}} \, ,
\end{align*}
such that the approximation error $\mathcal{E}_{app} \leq \epsilon$.

Secondly, by Corollary \ref{col5.2}, we have:
\begin{align*}
    \mathcal{E}_{sta}\leq C(\Omega, B_0, d, \bar{W}, \bar{L})\cdot \bar{M}^2   (B_{\bar{\boldsymbol{\theta}}}+\eta)^{2\bar{L}} N_{s}^{-\frac{1}{2}}  \big\{\log^{1/2}[(B_{\bar{\boldsymbol{\theta}}}+\eta) \bar{W} \bar{L} N_{s}]+\log^{1/2}(\xi^{-1})\big\} \, .
\end{align*}
Setting the Monte Carlo sample size $N_{s} = C \cdot \epsilon^{-C_3(\mu, d, \beta_0, n)}$ and $\xi = C \cdot \epsilon^{C_3(\mu, d, \beta_0, n)}$ with
$$
C_3(\mu, d, \beta_0, n) = 4\beta_0\log(d+1)+\frac{6d}{n-\mu-1}+12\beta_0 + 2\, , \quad \beta_0 = \max\{\beta, \,  2+2 \m d \m (n-\mu-1)^{-1}\} ,$$ 
the statistical error is controlled by $\mathcal{E}_{sta} \leq C \m \epsilon \log^{1/2} (C\m \epsilon^{-1}) = \tilde{\mathcal{O}}(\epsilon)$.

Finally, in order to bound the optimization error $\mathcal{E}^{-}_{opt} \le C\epsilon$ with probability at least $1- \xi$, we need to determine parameters $Q, R, \delta, T, \eta$ such that the following inequalities hold:
\[1 - \bar{\mathfrak{m}}R\Big[1-\delta^{\bar{W}(\bar{W}+1)\bar{L}}(2B_{\bar{\boldsymbol{\theta}}})^{-\bar{W}(\bar{W}+1)\bar{L}}\Big]^{Q} \geq 1 - \xi \, ,\\\]
   \[\frac{C_3 \cdot \bar{M}^2\cdot (B_{\bar{\boldsymbol{\theta}}}+\eta)^{3\bar{L}} \cdot \eta}{\sqrt{R}}+\frac{\bar{M}^2}{2 R}+\frac{C_1 \cdot \mathfrak{m} \cdot \bar{M}^2 \cdot (B_{\bar{\boldsymbol{\theta}}}+\eta)^{4\bar{L}}}{2\mm T} + C_4 \cdot \bar{M}^2 \cdot B^ {\raisebox{0.2ex}{$\scriptscriptstyle 3\bar{L}$}}_{\bar{\boldsymbol{\theta}}} \cdot \delta \leq C\epsilon\, .\]
By some calculation, we obtain that:
\begin{align*}
\delta = C \cdot \epsilon^{\frac{6d\log(d+1)+21d}{n-\mu-1}+6\log(d+1)+19} \, , \quad R = C \cdot \epsilon^{-6\log(d+1)\beta_0-18\beta_0-\frac{6d}{n-\mu-1}-2\beta-2} \, .
\end{align*}
Using the properties $1-x\leq \exp(-x)$ and $\exp(-t)\leq t^{-2}$ when $t\geq 0$, we solve for $Q$ as follows:
\[ Q = \sqrt{\frac{\bar{\mathfrak{m}}R}{\xi}}\left(\frac{\delta}{2B_{\bar{\boldsymbol{\theta}}}}\right)^{-\bar{W}(\bar{W}+1)\bar{L}} = C \cdot \epsilon^{-\frac{C_0d^3\log^2(d+1)}{n-\mu-1}-5\beta_0\log(d+1)-15\beta_0-\beta}\, . \]
Then, $\mathfrak{m}= \bar{\mathfrak{m}} \cdot R \cdot Q =C \cdot \epsilon^{-C_1(\mu, d, \beta, \beta_0, n)}$. In order to satisfy the following inequality:
   \[\frac{C_1 \cdot \mathfrak{m} \cdot \bar{M}^2 \cdot (B_{\bar{\boldsymbol{\theta}}}+\eta)^{4\bar{L}}}{2\mm T}  \leq C\epsilon\, ,\]
we can calculate that
\begin{align*}
T =C \cdot \epsilon^{-C_2(\mu, d, \beta, \beta_0, n)} \, .
\end{align*}
Thus, we complete the proof of our main theorem. $\hfill\Box$

\section{Related work} \label{rw}
\subsection{Approximation error}
Approximation error in the context of deep neural networks  refers to the difference between the target  function and neural network function.
The theoretical analysis of the approximation power of shallow sigmoidal networks dates back to the 1980s \cite{cybenko1989approximation,hornik1989multilayer,hornik1991approximation}, see the review paper \cite{pinkus1999approximation} and the reference therein for  shallow network approximations. In recent years, attention has shifted to ReLU networks due to their superior empirical performance in modern learning tasks. Yarotsky \cite{yarotsky2017error} was the first to demonstrate how to construct a ReLU network that achieves any desired approximation accuracy using the Taylor expansion. Inspired by this, modern approximation techniques for deep neural networks have emerged, utilizing network architecture parameters like depth, width, and size  to control approximation errors \cite{yarotsky2017error,yarotsky2018optimal,petersen2018optimal,zhou2020universality,shen2020deep,siegel2020approximation,shen2022optimal,lu2021deep}. For more information, see \cite{petersen2020neural,devore2021neural}. Neural networks with supper expressive power that can break the curse of dimensionality   has also been constructed  in \cite{yarotsky2020phase,yarotsky2021elementary,shen2021neural,shen2021deepnn,jiao2023deep}. Recently, such approximation theories have been extended to Sobolev spaces in the context of using deep learning models to solve PDEs. For instance, using techniques such as approximate partition of unity and averaged Taylor expansion, Gühring et al. extended Yarotsky's proof to  Sobolev spaces \cite{guhring2020error,guhring2021approximation}. 
One can also derive the approximation in Sobolev norm for deep networks with ReLU$^k$ as activations by observing the connection between deep neural networks and B-splines \cite{duan2021convergence,jiao2023improved}.

\subsection{Statistical (generalization) error}
Statistical (generalization) error in learning theory is described via the uniform law of large numbers over the network class. 
Classical methods in empirical process theory employ tools such as symmetrization and Lipschitz contraction to transform the study of generalization error into bounding the complexity of neural network classes, such as the Rademacher complexity, covering number, or VC-dimension. For detailed analysis, see \cite{VanJhon,van2000empirical,gine2021mathematical,cucker2002mathematical}. However, generalization analysis from the perspective of the uniform law of large numbers may lead to suboptimal error bounds \cite{bartlettspectrally}. Localized techniques that utilize the local structure of the hypothesis function class can reach sharp error bounds in scenarios where the Bernstein condition or off-set condition hold, see \cite{bartlett2005local,koltchinskii2006local,mendelson2018learning,xu2020towards,kanade2022exponential} and the references therein. 
One should note that the statistical error in DRM cannot be directly handled using the contraction principle, as the differential operator involved in the loss function is not Lipschitz continuous.
One way to address this challenge is to use the chain rule to represent the gradient of the employed neural network class as another neural network class, and then bound the complexity of the latter \cite{duan2021convergence}.
By expressing the gradient of the neural network as another neural network, we can leverage the properties of neural network classes, such as their Lipschitz continuity and covering numbers, to derive bounds on the statistical error. This technique provides a more rigorous treatment for understanding the generalization performance of deep PDE's solver \citep{jiao2022rate,duan2021convergence,lu2021machine,jiao2023rate,ji2024deep,yang2024deeper,jiao2023improved}.  

\subsection{Theory  on ERM with deep neural networks}
The convergence rate of Empirical Risk Minimization (ERM) with deep neural networks in regression, classification, and solving PDEs can be established within the framework of nonparametric estimation \cite{bauer2019deep,kohler2021rate,schmidt2020nonparametric,nakada2020adaptive,farrell2021deep,jiao2021deep,suzuki2018adaptivity,suzuki2021deep,weinan2019priori,hong2021rademacher,lu2021priori,hutzenthaler2020overcoming,shin2020convergence,lanthaler2022error,muller2021error,mishra2021enhancing,kutyniok2022theoretical,son2021sobolev,wang2022and,weinan2020comparative,jiao2022rate,duan2021convergence,lu2021machine,mishra2022estimates,ji2024deep,yang2024deeper,hu2023solving1, hu2023solving2, dai2023solving}. This is done by carefully balancing the trade-off between the approximation error and the statistical error, which  provides  theoretical guarantees on the performance of deep learning models. 

However, the aforementioned results on the convergence rate of ERM are only useful when the size of the neural network class is smaller than the number of training samples.
Taking deep ReLU neural networks as an example, their VC-dimension or covering number is bounded by their size \citep{bartlett2019nearly}. This means that the theoretical guarantees for the convergence of ERM can only be established in the under-parameterized regime, where the depth, width, and size of the neural network are chosen as a function of the sample size to balance the approximation error and the generalization error.
Recent works in \cite{jiao2023approximation,yang2024optimal,jiao2024pinns,jiao2024over}
 have proposed using the weight norm instead of the more commonly used measures of network width, depth, or size to characterize both the approximation error and the statistical error. This has enabled them to derive convergence rates for ERM in the over-parameterized regime, 
 This represents an important step towards theoretical understanding  of modern, over-parameterized neural networks without considering optimization error.

\subsection{Optimization error}
The prevailing analytical tools for analyzing the optimization error in deep neural networks are currently the neural tangent kernel and mean field theory in the over-parameterized or even infinite width setting, see  \citep{jacot2018neural,allen2019convergence,du2019gradient,zou2019improved,liu2022loss,chizat2019lazy,nguyen2021proof,lu2020mean,mahankali2024beyond} and the reference therein. 
A technical limitation of the aforementioned works is that they require the training dynamics to remain close to the initial parameter values, which is often not realistic in practical applications. 

In the over-parameterized scheme, the optimal value of the training loss is zero, i.e., the network perfectly interpolates the data. One one hand, when the network memorizes the data, it may have a significantly large upper bound for the weights \cite{vershynin2020memory,vardi2021optimal,shen2020deep,lu2021deep}, which can further lead to the size-independent generalization error becoming uncontrollable \cite{golowich2018size,jiao2023approximation}.
On the other hand, the magnitude of the norm constraint used in recent studies \cite{jiao2023approximation,yang2024optimal,jiao2024pinns,jiao2024over} for analyzing  ERM may not be large enough. Even randomized initialization in  NTK and mean field analyses may not satisfy this norm constraint.
These  are  the primary reasons that prevent researchers from simultaneously considering the three key errors (approximation, generalization, and optimization ) in over-parameterized modern  deep learning.
\subsection{Works on complete error analysis}
\cite{beck2022full, jentzen2023overall} conducted a comprehensive error analysis of deep regression in the under-parameterization setting. Building upon recent research on the estimation error of gradient descent in regression \cite{kohler2021rate, kohler2023rate, drews2023analysis}, \cite{jiao2024error} derived consistency results for DRM, encompassing all three types of errors. However, it should be noted that the results in \cite{jiao2024error} are specifically applicable to three-layer networks only. One major drawback of \cite{jiao2024error, kohler2021rate, kohler2023rate, drews2023analysis} is their reliance on the iteration being very close to the initialization, which is an unrealistic requirement. This restrictive condition significantly limits the practicality of the theory proposed in \cite{jiao2024error, kohler2021rate, kohler2023rate, drews2023analysis}, as real-world training of deep neural networks often involves substantial parameter updates that move far from the initial configuration.

\section{Conclusion} \label{conclu}
In this paper, we provide the first complete error analysis for the Deep Ritz method (DRM) that includes the approximation, generalization, and optimization errors in the scenario of over-parameterization. Our analysis is based on the projected gradient descent algorithm and does not require constraining the neural network weights near their initial values during the optimization process, thereby completely moving away from the lazy training framework. This marks a milestone in the field of theoretical understanding of  solving PDEs via deep learning. 

Several questions deserve further investigation. 
Firstly, 
our analytical techniques rely on the random initialization of over-parameterized neural networks. In the current analysis, we do not use any prior knowledge to design the parameter initialization method; instead, we choose a general uniform distribution. This results in a theoretically excessive number of training samples and number of iteration steps to achieve the desired accuracy. Therefore, exploring the effective utilization of prior information to enhance analysis results is an intriguing subject.  Secondly, the gradient descent algorithm used in our theoretical analysis is full gradient descent, which still has some gaps compared to the stochastic gradient descent (SGD) algorithm commonly used in practice. Finally, the analytical framework presented in this paper is highly versatile and can be directly applied to the analysis of other areas in deep PDE solving, such as PINNs and various inverse problems. We plan to thoroughly investigate these issues in the future. 

\section{Appendix}
The appendix is divided into three parts. In Appendix \ref{proof of app}, we provide a detailed explanation of how to construct the optimal approximation function in the Sobolev space using the $\mathcal{PNN}$ structure as discussed in Section \ref{app}. In Appendix \ref{proof of sta}, we present the complete proof of the statistical error upper bound estimation for the over-parameterized $\mathcal{PNN}$ network class as given in Section \ref{sta}. Note that some of the lemmas used in Appendices \ref{proof of app} and \ref{proof of sta} are derived from previous work or are classical results. For the sake of completeness, we have still provided detailed proofs of these lemmas. In Appendix \ref{proof of opt}, we present the proofs of the theorems and lemmas involved in the optimization error analysis in Section \ref{opt}, which are newly proposed in this paper.

\subsection{Detailed approximation error analysis}
\label{proof of app}
Despite the many existing approximation results, for the sake of completeness, we provide a comprehensive error analysis in this section.  Building on methods from \cite{yarotsky2017error}, \cite{jiao2023approximation}, and \cite{guhring2021approximation}, we derive an approximation error bound for neural networks in the \( W^{k,p} \) norm (\( n \geq k+1 \)) for functions in the \( W^{n,p} \) Sobolev space.  This bound is achieved by explicitly constructing  \( \tanh \)-activated  neural networks that approximate local Taylor polynomials.
 Notably, the constructed network in this paper has a parallel architecture, meaning the final neural network is a linear combination of many structurally similar fully connected sub-networks, as specified in Section \ref{subsec:topo}. 

We divide the analysis into four parts. Firstly, following the construction in \cite{guhring2021approximation}, we introduce an approximate partition of unity that is compatible with the $tanh$ activation function in Section \ref{app1}. Secondly, in Section \ref{app2}, we approximate a function $f$ by localized Taylor polynomials, where the localization is realized by an approximate partition of unity. Thirdly, in Section \ref{app3}, we approximate the multiplication of Taylor polynomials and the multi-dimensional partition of unity by neural networks with accuracy $\epsilon$.  Finally, in Section \ref{app4}, we construct the sum of all approximations of localized Taylor polynomials by neural networks.

\subsubsection{Approximate partition of unity}
\label{app1}
\begin{defn}Let $d, j, \tau, N \in \mathbb{N}$, $s\in \mathbb{R}$. We say that the collection of families of functions $(\Lambda^{(j, \tau, N, s)})_{N\in\mathbb{N}, s\in \mathbb{R}_{\geq 1}}$, where each $\Lambda^{(j, \tau, N, s)}:=\{\Psi_{\raisebox{0.2ex}{$\hspace{-0.1em}{\scalebox{0.6}{$\boldsymbol{m}$}}$}}^{s}: \boldsymbol{m} \in\{0, \ldots, N\}^{d}\}$ consists of $(N+1)^{d}$ functions $\Psi_{\raisebox{0.2ex}{$\hspace{-0.1em}{\scalebox{0.6}{$\boldsymbol{m}$}}$}}^{s}: \mathbb{R}^{d} \rightarrow \mathbb{R}$, is an exponential partition of unity of order $\tau$ and smoothness $j$, if the following conditions are met:

There exist some $D>0$, $C=C(k, d)>0$ and $S>0$ such that for all $N \in \mathbb{N}, s \geq S, k \in\{0, \ldots, j\}$ the following properties hold:

{\rm (i) } $\|\Psi_{\raisebox{0.2ex}{$\hspace{-0.1em}{\scalebox{0.6}{$\boldsymbol{m}$}}$}}^{s}\|_{W^{k, \infty}(\mathbb{R}^{d})} \leq C N^{k} \cdot s^{\max \{0, k-\tau\}}$ for every $\Psi_{\raisebox{0.2ex}{$\hspace{-0.1em}{\scalebox{0.6}{$\boldsymbol{m}$}}$}}^{s} \in \Lambda^{(j, \tau, N, s)}$.

{\rm (ii) } For $\Omega_{\raisebox{0.2ex}{{\scalebox{0.6}{$\boldsymbol{m}$}}}}^{c}=\{\boldsymbol{x} \in \mathbb{R}^{d}:\|\boldsymbol{x}-N^{-1}\boldsymbol{m}\|_{\infty} \geq N^{-1}\}$, we have
$$
\|\Psi_{\raisebox{0.2ex}{\hspace{-0.1em}{\scalebox{0.6}{$\boldsymbol{m}$}}}}^{s}\|_{W^{k, \infty}(\Omega_{\scriptscriptstyle \boldsymbol{m}}^{c})} \leq C N^{k} s^{\max \{0, k-\tau\}} e^{-D s}\, ,
$$
for every $\Psi_{\raisebox{0.2ex}{$\hspace{-0.1em}{\scalebox{0.6}{$\boldsymbol{m}$}}$}}^{s} \in \Lambda^{(j, \tau, N, s)}$.

{\rm (iii) } We have
$$
\bigg\|\mathbf{1}_{[0,1]^{d}} - \sum_{\boldsymbol{m} \in\{0, \ldots, N\}^{d}} \Psi_{\raisebox{0.2ex}{$\hspace{-0.1em}{\scalebox{0.6}{$\boldsymbol{m}$}}$}}^{s}\bigg\|_{W^{k, \infty}([0,1]^{d})} \leq C N^{k} s^{\max \{0, k-\tau\}} e^{-D s}\, ,
$$
for every $\Psi_{\raisebox{0.2ex}{$\hspace{-0.1em}{\scalebox{0.6}{$\boldsymbol{m}$}}$}}^{s} \in \Lambda^{(j, \tau, N, s)}$.

{\rm (iv) } There exists a function $\rho: \mathbb{R} \rightarrow \mathbb{R}$ such that for each $\Psi_{\raisebox{0.2ex}{$\hspace{-0.1em}{\scalebox{0.6}{$\boldsymbol{m}$}}$}}^{s} \in \Lambda$ there is a neural network $\psi_{\boldsymbol{\theta}}$ with $d$-dimensional input $\boldsymbol{x}$ and $d$-dimensional output, with two layers and $C$ nonzero weights, that satisfies
$$
\prod_{l=1}^{d}[\psi_{\boldsymbol{\theta}}(\boldsymbol{x})]_{l}=\Psi_{\raisebox{0.2ex}{$\hspace{-0.1em}{\scalebox{0.6}{$\boldsymbol{m}$}}$}}^{s}
$$
and $\|\psi_{\boldsymbol{\theta}}(\boldsymbol{x})\|_{W^{k, \infty}([0,1]^{d})} \leq C N^{k} \cdot s^{\max \{0, k-\tau\}}$\, . Furthermore, for the weights of $\psi_{\boldsymbol{\theta}}$ it holds that $\|\boldsymbol{\theta}\|_{\infty} \leq C s N$.
\end{defn}
\begin{defn}
\label{adef4.2}
Let $j\in \mathbb{N}$, $\tau =0$, $\rho$ be the $tanh$ activation function $\frac{e^{x}-e^{-x}}{e^{x}+e^{-x}}$. Define that for a scaling factor $s\geq 1$, the one-dimension bump functions
$$
\psi^s :\mathbb{R} \rightarrow \mathbb{R}, \quad \psi^{s}(x):= \frac{1}{2}[\rho(s(x+3 / 2))-\rho(s(x-3 / 2))]\, .
$$
For $N, d \in \mathbb{N}$ and $\boldsymbol{m} \in\{0, \ldots, N\}^{d}$ we define multi-dimensional bumps $\Psi_{\raisebox{0.2ex}{$\hspace{-0.1em}{\scalebox{0.6}{$\boldsymbol{m}$}}$}}^{s}: \mathbb{R}^{d} \rightarrow \mathbb{R}$ as a tensor product of scaled and shifted versions of $\psi^{s}$. Concretely, we set
$$
\Psi_{\raisebox{0.2ex}{$\hspace{-0.1em}{\scalebox{0.6}{$\boldsymbol{m}$}}$}}^{s}(\boldsymbol{x}):=\prod_{l=1}^{d} \psi^{s}\Big(3 N\Big(x_{l}-\frac{m_{l}}{N}\Big)\Big)\, .
$$
Finally for $s \geq 1$, the collection of bump functions is denoted by $\Lambda^{(j, 0 , N, s)}(\rho):=\{\Psi_{\raisebox{0.2ex}{$\hspace{-0.1em}{\scalebox{0.6}{$\boldsymbol{m}$}}$}}^{s}: \boldsymbol{m} \in \{0, \ldots, N\}^{d}\}$.
\end{defn}
\begin{lem}
\label{alem4.1}
The collection of families of functions $(\Lambda^{(j, 0 , N, s)}(\rho))_{N\in\mathbb{N},s \in \mathbb{R}_{\geq 1}}$ defined in Definition \ref{adef4.2} is an exponential PU of order $0$ and smoothness $j$.
\end{lem}
\begin{proof}
    See \cite{guhring2021approximation} (Lemma 4.5).
\end{proof}

\subsubsection{Approximate by  polynomials}
\label{app2}
\begin{prop}
\label{aprop4.1}
Let $d\in \mathbb{N}, j \in \mathbb{N}, k\in\{0,\cdots,j\}, n \in \mathbb{N}_{\geq k+1}$, $\boldsymbol{\alpha} \in \mathbb{N}^{d}$ and $1\leq p\leq \infty$. Let $\mu \in (0,1)$, $\rho$ be the $tanh$ activation function $\frac{e^{x}-e^{-x}}{e^{x}+e^{-x}}$. Denote  ${\boldsymbol{x}}^{\boldsymbol{\alpha}} = x_2^{\alpha_2} x_1^{\alpha_1}\cdots x_d^{\alpha_d}$ and $|{\boldsymbol{\alpha}}|_1 = \alpha_1 + \alpha_2 + \cdots +\alpha_d$. For $N\in \mathbb{N}$, set $s:=N^{\mu}$. Then for a collection of functions  $\Psi_{\raisebox{0.2ex}{$\hspace{-0.1em}{\scalebox{0.6}{$\boldsymbol{m}$}}$}}^{s} \in (\Lambda^{(j, 0, N, s)}(\rho))_{N\in\mathbb{N},s\in\mathbb{R}_{\geq 1}}$, there is a constant $C=C(d, n, p, k)>0$ and $\tilde{N}=\tilde{N}(d, p, \mu, k) \in \mathbb{N}$ such that for every $f \in$ $W^{n, p}([0,1]^{d})$ and every $\boldsymbol{m} \in\{0, \ldots, N\}^{d}$, there exist polynomials $p_{\scalebox{0.7}{$f, \mmn \boldsymbol{m}$}}(\boldsymbol{x})=\sum_{|\boldsymbol{\alpha}|_1 \leq n-1} c_{\scalebox{0.7}{$f, \mmn \boldsymbol{m}, \mmn \boldsymbol{\alpha}$}} {\boldsymbol{x}}^{\boldsymbol{\alpha}}$ with the following properties:

Set $f_{N}:=\sum_{\boldsymbol{m} \in\{0, \cdots, N\}^{d}} \Psi_{\raisebox{0.2ex}{$\hspace{-0.1em}{\scalebox{0.6}{$\boldsymbol{m}$}}$}}^{s} \mm p_{\scalebox{0.7}{$f, \mmn \boldsymbol{m}$}} .$ Then, the operator $T_{k}: W^{n, p}([0,1]^{d}) \rightarrow W^{k, p}([0,1]^{d})$ with $T_{k} f=$ $f-f_{N}$ is linear and bounded with
\begin{equation}
    \|T_kf\|_{W^{k,p}([0,1]^{d})}\leq C\|f\|_{W^{n,p}([0,1]^{d})}\cdot N^{-(n-k-\mu k)}\, ,
\end{equation}
for all $N \in \mathbb{N}$ with $N \geq \widetilde{N}$.
Moreover,
there is a constant $C=C(d,n,k)>0$ such that for any $f\in W^{n,p}([0,1]^{d})$ the coefficients of polynomials $p_{\scalebox{0.7}{$f, \mmn \boldsymbol{m}$}}$ satisfy
$$
|c_{\scalebox{0.7}{$f, \mmn \boldsymbol{m}, \mmn \boldsymbol{\alpha}$}}|\leq C\|\tilde{f}\|_{W^{n,p}(\Omega_{\m\boldsymbol{m},N})}N^{d/p}\, ,
$$
where $\Omega_{\boldsymbol{m},N}$ denotes the open ball w.r.t. $\|\cdot\|_{\infty}$ around $N^{-1}\boldsymbol{m}$ with radius $N^{-1}$, and $\tilde{f}\in W^{n,p}(\mathbb{R}^d)$ is an extension of $f$.

\end{prop}
\begin{proof}
See \cite{guhring2021approximation} (Lemma D.1).
\end{proof}

\subsubsection{Approximate polynomials with neural network}
\label{app3}
 The goal of this subsection is to demonstrate how to approximate the sums of localized polynomials $\sum_{ \boldsymbol{m} \in\{0, \ldots, N\}^{d}}\sum_{ |\boldsymbol{\alpha}|_1 \leq n-1} c_{\scalebox{0.7}{$f, \mmn \boldsymbol{m}, \mmn \boldsymbol{\alpha}$}} \Psi_{\raisebox{0.2ex}{$\hspace{-0.1em}{\scalebox{0.6}{$\boldsymbol{m}$}}$}}^{s} \mm {\boldsymbol{x}}^{\boldsymbol{\alpha}}$ by neural networks 
  since  $$f_{N}:=\sum_{ \boldsymbol{m} \in\{0, \ldots, N\}^{d}}\sum_{ |\boldsymbol{\alpha}|_1 \leq n-1} c_{\scalebox{0.7}{$f, \mmn \boldsymbol{m}, \mmn \boldsymbol{\alpha}$}} \Psi_{\raisebox{0.2ex}{$\hspace{-0.1em}{\scalebox{0.6}{$\boldsymbol{m}$}}$}}^{s} \mm {\boldsymbol{x}}^{\boldsymbol{\alpha}}\, .$$ 
  We first consider the approximation of the quadratic function $f(x)=x^2$.
\begin{lem}
\label{aprop4.2}
For any $0 < \epsilon < 1$ and a positive constant $C$, there exists $\phi_{\boldsymbol{\theta}} \in \mathcal{NN}(W,L,B_{\boldsymbol{\theta}})$ where $W = 2$, $L = 2$, and $B_{\boldsymbol{\theta}} = C \cdot \epsilon^{-2}$, such that
$$
\|x^2-\phi_{\boldsymbol{\theta}}(x)\|_{W^{k,\infty}([0,1])}\leq \epsilon\, .
$$
\end{lem}
\begin{proof}
Let $\Omega=[0,1]$, since $\rho(x)=tanh(x)$, for some $n\in\mathbb{N}$ there exists $x_0\in\Omega$ such that $\rho^{(r)}(x_0)\neq 0,\ r\in \{1,\cdots,n\}$. According to \cite{guhring2021approximation} (Proposition 4.7) choose $C_0>1$ so that $[x_0-C_0^{-1}n, x_0+C_0^{-1}n]\subset \Omega$. Moreover, let $\delta\geq C_0$ be arbitrary. Define the function
\begin{align*}
\rho_{\delta}^{r}: \mathbb{R} \rightarrow \mathbb{R}\, , \quad x \mapsto \frac{\delta^{r}}{\rho^{(r)}(x_{0})} \sum_{j=0}^{r}(-1)^{j}\Bigg(\!\begin{array}{c}
r \\
j
\end{array} \!\Bigg) \cdot \rho\Big(x_{0}-j \frac{x}{\delta}\Big)\, .
\end{align*}
Then $\rho_{\delta}^r|_{\Omega} \in C^{n+1}(\Omega)$. Using the Taylor expansion we have that for every $k=0, \cdots, n$ and every $x\in \Omega$,
\begin{align*}
        |(\rho_{\delta}^{r})^{(k)}(x)&-(x^r)^{(k)}| \\
    & \leq 2^n(n+1)^{n+1}n!\cdot \frac{\|\rho\|_{C^{n+1}(\Omega)}}{\min_{i=0,\cdots,n}|\rho^{(i)}(x_0)|}\max\{|\Omega|,1\}^{n+1}\cdot \frac{1}{\delta} =: \frac{C^{\prime}(n,\rho)}{\delta}\, .
\end{align*}
This implies that there exists some $C \geq \max \{C_{0}, C^{\prime}(n, \rho)\}$ such that for every $\epsilon \in(0,1)$ and the neural network $\boldsymbol{\theta}:=((\boldsymbol{A}_{0}, \boldsymbol{b}_{0}),(\boldsymbol{A}_{1}, \boldsymbol{b}_{1}))$ with
\begin{align*}
\boldsymbol{A}_{0} &:=\Big(-\frac{\epsilon}{C}, \ldots,-\frac{r \epsilon}{C}\Big)^{\rm T} \in \mathbb{R}^{r\times1}\, , \qquad
\boldsymbol{b}_{0} :=(x_{0}, \ldots, x_{0})^{\rm T} \in \mathbb{R}^{r}\, , \\
\boldsymbol{A}_{1} &:=\frac{C^{r}}{\epsilon^{r} \rho^{(r)}(x_{0})}\Bigg((-1)^{0}\Bigg(\!\begin{array}{l}
r \\
1
\end{array}\!\Bigg),(-1)^{1}\Bigg(\!\begin{array}{l}
r \\
2
\end{array}\!\Bigg), \ldots,(-1)^{r}\Bigg(\!\begin{array}{l}
r \\
r
\end{array}\!\Bigg)\Bigg) \in \mathbb{R}^{1\times r}\, , \\
\boldsymbol{b}_{1} &:=\frac{C^{r}}{\epsilon^{r} \rho^{(r)}(x_0)}\cdot \rho(x_0) \in \mathbb{R}\, ,
\end{align*}
fulfills $\|\phi_{\boldsymbol{\theta}}(x)-x^{r}\|_{C^{n}(\Omega)} \leq \epsilon.$
Therefore, $L(\phi_{\boldsymbol{\theta}})=2$, $W(\phi_{\boldsymbol{\theta}}) =r$ and $B_{\boldsymbol{\theta}}(\phi_{\boldsymbol{\theta}}) = C \cdot \epsilon^{-2}$. The proof can be done when $r=2$.
\end{proof}

Using the relation $xy = 4^{-1}[(x+y)^2-(x-y)^2]$, we can approximate the product function by neural networks and then further approximate any monomials $x_1\cdots x_d$.
\begin{lem}
\label{alem4.3}
For any $0<\epsilon<1$ and a positive constant $C$, there exists $\phi_{\boldsymbol{\theta}} \in \mathcal{N N}(W,L,B_{\boldsymbol{\theta}})$ where $W = 4$, $L=2$, $B_{\boldsymbol{\theta}}=C \cdot \epsilon^{-2}$ such that
$$
\|xy-\phi_{\boldsymbol{\theta}}(x,y)\|_{W^{k,\infty}([0,1]^2)}\leq \epsilon\, .
$$
\end{lem}
\begin{proof}
Lemma \ref{aprop4.2} yields that there exists a neural network $\phi_{\hat{\boldsymbol{\theta}}}$ with $L=2$, $W=2$ and $B_{\boldsymbol{\theta}}=C\cdot \epsilon^{-2}$ such that for all $k\in\{0,1,\cdots,j\}$, we have $\|x^2-\phi_{\hat{\boldsymbol{\theta}}}(x)\|_{W^{k,\infty}([0,1])}\leq \epsilon,
$ where $$\phi_{\hat{\boldsymbol{\theta}}}(x)= \frac{C^{2}}{\epsilon^{2} \rho^{(2)}(x_{0})}\bigg[\rho(x_0)-2\rho\Big(x_0-\frac{\epsilon}{C}x\Big)+\rho\Big(x_0-\frac{2\epsilon}{C}x\Big)\bigg]\, .$$

Then, we construct a neural network which implements an approximate multiplication via the polarization identity $xy = 4^{-1}[(x+y)^2-(x-y)^2]$ for $x,y\in\mathbb{R}$. In detail, we define the neural network $\boldsymbol{\theta}:=((\boldsymbol{A}_{0}, \boldsymbol{b}_{0}), (\boldsymbol{A}_{1}, \boldsymbol{b}_{1}))$ with
\begin{align*}
\boldsymbol{A}_{0} &:=\left(\!\!\begin{array}{rr}
-\frac{\epsilon}{C} & -\frac{\epsilon}{C}\\
-\frac{2 \epsilon}{C} & -\frac{2 \epsilon}{C}\\
-\frac{\epsilon}{C} & \frac{\epsilon}{C}\\
-\frac{2 \epsilon}{C} & \frac{2 \epsilon}{C}
\end{array}\right) \in \mathbb{R}^{4\times 2}\, , ~~~~~~~~~~~~~~~~~~\m
\boldsymbol{b}_{0} :=(x_{0}, x_{0}, x_{0}, x_{0})^{\rm T} \in \mathbb{R}^{4}\, , \\
\boldsymbol{A}_{1} &:=\cdot \frac{C^{2}}{4\epsilon^{2} \rho^{(2)}(x_{0})}(-2, 1, 2, -1) \in \mathbb{R}^{1\times 4}\, , \qquad
\boldsymbol{b}_{1} :=0 \in \mathbb{R}\, ,
\end{align*}
which fulfills for all $(x,y)\in \mathbb{R}^2$ that
\begin{align*}
    \phi_{\boldsymbol{\theta}}(x,y)=\,&\frac{1}{4}\cdot \frac{C^{2}}{\epsilon^{2} \rho^{(2)}(x_{0})}\Big[-2\rho\Big(x_0-\frac{\epsilon}{C}(x+y)\Big)+\rho\Big(x_0-\frac{2\epsilon}{C}(x+y)\Big)\\
    &+2\rho\Big(x_0-\frac{\epsilon}{C}(x-y)\Big)-\rho\Big(x_0-\frac{2\epsilon}{C}(x-y)\Big)\Big]\\
    =\,&\frac{1}{4}\big(\phi_{\hat{\boldsymbol{\theta}}}(x+y)-\phi_{\hat{\boldsymbol{\theta}}}(x-y)\big)\, .
\end{align*}
Therefore, for any $\epsilon > 0$, we have $\|xy - \phi_{\boldsymbol{\theta}}(x,y)\|_{W^{k,\infty}([0,1]^2)} \leq 4^{-1} \cdot 2\epsilon = 2^{-1}\epsilon < \epsilon$. Here, $W = 4$, $L = 2$, and $B_{\boldsymbol{\theta}} = C \cdot \epsilon^{-2}$.  
\end{proof}

By rescaling, we have the following modification of Lemma \ref{alem4.3}.
\begin{lem}
\label{alem4.4}
For any $0<\epsilon<\frac{1}{(b-a)^2}$, $a,b\in \mathbb{R}$ with $a<b$ and a positive constant $C$, there exists $\phi_{\boldsymbol{\theta}} \in \mathcal{N N}(W,L,B_{\boldsymbol{\theta}})$ where $W = 4$, $L=3$, $B_{\boldsymbol{\theta}}=C\cdot \epsilon^{-2}$ such that
$$
\|x y-\phi_{\boldsymbol{\theta}}(x, y)\|_{W^{k,\infty}([a,b]^2)} \leq (b-a)^2\epsilon\, .
$$
\end{lem}
\begin{proof}
By Lemma \ref{alem4.3}, there exists  $\hat{\phi}_{\boldsymbol{\theta}} \in \mathcal{N N}(\hat{W},\hat{L},\hat{B}_{\boldsymbol{\theta}})$ with $\hat{W}= 4$, $\hat{L}=2$, and $\hat{B}_{\boldsymbol{\theta}}=C\cdot \epsilon^{-2}$ such that $\|\hat{x} \hat{y}-\hat{\phi}_{\boldsymbol{\theta}}(\hat{x}, \hat{y})\|_{W^{k,\infty}([0,1]^2)} \leq \epsilon$. By setting $x=a+(b-a)\hat{x}$ and $y=a+(b-a)\hat{y}$ for any $\hat{x},\hat{y}\in[0,1]$, we define the following network $\phi_{\boldsymbol{\theta}}$
$$
\phi_{\boldsymbol{\theta}}=(b-a)^2\hat{\phi}_{\boldsymbol{\theta}}\Big(\frac{x-a}{b-a}, \frac{y-a}{b-a}\Big)+a(x-a)+a(y-a)+a^2\, .
$$
Note that $a(x-a)+a(y-a)$ is positive. Hence, the width of $\phi_{\boldsymbol{\theta}}$ can be the same as $\hat{\phi}_{\boldsymbol{\theta}}$, $W = 4$ and $L=\hat{L}+1=3$. By $$xy=(b-a)^2\Big(\frac{x-a}{b-a}\cdot \frac{y-a}{b-a}\Big) + a(x-a)+a(y-a)+a^2\, ,$$ we have
\begin{align*}
&\|\phi_{\boldsymbol{\theta}}(x,y)-xy\|_{W^{k,\infty}([a,b]^2)}\\
= & (b-a)^2 \bigg\|\hat{\phi}_{\boldsymbol{\theta}}\bigg(\frac{x-a}{b-a}, \frac{y-a}{b-a}\bigg)-\bigg(\frac{x-a}{b-a}\cdot \frac{y-a}{b-a}\bigg)\bigg\|_{W^{k,\infty}([a,b]^2)} \leq (b-a)^2 \epsilon\, .
\end{align*}
\end{proof}

We can then obtain the following lemma through induction.
\begin{lem}
\label{alem4.5}
  Let $d\geq 2$ and $C$ be a positive constant. For some sufficiently small $\epsilon^{*}>0$ and any $0<\epsilon < \epsilon^{*}$, there exists $\phi_{\boldsymbol{\theta}} \in \mathcal{N N}(W, L, B_{\boldsymbol{\theta}})$ where $W = 2^{\lceil \log_2 d \rceil+1}$, $L=\lceil \log_2 d \rceil+1$, $B_{\boldsymbol{\theta}}=C\cdot \epsilon^{-2}$ such that
$$
\|x_1\cdots x_d-\phi_{\boldsymbol{\theta}}(\boldsymbol x)\|_{W^{k,\infty}([0,1]^{d})}\leq C(d)\epsilon\, , \quad \boldsymbol x=(x_1, \cdots, x_d)^{\top} \in [0,1]^{d}\, .
$$
\end{lem}
\begin{proof}
We first consider the case $d=2^{\kappa}$ for some ${\kappa}\in \mathbb{N}$. For ${\kappa}=1$, by Lemma \ref{alem4.3}, there exists $\phi_{\boldsymbol{\theta}_1}\in \mathcal{N N}(4,2,C\epsilon^{-2})$ such that $|x_1x_2-\phi_{\boldsymbol{\theta}_1}(x_1,x_2)|_{W^{k,\infty}([0,1]^2)}\leq \epsilon$. For ${\kappa}=2$, by some simple calculations, we can approximate $x_1  x_2 \mm x_3\mm x_4$ by $\phi_{\boldsymbol{\theta}_2}(x_1,x_2,x_3,x_4)$, where
\begin{align*}
    \phi_{\boldsymbol{\theta}_2}(x_1,\ldots,x_4) = &\frac{C^{2}}{4\epsilon^{2} \rho^{(2)}(x_{0})}\Big[-\rho\Big(x_0-\frac{2\epsilon}{C}(x_1x_2-x_3x_4)\Big)+\rho\Big(x_0-\frac{2\epsilon}{C}(x_1x_2+x_3x_4)\Big)\\
    &+ 2\rho\Big(x_0-\frac{\epsilon}{C}(x_1x_2-x_3x_4)\Big)-2\rho\Big(x_0-\frac{\epsilon}{C}(x_1x_2+x_3x_4)\Big)\Big]\, .
\end{align*}
From Lemma \ref{alem4.3} we have that
$$
    \|x_1x_2-\phi_{\boldsymbol{\theta}_1}(x_1,x_2)\|_{W^{k,\infty}([0,1]^2)} \leq \epsilon\, ,\\
 $$ 
\begin{align*}
\phi_{\boldsymbol{\theta}_1}(x_1,x_2) = \frac{1}{4}&\cdot \frac{C^{2}}{\epsilon^{2} \rho^{(2)}(x_{0})}\Big[-2\rho\Big(x_0-\frac{\epsilon}{C}(x_1+x_2)\Big)+\rho\Big(x_0-\frac{2\epsilon}{C}(x_1+x_2)\Big)\\
    &+2\rho\Big(x_0-\frac{\epsilon}{C}(x_1-x_2)\Big)-\rho\Big(x_0-\frac{2\epsilon}{C}(x_1-x_2)\Big)\Big]\, ,
\end{align*}
and the same is true for $x_3x_4$, then when $4\epsilon<1$ the approximation error is
\begin{align*}
    &\|x_1x_2 \mm x_3 \mm x_4 - \phi_{\boldsymbol{\theta}_2}(x_1,x_2,x_3,x_4)\|_{W^{k,\infty}([0,1]^4)} \\
    \leq& \|x_1x_2 \mm x_3 \mm x_4-\phi_{\boldsymbol{\theta}_1}(x_1,x_2)\phi_{\boldsymbol{\theta}_1}(x_3,x_4)\|_{W^{k,\infty}([0,1]^4)}\\ & +\|\phi_{\boldsymbol{\theta}_1}(x_1,x_2)\phi_{\boldsymbol{\theta}_1}(x_3,x_4)-\phi_{\boldsymbol{\theta}_1}(\phi_{\boldsymbol{\theta}_1}(x_1,x_2),\phi_{\boldsymbol{\theta}_1}(x_3,x_4))\|_{W^{k,\infty}([0,1]^4)}\\
    \leq& \|x_1x_2 \mm x_3 \mm x_4-x_1x_2\phi_{\boldsymbol{\theta}_1}(x_3,x_4)\|_{W^{k,\infty}([0,1]^4)}\\
    & +\|x_1x_2\phi_{\boldsymbol{\theta}_1}(x_3,x_4)-\phi_{\boldsymbol{\theta}_1}(x_1,x_2)\phi_{\boldsymbol{\theta}_1}(x_3,x_4)\|_{W^{k,\infty}([0,1]^4)}\\
    & +\|\phi_{\boldsymbol{\theta}_1}(x_1,x_2)\phi_{\boldsymbol{\theta}_1}(x_3,x_4)-\phi_{\boldsymbol{\theta}_1}(\phi_{\boldsymbol{\theta}_1}(x_1,x_2),\phi_{\boldsymbol{\theta}_1}(x_3,x_4))\|_{W^{k,\infty}([0,1]^4)}\\
    \leq & \epsilon+ \epsilon(1+\epsilon)+\epsilon(1+2\epsilon)^2 \leq 4\epsilon\, .
\end{align*}
Therefore, we can define the neural network $\boldsymbol{\theta}_2:=((\boldsymbol{A}_{0},\boldsymbol{b}_{0}),(\boldsymbol{A}_{1},\boldsymbol{b}_{1}),(\boldsymbol{A}_{2},\boldsymbol{b}_{2}))$ with
\begin{align*}
\boldsymbol{A}_{0} &:=\left(\begin{array}{rrrr}
-\frac{\epsilon}{C} & -\frac{\epsilon}{C}& 0 & 0\\
-\frac{2 \epsilon}{C} & -\frac{2 \epsilon}{C}& 0 & 0\\
-\frac{\epsilon}{C} & \frac{\epsilon}{C}& 0 & 0\\
-\frac{2 \epsilon}{C} & \frac{2 \epsilon}{C}& 0 & 0\\
0 & 0& -\frac{\epsilon}{C} & -\frac{\epsilon}{C} \\
0 & 0& -\frac{2 \epsilon}{C} & -\frac{2 \epsilon}{C}\\
0 & 0& -\frac{\epsilon}{C} & \frac{\epsilon}{C}\\
0 & 0& -\frac{2 \epsilon}{C} & \frac{2 \epsilon}{C}
\end{array}\right) \in \mathbb{R}^{8\times4}\, , \qquad \ \boldsymbol{b}_{0} :=\left(\begin{array}{c}
x_0\\
x_0\\
x_0\\
x_0\\
x_0\\
x_0\\
x_0\\
x_0\\
\end{array}\right)\, ,
\end{align*}
\begin{align*}
\boldsymbol{A}_{1} &:=\frac{C}{4\epsilon\rho^{(2)}(x_0)}\left(\begin{array}{rrrrrrrr}
2 & -1& -2 & 1& 2 & -1 & -2 & 1\\
4 & -2& -4 & 2& 4 & -2& -4& 2\\
2 & -1& -2 & 1& -2 & 1 & 2 & -1\\
4 & -2& -4 & 2&  -4& 2& 4& -2\\
\end{array}\right) \in \mathbb{R}^{4\times 8}\, , \ \boldsymbol{b}_{1} :=\left(\begin{array}{c}
x_0\\
x_0\\
x_0\\
x_0\\
\end{array}\right)\, , \\
\boldsymbol{A}_{2} &:=\frac{1}{4}\cdot \frac{C^{2}}{\epsilon^{2} \rho^{(3)}(x_{0})}(-2, 1, 2, -1) \in \mathbb{R}^{1\times 4}\, , \ \boldsymbol{b}_{2} :=0\, ,
\end{align*}
which satisfies our construction.
We have that $W(\phi_{\boldsymbol{\theta}_2})=2^3$, $L(\phi_{\boldsymbol{\theta}_2})=3$ and $B_{\boldsymbol{\theta}}(\phi_{\boldsymbol{\theta}_2})=C\cdot \epsilon^{-2}$.

Next, we inductively show that $\phi_{\boldsymbol{\theta}_{\kappa}} \in \mathcal{NN}(2^{{\kappa}+1}, {\kappa}+1, C\cdot \epsilon^{-2})$. It is obvious that the assertion is true for ${\kappa}=1$ and ${\kappa}=2$ by construction. Assume that the assertion is true for some ${\kappa} \in \mathbb{N}$, then $W(\phi_{\boldsymbol{\theta}_{\kappa}}) = 2^{{\kappa}+1}$, $L(\phi_{\boldsymbol{\theta}_{\kappa}}) = {\kappa}+1$, the neural network $\boldsymbol{\theta}_{\kappa} = ((\boldsymbol{A}_{0}, \boldsymbol{b}_{0}), \dots, (\boldsymbol{A}_{{\kappa}}, \boldsymbol{b}_{{\kappa}}))$ with $B_{\boldsymbol{\theta}}(\phi_{\boldsymbol{\theta}_{{\kappa}}}) = C\cdot \epsilon^{-2}$, and when $4^{{\kappa}}\epsilon<1$,
$$
\|x_1x_2\cdots\ x_{2^{\kappa}}-\phi_{\boldsymbol{\theta}_{\kappa}}(x_1,x_2,\cdots,x_{2^{\kappa}})\|_{W^{k,\infty}([0,1]^{2^{\kappa}})}\leq 4^{{\kappa}-1}\epsilon\, .
$$
Then, by our construction $$W(\phi_{\boldsymbol{\theta}_{{\kappa}+1}})= 2W(\phi_{\boldsymbol{\theta}_{\kappa}})=2^{{\kappa}+2}, L(\phi_{\boldsymbol{\theta}_{{\kappa}+1}})= L(\phi_{\boldsymbol{\theta}_{\kappa}})+1 ={\kappa}+2$$ and $B_{\boldsymbol{\theta}}(\phi_{\boldsymbol{\theta}_{{\kappa}+1}}) = C\cdot \epsilon^{-2}$,
which means $\phi_{\boldsymbol{\theta}_{{\kappa}+1}} \in \mathcal{N N}(2^{{\kappa}+2}, {\kappa}+2, C\cdot \epsilon^{-2}).$ And
$$
\|x_1\cdots\ x_{2^{{\kappa}+1}}-\phi_{\boldsymbol{\theta}_{{\kappa}+1}}(x_1,\cdots, x_{2^{{\kappa}+1}})\|_{W^{k,\infty}([0,1]^{2^{{\kappa}+1}})}\leq 4\cdot 4^{{\kappa}-1} \epsilon = 4^{{\kappa}}\epsilon\, .
$$
Hence the assertion is true for ${\kappa}+1$.

For general $d \geq 2$, we choose ${\kappa}=\lceil \log_2 d \rceil$, then $2^{{\kappa}-1}<d \leq 2^{\kappa}$. We  define the target function $\phi_{\boldsymbol{\theta}}$ by
$$
\phi_{\boldsymbol{\theta}}(\boldsymbol{x}):=\phi_{\boldsymbol{\theta}_{\kappa}}\Bigg(\Bigg(\!\begin{array}{c}
\boldsymbol{I}_{d} \\
\mathbf{0}_{\left(2^{{\kappa}}-d\right) \times d}
\end{array}\!\Bigg) \boldsymbol{x}+\Bigg(\!\begin{array}{c}
\mathbf{0}_{d \times 1} \\
\mathbf{1}_{\left(2^{\kappa}-d\right) \times 1}
\end{array}\!\Bigg)\Bigg)\, ,
$$
where $\boldsymbol{I}_{d}$ is $d \times d$ identity matrix, $\mathbf{0}_{p \times q}$ is $p \times q$ zero matrix and $\mathbf{1}_{(2^{{\kappa}}-d) \times 1}$ is all ones vector. In this case $W=W(\phi_{\boldsymbol{\theta}_{\kappa}})=2^{{\kappa}+1}=2^{\lceil \log_2 d \rceil+1}$, $L= L(\phi_{\boldsymbol{\theta}_{\kappa}}) = {\kappa}+1 = \lceil \log_2 d \rceil +1$ and $B_{\boldsymbol{\theta}}(\phi_{\boldsymbol{\theta}_{{\kappa}}})= C\cdot \epsilon^{-2}$. The approximation error is
$$
\|x_{1} \cdots x_{d}-\phi_{\boldsymbol{\theta}}(\boldsymbol{x})\|_{W^{k,\infty}([0,1]^{d})} = \|x_{1} \cdots x_{2^{\kappa}}-\phi_{\boldsymbol{\theta}}(\boldsymbol{x})\|_{W^{k,\infty}([0,1]^{d})}\leq 4^{{\kappa}-1} \epsilon< C(d)\epsilon\, .
$$
\end{proof}

In Lemma \ref{alem4.5}, we construct neural networks to approximate monomials. In order to approximate $$f_{N}=\sum_{\scriptscriptstyle \boldsymbol{m} \in\{0, \ldots, N\}^{d}}\sum_{\scriptscriptstyle |\boldsymbol{\alpha}|_1 \leq n-1} c_{\scalebox{0.7}{$f, \mmn \boldsymbol{m}, \mmn \boldsymbol{\alpha}$}} \Psi_{\raisebox{0.2ex}{$\hspace{-0.1em}{\scalebox{0.6}{$\boldsymbol{m}$}}$}}^{s} \mm {\boldsymbol{x}}^{\boldsymbol{\alpha}},$$ we first construct a neural network $\phi_{\boldsymbol{\theta}}^{\scalebox{0.6}{$\boldsymbol{m}, \mmn \boldsymbol{\alpha}$}}$ to approximate $\Psi_{\raisebox{0.2ex}{$\hspace{-0.1em}{\scalebox{0.6}{$\boldsymbol{m}$}}$}}^{s} \mm {\boldsymbol{x}}^{\boldsymbol{\alpha}}$.
\begin{lem}
\label{alem4.6}
Let $C$ be a positive number and $C(n,d)$ be a polynomial that depends on $n$ and $d$ and $s\geq 1$. For some sufficiently small $\epsilon^{*}>0$ and any $0<\epsilon < \epsilon^{*}$, there exists $\phi_{\boldsymbol{\theta}}^{\scalebox{0.6}{$\boldsymbol{m}, \mmn \boldsymbol{\alpha}$}} \in \mathcal{NN}(W,L,B_{\boldsymbol{\theta}})$ where
\begin{align*}
W & = 2^{\lceil \log_2 (d + |\boldsymbol{\alpha}|_1) \rceil + 1}\, , \\
L & = \lceil \log_2 (d + |\boldsymbol{\alpha}|_1) \rceil + 2\, , \\
B_{\boldsymbol{\theta}} & = \max\{3Ns, (3d+3/2)s, C\epsilon^{-2} \}\, ,
\end{align*}
such that
$$
\|\Psi_{\raisebox{0.2ex}{$\hspace{-0.1em}{\scalebox{0.6}{$\boldsymbol{m}$}}$}}^{s} \mm {\boldsymbol{x}}^{\boldsymbol{\alpha}} - \phi_{\boldsymbol{\theta}}^{\scalebox{0.6}{$\boldsymbol{m}, \mmn \boldsymbol{\alpha}$}}(\boldsymbol x)\|_{W^{k,\infty}([0,1]^{d})} \leq C(n,d)\epsilon\, , \quad \boldsymbol x = (x_1, \cdots, x_d)^{\top} \in [0,1]^{d}\, ,
$$
for all $\boldsymbol{\alpha} \in \mathbb{N}^{d}$ with $|\boldsymbol{\alpha}|_1 \leq n-1$, and for all $\boldsymbol{m} \in \{0, \cdots, N\}^d$.
\end{lem}
\begin{proof}
For $N, d \in \mathbb{N}$ and $\boldsymbol{m} \in\{0, \ldots, N\}^{d}$, we define $$\Psi_{\raisebox{0.2ex}{$\hspace{-0.1em}{\scalebox{0.6}{$\boldsymbol{m}$}}$}}^{s}(\boldsymbol{x}):=\prod_{l=1}^{d} \psi^{s}(3 N(x_{l}-N^{-1}m_{l})).$$ Denote by $\psi^{s}_{l}:=\psi^{s}(3 N(x_{l}-N^{-1}m_{l}))$, then $\Psi_{\raisebox{0.2ex}{$\hspace{-0.1em}{\scalebox{0.6}{$\boldsymbol{m}$}}$}}^{s} \mm {\boldsymbol{x}}^{\boldsymbol{\alpha}}= \prod_{l=1}^{d} \psi^{s}_{l} {\boldsymbol{x}}^{\boldsymbol{\alpha}}$. Define $$\boldsymbol x=(x_{1}, \cdots, x_{d})^{\top} \in [0,1]^{d}\, , \quad  \bar{\boldsymbol x}=(\psi^{s}_{1}, \cdots, \psi^{s}_{d},\  \underbrace{\cdots, x_i, \cdots}_{|\boldsymbol{\alpha}|_1}\ )^{\top} \in [0,1]^{d+|\boldsymbol{\alpha}|_1}\, .$$ As we have shown in Lemma \ref{alem4.5}, let ${\kappa}=\lceil \log_2 (d+|\boldsymbol{\alpha}|_1) \rceil$ we can construct a neural network  $$\bar{\boldsymbol{\theta}}_{\kappa} =((\boldsymbol{A}_{0},\boldsymbol{b}_{0}),\cdots,(\boldsymbol{A}_{{\kappa}},\boldsymbol{b}_{{\kappa}}))$$ with $$W(\phi_{\bar{\boldsymbol{\theta}}_{\kappa}}(\bar{\boldsymbol{x}}))=2^{{\kappa}+1} \, , \quad L(\phi_{\bar{\boldsymbol{\theta}}_{\kappa}}(\bar{\boldsymbol{x}}))={\kappa}+1\, , \quad K(\phi_{\bar{\boldsymbol{\theta}}_{\kappa}}(\bar{\boldsymbol{x}}))=C{\epsilon}^{-({\kappa}+1)} \, ,$$ which satisfies
$$
\bigg\|\prod_{l=1}^{d} \psi^{s}_{l} {\boldsymbol{x}}^{\boldsymbol{\alpha}}-\phi_{\bar{\boldsymbol{\theta}}_{\kappa}}(\bar{\boldsymbol{x}})\bigg\|_{W^{k,\infty}([0,1]^{d+|\boldsymbol{\alpha}|_1})}\leq 4^{{\kappa}}\epsilon\, .
$$
Then we need to replace the input $\bar{\boldsymbol{x}}$ with $\boldsymbol{x}$. Since $$\psi^{s}(x):= 2^{-1}[\rho(s(x+3 / 2))-\rho(s(x-3 / 2))],$$ $$\psi^{s}_{l}(\boldsymbol{x}):=\psi^{s}(3 N(x_{l}-N^{-1}m_{l})),$$ we can add one hidden layer to the neural network $\bar{\boldsymbol{\theta}}_{\kappa}$ before its first layer. Denote the new neural network as 
\[
\boldsymbol{\theta}_{{\kappa}}:=((\boldsymbol{A}_{0},\boldsymbol{b}_{0}),(\boldsymbol{A}_{1},\boldsymbol{b}_{1}),\cdots,(\boldsymbol{A}_{\ell},\boldsymbol{b}_{\ell}),\cdots, (\boldsymbol{A}_{{\kappa}+1},\boldsymbol{b}_{{\kappa}+1}))\, , \]
where $$\ell \in \{0,1,\ldots, {\kappa}+1\}.$$ The parameter matrix of each layer can be written as follows.
\begin{itemize}
\item $\ell=0$. \begin{align*}
\boldsymbol{A}_{0} :=\left(\begin{array}{cccc}
3Ns         & \cdots & 0 \\
3Ns         & \cdots & 0 \\
\vdots&  \ddots & \vdots\\
0           & \cdots & 3Ns\\
0           & \cdots & 3Ns\\
-\frac{\epsilon}{C}&\\
&\ddots &   \\
&&-\frac{\epsilon}{C} \\
0& \\
&\ddots \\
&&0 \\
\end{array}\right) \in \mathbb{R}^{(d+2^{\kappa})\times d}
\, , \qquad  \boldsymbol{b}_{0} :=\left(\setlength{\arraycolsep}{1.5pt}
\begin{array}{c}
-3m_1s+\frac{3}{2}s\\
-3m_1s-\frac{3}{2}s\\
\vdots\\
-3m_ds+\frac{3}{2}s\\
-3m_ds-\frac{3}{2}s\\
x_0\\
\vdots \\
x_0\\
1\\
\vdots\\
1\\
\end{array}\right)\in \mathbb{R}^{d+2^{\kappa}}\, .
\end{align*}
\item 
$\ell=1$.
$$
\boldsymbol{A}_{1} :=
\left(\!\setlength{\arraycolsep}{1.2pt}
\begin{array}{ccccccccccccccc}
 -\frac{\epsilon}{2C}   & \frac{\epsilon}{2C} & -\frac{\epsilon}{2C} & \frac{\epsilon}{2C}\\
 -\frac{\epsilon}{2C}   & \frac{\epsilon}{2C} & -\frac{\epsilon}{2C} & \frac{\epsilon}{2C}\\
 -\frac{\epsilon}{2C}   & \frac{\epsilon}{2C} & -\frac{\epsilon}{2C} & \frac{\epsilon}{2C}\\
 -\frac{\epsilon}{2C}   & \frac{\epsilon}{2C} & -\frac{\epsilon}{2C} & \frac{\epsilon}{2C}\\
 & & & & \ddots\\
 & & & & & \frac{1}{\rho^{(1)}(x_0)}&\frac{1}{\rho^{(1)}(x_0)}\\
 & & & & & \frac{2}{\rho^{(1)}(x_0)}&\frac{2}{\rho^{(1)}(x_0)}\\
 & & & & & \frac{1}{\rho^{(1)}(x_0)}&-\frac{1}{\rho^{(1)}(x_0)}\\
 & & & & & \frac{2}{\rho^{(1)}(x_0)}&-\frac{2}{\rho^{(1)}(x_0)}\\
 & & & & & & & \ddots\\
 & & & & & & & &-\frac{\epsilon}{C\rho(1)}&-\frac{\epsilon}{C\rho(1)}\\
 & & & & & & & &-\frac{2\epsilon}{C\rho(1)}&-\frac{2\epsilon}{C\rho(1)}\\
 & & & & & & & &-\frac{\epsilon}{C\rho(1)}&\frac{\epsilon}{C\rho(1)}\\
 & & & & & & & &-\frac{2\epsilon}{C\rho(1)}&\frac{2\epsilon}{C\rho(1)}
\end{array}\!\right)\, , \quad \boldsymbol{b}_{1} :=\left(\!\setlength{\arraycolsep}{1.5pt}
\begin{array}{c}
x_0\\
x_0\\
x_0\\
x_0\\
\vdots\\
x_0-\frac{2\rho(x_0)}{\rho^{(1)}(x_0)}\\
x_0-\frac{4\rho(x_0)}{\rho^{(1)}(x_0)}\\
x_0\\
x_0\\
\vdots \\
x_0\\
x_0\\
x_0 \\
x_0
\end{array}\!\right)\, .
$$

$\boldsymbol{A}_1 \in \mathbb{R}^{(2^{\kappa+1})\times(d+2^{\kappa})}\, , \qquad \boldsymbol{b}_{1} \in \mathbb{R}^{2^{{\kappa}+1}}. $

\
\item 
 $2\leq \ell \leq {\kappa}$.
\begin{align*}
&\boldsymbol{A}_{\ell} :=\frac{C}{4\epsilon\rho^{(2)}(x_0)}\left(\setlength{\arraycolsep}{1.5pt}
\begin{array}{rrrrrrrrrrrrrrrrr}
2 & -1& -2 & 1& 2 & -1 & -2 & 1\\
4 & -2& -4 & 2& 4 & -2& -4& 2\\
2 & -1& -2 & 1& -2 & 1 & 2 & -1\\
4 & -2& -4 & 2&  -4& 2& 4& -2\\
&&&&&&&&\ddots\\
&&&&&&&&&2 & -1& -2 & 1& 2 & -1 & -2 & 1\\
&&&&&&&&&4 & -2& -4 & 2& 4 & -2& -4& 2\\
&&&&&&&&&2 & -1& -2 & 1& -2 & 1 & 2 & -1\\
&&&&&&&&&4 & -2& -4 & 2&  -4& 2& 4& -2\\
\end{array}\right)\, ,\\
&\boldsymbol{A}_{\ell} \in \mathbb{R}^{(2^{{\kappa}-\ell+2})\times(2^{{\kappa}-\ell+3})}\, , \qquad \boldsymbol{b}_{\ell} :=(x_{0},\cdots , x_{0})^{\rm T} \in \mathbb{R}^{2^{{\kappa}-\ell+2}}\, . 
\end{align*}
\item 
The last layer.
\begin{align*}
&\boldsymbol{A}_{{\kappa}+1} :=\frac{1}{4}\cdot \frac{C^{2}}{\epsilon^{2} \rho^{(2)}(x_{0})}(-2, 1, 2, -1) \in \mathbb{R}^{1\times 4}\, , \qquad \boldsymbol{b}_{{\kappa}+1} :=0 \in \mathbb{R}\, .
\end{align*}
\end{itemize}
In conclusion, when $W = W(\phi_{\boldsymbol{\theta}_{\kappa}})=2^{{\kappa}+1}=2^{\lceil \log_2 (d+|\boldsymbol{\alpha}|_1) \rceil+1}$, $L = L(\phi_{\boldsymbol{\theta}_{\kappa}}) = {\kappa}+2 = \lceil \log_2 (d+|\boldsymbol{\alpha}|_1) \rceil +1$ and $B_{\boldsymbol{\theta}}(\phi_{\boldsymbol{\theta}_{{\kappa}}})= \max\{3Ns, (3d+3/2)s, C\epsilon^{-2} \}$. The approximation error is
$$
\|\Psi_{\raisebox{0.2ex}{$\hspace{-0.1em}{\scalebox{0.6}{$\boldsymbol{m}$}}$}}^{s} \mm {\boldsymbol{x}}^{\boldsymbol{\alpha}}-\phi_{\boldsymbol{\theta}_{\kappa}}(\boldsymbol{x})\|_{W^{k,\infty}([0,1]^{d})}\leq C(n,d)\epsilon\, ,
$$
where $C(n,d)$ is a polynomial which only depends on $n$ and $d$. Denote $\phi_{\boldsymbol{\theta}_{\kappa}}(\boldsymbol{x})$ as $\phi_{\boldsymbol{\theta}}^{\scalebox{0.6}{$\boldsymbol{m}, \mmn \boldsymbol{\alpha}$}}$, the proof is complete.
\end{proof}

Then we can construct the parallel neural network $\Phi_{\bar{\mathfrak{m}}, \bar{\boldsymbol{\theta}}}\in \mathcal{PNN}(\bar{\mathfrak{m}}, \bar{M}, \{\bar{W}, \bar{L}, B_{\bar{\boldsymbol{\theta}}}\})$ to approximate $f_{N}$.
\begin{thm}
\label{athm4.1}
Let $C$ be a positive number and $C(n,d)$ be a polynomial that depends on $n$ and $d$. For some sufficiently small $\epsilon^{*}>0$ and any $0<\epsilon < \epsilon^{*}$, there exists a neural network $\Phi_{\bar{\mathfrak{m}}, \bar{\boldsymbol{\theta}}}\in \mathcal{PNN}(\bar{\mathfrak{m}}, \bar{M}, \{\bar{W}, \bar{L}, B_{\bar{\boldsymbol{\theta}}}\})$ with 
\begin{align*}
\bar{\mathfrak{m}}&=(N+1)^dn \mm d^{ \mm n-1}\, ,\\
\bar{M} &= CN^{d/p}(N+1)^dn \mm d^{ \mm n-1}\, ,\\
\bar{W} &= 2^{\lceil \log_2 (d+|\boldsymbol{\alpha}|_1) \rceil+1}\, ,\\
\bar{L} &= \lceil \log_2 (d+|\boldsymbol{\alpha}|_1) \rceil+2\, ,\\
B_{\bar{\boldsymbol{\theta}}} &= \max\{3Ns, (3d+3/2)s, C\epsilon^{-2} \}\, ,
\end{align*}
such that
$$
\|f_N(\boldsymbol x)-\Phi_{\bar{\mathfrak{m}}, \bar{\boldsymbol{\theta}}}\|_{W^{k,\infty}([0,1]^{d})}\leq C(n,d)(N+1)^{d}\epsilon\, , \quad \boldsymbol x=(x_1, \cdots, x_d)^{\top} \in [0,1]^{d}\, ,
$$
for all $\boldsymbol{\alpha} \in \mathbb{N}^{d}$ with $|\boldsymbol{\alpha}|_1  \leq n-1$.
\end{thm}
\begin{proof}
In Lemma \ref{alem4.6}, $\Psi_{\raisebox{0.2ex}{$\hspace{-0.1em}{\scalebox{0.6}{$\boldsymbol{m}$}}$}}^{s} \mm {\boldsymbol{x}}^{\boldsymbol{\alpha}}$ can be approximated by a neural network $\phi_{\bar{\boldsymbol{\theta}}}^{\scalebox{0.6}{$\boldsymbol{m}, \mmn \boldsymbol{\alpha}$}} \in \mathcal{N N}(\bar{W}, \bar{L}, B_{\bar{\boldsymbol{\theta}}})$ with $\bar{W} = 2^{\lceil \log_2 (d+|\boldsymbol{\alpha}|_1) \rceil+1}$, $\bar{L} = \lceil \log_2 (d+|\boldsymbol{\alpha}|_1) \rceil+2$ and $B_{\bar{\boldsymbol{\theta}}} = \max\{3Ns, (3d+3/2)s, C\epsilon^{-2} \}$. According to Proposition \ref{aprop4.1}, we can approximate any $f\in W^{n,p}([0,1]^{d})$ by approximating
$$
f_N(\boldsymbol{x})= \sum_{\boldsymbol{m} \in\{0, \ldots, N\}^{d}}\sum_{|\boldsymbol{\alpha}|_1 \leq n-1} c_{\scalebox{0.7}{$f, \mmn \boldsymbol{m}, \mmn \boldsymbol{\alpha}$}} \Psi_{\raisebox{0.2ex}{$\hspace{-0.1em}{\scalebox{0.6}{$\boldsymbol{m}$}}$}}^{s} \mm {\boldsymbol{x}}^{\boldsymbol{\alpha}}\, .
$$
In Proposition \ref{aprop4.1},
$$
|c_{\scalebox{0.7}{$f, \mmn \boldsymbol{m}, \mmn \boldsymbol{\alpha}$}}|\leq C \|\tilde{f}\|_{W^{n,p}(\Omega_{\boldsymbol{m},N})}N^{d/p}\leq C \|f\|_{W^{n,p}([0,1]^{d})}N^{d/p}\leq CN^{d/p}\, .
$$
Observe that $\sum_{|\boldsymbol{\alpha}|_1\leq n-1}1 = \sum_{j=0}^{n-1}\sum_{|\boldsymbol{\alpha}|_1=j}1\leq \sum_{j=0}^{n-1}d^{\mm j}\leq nd^{\mm n-1}$.
Assume $f_{N}$ can be approximate by neural network $\Phi_{\bar{\mathfrak{m}}, \bar{\boldsymbol{\theta}}}\in \mathcal{PNN}(\bar{\mathfrak{m}}, \bar{M}, \{\bar{W}, \bar{L}, B_{\bar{\boldsymbol{\theta}}}\})$,
$$
\Phi_{\bar{\mathfrak{m}}, \bar{\boldsymbol{\theta}}}(\boldsymbol{x}) = \sum_{\boldsymbol{m} \in\{0, \ldots, N\}^{d}}\sum_{|\boldsymbol{\alpha}|_1 \leq n-1} c_{\scalebox{0.7}{$f, \mmn \boldsymbol{m}, \mmn \boldsymbol{\alpha}$}} \phi_{\bar{\boldsymbol{\theta}}}^{\scalebox{0.6}{$\boldsymbol{m}, \mmn \boldsymbol{\alpha}$}}(\boldsymbol{x}) = \sum_{k = 1}^{\bar{\mathfrak{m}}} c_{k} \phi^{k}_{\bar{\boldsymbol{\theta}}}(\boldsymbol{x})\, ,
$$
where $\bar{\mathfrak{m}}=(N+1)^dn \mm d^{\mm n-1}$ and $\bar{M} =CN^{d/p}(N+1)^dn \mm d^{ \mm n-1}$, then
\begin{align*}
\big\|f_{N}(\boldsymbol{x}) - \Phi_{\bar{\mathfrak{m}}, \bar{\boldsymbol{\theta}}}(\boldsymbol{x})\big\|&_{W^{k,p}([0,1]^{d})}=\Big\|f_{N}(\boldsymbol{x}) - \sum_{\boldsymbol{m} \in\{0, \ldots, N\}^{d}}\sum_{|\boldsymbol{\alpha}|_1 \leq n-1} c_{\scalebox{0.7}{$f, \mmn \boldsymbol{m}, \mmn \boldsymbol{\alpha}$}} \mm \phi_{\bar{\boldsymbol{\theta}}}^{\scalebox{0.6}{$\boldsymbol{m}, \mmn \boldsymbol{\alpha}$}}(\boldsymbol{x})\Big\|_{W^{k,p}([0,1]^{d})}\\
& \leq \sum_{\boldsymbol{m} \in\{0, \ldots, N\}^{d}}\sum_{|\boldsymbol{\alpha}|_1 \leq n-1} |c_{\scalebox{0.7}{$f, \mmn \boldsymbol{m}, \mmn \boldsymbol{\alpha}$}}|\cdot \big\|\Psi_{\raisebox{0.2ex}{$\hspace{-0.1em}{\scalebox{0.6}{$\boldsymbol{m}$}}$}}^{s} \mm {\boldsymbol{x}}^{\boldsymbol{\alpha}}-\phi_{\bar{\boldsymbol{\theta}}}^{\scalebox{0.6}{$\boldsymbol{m}, \mmn \boldsymbol{\alpha}$}}(\boldsymbol{x})\big\|_{W^{k,p}([0,1]^{d})}\\
&\leq C(n,d)N^{d/p}\mm 2^d \mm n \mm d^{\mm n-1}(N+1)^d\epsilon\, .
\end{align*}
When $p = \infty$, we have $\big\|f_{N}(\boldsymbol{x}) - \Phi_{\bar{\mathfrak{m}}, \bar{\boldsymbol{\theta}}}(\boldsymbol{x})\big\|_{W^{k,\infty}([0,1]^{d})}\leq C(n,d)(N+1)^d\epsilon$.
\end{proof}

\subsubsection{Approximation bound for neural networks}
\label{app4}
In this section, by combining Proposition \ref{aprop4.1} with Theorem \ref{athm4.1}, we can derive an upper bound for the approximation error.

\begin{thm}
\label{athm4.2}
Let $n, k, d, \bar{\mathfrak{m}} \in \mathbb{N}$, $n\geq k+1$ and $1\leq p\leq \infty$, $C$ be a positive number and $C(n,d)$ be a polynomial that depends on $n$ and $d$. Let $f \in \mathcal{F}_{n, d, p}$. For some sufficiently small $\epsilon^{*}>0$ and any $0<\epsilon < \epsilon^{*}$, there exists a neural network $\Phi_{\bar{\mathfrak{m}}, \bar{\boldsymbol{\theta}}} \in \mathcal{PNN}(\bar{\mathfrak{m}}, \bar{M}, \{\bar{W}, \bar{L}, B_{\bar{\boldsymbol{\theta}}}\})$ with $\bar{\mathfrak{m}}=C_1(n, d, k)\epsilon^{-\frac{d}{n-k-\mu k}}$, $\bar{M} = C_2(n, d, k)\epsilon^{-\frac{d(p+1)}{(n-k-\mu k)p}}$ , $\bar{W} = 2^{\lceil \log_2 (d+|\boldsymbol{\alpha}|_1) \rceil+1}$, $\bar{L}= \lceil \log_2 (d+|\boldsymbol{\alpha}|_1) \rceil+2$, and $B_{\bar{\boldsymbol{\theta}}} = C_3(n,d,k)\epsilon^{-2-\frac{2d}{n-k-\mu k}}$ such that
$$
\big\|f-\Phi_{\bar{\mathfrak{m}}, \bar{\boldsymbol{\theta}}}\big\|_{W^{k, p}([0,1]^{d})} \leq \epsilon\, ,
$$
where $\mu$ is an arbitrarily small positive number.
\end{thm}
\begin{proof}
We divide the proof into two steps: First, we approximate the function $f$ by a sum of localized polynomials. Afterwards, we proceed by approximating this sum by a neural network. For the first step, by Proposition \ref{aprop4.1}, we can set
$$
N:=\bigg\lceil\Big(\frac{\tilde{\epsilon}}{2C}\Big)^{-1 /(n-k-\mu k)}\bigg\rceil \quad \text { and } \quad s:=N^{\mu}\, ,
$$
then there exist polynomials $p_{\scalebox{0.7}{$\boldsymbol{m}$}}(\boldsymbol{x})=\sum_{|\boldsymbol{\alpha}|_1 \leq n-1} c_{\scalebox{0.7}{$f, \mmn \boldsymbol{m}, \mmn \boldsymbol{\alpha}$}} {\boldsymbol{x}}^{\boldsymbol{\alpha}}$ for $\boldsymbol{m} \in\{0, \cdots, N\}^{d}$ such that
$$
\bigg\|f-\sum_{\boldsymbol{m} \in\{0, \cdots, N\}^{d}} \Psi_{\raisebox{0.2ex}{$\hspace{-0.1em}{\scalebox{0.6}{$\boldsymbol{m}$}}$}}^{s} \mm p_{\scalebox{0.7}{$\boldsymbol{m}$}}\bigg\|_{W^{k, p}([0,1]^{d})} \leq C \bigg(\frac{1}{N}\bigg)^{n-k-\mu k} \leq C \cdot \frac{\tilde{\epsilon}}{2C}=\frac{\tilde{\epsilon}}{2}\, .
$$
Secondly, combining with Theorem \ref{athm4.1} and plugging in $N$, $s$ calculated in the first step, we can get that when $\big\|f_{N}(\boldsymbol{x}) - \Phi_{\bar{\mathfrak{m}}, \bar{\boldsymbol{\theta}}}(\boldsymbol{x})\big\|_{W^{k,\infty}([0,1]^{d})}\leq C(n,d)(N+1)^d\epsilon = 2^{-1}\tilde{\epsilon}$, $\epsilon = C \m \tilde{\epsilon}^{\mm 1+\frac{d}{n-k-\mu k}}$. Then the total approximation error is
$$
\big\|f-\Phi_{\bar{\mathfrak{m}}, \bar{\boldsymbol{\theta}}}\big\|_{W^{k, p}([0,1]^{d})} \leq \big\|f-f_N\big\|_{W^{k, p}([0,1]^{d})}+ \big\|f_N-\Phi_{\bar{\mathfrak{m}}, \bar{\boldsymbol{\theta}}}\big\|_{W^{k,p}([0,1]^{d})}\leq\tilde{\epsilon}\, ,
$$
and $\bar{\mathfrak{m}}=C_1(n, d, k)\tilde{\epsilon}^{-\frac{d}{n-k-\mu k}}$, $\bar{M} = C_2(n, d, k)\tilde{\epsilon}^{-\frac{d(p+1)}{(n-k-\mu k)p}}$ , $\bar{W} = 2^{\lceil \log_2 (d+|\boldsymbol{\alpha}|_1) \rceil+1}$, $\bar{L}= \lceil \log_2 (d+|\boldsymbol{\alpha}|_1) \rceil+2$, and $B_{\bar{\boldsymbol{\theta}}} = C_3(n,d,k)\tilde{\epsilon}^{-2-\frac{2d}{n-k-\mu k}}$.
\end{proof}

\subsection{Detailed statistical error analysis}
\label{proof of sta}
In this section, following the approach of \cite{kohler2023opt} and \cite{jiao2023error}, and for the sake of completeness, we present a specific proof for controlling the upper bound of the statistical error. The neural network space is denoted as $\mathcal{PNN}(\mathfrak{m}, M, \{W, L, B_{\boldsymbol{\theta}}\})$, where each element in this space is a parallel neural network composed of $\mathfrak{m}$ sub-neural networks with width $W$, depth $L$, and the uniform upper bound of the weights $B_{\boldsymbol{\theta}}$. Moreover, each element satisfies $\sum_{k=1}^{\mathfrak{m}} |c_k| \leq M$.  Recall the definition of statistical error in Section \ref{sta}, 
$$
\mathcal{E}_{sta} := \sup_{u\in\mathcal{PNN}}\big|\mathcal{L}(u)-\widehat{\mathcal{L}}(u)\big|\, .
$$
The task is to control $\mathcal{E}_{sta}$ with high probability, which involves initially controlling its expectation, $\mathbb{E}[\mathcal{E}_{sta}]$. Our proof is organized into four steps. First, through careful computation and the application of the triangle inequality, we provide a detailed decomposition of $\mathbb{E}[\mathcal{E}_{sta}]$. These components will be analyzed individually in the subsequent sections.

\begin{lem}
\label{lem4.1}
The expectation of $\mathcal{E}_{sta}$ w.r.t. the sample points can be decomposed as follows
\begin{align*}
\mathbb{E}[\mathcal{E}_{sta}]&=\underset{\{X_p\}_{p=1}^{N_{\rm in}},\{Y_p\}_{p=1}^{N_{b}}}{\mathbb{E}}\bigg[\sup_{u_{\mathfrak{m}, \boldsymbol{\theta}}\in\mathcal{PNN}}\big|\mathcal{L}(u_{\mathfrak{m}, \boldsymbol{\theta}})-\widehat{\mathcal{L}}(u_{\mathfrak{m}, \boldsymbol{\theta}})\big|\bigg] \\
&\leq \sum_{i=1}^{4}\underset{\{X_p\}_{p=1}^{N_{\rm in}},\{Y_p\}_{p=1}^{N_{b}}}{\mathbb{E}}\sup_{u_{\mathfrak{m}, \boldsymbol{\theta}}\in\mathcal{PNN}}\big|\mathcal{L}_i(u_{\mathfrak{m}, \boldsymbol{\theta}})-\widehat{\mathcal{L}}_i(u_{\mathfrak{m}, \boldsymbol{\theta}})\big| = \sum_{i=1}^{4} \mathbb{E}[\mathcal{E}^{i}_{sta}]\, .
\end{align*}
where
\begin{align*}
	&\mathcal{L}_1(u_{\mathfrak{m}, \boldsymbol{\theta}})=\frac{|\Omega|}{2}\mathbb{E}_{X\sim U(\Omega)}\|\nabla u_{\mathfrak{m}, \boldsymbol{\theta}}(X)\|^2\, ,
	&\mathcal{L}_2(u_{\mathfrak{m}, \boldsymbol{\theta}})&=\frac{|\Omega|}{2}\mathbb{E}_{X\sim U(\Omega)}\omega(X)u_{\mathfrak{m}, \boldsymbol{\theta}}^2(X)\, ,\\
	&\mathcal{L}_3(u_{\mathfrak{m}, \boldsymbol{\theta}})=-|\Omega|\mathbb{E}_{X\sim U(\Omega)}h(X)u_{\mathfrak{m}, \boldsymbol{\theta}}(X)\, ,
	&\mathcal{L}_4(u_{\mathfrak{m}, \boldsymbol{\theta}})&=-|\partial\Omega|\mathbb{E}_{Y\sim U(\partial\Omega)}g(Y)Tu_{\mathfrak{m}, \boldsymbol{\theta}}(Y)\, .
\end{align*}
	and $\widehat{\mathcal{L}}_i(u_{\mathfrak{m}, \boldsymbol{\theta}})$ is the discrete version of $\mathcal{L}_i(u_{\mathfrak{m}, \boldsymbol{\theta}})$, for example,
	\begin{equation*}
	\widehat{\mathcal{L}}_1(u_{\mathfrak{m}, \boldsymbol{\theta}})=\frac{|\Omega|}{2N_{\rm in}}\sum_{p=1}^{N_{\rm in}}\|\nabla u_{\mathfrak{m}, \boldsymbol{\theta}}(X_p)\|^2\, .
	\end{equation*}
\end{lem}

To facilitate subsequent analysis, we first introduce the definition of the Rademacher complexity, which will aid us in bounding $\mathbb{E}[\mathcal{E}^{i}_{sta}]$ through the technique of symmetrization.

\begin{defn}
	Two types of Rademacher complexity of  function class $\mathcal{F}$ associate with random sample $\{X_k\}_{k=1}^{N}$ are defined as
	\begin{align*}
		\mathfrak{R}_{N}(\mathcal{F}) & = \mathbb{E}_{\{X_k,\sigma_k\}_{k=1}^{N}}\bigg[\sup_{u\in \mathcal{F}}\frac{1}{N}\sum_{k=1}^N \sigma_k u(X_k)\bigg]\, ,\\
  \hat{\mathfrak{R}}_{N}(\mathcal{F}) & = \mathbb{E}_{\{X_k,\sigma_k\}_{k=1}^{N}}\bigg[\sup_{u\in \mathcal{F}}\frac{1}{N}\Big|\sum_{k=1}^N \sigma_k u(X_k)\Big|\bigg]\, .
	\end{align*}
 where,   $\{\sigma_k\}_{k=1}^N$ are $N$ i.i.d  Rademacher variables with $\mathbb{P}(\sigma_k = 1) = \mathbb{P}(\sigma_k = -1) = \frac{1}{2}.$
\end{defn}

For Rademacher complexity $\mathfrak{R}_N(\mathcal{F})$, we have following two structural results.
\begin{prop} \label{structural result of Rademacher}
	Assume that $\omega:\Omega\to\mathbb{R}$ and $|\omega(x)|\leq\mathcal{B}$ for all $x\in\Omega$, then for any function class $\mathcal{F}$, there holds
	\begin{equation*}
		\mathfrak{R}_N(\omega\cdot\mathcal{F})\leq\mathcal{B}\mathfrak{R}_N(\mathcal{F})\, ,
	\end{equation*}
	where $\omega\cdot\mathcal{F}:=\{\bar{u}:\bar{u}(x)=\omega(x)u(x),u\in\mathcal{F}\}$.
\end{prop}
\begin{proof}
	\begin{align*}
		&\mathfrak{R}_N(\omega\cdot\mathcal{F})
		=\frac{1}{N}\mathbb{E}_{\{X_k,\sigma_k\}_{k=1}^{N}}\sup_{u\in\mathcal{F}}\sum_{k=1}^{N}\sigma_k \m \omega(X_k)u(X_k)\\
		&=\frac{1}{2N}\mathbb{E}_{\{X_k\}_{k=1}^{N}}\mathbb{E}_{\{\sigma_k\}_{k=2}^{N}}\sup_{u\in\mathcal{F}}\bigg[\omega(X_1)u(X_1)+\sum_{k=2}^{N}\sigma_k \m \omega(X_k)u(X_k)\bigg]\\
		&\quad\ +\frac{1}{2N}\mathbb{E}_{\{X_k\}_{k=1}^{N}}\mathbb{E}_{\{\sigma_k\}_{k=2}^{N}}\sup_{u\in\mathcal{F}}\bigg[-\omega(X_1)u(X_1)+\sum_{k=2}^{N}\sigma_k \m \omega(X_k)u(X_k)\bigg]\\
		&=\frac{1}{2N}\mathbb{E}_{\{X_k\}_{k=1}^{N}}\mathbb{E}_{\{\sigma_k\}_{k=2}^{N}}\\
		&\quad\sup_{u,u^{\prime}\in\mathcal{F}}\bigg[\omega(X_1)[u(X_1)-u^{\prime}(X_1)]+\sum_{k=2}^{N}\sigma_k \m \omega(X_k)u(X_k)+\sum_{k=2}^{N}\sigma_k \m \omega(X_k)u^{\prime}(X_k)\bigg]\\
		&\leq\frac{1}{2N}\mathbb{E}_{\{X_k\}_{k=1}^{N}}\mathbb{E}_{\{\sigma_k\}_{k=2}^{N}}\\
		&\quad\sup_{u,u^{\prime}\in\mathcal{F}}\bigg[\mathcal{B}|u(X_1)-u^{\prime}(X_1)|+\sum_{k=2}^{N}\sigma_k \m \omega(X_k)u(X_k)+\sum_{k=2}^{N}\sigma_k \m \omega(X_k)u^{\prime}(X_k)\bigg]\\
		&=\frac{1}{2N}\mathbb{E}_{\{X_k\}_{k=1}^{N}}\mathbb{E}_{\{\sigma_k\}_{k=2}^{N}}\\
		&\quad\sup_{u,u^{\prime}\in\mathcal{F}}\bigg[\mathcal{B}[u(X_1)-u^{\prime}(X_1)]+\sum_{k=2}^{N}\sigma_k \m \omega(X_k)u(X_k)+\sum_{k=2}^{N}\sigma_k \m \omega(X_k)u^{\prime}(X_k)\bigg]\\
		&=\frac{1}{N}\mathbb{E}_{\{X_k,\sigma_k\}_{k=1}^{N}}\sup_{u\in\mathcal{F}}\bigg[\sigma_1\mathcal{B}u(X_1)+\sum_{k=2}^{N}\sigma_k \m \omega(X_k)u(X_k)\bigg]\\
		&\leq\cdots\leq\frac{\mathcal{B}}{N}\mathbb{E}_{\{X_k,\sigma_k\}_{k=1}^{N}}\sup_{u\in\mathcal{F}}\sum_{k=1}^{N}\sigma_k \m u(X_k)=\mathcal{B}\mm \mathfrak{R}_{N}(\mathcal{F})\, .
	\end{align*}
\end{proof}

\begin{prop}
\label{Talagrand}
Let $\Phi: \mathbb{R} \rightarrow \mathbb{R}$ be a $\lambda$-Lipschitz. Then, for any hypothesis set $H$ of real-valued functions, the following inequality holds:
$$
\mathfrak{R}_N(\Phi \circ H) \leq \lambda \mathfrak{R}_N(H)\, .
$$
\end{prop}
\begin{proof}
First we fix a sample $S=(x_1, \ldots, x_N)$, then, by definition,
\begin{align*}
\mathfrak{R}_N(\Phi \circ H) &=\frac{1}{N} \underset{\boldsymbol{\sigma}}{\mathbb{E}}\bigg[\sup _{h \in H} \sum_{i=1}^N \sigma_i(\Phi \circ h)(x_i)\bigg] \\
&=\frac{1}{N} \underset{\sigma_1, \ldots, \sigma_{N-1}}{\mathbb{E}}\bigg[\underset{\sigma_N}{\mathbb{E}}\Big[\sup _{h \in H} u_{N-1}(h)+\sigma_N(\Phi \circ h)(x_N)\Big]\bigg]\, ,
\end{align*}
where $u_{N-1}(h)=\sum_{i=1}^{N-1} \sigma_i(\Phi \circ h)(x_i)$. By definition of the supremum, for any $\epsilon>0$, there exist $h_1, h_2 \in H$ such that
\begin{align*}
& u_{N-1}(h_1)+(\Phi \circ h_1)(x_N) \geq(1-\epsilon)\Big[\sup _{h \in H} u_{N-1}(h)+(\Phi \circ h)(x_N)\Big] \\
\text { and } & u_{N-1}(h_2)-(\Phi \circ h_2)(x_N) \geq(1-\epsilon)\Big[\sup _{h \in H} u_{N-1}(h)-(\Phi \circ h)(x_N)\Big] \, .
\end{align*}
Thus, for any $\epsilon>0$, by definition of $\mathbb{E}_{\sigma_N}$,
\begin{align*}
&(1-\epsilon) \underset{\sigma_N}{\mathbb{E}}\Big[\sup _{h \in H} u_{N-1}(h)+\sigma_N(\Phi \circ h)(x_N)\Big] \\
&\quad=(1-\epsilon)\bigg[\frac{1}{2} \sup _{h \in H} u_{N-1}(h)+(\Phi \circ h)(x_N)+\frac{1}{2} \sup _{h \in H} u_{N-1}(h)-(\Phi \circ h)(x_N)\bigg] \\
&\quad \leq \frac{1}{2}[u_{N-1}(h_1)+(\Phi \circ h_1)(x_N)]+\frac{1}{2}[u_{N-1}(h_2)-(\Phi \circ h_2)(x_N)] \, .
\end{align*}
Let $s=\operatorname{sgn}(h_1(x_N)-h_2(x_N))$. Then, the previous inequality implies
\begin{align*}
&(1-\epsilon) \underset{\sigma_N}{\mathbb{E}}\bigg[\sup _{h \in H} u_{N-1}(h)+\sigma_N(\Phi \circ h)(x_N)\bigg] \\
&\leq \frac{1}{2}[u_{N-1}(h_1)+u_{N-1}(h_2)+s \lambda(h_1(x_N)-h_2(x_N))] \\
&=\frac{1}{2}[u_{N-1}(h_1)+s \lambda h_1(x_N)]+\frac{1}{2}[u_{N-1}(h_2)-s\lambda h_2(x_N)] \\
&\leq \frac{1}{2} \sup _{h \in H}[u_{N-1}(h)+s\lambda h(x_N)]+\frac{1}{2} \sup _{h \in H}[u_{N-1}(h)-s\lambda h(x_N)] \\
&=\underset{\sigma_N}{\mathbb{E}}\Big[\sup _{h \in H} u_{N-1}(h)+\sigma_N \lambda h(x_N)\Big]\, ,
\end{align*}
where the first step is due to the Lipschitz inequality. Since the inequality holds for all $\epsilon>0$, we have
$$
\underset{\sigma_N}{\mathbb{E}}\Big[\sup _{h \in H} u_{N-1}(h)+\sigma_N(\Phi \circ h)(x_N)\Big] \leq \underset{\sigma_N}{\mathbb{E}}\Big[\sup _{h \in H} u_{N-1}(h)+\sigma_N \lambda h(x_N)\Big] \, .
$$
Proceeding in the same way for all other $\sigma_i (i \neq N)$ proves the lemma.
\end{proof}

Secondly, we will bound $\mathbb{E}[\mathcal{E}^{i}_{sta}]$ in terms of Rademacher complexity. Define some neural network function class as follows:
\begin{align*}
&\mathcal{F}_1=\{\pm f:\Omega \to \mathbb{R} \mid \exists \  u_{\mathfrak{m}, \boldsymbol{\theta}}\in\mathcal{PNN}\, \ s.t.\ f(\boldsymbol{x}; \boldsymbol{\theta})=\partial_{x_1}u_{\mathfrak{m}, \boldsymbol{\theta}}(\boldsymbol{x})\}\, ,\\
& \mathcal{F}_2=\{\pm f:\Omega\to \mathbb{R} \mid \exists \ u_{\mathfrak{m}, \boldsymbol{\theta}}\in\mathcal{PNN} \ s.t.\  f(\boldsymbol{x}; \boldsymbol{\theta})=u_{\mathfrak{m}, \boldsymbol{\theta}}(\boldsymbol{x})\}\, , \\
& \mathcal{F}_3=\{\pm f:\Omega\to \mathbb{R} \mid \exists \ u_{\mathfrak{m}, \boldsymbol{\theta}}\in\mathcal{PNN} \ s.t.\  f(\boldsymbol{x}; \boldsymbol{\theta})=u_{\mathfrak{m}, \boldsymbol{\theta}}(\boldsymbol{x})|_{\partial\Omega}\}.
\end{align*} 
Denote the sub-network function class as
\begin{align*}
&\mathcal{F}_{1, sub}=\{\pm f:\Omega \to \mathbb{R} \mid \exists \ \phi_{\boldsymbol{\theta}}(\boldsymbol{x})\in\mathcal{NN}(W, L, B_{\boldsymbol{\theta}})\, \ s.t.\ f(\boldsymbol{x}; \boldsymbol{\theta})=\partial_{x_1}\phi_{\boldsymbol{\theta}}(\boldsymbol{x})\}\, ,\\
& \mathcal{F}_{2, sub}=\{\pm f:\Omega\to \mathbb{R} \mid \exists \ \phi_{\boldsymbol{\theta}}(\boldsymbol{x})\in\mathcal{NN}(W, L, B_{\boldsymbol{\theta}})\ s.t.\  f(\boldsymbol{x}; \boldsymbol{\theta})={\phi_{\boldsymbol{\theta}}(\boldsymbol{x})}\}\, ,\\
& \mathcal{F}_{3, sub}=\{\pm f:\Omega\to \mathbb{R} \mid \exists \ \phi_{\boldsymbol{\theta}}(\boldsymbol{x})\in\mathcal{NN}(W, L, B_{\boldsymbol{\theta}})\ s.t.\  f(\boldsymbol{x}; \boldsymbol{\theta})={\phi_{\boldsymbol{\theta}}(\boldsymbol{x})}|_{\partial\Omega}\}\, .
\end{align*}

When $u_{\mathfrak{m}, \boldsymbol{\theta}} \in \mathcal{PNN}(\mathfrak{m}, M, \{W, L, B_{\boldsymbol{\theta}}\})$, we have $\|\boldsymbol{\theta}^{\mathfrak{m}}_{\rm out}\|_1 \leq M$, i.e., $M$  regulates the magnitude of the linear coefficients which connect all the sub-networks. The following lemma reveals an important fact: the Rademacher complexity of $\mathcal{F}_i$ can be controlled by $M$ and the complexity of 
$\mathcal{F}_{i, sub}$. This suggests that the network's overall complexity may not be affected by the number of sub-networks $\mathfrak{m}$, aiding us in managing the statistical error within the over-parameterized setting, where $\mathfrak{m}$ can grow arbitrarily large.
\begin{lem}\label{relation of R}
For $i = 1, 2, 3,$
$$
\mathfrak{R}_{N_{\rm in}}(\mathcal{F}_i) \leq M \cdot \hat{\mathfrak{R}}_{N_{\rm in}}(\mathcal{F}_{i,sub})\, .
$$
\end{lem}
\begin{proof}
We present the proof with respect to $\mathcal{F}_1$.
\begin{align*}
\mathfrak{R}_{N_{\rm in}}(\mathcal{F}_{1})& =  \mathbb{E}_{\{X_p,\sigma_p\}_{p=1}^{N_{\rm in}}}\bigg[\sup_{u_{\mathfrak{m}, \boldsymbol{\theta}} \in \mathcal{PNN}}\frac{1}{N_{\rm in}}\sum_{p=1}^{N_{\rm in}} \sigma_p \partial_{x_1}u_{\mathfrak{m}, \boldsymbol{\theta}}(X_p)\bigg] \\
& = \mathbb{E}_{\{X_p,\sigma_p\}_{p=1}^{N_{\rm in}}}\bigg[\sup_{\boldsymbol{\theta}_{\rm total}^{\mathfrak{m}} \in \Theta^{\mathfrak{m}}}\frac{1}{N_{\rm in}}\sum_{p=1}^{N_{\rm in}}\sigma_p \bigg[\sum_{k=1}^{\mathfrak{m}} c_{k} \cdot \partial_{x_1}\phi_{\boldsymbol{\theta}}^{k}(X_p)\bigg] \bigg] \\
& = \mathbb{E}_{\{X_p,\sigma_p\}_{p=1}^{N_{\rm in}}} \bigg[\sup_{\boldsymbol{\theta}_{\rm total}^{\mathfrak{m}} \in \Theta^{\mathfrak{m}}} \sum_{k=1}^{\mathfrak{m}} c_{k} \cdot   \frac{1}{N_{\rm in}}\sum_{p=1}^{N_{\rm in}} \sigma_p  \partial_{x_1}\phi_{\boldsymbol{\theta}}^{k}(X_p) \bigg] \\
& \leq \mathbb{E}_{\{X_p,\sigma_p\}_{p=1}^{N_{\rm in}}} \bigg[\sup_{\boldsymbol{\theta}_{\rm total}^{\mathfrak{m}} \in \Theta^{\mathfrak{m}}} \sum_{k=1}^{\mathfrak{m}} |c_{k}| \cdot   \Big|\frac{1}{N_{\rm in}}\sum_{p=1}^{N_{\rm in}} \sigma_p  \partial_{x_1}\phi_{\boldsymbol{\theta}}^{k}(X_p)\Big| \bigg]\\
&\leq \mathbb{E}_{\{X_p,\sigma_p\}_{p=1}^{N_{\rm in}}} \bigg[\sup_{\boldsymbol{\theta}_{\rm total}^{\mathfrak{m}} \in \Theta^{\mathfrak{m}}} \sum_{k=1}^{\mathfrak{m}} |c_{k}| \cdot \sup_{  k\in \{1, \ldots, \mathfrak{m}\} }  \Big|\frac{1}{N_{\rm in}}\sum_{p=1}^{N_{\rm in}} \sigma_p  \partial_{x_1}\phi_{\boldsymbol{\theta}}^{k}(X_p)\Big| \bigg]\\
& \leq M \cdot \mathbb{E}_{\{X_p,\sigma_p\}_{p=1}^{N_{\rm in}}} \bigg[\sup_{ \boldsymbol{\theta}_{\rm total}^{\mathfrak{m}} \in \Theta^{\mathfrak{m}}, k\in \{1, \ldots, \mathfrak{m}\} }  \Big|\frac{1}{N_{\rm in}}\sum_{p=1}^{N_{\rm in}} \sigma_p  \partial_{x_1}\phi_{\boldsymbol{\theta}}^{k}(X_p)\Big| \bigg]\\
& = M \cdot \mathbb{E}_{\{X_p,\sigma_p\}_{p=1}^{N_{\rm in}}} \bigg[\sup_{ \boldsymbol{\theta}_{1}\in \Theta}  \Big|\frac{1}{N_{\rm in}}\sum_{p=1}^{N_{\rm in}} \sigma_p  \partial_{x_1}\phi_{\boldsymbol{\theta}}^{1}(X_p)\Big| \bigg]\\
& = M \cdot \hat{\mathfrak{R}}_{N_{\rm in}}(\mathcal{F}_{1,sub})\, ,
\end{align*}
where the H$\mathrm{\ddot{o}}$lder inequality is used in the first inequality. 
\end{proof}

Next, from Lemma \ref{f-lip} to Lemma \ref{Lip of Fi}, we will provide several results concerning the sub-neural network $\mathcal{NN}(W, L, B_{\boldsymbol{\theta}})$. Note that the activation function used in this paper is $tanh$, which is 1-Lipschitz and also has a 1-Lipschitz continuous gradient.

\begin{lem} \label{f-lip}
Let $W, L \in\mathbb{N}^{+}$, $B_{\boldsymbol{\theta}}\ge 1$. Let $\rho(x)$ be the hyperbolic tangent function $\frac{e^{x}-e^{-x}}{e^{x}+e^{-x}}$. Then,  for any $\phi_{\boldsymbol{\theta}}(\boldsymbol{x})\in \mathcal{NN}(W, L, B_{\boldsymbol{\theta}})$, we have
\begin{align*}
\big|\phi_{\boldsymbol{\theta}}(\boldsymbol{x})\big|\leq \big(W+1\big)B_{\boldsymbol{\theta}}\, .
\end{align*}
Moreover, for any $\phi_{\boldsymbol{\theta}}(\boldsymbol{x}), \phi_{\tilde{\boldsymbol{\theta}}}(\boldsymbol{x})\in \mathcal{NN}(W, L, B_{\boldsymbol{\theta}})$,
\begin{equation*}
	|\phi_{\boldsymbol{\theta}}(\boldsymbol{x})-\phi_{\tilde{\boldsymbol{\theta}}}(\boldsymbol{x})|
	\leq 2W^{L}\sqrt{L}\cdot B_{\boldsymbol{\theta}}^{L-1} \big\|\boldsymbol{\theta}-\tilde{\boldsymbol{\theta}}\big\|_2\, ,\quad \forall \boldsymbol{x}\in\Omega\, .
\end{equation*}
\end{lem}
\begin{proof}
For any $\phi_{\boldsymbol{\theta}}(\boldsymbol{x})\in \mathcal{NN}(W, L, B_{\boldsymbol{\theta}})$, we have
\begin{align*}
\big|\phi_{\boldsymbol{\theta}}(\boldsymbol{x})\big|\leq \big(N_{L-1}+1\big)B_{\boldsymbol{\theta}}\leq \big(W+1\big)B_{\boldsymbol{\theta}}\, .
\end{align*}

Denote $\phi_{q}^{(\ell)}$ as the $q$-th output of the $\ell$-th layer. $\mathfrak{n}_{\ell},N_{\ell}\in\mathbb{N}^{+}$, $\ell \in \{1,\ldots, L-1\}$, $N_L =1$. For $\ell=2,\cdots, L$ (the argument for the case of $\ell= L$ is slightly different),
\begin{align*}
	\big|\phi_q^{(\ell)}-\tilde{\phi}_q^{(\ell)}\big|
	&=\bigg|\rho\bigg(\sum_{j=1}^{N_{\ell-1}}a_{qj}^{(\ell-1)}\phi_j^{(\ell-1)}+b_q^{(\ell-1)}\bigg)-\rho\bigg(\sum_{j=1}^{N_{\ell-1}}\tilde{a}_{qj}^{(\ell-1)}\tilde{\phi}_j^{(\ell-1)}+\tilde{b}_q^{(\ell-1)}\bigg)\bigg|\\
	&\leq \bigg|\sum_{j=1}^{N_{\ell-1}}a_{qj}^{(\ell-1)}\phi_j^{(\ell-1)}-\sum_{j=1}^{N_{\ell-1}}\tilde{a}_{qj}^{(\ell-1)}\tilde{\phi}_j^{(\ell-1)}+b_q^{(\ell-1)}-\tilde{b}_q^{(\ell-1)}\bigg|\\
	&\leq \sum_{j=1}^{N_{\ell-1}}\big|a_{qj}^{(\ell-1)}\big|\big|\phi_j^{(\ell-1)}-\tilde{\phi}_j^{(\ell-1)}\big|+\sum_{j=1}^{N_{\ell-1}}\big|a_{qj}^{(\ell-1)}-\tilde{a}_{qj}^{(\ell-1)}\big|\big|\tilde{\phi}_j^{(\ell-1)}\big|+\big|b_q^{(\ell-1)}-\tilde{b}_q^{(\ell-1)}\big|\\
	&\leq B_{\boldsymbol{\theta}}\sum_{j=1}^{N_{\ell-1}}\big|\phi_j^{(\ell-1)}-\tilde{\phi}_j^{(\ell-1)}\big|+\sum_{j=1}^{N_{\ell-1}}\big|a_{qj}^{(\ell-1)}-\tilde{a}_{qj}^{(\ell-1)}\big|+\big|b_q^{(\ell-1)}-\tilde{b}_q^{(\ell-1)}\big|\, .
\end{align*}
For $\ell=1$,
\begin{align*}
	\big|\phi_q^{(1)}-\tilde{\phi}_q^{(1)}\big|
	&= \bigg|\rho\bigg(\sum_{j=1}^{N_0}a_{qj}^{(0)}x_j + b_q^{(0)}\bigg)-\rho\bigg(\sum_{j=1}^{N_0}\tilde{a}_{qj}^{(0)}x_j+\tilde{b}_q^{(0)}\bigg)\bigg|\\
	&\le \sum_{j=1}^{N_0}\big|a_{qj}^{(0)}-\tilde{a}_{qj}^{(0)}\big|+\big|b_q^{(0)}-\tilde{b}_q^{(0)}\big|
	=\sum_{j=1}^{\mathfrak{n}_1}\big|\theta_j-\tilde{\theta}_j\big|\, .
\end{align*}
For $\ell=2$,
\begin{align*}
\big|\phi_q^{(2)}-\tilde{\phi}_q^{(2)}\big|
	&\leq B_{\boldsymbol{\theta}}\sum_{j=1}^{N_{1}}\big|\phi_j^{(1)}-\tilde{\phi}_j^{(1)}\big|+\sum_{j=1}^{N_{1}}\big|a_{qj}^{(1)}-\tilde{a}_{qj}^{(1)}\big|+\big|b_q^{(1)}-\tilde{b}_q^{(1)}\big|\\
	&\leq B_{\boldsymbol{\theta}}\sum_{j=1}^{N_{1}}\sum_{k=1}^{\mathfrak{n}_1}\big|\theta_k-\tilde{\theta}_k\big|+\sum_{j=1}^{N_{1}}\big|a_{qj}^{(1)}-\tilde{a}_{qj}^{(1)}\big|+\big|b_q^{(1)}-\tilde{b}_q^{(1)}\big|\\
	&\leq N_{1}B_{\boldsymbol{\theta}}\sum_{j=1}^{\mathfrak{n}_2}\big|\theta_j-\tilde{\theta}_j\big|\, .
\end{align*}
Assuming that for $\ell\geq2$,
\begin{align*}
\big|\phi_q^{(\ell)}-\tilde{\phi}_q^{(\ell)}\big|
\leq \bigg(\prod_{i=1}^{\ell-1}N_i\bigg) B_{\boldsymbol{\theta}}^{\ell-1}\sum_{j=1}^{\mathfrak{n}_{\ell}}\big|\theta_j-\tilde{\theta}_j\big|\, ,
\end{align*}
we have
\begin{align*}
	\big|\phi_q^{(\ell+1)}-\tilde{\phi}_q^{(\ell+1)}\big|
	&\leq B_{\boldsymbol{\theta}}\sum_{j=1}^{N_{\ell}}\big|\phi_j^{(\ell)}-\tilde{\phi}_j^{(\ell)}\big|+\sum_{j=1}^{N_{\ell}}\big|a_{qj}^{(\ell)}-\tilde{a}_{qj}^{(\ell)}\big|+\big|b_q^{(\ell)}-\tilde{b}_q^{(\ell)}\big|\\
	&\leq B_{\boldsymbol{\theta}}\sum_{j=1}^{N_{\ell}}\bigg(\prod_{i=1}^{\ell-1}N_i\bigg) B_{\boldsymbol{\theta}}^{\ell-1}\sum_{k=1}^{\mathfrak{n}_{\ell}}\big|\theta_k-\tilde{\theta}_k\big|+\sum_{j=1}^{N_{\ell}}\big|a_{qj}^{(\ell)}-\tilde{a}_{qj}^{(\ell)}\big|+\big|b_q^{(\ell)}-\tilde{b}_q^{(\ell)}\big|\\
	&\leq \bigg(\prod_{i=1}^{\ell}N_i\bigg) B_{\boldsymbol{\theta}}^{\ell}\sum_{j=1}^{\mathfrak{n}_{\ell+1}}\big|\theta_j-\tilde{\theta}_j\big|\, .
\end{align*}
Hence by induction and H$\mathrm{\ddot{o}}$lder inequality we conclude that
\begin{align*}
|\phi_{\boldsymbol{\theta}}(\boldsymbol{x})-\phi_{\tilde{\boldsymbol{\theta}}}(\boldsymbol{x})|
&\leq \bigg(\prod_{i=1}^{L-1}N_i\bigg) B_{\boldsymbol{\theta}}^{L-1}\sum_{j=1}^{\mathfrak{n}_{L}}\big|\theta_j-\tilde{\theta}_j\big|\\
&\leq\sqrt{\mathfrak{n}_{L}}B_{\boldsymbol{\theta}}^{L-1}\bigg(\prod_{i=1}^{L-1}N_i\bigg)\big\|\boldsymbol{\theta}-\tilde{\boldsymbol{\theta}}\big\|_2 \leq  2W^{L}\sqrt{L}\cdot B_{\boldsymbol{\theta}}^{L-1} \big\|\boldsymbol{\theta}-\tilde{\boldsymbol{\theta}}\big\|_2\, .
\end{align*}
\end{proof}

\begin{lem}\label{df-bound}
Let $W, L\in\mathbb{N}^{+}$, $B_{\boldsymbol{\theta}}\ge 1$. Let $\rho(x)$ be the hyperbolic tangent function $\frac{e^{x}-e^{-x}}{e^{x}+e^{-x}}$. Then,  for any $\phi_{\boldsymbol{\theta}}(\boldsymbol{x})\in \mathcal{NN}(W, L, B_{\boldsymbol{\theta}})$, we have
\begin{align*}
	\big|\partial_{x_m}\phi_q^{(\ell)}\big|\leq W^{\ell -1}B_{\boldsymbol{\theta}}^{\ell}\, ,
	\ \ell=1,2,\cdots,L-1\, ; \qquad
	\big|\partial_{x_m}\phi_{\boldsymbol{\theta}}(\boldsymbol{x})\big|\leq W^{L-1}B_{\boldsymbol{\theta}}^{L}\, .
\end{align*}
\end{lem}
\begin{proof}
For $\ell=1,2,\cdots,L-1$, $\mathfrak{n}_{\ell},N_{\ell}\in\mathbb{N}^{+}$,
\begin{align*}
	\big|\partial_{x_m}\phi_q^{(\ell)}\big|
	&=\bigg|\sum_{j=1}^{N_{\ell-1}}a_{qj}^{(\ell-1)}\partial_{x_m}\phi_j^{(\ell-1)}\rho^{\prime}\bigg(\sum_{j=1}^{N_{\ell-1}}a_{qj}^{(\ell-1)}\phi_j^{(\ell-1)}+b_q^{(\ell-1)}\bigg)\bigg|
	\leq B_{\boldsymbol{\theta}}\sum_{j=1}^{N_{\ell-1}}\big|\partial_{x_m}\phi_j^{(\ell-1)}\big|\\
	&\leq(B_{\boldsymbol{\theta}})^2\sum_{k=1}^{N_{\ell-1}}\sum_{j=1}^{N_{\ell-2}}\big|\partial_{x_m}\phi_j^{(\ell-2)}\big|
	=N_{\ell-1}(B_{\boldsymbol{\theta}})^2\sum_{j=1}^{N_{\ell-2}}\big|\partial_{x_m}\phi_j^{(\ell-2)}\big|\\
	&\leq\cdots\leq\bigg(\prod_{i=2}^{\ell-1}N_i\bigg)(B_{\boldsymbol{\theta}})^{\ell-1}\sum_{j=1}^{N_{1}}\big|\partial_{x_m}\phi_j^{(1)}\big| \\
	&\leq\bigg(\prod_{i=2}^{\ell-1}N_i\bigg)(B_{\boldsymbol{\theta}})^{\ell-1}\sum_{j=1}^{N_{1}}B_{\boldsymbol{\theta}}=\bigg(\prod_{i=1}^{\ell-1}N_i\bigg)(B_{\boldsymbol{\theta}})^{\ell}\leq  W^{\ell -1}B_{\boldsymbol{\theta}}^{\ell}\, .
\end{align*}
The bound for $|\partial_{x_m}\phi_{\boldsymbol{\theta}}(\boldsymbol{x})|$ can be derived similarly.
\end{proof}

\begin{lem}\label{df-lip}
Let $W, L\in\mathbb{N}^{+}$, $B_{\boldsymbol{\theta}}\ge 1$. Let $\rho(x)$ be the hyperbolic tangent function $\frac{e^{x}-e^{-x}}{e^{x}+e^{-x}}$. Let $m =1,\cdots,d$. Then,  for any $\phi_{\boldsymbol{\theta}}(\boldsymbol{x}), \phi_{\tilde{\boldsymbol{\theta}}}(\boldsymbol{x})\in \mathcal{NN}(W, L, B_{\boldsymbol{\theta}})$,
\begin{equation*}
	|\partial_{x_m}\phi_{\boldsymbol{\theta}}(\boldsymbol{x})-\partial_{x_m}\phi_{\tilde{\boldsymbol{\theta}}}(\boldsymbol{x})|
	\leq 2W^{2L-1}\sqrt{L}(L+1)\cdot B_{\boldsymbol{\theta}}^{2L}\big\|\boldsymbol{\theta}-\tilde{\boldsymbol{\theta}}\big\|_2\, ,\quad \forall \boldsymbol{x}\in\Omega\, .
\end{equation*}
\end{lem}
\begin{proof}
Note that $\mathfrak{n}_{\ell},N_{\ell}\in\mathbb{N}^{+}$, $\ell \in \{1,\ldots, L-1\}$, $N_L =1$. For $\ell=1$,
\begin{align*}
&\big|\partial_{x_m}\phi_q^{(1)}-\partial_{x_m}\tilde{\phi}_q^{(1)}\big| \\
=&\bigg|a_{qp}^{(0)}\rho^{\prime}\bigg(\sum_{j=1}^{N_0}a_{qj}^{(0)}x_j+b_q^{(0)}\bigg)-\tilde{a}_{qp}^{(0)}\rho^{\prime}\bigg(\sum_{j=1}^{N_0}\tilde{a}_{qj}^{(0)}x_j+\tilde{b}_q^{(0)}\bigg)\bigg|\\
\leq&\big|a_{qp}^{(0)}-\tilde{a}_{qp}^{(0)}\big|\bigg|\rho^{\prime}\bigg(\sum_{j=1}^{N_0}a_{qj}^{(0)}x_j+b_q^{(0)}\bigg)\bigg| +\big|\tilde{a}_{qp}^{(0)}\big| \bigg|\rho^{\prime}\bigg(\sum_{j=1}^{N_0}a_{qj}^{(0)}x_j+b_{q}^{(0)}\bigg)-\rho^{\prime}\bigg(\sum_{j=1}^{N_0}\tilde{a}_{qj}^{(0)}x_j+\tilde{b}_q^{(0)}\bigg)\bigg|\\
\leq& \big|a_{qp}^{(0)}-\tilde{a}_{qp}^{(0)}\big|+B_{\boldsymbol{\theta}}\sum_{j=1}^{N_0}\big|a_{qj}^{(0)}-\tilde{a}_{qj}^{(0)}\big|+B_{\boldsymbol{\theta}}\big|{b}_q^{(0)}-\tilde{b}_q^{(0)}\big|\leq 2B_{\boldsymbol{\theta}}\sum_{k=1}^{\mathfrak{n}_{1}}\big|\theta_k-\tilde{\theta}_k\big| \, .
\end{align*}
For $\ell\geq 2$, we establish the Recurrence relation:
\begin{align*}
	&\big|\partial_{x_m}\phi_q^{(\ell)}-\partial_{x_m}\tilde{\phi}_q^{(\ell)}\big|\\
	&\leq\sum_{j=1}^{N_{\ell-1}}\big|a_{qj}^{(\ell-1)}\big|\big|\partial_{x_m}\phi_j^{(\ell-1)}\big|\bigg|\rho^{\prime}\bigg(\sum_{j=1}^{N_{\ell-1}}a_{qj}^{(\ell-1)}\phi_j^{(\ell-1)}+b_q^{(\ell-1)}\bigg)-\rho^{\prime}\bigg(\sum_{j=1}^{N_{\ell-1}}\tilde{a}_{qj}^{(\ell-1)}\tilde{\phi}_j^{(\ell-1)}+\tilde{b}_q^{(\ell-1)}\bigg)\bigg|\\
	&\quad+\sum_{j=1}^{N_{\ell-1}}\big|a_{qj}^{(\ell-1)}\partial_{x_m}\phi_j^{(\ell-1)}-\tilde{a}_{qj}^{(\ell-1)}\partial_{x_m}\tilde{\phi}_j^{(\ell-1)}\big|\bigg|\rho^{\prime}\bigg(\sum_{j=1}^{N_{\ell-1}}\tilde{a}_{qj}^{(\ell-1)}\tilde{\phi}_j^{(\ell-1)}+\tilde{b}_q^{(\ell-1)}\bigg)\bigg|\\
	&\leq B_{\boldsymbol{\theta}}\sum_{j=1}^{N_{\ell-1}}\big|\partial_{x_m}\phi_j^{(\ell-1)}\big|\bigg(\sum_{j=1}^{N_{\ell-1}}\big|a_{qj}^{(\ell-1)}\phi_j^{(\ell-1)}-\tilde{a}_{qj}^{(\ell-1)}\tilde{\phi}_j^{(\ell-1)}\big|+\big|b_q^{(\ell-1)}-\tilde{b}_q^{(\ell-1)}\big|\bigg)\\
	&\quad+\sum_{j=1}^{N_{\ell-1}}\big|a_{qj}^{(\ell-1)}\partial_{x_m}\phi_j^{(\ell-1)}-\tilde{a}_{qj}^{(\ell-1)}\partial_{x_m}\tilde{\phi}_j^{(\ell-1)}\big|\\
	&\leq B_{\boldsymbol{\theta}}\sum_{j=1}^{N_{\ell-1}}\big|\partial_{x_m}\phi_j^{(\ell-1)}\big|\bigg(\sum_{j=1}^{N_{\ell-1}}\big|a_{qj}^{(\ell-1)}-\tilde{a}_{qj}^{(\ell-1)}\big| + B_{\boldsymbol{\theta}}\sum_{j=1}^{N_{\ell-1}}\big|\phi_j^{(\ell-1)}-\tilde{\phi}_j^{(\ell-1)}\big|+\big|b_q^{(\ell-1)}-\tilde{b}_q^{(\ell-1)}\big|\bigg)\\
	&\quad+B_{\boldsymbol{\theta}}\sum_{j=1}^{N_{\ell-1}}\big|\partial_{x_m}\phi_j^{(\ell-1)}-\partial_{x_m}\tilde{\phi}_j^{(\ell-1)}\big|+\sum_{j=1}^{N_{\ell-1}}\big|a_{qj}^{(\ell-1)}-\tilde{a}_{qj}^{(\ell-1)}\big|\big|\partial_{x_m}\tilde{\phi}_j^{(\ell-1)}\big|\\
	&\leq B_{\boldsymbol{\theta}}\sum_{j=1}^{N_{\ell-1}}\big|\partial_{x_m}\phi_j^{(\ell-1)}\big|\bigg(\sum_{j=1}^{N_{\ell-1}}\big|a_{qj}^{(\ell-1)}-\tilde{a}_{qj}^{(\ell-1)}\big|+B_{\boldsymbol{\theta}}\sum_{j=1}^{N_{\ell-1}}\big|\phi_j^{(\ell-1)}-\tilde{\phi}_j^{(\ell-1)}\big|+\big|b_q^{(\ell-1)}-\tilde{b}_q^{(\ell-1)}\big|\bigg)\\
	&\quad+B_{\boldsymbol{\theta}}\sum_{j=1}^{N_{\ell-1}}\big|\partial_{x_m}\phi_j^{(\ell-1)}-\partial_{x_m}\tilde{\phi}_j^{(\ell-1)}\big|+\sum_{j=1}^{N_{\ell-1}}\big|a_{qj}^{(\ell-1)}-\tilde{a}_{qj}^{(\ell-1)}\big|\big|\partial_{x_m}\tilde{\phi}_j^{(\ell-1)}\big|\\
	&\leq B_{\boldsymbol{\theta}}\bigg(\prod_{i=1}^{\ell-1}N_i\bigg)B_{\boldsymbol{\theta}}^{\ell}\bigg(\sum_{j=1}^{N_{\ell-1}}\big|a_{qj}^{(\ell-1)}-\tilde{a}_{qj}^{(\ell-1)}\big|\!+\!B_{\boldsymbol{\theta}}\sum_{j=1}^{N_{\ell-1}}
	\bigg(\prod_{i=1}^{\ell-2}N_i\bigg) B_{\boldsymbol{\theta}}^{\ell-2}\sum_{k=1}^{\mathfrak{n}_{\ell-1}}\big|\theta_k-\tilde{\theta}_k\big|+\big|b_q^{(\ell-1)}-\tilde{b}_q^{(\ell-1)}\big|\bigg)\\
	&\quad+B_{\boldsymbol{\theta}}\sum_{j=1}^{N_{\ell-1}}\big|\partial_{x_m}\phi_j^{(\ell-1)}-\partial_{x_m}\tilde{\phi}_j^{(\ell-1)}\big|+\sum_{j=1}^{N_{\ell-1}}\big|a_{qj}^{(\ell-1)}-\tilde{a}_{qj}^{(\ell-1)}\big|\bigg(\prod_{i=1}^{\ell-2}N_i\bigg)B_{\boldsymbol{\theta}}^{\ell-1}\\
	&\leq B_{\boldsymbol{\theta}}\sum_{j=1}^{N_{\ell-1}}\big|\partial_{x_m}\phi_j^{(\ell-1)}-\partial_{x_m}\tilde{\phi}_j^{(\ell-1)}\big| + B_{\boldsymbol{\theta}}^{2\ell}\bigg(\prod_{i=1}^{\ell-1}N_i\bigg)^{\mkern-4mu 2}\sum_{k=1}^{\mathfrak{n}_{\ell}}\big|\theta_k-\tilde{\theta}_k\big| \, .
\end{align*}
For $\ell=2$,
\begin{align*}
    \big|\partial_{x_m}\phi_q^{(2)}-\partial_{x_m}\tilde{\phi}_q^{(2)}\big|&\leq B_{\boldsymbol{\theta}}\sum_{j=1}^{N_{1}}\big|\partial_{x_m}\phi_j^{(1)}-\partial_{x_m}\tilde{\phi}_j^{(1)}\big|
	+B_{\boldsymbol{\theta}}^{4}N_{1}^2\sum_{k=1}^{\mathfrak{n}_{2}}\big|\theta_k-\tilde{\theta}_k\big|\\
	&\leq2B_{\boldsymbol{\theta}}^2N_{1}\sum_{k=1}^{\mathfrak{n}_{1}}\big|\theta_k-\tilde{\theta}_k\big|+B_{\boldsymbol{\theta}}^{4}N_{1}^2\sum_{k=1}^{\mathfrak{n}_{2}}\big|\theta_k-\tilde{\theta}_k\big|
	\leq 3B_{\boldsymbol{\theta}}^{4}N_{1}^2\sum_{k=1}^{\mathfrak{n}_{2}}\big|\theta_k-\tilde{\theta}_k\big| \, .
\end{align*}
Assuming that for $\ell\geq2$,
\begin{equation*}
	\big|\partial_{x_m}\phi_q^{(\ell)}-\partial_{x_m}\tilde{\phi}_q^{(\ell)}\big|\leq (\ell+1) B_{\boldsymbol{\theta}}^{2\ell}\bigg(\prod_{i=1}^{\ell-1}N_i\bigg)^{\mkern-4mu 2}\sum_{k=1}^{\mathfrak{n}_{\ell}}\big|\theta_k-\tilde{\theta}_k\big| \, ,
\end{equation*}
we have
\begin{align*}
	&\big|\partial_{x_m}\phi_q^{(\ell+1)}-\partial_{x_m}\tilde{\phi}_q^{(\ell+1)}\big| \\
	\leq& B_{\boldsymbol{\theta}}\sum_{j=1}^{N_{\ell}}\big|\partial_{x_m}\phi_j^{(\ell)}-\partial_{x_m}\tilde{\phi}_j^{(\ell)}\big|
	+B_{\boldsymbol{\theta}}^{2\ell+2}\bigg(\prod_{i=1}^{\ell}N_i\bigg)^{\mkern-4mu 2}\sum_{k=1}^{\mathfrak{n}_{\ell+1}}\big|\theta_k-\tilde{\theta}_k\big|\\
	\leq & B_{\boldsymbol{\theta}}\sum_{j=1}^{N_{\ell}}(\ell+1) B_{\boldsymbol{\theta}}^{2\ell}\bigg(\prod_{i=1}^{\ell-1}N_i\bigg)^{\mkern-4mu 2}\sum_{k=1}^{\mathfrak{n}_{\ell}}\big|\theta_k-\tilde{\theta}_k\big|
	+B_{\boldsymbol{\theta}}^{2\ell+2}\bigg(\prod_{i=1}^{\ell}N_i\bigg)^{\mkern-4mu 2}\sum_{k=1}^{\mathfrak{n}_{\ell+1}}\big|\theta_k-\tilde{\theta}_k\big|\\
	\leq& (\ell+2)B_{\boldsymbol{\theta}}^{2\ell+2}\bigg(\prod_{i=1}^{\ell}N_i\bigg)^{\mkern-4mu 2}\sum_{k=1}^{\mathfrak{n}_{\ell+1}}\big|\theta_k-\tilde{\theta}_k\big| \, .
\end{align*}
Hence by by induction and H$\mathrm{\ddot{o}}$lder inequality we conclude that
\begin{align*}
\big|\partial_{x_m}\phi_{\boldsymbol{\theta}}-\partial_{x_m}\phi_{\tilde{\boldsymbol{\theta}}}\big|  &=  \big|\partial_{x_m}\phi-\partial_{x_m}\tilde{\phi}\big| \leq (L+1)B_{\boldsymbol{\theta}}^{2L}\bigg(\prod_{i=1}^{L-1}N_i\bigg)^{\mkern-4mu 2}\sum_{k=1}^{\mathfrak{n}_{L}}\big|\theta_k-\tilde{\theta}_k\big|\\
&\leq \sqrt{\mathfrak{n}_{L}}(L+1)B_{\boldsymbol{\theta}}^{2L}\bigg(\prod_{i=1}^{L-1}N_i\bigg)^{\mkern-4mu 2}\big\|\boldsymbol{\theta}-\tilde{\boldsymbol{\theta}}\big\|_2 \leq 2W^{2L-1}\sqrt{L}(L+1)\cdot B_{\boldsymbol{\theta}}^{2L}\big\|\boldsymbol{\theta}-\tilde{\boldsymbol{\theta}}\big\|_2\, .
\end{align*}
\end{proof}

\begin{lem} \label{Lip of Fi}
Let $W, L\in\mathbb{N}^{+}$, $B_{\boldsymbol{\theta}}\ge 1$. Let $\rho(x)$ be the hyperbolic tangent function $\frac{e^{x}-e^{-x}}{e^{x}+e^{-x}}$. Then, for any $f_{i}(\boldsymbol{x}; \boldsymbol{\theta}),f_{i}(\boldsymbol{x}; \tilde{\boldsymbol{\theta}})\in\mathcal{F}_{i, sub}$, $i=1,\cdots,3$, we have
\begin{align*}
\big|f_{i}(\boldsymbol{x}; \boldsymbol{\theta})\big| &\le B_i\, , \quad \forall x\in\Omega\, ,\\
\big|f_{i}(\boldsymbol{x}; \boldsymbol{\theta})-f_{i}(\boldsymbol{x}; \tilde{\boldsymbol{\theta}})\big| &\leq L_i \big\|\boldsymbol{\theta} -\tilde{\boldsymbol{\theta}}\big\|_2\, ,\quad \forall x\in\Omega\, ,
\end{align*}
with
\begin{align*}
B_1&=W^{L-1}\cdot B_{\boldsymbol{\theta}}^{L}\, ,
&B_2&= B_3 = (W+1)\cdot B_{\boldsymbol{\theta}}\, ,\\
L_1&=2W^{2L-1}\sqrt{L}(L+1)\cdot B_{\boldsymbol{\theta}}^{2L}\, ,
&L_2&=L_3 = 2W^{L}\sqrt{L}\cdot B_{\boldsymbol{\theta}}^{L-1}\, ,
\end{align*}
\end{lem}
\begin{proof}
    Direct result from Lemma \ref{f-lip}, Lemma \ref{df-bound} and Lemma \ref{df-lip}.
\end{proof}

Thus, we obtain the following conclusion, which bounds $\mathbb{E}[\mathcal{E}^{i}_{sta}] $ in terms of the Rademacher complexity of $\mathcal{F}_{i, sub}$.
\begin{lem}
\label{sta1}
The element $u_{\mathfrak{m}, \boldsymbol{\theta}}\in \mathcal{PNN}(\mathfrak{m}, M, \{W, L, B_{\boldsymbol{\theta}}\})$, then
\begin{align*}
&\mathbb{E}[\mathcal{E}^{1}_{sta}] \leq 2|\Omega| d M^2\cdot W^{L-1} B_{\boldsymbol{\theta}}^{L}\cdot \hat{\mathfrak{R}}_{N_{\rm in}}(\mathcal{F}_{1, sub})\, , ~~~~~~~~~ \mathbb{E}[\mathcal{E}^{3}_{sta}]\leq 2B_0|\Omega|M\cdot\hat{\mathfrak{R}}_{N_{\rm in}}(\mathcal{F}_{2, sub})\, ,\\
&\mathbb{E}[\mathcal{E}^{2}_{sta}] \leq 2 B_0 |\Omega| M^2\cdot (W+1) B_{\boldsymbol{\theta}} \cdot \hat{\mathfrak{R}}_{N_{\rm in}}(\mathcal{F}_{2, sub})\, , ~~~~~ \mathbb{E}[\mathcal{E}^{4}_{sta}] \leq 2B_0 |\partial\Omega|M\cdot \hat{\mathfrak{R}}_{N_{b}}(\mathcal{F}_{3, sub})\, .
\end{align*}
\end{lem}
\begin{proof}
\textbf{Step 1.} We present the proof with respect to $\mathcal{L}_1$ and $\mathcal{L}_2$, other inequalities can be shown similarly. Denote
\begin{align*}
&\mathcal{F}_1^{\prime} = \big\{\pm f:\Omega \to \mathbb{R} \mid \exists \  u_{\mathfrak{m}, \boldsymbol{\theta}}\in\mathcal{PNN}\, \ s.t.\ f(\boldsymbol{x}; \boldsymbol{\theta})=\big(\partial_{x_1}u_{\mathfrak{m}, \boldsymbol{\theta}}(\boldsymbol{x})\big)^2\big\}\, ,\\
& \mathcal{F}_2^{\prime} = \big\{\pm f:\Omega\to \mathbb{R} \mid \exists \ u_{\mathfrak{m}, \boldsymbol{\theta}}\in\mathcal{PNN} \ s.t.\  f(\boldsymbol{x}; \boldsymbol{\theta})=u^2_{\mathfrak{m}, \boldsymbol{\theta}}(\boldsymbol{x})\big\}\, .
\end{align*} 
We take $\{\tilde{X_p}\}_{p=1}^{N_{\rm in}}$ as an independent copy of $\{{X_p}\}_{p=1}^{N_{\rm in}}$, then
\begin{align*}
	\mathcal{L}_1(u_{\mathfrak{m}, \boldsymbol{\theta}})-\widehat{\mathcal{L}}_1(u_{\mathfrak{m}, \boldsymbol{\theta}})
	&=\frac{|\Omega|}{2}\bigg[\mathbb{E}_{X\sim U(\Omega)}\|\nabla u_{\mathfrak{m}, \boldsymbol{\theta}}(X)\|^2-\sum_{p=1}^{N_{\rm in}}\|\nabla u_{\mathfrak{m}, \boldsymbol{\theta}}(X_p)\|^2\bigg]\\	
	&=\frac{|\Omega|}{2N_{\rm in}}\mathbb{E}_{\{\tilde{X_p}\}_{p=1}^{N_{\rm in}}}\sum_{m=1}^{d}\sum_{p=1}^{N_{\rm in}}\Big[(\partial_{x_m}u_{\mathfrak{m}, \boldsymbol{\theta}}(\tilde{X_p}))^2-(\partial_{x_m}u_{\mathfrak{m}, \boldsymbol{\theta}}(X_p))^2\Big]\, .
\end{align*}
Hence
\begin{align*}
	&\mathbb{E}[\mathcal{E}^{1}_{sta}] =\mathbb{E}_{\{{X_p}\}_{p=1}^{N_{\rm in}}}\sup_{u_{\mathfrak{m}, \boldsymbol{\theta}}\in\mathcal{PNN}}\big|\mathcal{L}_1(u_{\mathfrak{m}, \boldsymbol{\theta}})-\widehat{\mathcal{L}}_1(u_{\mathfrak{m}, \boldsymbol{\theta}})\big|\\
    & \leq \frac{|\Omega|}{2N_{\rm in}} \mathbb{E}_{\{{X_p}\}_{p=1}^{N_{\rm in}}}\sup_{u_{\mathfrak{m}, \boldsymbol{\theta}}\in\mathcal{PNN}}\bigg|\mathbb{E}_{\{\tilde{X_p}\}_{p=1}^{N_{\rm in}}}\sum_{m=1}^{d}\sum_{p=1}^{N_{\rm in}}\Big[(\partial_{x_m}u_{\mathfrak{m}, \boldsymbol{\theta}}(\tilde{X_p}))^2-(\partial_{x_m}u_{\mathfrak{m}, \boldsymbol{\theta}}(X_p))^2\Big]\bigg|\\
    & \leq \frac{|\Omega|}{2N_{\rm in}} \mathbb{E}_{\{{X_p}\}_{p=1}^{N_{\rm in}}}\sup_{u_{\mathfrak{m}, \boldsymbol{\theta}}\in\mathcal{PNN}}\mathbb{E}_{\{\tilde{X_p}\}_{p=1}^{N_{\rm in}}}\sum_{m=1}^{d}\bigg|\sum_{p=1}^{N_{\rm in}}\Big[(\partial_{x_m}u_{\mathfrak{m}, \boldsymbol{\theta}}(\tilde{X_p}))^2-(\partial_{x_m}u_{\mathfrak{m}, \boldsymbol{\theta}}(X_p))^2\Big]\bigg|\\
    & \leq \frac{|\Omega|}{2N_{\rm in}} \mathbb{E}_{\{{X_p}, \tilde{X_p}\}_{p=1}^{N_{\rm in}}}\sup_{u_{\mathfrak{m}, \boldsymbol{\theta}}\in\mathcal{PNN}}\sum_{m=1}^{d}\bigg|\sum_{p=1}^{N_{\rm in}}\Big[(\partial_{x_m}u_{\mathfrak{m}, \boldsymbol{\theta}}(\tilde{X_p}))^2-(\partial_{x_m}u_{\mathfrak{m}, \boldsymbol{\theta}}(X_p))^2\Big]\bigg|\\
    &\leq \frac{|\Omega|}{2N_{\rm in}} \mathbb{E}_{\{{X_p}, \tilde{X_p}, \sigma_p\}_{p=1}^{N_{\rm in}}}\sup_{u_{\mathfrak{m}, \boldsymbol{\theta}}\in\mathcal{PNN}}\sum_{m=1}^{d}\bigg|\sum_{p=1}^{N_{\rm in}}\sigma_p \Big[(\partial_{x_m}u_{\mathfrak{m}, \boldsymbol{\theta}}(\tilde{X_p}))^2-(\partial_{x_m}u_{\mathfrak{m}, \boldsymbol{\theta}}(X_p))^2\Big]\bigg|\\
	&\leq\frac{|\Omega|}{2N_{\rm in}}\mathbb{E}_{\{{X_p}, \tilde{X_p}, \sigma_p\}_{p=1}^{N_{\rm in}}}\sup_{u_{\mathfrak{m}, \boldsymbol{\theta}}\in\mathcal{PNN}}\sum_{m=1}^{d}\bigg|\sum_{p=1}^{N_{\rm in}}\sigma_p (\partial_{x_m}u_{\mathfrak{m}, \boldsymbol{\theta}}(\tilde{X_p}))^2\bigg|\\
    &\quad +\frac{|\Omega|}{2N_{\rm in}}\mathbb{E}_{\{{X_p}, \tilde{X_p}, \sigma_p\}_{p=1}^{N_{\rm in}}}\sup_{u_{\mathfrak{m}, \boldsymbol{\theta}}\in\mathcal{PNN}}\sum_{m=1}^{d}\bigg|\sum_{p=1}^{N_{\rm in}}-\sigma_p (\partial_{x_m}u_{\mathfrak{m}, \boldsymbol{\theta}}(X_p))^2\bigg|\\
	&=\frac{|\Omega|}{N_{\rm in}}\mathbb{E}_{\{{X_p},{\sigma_p}\}_{p=1}^{N_{\rm in}}}\sup_{u_{\mathfrak{m}, \boldsymbol{\theta}}\in\mathcal{PNN}}\sum_{m=1}^{d}\bigg|\sum_{p=1}^{N_{\rm in}} \sigma_p (\partial_{x_m}u_{\mathfrak{m}, \boldsymbol{\theta}}(X_p))^2\bigg|\\
	&\leq\frac{|\Omega|}{N_{\rm in}}\sum_{m=1}^{d}\mathbb{E}_{\{{X_p},{\sigma_p}\}_{p=1}^{N_{\rm in}}}\sup_{u_{\mathfrak{m}, \boldsymbol{\theta}}\in\mathcal{PNN}}\bigg|\sum_{p=1}^{N_{\rm in}} \sigma_p (\partial_{x_m}u_{\mathfrak{m}, \boldsymbol{\theta}}(X_p))^2\bigg|\\
 &=\frac{|\Omega|}{N_{\rm in}}d \cdot \mathbb{E}_{\{{X_p},{\sigma_p}\}_{p=1}^{N_{\rm in}}}\sup_{u_{\mathfrak{m}, \boldsymbol{\theta}}\in\mathcal{PNN}}\bigg|\sum_{p=1}^{N_{\rm in}} \sigma_p (\partial_{x_1}u_{\mathfrak{m}, \boldsymbol{\theta}}(X_p))^2\bigg|\\
	&= \frac{|\Omega|}{N_{\rm in}}d\cdot \mathbb{E}_{\{{X_p},{\sigma_p}\}_{p=1}^{N_{\rm in}}}\sup_{f\in\mathcal{F}_{1}^{\prime}}\sum_{p=1}^{N_{\rm in}}\sigma_p f({X_p}) \le d \mm |\Omega| \mm\mathfrak{R}_{N_{\rm in}}(\mathcal{F}_{1}^{\prime})\, ,
\end{align*}
where the fourth step is due to the fact that the insertion of Rademacher variables does not change the distribution, the nineth step is due to the symmetric structure of the neural network function in $\mathcal{PNN}$ w.r.t. its individual variable components and the tenth step is due to the fact that $\mathcal{F}_{1}^{\prime}$ is symmetric (i.e., if $f\in \mathcal{F}_{1}^{\prime}$, then $-f \in \mathcal{F}_{1}^{\prime}$). 

Similarly, 
\begin{align*}
	\mathbb{E}[\mathcal{E}^{2}_{sta}] & = \mathbb{E}_{\{{X_p}\}_{p=1}^{N_{\rm in}}}\sup_{u_{\mathfrak{m}, \boldsymbol{\theta}}\in\mathcal{PNN}}\big|\mathcal{L}_2(u_{\mathfrak{m}, \boldsymbol{\theta}})-\widehat{\mathcal{L}}_2(u_{\mathfrak{m}, \boldsymbol{\theta}})\big|\\
 &\leq\frac{|\Omega|}{N_{\rm in}}\mathbb{E}_{\{{X_p},{\sigma_p}\}_{p=1}^{N_{\rm in}}}\sup_{u_{\mathfrak{m},\boldsymbol{\theta}}\in\mathcal{PNN}}\bigg|\sum_{p=1}^{N_{\rm in}} \sigma_p u^2_{\mathfrak{m}, \boldsymbol{\theta}}(X_p)\bigg|\leq \frac{|\Omega|}{N_{\rm in}}\mathbb{E}_{\{{X_p},{\sigma_p}\}_{p=1}^{N_{\rm in}}}\sup_{f\in\mathcal{F}_{2}^{\prime}}\bigg|\sum_{p=1}^{N_{\rm in}}\sigma_p \mm \omega({X_p})f({X_p})\bigg|\\
	&= \frac{|\Omega|}{N_{\rm in}}\mathbb{E}_{\{{X_p},{\sigma_p}\}_{p=1}^{N_{\rm in}}}\sup_{f\in\mathcal{F}_{2}^{\prime}}\sum_{p=1}^{N_{\rm in}}\sigma_p \mm \omega({X_p})f({X_p})= |\Omega|\mathfrak{R}_N(\omega\cdot\mathcal{F}_{2}^{\prime})
	\le B_0|\Omega| \mathfrak{R}_{N_{\rm in}}(\mathcal{F}_{2}^{\prime})\, ,
\end{align*}
where we use Proposition \ref{structural result of Rademacher} in the last step. And
\begin{align*}
	\mathbb{E}[\mathcal{E}^{3}_{sta}]\leq 2B_0|\Omega| \mathfrak{R}_{N_{\rm in}}(\mathcal{F}_{2})\, ,\quad 
	\mathbb{E}[\mathcal{E}^{4}_{sta}]	\leq 2B_0 |\partial\Omega| \mathfrak{R}_{N_{b}}(\mathcal{F}_{3})\, .
\end{align*}
\textbf{Step 2.}
Since $u_{\mathfrak{m}, \boldsymbol{\theta}} (\boldsymbol{x}) = \sum_{k=1}^{\mathfrak{m}} c_{k} \phi_{\boldsymbol{\theta}}^{k}(\boldsymbol{x})$, by the proof of Lemma \ref{Lip of Fi}, 
\begin{align*}
|u_{\mathfrak{m}, \boldsymbol{\theta}}| &= \bigg|\sum_{k=1}^{\mathfrak{m}} c_{k}  \phi_{\boldsymbol{\theta}}^{k}(\boldsymbol{x})\bigg| \leq M\cdot (W+1) B_{\boldsymbol{\theta}}\, ,\\
|\partial_{x_m}u_{\mathfrak{m}, \boldsymbol{\theta}}| &= \bigg|\sum_{k=1}^{\mathfrak{m}} c_{k}  \partial_{x_m}\phi_{\boldsymbol{\theta}}^{k}(\boldsymbol{x})\bigg| \leq M\cdot W^{L-1} B_{\boldsymbol{\theta}}^{L}\, .
\end{align*}
Then,
\begin{align*}
|u^2_{\mathfrak{m}, \boldsymbol{\theta}}(X_p) - \hat{u}^2_{\mathfrak{m}, \boldsymbol{\theta}}(X_p)| &\leq 2 M\cdot (W+1) B_{\boldsymbol{\theta}} \cdot |u_{\mathfrak{m}, \boldsymbol{\theta}}(X_p) - \hat{u}_{\mathfrak{m}, \boldsymbol{\theta}}(X_p)|\, ,\\
\big|(\partial_{x_m}u_{\mathfrak{m}, \boldsymbol{\theta}}(X_p))^2 - (\partial_{x_m}\hat{u}_{\mathfrak{m}, \boldsymbol{\theta}}(X_p))^2\big| &\leq 2 M\cdot W^{L-1} B_{\boldsymbol{\theta}}^{L}\cdot |\partial_{x_m} u_{\mathfrak{m}, \boldsymbol{\theta}}(X_p) - \partial_{x_m} \hat{u}_{\mathfrak{m}, \boldsymbol{\theta}}(X_p)|\, .\\
\end{align*}
By  Proposition \ref{Talagrand}, 
\begin{align*}
\mathfrak{R}_{N_{\rm in}}(\mathcal{F}_{1}^{\prime}) \leq 2 M\cdot W^{L-1} B_{\boldsymbol{\theta}}^{L}\cdot \mathfrak{R}_{N_{\rm in}}(\mathcal{F}_{1})\, , \quad
\mathfrak{R}_{N_{\rm in}}(\mathcal{F}_{2}^{\prime}) \leq  2M\cdot (W+1) B_{\boldsymbol{\theta}} \cdot \mathfrak{R}_{N_{\rm in}}(\mathcal{F}_{2})\, .
\end{align*}
\textbf{Step 3.} By Lemma \ref{relation of R}, we have $\mathfrak{R}_{N_{\rm in}}(\mathcal{F}_i) \leq M \cdot \hat{\mathfrak{R}}_{N_{\rm in}}(\mathcal{F}_{i,sub})$, then
\begin{align*}
&\mathbb{E}[\mathcal{E}^{1}_{sta}] \leq 2|\Omega| d M^2\cdot W^{L-1} B_{\boldsymbol{\theta}}^{L}\cdot \hat{\mathfrak{R}}_{N_{\rm in}}(\mathcal{F}_{1, sub})\, , ~~~~~~~~~ \mathbb{E}[\mathcal{E}^{3}_{sta}]\leq 2B_0|\Omega|M\cdot\hat{\mathfrak{R}}_{N_{\rm in}}(\mathcal{F}_{2, sub})\, ,\\
&\mathbb{E}[\mathcal{E}^{2}_{sta}] \leq 2 B_0 |\Omega| M^2\cdot (W+1) B_{\boldsymbol{\theta}} \cdot \hat{\mathfrak{R}}_{N_{\rm in}}(\mathcal{F}_{2, sub})\, , ~~~~~ \mathbb{E}[\mathcal{E}^{4}_{sta}] \leq 2B_0 |\partial\Omega|M\cdot \hat{\mathfrak{R}}_{N_{b}}(\mathcal{F}_{3, sub})\, .
\end{align*}
\end{proof}

Thirdly, we provide an upper bound for each  $\hat{\mathfrak{R}}_{N_{\rm in}}(\mathcal{F}_{i,sub}), i = 1, \ldots, 3$ in terms of the covering number   of $\mathcal{F}_{i,sub}$  by using Dudley's entropy formula (Proposition \ref{Dudley}).

\begin{defn}
An $\epsilon$-cover of a set $T$ in a metric space $(S, \tau)$
is a subset $T_c\subset S$ such  that for each $t\in T$, there exists a $t_c\in T_c$ such that $\tau(t, t_c) \leq\epsilon$. The $\epsilon$-covering number of $T$, denoted as $\mathcal{C}(\epsilon, T,\tau)$ is  defined to be the minimum cardinality among all $\epsilon$-cover of $T$ with respect to the metric $\tau$.
\end{defn}

\begin{prop}
\label{Massart}
Let $A \subseteq \mathbb{R}^{m}$ be a finite set, with $r=\max _{\boldsymbol{x} \in A}\|\boldsymbol{x}\|_{2}$, then the following holds:
$$
\underset{\boldsymbol{\sigma}}{\mathbb{E}}\bigg[\frac{1}{m} \sup _{\boldsymbol{x} \in A} \Big| \sum_{i=1}^{m} \sigma_{i} x_{i}\Big| \bigg] \leq \frac{r \sqrt{2 \log (2|A|)}}{m}\, ,
$$
where $\sigma_{i}$ are independent uniform random variables taking values in $\{-1,+1\}$ and $x_{1}, \ldots, x_{m}$ are the components of vector $\boldsymbol{x}$.
\end{prop}
\begin{proof}
For any $t>0$, using Jensen's inequality, rearranging terms, and bounding the supremum by a sum, we obtain:
\begin{align*}
\exp \bigg[t \underset{\boldsymbol{\sigma}}{\mathbb{E}}\bigg(\sup _{x \in A} \Big|\sum_{i=1}^{m} \sigma_{i} x_{i}\Big|\bigg)\bigg] & \leq \underset{\boldsymbol{\sigma}}{\mathbb{E}}\bigg[\exp \bigg(t \sup _{x \in A} \Big|\sum_{i=1}^{m} \sigma_{i} x_{i}\Big|\bigg)\bigg] \\
&=\underset{\boldsymbol{\sigma}}{\mathbb{E}}\bigg[\sup _{x \in A} \exp \bigg(t \Big|\sum_{i=1}^{m} \sigma_{i} x_{i}\Big|\bigg)\bigg] \leq \sum_{x \in A} \underset{\boldsymbol{\sigma}}{\mathbb{E}}\bigg[\exp \bigg(t \Big|\sum_{i=1}^{m} \sigma_{i} x_{i}\Big|\bigg)\bigg]\, .
\end{align*}
Since $e^{|x|} \leq e^{x}+e^{-x}$ and $e^{x}+e^{-x} \leq 2e^{x^2/2}$, it then holds that
\begin{align*}
    \sum_{x \in A} \underset{\boldsymbol{\sigma}}{\mathbb{E}}\bigg[\exp \bigg(t \Big|\sum_{i=1}^{m} \sigma_{i} x_{i}\Big|\bigg)\bigg] &\leq \sum_{x \in A} \underset{\boldsymbol{\sigma}}{\mathbb{E}}\bigg[\exp \bigg(t \sum_{i=1}^{m} \sigma_{i} x_{i}\bigg)\bigg] + \sum_{x \in A} \underset{\boldsymbol{\sigma}}{\mathbb{E}}\bigg[\exp \bigg(-t \sum_{i=1}^{m} \sigma_{i} x_{i}\bigg)\bigg] \\
    &= \sum_{x \in A} \prod_{i=1}^{m} \underset{\sigma_{i}}{\mathbb{E}}\big[\exp \mm (t \mm \sigma_{i} \mm x_{i})\big] + \sum_{x \in A} \prod_{i=1}^{m} \underset{\sigma_{i}}{\mathbb{E}}\big[\exp \mm (-t \mm \sigma_{i} \mm x_{i})\big] \\
    &= 2 \sum_{x \in A} \prod_{i=1}^{m} \frac{\exp\mm(t \mm x_i)+\exp\mm(-t\mm x_i)}{2} \leq 2 \sum_{x \in A} \prod_{i=1}^{m} \exp\bigg(\frac{t^2 x_i^2}{2}\bigg) \\
    &=2\sum_{x \in A} \exp \bigg(\frac{t^{2}}{2} \sum_{i=1}^{m} x_{i}^{2}\bigg) \leq 2\sum_{x \in A} \exp \bigg(\frac{t^{2} r^{2}}{2}\bigg)=2|A| e^{\frac{t^{2} r^{2}}{2}} \, .
\end{align*}
Taking the log of both sides and dividing by $t$ gives us:
$$
\underset{\boldsymbol{\sigma}}{\mathbb{E}}\bigg[\sup _{x \in A} \sum_{i=1}^{m} \sigma_{i} x_{i}\bigg] \leq \frac{\log (2|A|)}{t}+\frac{t r^{2}}{2} \, .
$$
If we choose $t=\frac{\sqrt{2 \log (2|A|)}}{r}$, which minimizes this upper bound, we get:
$$
\underset{\boldsymbol{\sigma}}{\mathbb{E}}\bigg[\sup _{x \in A} \sum_{i=1}^{m} \sigma_{i} x_{i}\bigg] \leq r \sqrt{2 \log (2|A|)}\, .
$$
Dividing both sides by $m$ leads to the statement of the lemma.
\end{proof}

\begin{prop}
\label{Dudley}
Let $\mathcal{F}$ be a class of functions from $\Omega$ to $\mathbb{R}$ such that $0\in \mathcal{F}$ and the diameter of $\mathcal{F}$ is less than $M$, i.e., $\|u\|_{L^{\infty}(\Omega)}\leq \mathcal{B}, \forall u\in \mathcal{F}$. Then
\begin{equation*}
  \hat{\mathfrak{R}}_N(\mathcal{F}) \leq \inf_{0<\delta<\mathcal{B}}\bigg(4\delta+\frac{12}{\sqrt{N}}\int_{\delta}^{\mathcal{B}}\sqrt{\log\mathcal{C}(\epsilon,\mathcal{F},d_{\infty})}\, \mathrm{d}\epsilon\bigg)\, .
\end{equation*}
\end{prop}
\begin{proof}
The proof is based on the chaining method. Set $\epsilon_{k}=2^{-k+1} \mathcal{B} .$ We denote by $\mathcal{F}_{k}$ such that $\mathcal{F}_{k}$ is an $\epsilon_{k}$-cover of $\mathcal{F}$ and $|\mathcal{F}_{k}|=\mathcal{C}(\epsilon_{k}, \mathcal{F}, d_{\infty}) .$ Hence for any $u \in \mathcal{F}$, there exists $u_{k} \in \mathcal{F}_{k}$ such that
$$
d_{\infty}(u,u_k)\leq \epsilon_{k} \, .
$$
Let $K$ be a positive integer determined later. We have
\begin{align*}
&\hat{\mathfrak{R}}_N(\mathcal{F}) =\underset{\{\sigma_{i}, X_{i}\}_{i=1}^{N}}{\mathbb{E}} \bigg[\sup _{u \in \mathcal{F}} \frac{1}{N} \Big| \sum_{i=1}^{N} \sigma_{i} u(X_{i})\Big|\bigg]\\
&=\underset{\{\sigma_{i}, X_{i}\}_{i=1}^{N}}{\mathbb{E}} \frac{1}{N} \sup _{u \in \mathcal{F}}\bigg[ \Big|\sum_{i=1}^{N} \sigma_{i}(u(X_{i})-u_{K}(X_{i})) +\sum_{j=1}^{K-1} \sum_{i=1}^{N} \sigma_{i}(u_{j+1}(X_{i})-u_{j}(X_{i}))+\sum_{i=1}^{N} \sigma_{i} u_{1}(X_{i})\Big|\bigg]\\
&\leq \underset{\{\sigma_{i}, X_{i}\}_{i=1}^{N}}{\mathbb{E}} \bigg[\sup _{u \in \mathcal{F}} \frac{1}{N} \Big|\sum_{i=1}^{N} \sigma_{i}(u(X_{i})-u_{K}(X_{i}))\Big|\bigg]+\underset{\{\sigma_{i}, X_{i}\}_{i=1}^{N}}{\mathbb{E}} \bigg[\sup _{u_1 \in \mathcal{F}_{1}} \frac{1}{N} \Big|\sum_{i=1}^{N} \sigma_{i} u_1(X_{i})\Big|\bigg]\\
&\quad +\sum_{j=1}^{K-1} \underset{\{\sigma_{i}, X_{i}\}_{i=1}^{N}}{\mathbb{E}}\bigg[\sup _{u \in \mathcal{F}} \frac{1}{N} \Big|\sum_{i=1}^{N} \sigma_{i}(u_{j+1}(X_{i})-u_{j}(X_{i}))\Big|\bigg] \, .
\end{align*}

We can choose $\mathcal{F}_{1}=\{0\}$ to eliminate the third term. For the first term,
$$
\underset{\{\sigma_{i}, X_{i}\}_{i=1}^{N}}{\mathbb{E}} \sup _{u \in \mathcal{F}} \frac{1}{N}\Big|\sum_{i=1}^{N} \sigma_{i}(u(X_{i})-u_{K}(X_{i}))\Big| \leq \underset{\{\sigma_{i}, X_{i}\}_{i=1}^{N}}{\mathbb{E}} \sup _{u \in \mathcal{F}} \frac{1}{N} \sum_{i=1}^{N}|\sigma_{i}|\|u-u_{K}\|_{L^\infty} \leq \epsilon_{K} \, .
$$
For the second term, for any fixed samples $\{X_{i}\}_{i=1}^{N}$, we define
$$
V_{j}:=\big\{(u_{j+1}(X_{1})-u_{j}(X_{1}), \ldots, u_{j+1}(X_{N})-u_{j}(X_{N})) \in \mathbb{R}^{N}: u \in \mathcal{F}\big\} \, .
$$
Then, for any $\boldsymbol{v}^{j} \in V_{j}$,
\begin{align*}
\|\boldsymbol{v}^{j}\|_{2} &=\bigg(\sum_{i=1}^{N}|u_{j+1}(X_{i})-u_{j}(X_{i})|^{2}\bigg)^{1 / 2} \leq \sqrt{N}\|u_{j+1}-u_{j}\|_{L^\infty} \\
& \leq \sqrt{N}\|u_{j+1}-u\|_{L^\infty}+\sqrt{N}\|u_{j}-u\|_{L^\infty}=\sqrt{N} \epsilon_{j+1}+\sqrt{N} \epsilon_{j}=3 \sqrt{N} \epsilon_{j+1} \, .
\end{align*}
Applying Proposition \ref{Massart}, we have
\begin{align*}
&\sum_{j=1}^{K-1} \underset{\{\sigma_{i}\}_{i=1}^{N}}{\mathbb{E}}\bigg[\sup _{u \in \mathcal{F}} \frac{1}{N} \Big| \sum_{i=1}^{N} \sigma_{i}(u_{j+1}(X_{i})-u_{j}(X_{i}))\Big| \bigg] \\
&=\sum_{j=1}^{K-1} \underset{\{\sigma_{i}\}_{i=1}^{N}}{\mathbb{E}} \bigg[\sup _{\boldsymbol{v}^{j} \in V_{j}} \frac{1}{N} \Big|\sum_{i=1}^{N} \sigma_{i} {v}_{i}^{\m j}\Big|\bigg] \leq \sum_{j=1}^{K-1} \frac{6 \epsilon_{j+1}}{\sqrt{N}} \sqrt{2 \log (2|V_{j}|)} \, .
\end{align*}
By the definition of $V_{j}$, we know that $|V_{j}| \leq|\mathcal{F}_{j}||\mathcal{F}_{j+1}| \leq|\mathcal{F}_{j+1}|^{2}$. Hence
$$
\sum_{j=1}^{K-1} \underset{\{\sigma_{i}, X_{i}\}_{i=1}^{N}}{\mathbb{E}} \bigg[\sup _{u \in \mathcal{F}} \frac{1}{N} \sum_{i=1}^{N} \sigma_{i}(u_{j+1}(X_{i})-u_{j}(X_{i}))\bigg] \leq \sum_{j=1}^{K-1} \frac{6 \epsilon_{j+1}}{\sqrt{N}} \sqrt{\log |\mathcal{F}_{j+1}|} \, .
$$
Now we obtain
\begin{align*}
\hat{\mathfrak{R}}_N(\mathcal{F}) & \leq \epsilon_{K}+\sum_{j=1}^{K-1} \frac{6 \epsilon_{j+1}}{\sqrt{N}} \sqrt{\log |\mathcal{F}_{j+1}|} \\
&=\epsilon_{K}+\frac{12}{\sqrt{N}} \sum_{j=1}^{K}(\epsilon_{j}-\epsilon_{j+1}) \sqrt{\log \mathcal{C}(\epsilon_{j}, \mathcal{F},d_{\infty})} \\
& \leq \epsilon_{K}+\frac{12}{\sqrt{N}} \int_{\epsilon_{K+1}}^{\mathcal{B}} \sqrt{\log \mathcal{C}(\epsilon, \mathcal{F},d_{\infty})} \, {\rm d} \epsilon \\
&\leq \inf_{0<\delta<\mathcal{B}}\bigg(4\delta+\frac{12}{\sqrt{N}}\int_{\delta}^{\mathcal{B}}\sqrt{\log\mathcal{C}(\epsilon,\mathcal{F},d_{\infty})} \, {\rm d}\epsilon\bigg)\, .
\end{align*}
where, last inequality holds since for $0\leq \delta \leq \mathcal{B}$, we can choose $K$ to be the largest integer such that $\epsilon_{K+1} > \delta$, at this time $\epsilon_K \leq 4\epsilon_{K+2} \leq 4\delta $.
\end{proof}

In Euclidean space, we can establish an upper bound of covering number for a bounded set easily.
\begin{prop} \label{covering number Euclidean space}
	Suppose that $T\subset\mathbb{R}^d$ and $\|t\|_2\leq B$ for $t\in T$, then
	\begin{equation*}
		\mathcal{C}(\epsilon,T,\|\cdot\|_2)\leq \bigg(\frac{2B\sqrt{d}}{\epsilon}\bigg)^{\!d}\, .
	\end{equation*}
\end{prop}
\begin{proof}
	Let $m=\lfloor 2B\sqrt{d} \epsilon^{-1}\rfloor$ and define
	\begin{equation*}
		T_c=\bigg\{-B+\frac{\epsilon}{\sqrt{d}},-B+\frac{2\epsilon}{\sqrt{d}},\cdots,-B+\frac{m\epsilon}{\sqrt{d}}\bigg\}^d\, ,
	\end{equation*}
	then for $t\in T$, there exists $t_c\in T_c$ such that
	\begin{equation*}
		\|t-t_c\|_2\leq\sqrt{\sum_{i=1}^{d}\bigg(\frac{\epsilon}{\sqrt{d}}\bigg)^2}=\epsilon \, .
	\end{equation*}
	Hence
	\begin{equation*}
		\mathcal{C}(\epsilon,T,\|\cdot\|_2)\leq|T_c|=m^d\leq\bigg(\frac{2B\sqrt{d}}{\epsilon}\bigg)^d\, .
	\end{equation*}
\end{proof}

A Lipschitz parameterization allows us to translates a cover of the function space into a cover of the parameter space. Such a property plays an essential role in our analysis of statistical error.
\begin{prop} \label{covering number Lipshcitz parameterization}
	Let $\mathcal{F}$ be a parameterized class of functions: $\mathcal{F} = \{f(\boldsymbol{x}; \boldsymbol{\theta}) : \boldsymbol{\theta}\in\Theta \}$. Let $\|\cdot\|_{\Theta}$ be a norm on $\Theta$ and let $\|\cdot\|_{\mathcal{F}}$ be a norm on $\mathcal{F}$. Suppose that the mapping $\boldsymbol{\theta} \mapsto f(\boldsymbol{x};\boldsymbol{\theta})$ is $\lambda$-Lipschitz, that is,
	\begin{equation*}
		\big\|f(\boldsymbol{x};\boldsymbol{\theta})-f\big(\boldsymbol{x};\tilde{\boldsymbol{\theta}}\big)\big\|_{\mathcal{F}} \leq \lambda\big\|\boldsymbol{\theta}-\tilde{\boldsymbol{\theta}}\big\|_{\Theta} \, ,
	\end{equation*}
	then for any $\epsilon>0$, 
 $$\mathcal{C}(\epsilon, \mathcal{F},\|\cdot\|_{\mathcal{F}}) \leq \mathcal{C}(\epsilon \lambda^{-1}, \Theta,\|\cdot\|_{\Theta}) \, .$$
\end{prop}
\begin{proof}
	Suppose that $\mathcal{C}(\epsilon  \lambda^{-1}, \Theta,\|\cdot\|_{\Theta})=n$ and $\{\boldsymbol{\theta}_i\}_{i=1}^{n}$ is an $\epsilon \lambda^{-1}$-cover of $\Theta$. Then for any $\boldsymbol{\theta}\in\Theta$, there exists $1\leq i\leq n$ such that
	\begin{equation*}
		\|f(\boldsymbol{x};\boldsymbol{\theta})-f(\boldsymbol{x};{\boldsymbol{\theta}}_i)\|_{\mathcal{F}} \leq \lambda\|\boldsymbol{\theta}-{\boldsymbol{\theta}}_i\|_{\Theta}\leq\epsilon \, .
	\end{equation*}
	Hence $\{f(\boldsymbol{x};\boldsymbol{\theta}_i)\}_{i=1}^{n}$ is an $\epsilon$-cover of $\mathcal{F}$, implying that $\mathcal{C}(\epsilon, \mathcal{F},\|\cdot\|_{\mathcal{F}}) \leq n$.
\end{proof}

Combining Proposition \ref{Dudley}, Proposition \ref{covering number Euclidean space} and Proposition \ref{covering number Lipshcitz parameterization}, we can obtain an upper bound for each  $\hat{\mathfrak{R}}_{N_{\rm in}}(\mathcal{F}_{i,sub}), i = 1, \ldots, 3$. 

Finally, using  the McDiarmid's inequality, we can obtain high probability control over $\mathcal{E}_{sta}$.

\begin{prop} \label{lem:McD}
 (McDiarmid's inequality) Let $g$ be a function from $\Omega_1 \times \Omega_2 \times \cdots \times \Omega_n$ to $\mathbb{R}$. Suppose that function $g$ satisfies the bounded differences property, i.e., there exists constants $c_1, \ldots, c_n>0$ such that for any $x_1 \in \Omega_1, \ldots, x_n \in \Omega_n$
$$
\sup _{\tilde{x}_i \in \Omega_i}|g(x_1, \cdots, \tilde{x}_i, \cdots, x_n)-g(x_1, \cdots, x_i, \cdots, x_n)| \leq c_i\, , \quad i = 1, \ldots, n \, .
$$
Let $\{X_i\}_{i=1}^n$ be independent variables, where $X_i \in \Omega_i$, then for any $\tau>0$, we have
$$
|g(X_1, \cdots, X_n)-\mathbb{E} [g(X_1, \cdots, X_n)]| \leq \tau
$$
with probability at least $1-2 \exp[-2 \tau^2 (\sum_{i=1}^n c_i^2)^{-1}]$.
\end{prop}
\begin{proof}
    See \cite{mcdiarmid1989method}.
\end{proof}

\begin{thm}
Let $\mathcal{PNN} = \mathcal{PNN}(\mathfrak{m}, M, \{W, L, B_{\boldsymbol{\theta}}\})$. Let $N_{\rm in} = N_{b} = N_{s}$ in the Monte Carlo sampling. Let $0<\xi<1$. Then, with probability at least $1-\xi$, it holds that
\begin{align*}
    \mathcal{E}_{sta}&=\sup_{u_{\mathfrak{m}, \boldsymbol{\theta}}\in\mathcal{PNN}}\big|\mathcal{L}(u_{\mathfrak{m}, \boldsymbol{\theta}})-\widehat{\mathcal{L}}(u_{\mathfrak{m}, \boldsymbol{\theta}})\big| \\
    &\leq C(\Omega, B_0, d, W, L)\cdot M^2  B_{\boldsymbol{\theta}}^{2L} N_{s}^{-\frac{1}{2}}  \big(\sqrt{\log(B_{\boldsymbol{\theta}} WLN_{s})}+\sqrt{\log\xi^{-1}}\big) \, ,
\end{align*}
where $C(\Omega, B_0, d, W, L)$ is a universal constant which only depends on $\Omega, B_0, d, W$ and $L$.
\end{thm}
\begin{proof}
By Lemma \ref{lem4.1} and Lemma \ref{sta1}, we have that
\begin{align*}
\underset{\{X_p\}_{p=1}^{N_{\rm in}},\{Y_p\}_{p=1}^{N_{b}}}{\mathbb{E}}&\bigg[\sup_{u_{\mathfrak{m}, \boldsymbol{\theta}}\in\mathcal{PNN}}\big|\mathcal{L}(u_{\mathfrak{m}, \boldsymbol{\theta}})-\widehat{\mathcal{L}}(u_{\mathfrak{m}, \boldsymbol{\theta}})\big|\bigg] \leq d\cdot |\Omega|\cdot W^{L-1}\cdot 2M^2 \cdot B_{\boldsymbol{\theta}}^{L} \cdot \hat{\mathfrak{R}}_{N_{\rm in}}(\mathcal{F}_{1, sub})\\
&+ 2B_0 \cdot|\Omega| \cdot \big(M(W+1)B_{\boldsymbol{\theta}}+1\big)\cdot M\cdot\hat{\mathfrak{R}}_{N_{\rm in}}(\mathcal{F}_{2, sub}) +  2B_0 |\partial\Omega|M\cdot \hat{\mathfrak{R}}_{N_{b}}(\mathcal{F}_{3, sub})\, .
\end{align*}
Using Proposition \ref{Dudley}, for $i =1, \ldots, 3$,
\begin{equation*}
  \hat{\mathfrak{R}}_N(\mathcal{F}_{i, sub}) \leq \inf_{0<\delta<B_i}\bigg(4\delta+\frac{12}{\sqrt{N}}\int_{\delta}^{B_i}\sqrt{\log\mathcal{C}(\epsilon,\mathcal{F}_{i, sub},d_{\infty})}\, \mathrm{d}\epsilon\bigg) \, .
\end{equation*}
Combining Lemma \ref{Lip of Fi} and Proposition \ref{covering number Lipshcitz parameterization}, we have that
\begin{equation*}
\mathcal{C}(\epsilon,\mathcal{F}_{i, sub},d_{\infty}) \leq \mathcal{C}(\epsilon L_i^{-1},\mathcal{F}_{i, sub},d_{\infty}) \ ,
\end{equation*}
and each of them can be bounded through Proposition \ref{covering number Euclidean space}:
\begin{equation*}
\mathcal{C}(\epsilon,\mathcal{F}_{i, sub},d_{\infty}) \leq \bigg(\frac{2B_{\boldsymbol{\theta}}\sqrt{\mathfrak{n}_{L}}\cdot L_i}{\epsilon}\bigg)^{\mathfrak{n}_{L}} \leq \bigg(\frac{2B_{\boldsymbol{\theta}}\sqrt{W(W+1)L}\cdot L_1}{\epsilon}\bigg)^{W(W+1)L}.
\end{equation*}
Therefore,
\begin{align*}
\hat{\mathfrak{R}}_{N_{\rm in}}(\mathcal{F}_{1, sub})& \leq \inf_{0<\delta<B_1}\bigg(4\delta+\frac{12}{\sqrt{N_{\rm in}}}\int_{\delta}^{B_1}\sqrt{\log\mathcal{C}(\epsilon,\mathcal{F}_{1, sub},d_{\infty})}\, \mathrm{d}\epsilon\bigg)\\
&\leq \inf_{0<\delta<B_1}\bigg(4\delta+\frac{12}{\sqrt{N_{\rm in}}}\int_{\delta}^{B_1}\sqrt{W(W+1)L} \mm \m{\log^{1/2}\!\big[2L(L+1)W^{2L}B_{\boldsymbol{\theta}}^{2L+1}\epsilon^{-1}\big]}\, \mathrm{d}\epsilon\bigg)\\
&\leq \inf_{0<\delta<B_1}\Big(4\delta+12W\sqrt{L} \cdot B_1 N^{-\frac{1}{2}}_{\rm in} \log^{1/2} \!\big[2L(L+1)W^{2L}B_{\boldsymbol{\theta}}^{2L+1}\delta^{-1}\big]\Big)\, .
\end{align*}
Choosing $\delta=N^{-\frac{1}{2}}_{\rm in}<B_1/2$ and applying Lemma $\ref{Lip of Fi}$, we have
\begin{align*}
\hat{\mathfrak{R}}_{N_{\rm in}}(\mathcal{F}_{1, sub})& \leq 12W(W+1)L\cdot B_{\boldsymbol{\theta}}^{L}\cdot N^{-\frac{1}{2}}_{\rm in} \cdot \sqrt{\log(B_{\boldsymbol{\theta}} WLN_{s})}\, .
\end{align*}
Moreover, $\hat{\mathfrak{R}}_{N_{\rm in}}(\mathcal{F}_{2, sub})$ and $ \hat{\mathfrak{R}}_{N_{\rm in}}(\mathcal{F}_{3, sub})$ can also be bounded in a similar way to $\hat{\mathfrak{R}}_{N_{\rm in}}(\mathcal{F}_{1, sub})$. Add them together and we get
\begin{align} \label{eq:expgene}
&\underset{\{X_p\}_{p=1}^{N_{\rm in}},\{Y_p\}_{p=1}^{N_{b}}}{\mathbb{E}}\bigg[\sup_{u_{\mathfrak{m}, \boldsymbol{\theta}}\in\mathcal{PNN}}\big|\mathcal{L}(u_{\mathfrak{m}, \boldsymbol{\theta}})-\widehat{\mathcal{L}}(u_{\mathfrak{m}, \boldsymbol{\theta}})\big|\bigg] \\
&~~~~~~~~~~~~~~~~~\leq C_1(\Omega, B_0, d, W, L)\cdot M^2  B_{\boldsymbol{\theta}}^{2L}N_{s}^{-\frac{1}{2}}  \sqrt{\log(B_{\boldsymbol{\theta}} WLN_{s})} \notag \, .
\end{align}
Now we define 
\begin{align*}
    \gamma(X_1, \ldots, X_{N_{\rm in}}, Y_1, \ldots, Y_{N_b}) = \sup_{u_{\mathfrak{m}, \boldsymbol{\theta}}\in\mathcal{PNN}}|\mathcal{L}(u_{\mathfrak{m}, \boldsymbol{\theta}})-\widehat{\mathcal{L}}(u_{\mathfrak{m}, \boldsymbol{\theta}})| \, ,
\end{align*}
Here, we expand $\mathcal{L}(u_{\mathfrak{m}, \boldsymbol{\theta}})$ as follows
\begin{align*}
    \mathcal{L}(u_{\mathfrak{m}, \boldsymbol{\theta}})&=|\Omega|\underset{X\sim U(\Omega)}{\mathbb{E}}\bigg[\frac{\| \nabla u_{\mathfrak{m}, \boldsymbol{\theta}}(X)\|_2^2}{2}+ \frac{\omega(X)u_{\mathfrak{m}, \boldsymbol{\theta}}^2(X)}{2} -u_{\mathfrak{m}, \boldsymbol{\theta}}(X)h(X)\bigg] \\
    &~~~~-|\partial\Omega|\underset{Y\sim U(\partial\Omega)}{\mathbb{E}}[Tu(Y)g(Y)] \, ,
\end{align*}
 Also, we expand $\widehat{\mathcal{L}}(u_{\mathfrak{m}, \boldsymbol{\theta}})$ as follows:
\begin{align*}
    \widehat{\mathcal{L}}(u_{\mathfrak{m}, \boldsymbol{\theta}})&=\frac{|\Omega|}{N_{\rm in}}\sum_{p=1}^{N_{\rm in}}\bigg[\frac{\| \nabla u_{\mathfrak{m}, \boldsymbol{\theta}}(X_p)\|_2^2}{2} +\frac{\omega(X_p)u_{\mathfrak{m}, \boldsymbol{\theta}}^2(X_p)}{2}- u_{\mathfrak{m}, \boldsymbol{\theta}}(X_p)h(X_p)\bigg] \\
    &~~~~-\frac{|\partial\Omega|}{N_b}\sum_{p=1}^{N_b}[u_{\mathfrak{m}, \boldsymbol{\theta}}(Y_p)g(Y_p)] \, ,
\end{align*}
We then examine the difference of $\gamma(X_1, \ldots, X_{N_{\rm in}}, Y_1, \ldots, Y_{N_b})$:
\begin{align*}
    &|\gamma(X_1, \ldots, X_i, \ldots, Y_{N_b}) - \gamma(X_1, \ldots, X^{\prime}_i, \ldots, Y_{N_b})|\\
    &~~~\leq \frac{|\Omega|}{N_{\rm in}}\sup_{u_{\mathfrak{m}, \boldsymbol{\theta}} \in \mathcal{PNN}}\bigg|\frac{\| \nabla u_{\mathfrak{m}, \boldsymbol{\theta}}(X_i)\|_2^2-\| \nabla u_{\mathfrak{m}, \boldsymbol{\theta}}(X^{\prime}_i)\|_2^2}{2} +\frac{\omega(X_i)u_{\mathfrak{m}, \boldsymbol{\theta}}^2(X_i)-\omega(X^{\prime}_i)u_{\mathfrak{m}, \boldsymbol{\theta}}^2(X^{\prime}_i)}{2} \\
    &~~~~~~~~~~~~~~~~~~~~~~~~~~+u_{\mathfrak{m}, \boldsymbol{\theta}}(X^{\prime}_i)h(X^{\prime}_i)-u_{\mathfrak{m}, \boldsymbol{\theta}}(X_i)h(X_i)\bigg| \\
    &~~~\leq 4|\Omega|N^{-1}_{\rm in}d \mm(B_0+1)W^{2L-2}B_{\boldsymbol{\theta}}^{2L} \, ,
\end{align*}
where we have utilized the boundedness properties outlined in Lemma \ref{Lip of Fi}. We also have
\begin{align*}
    &|\gamma(X_1, \ldots, Y_j, \ldots, Y_{N_b}) - \gamma(X_1, \ldots, Y^{\prime}_j, \ldots, Y_{N_b})|\\
    &~~~~~~\leq \frac{|\partial\Omega|}{N_b}\big|u_{\mathfrak{m}, \boldsymbol{\theta}}(Y^{\prime}_j)g(Y^{\prime}_j)-u_{\mathfrak{m}, \boldsymbol{\theta}}(Y_j)g(Y_j)\big| \leq 2 |\partial\Omega| N_b^{-1} B_0 (W+1)B_{\boldsymbol{\theta}} \, .
\end{align*}
Then by Proposition \ref{lem:McD} and (\ref{eq:expgene}), it holds that
\begin{align*}
    &\sup_{u_{\mathfrak{m}, \boldsymbol{\theta}}\in\mathcal{PNN}}\big|\mathcal{L}(u_{\mathfrak{m}, \boldsymbol{\theta}})-\widehat{\mathcal{L}}(u_{\mathfrak{m}, \boldsymbol{\theta}})\big| \\
    &~~~~~~~~~~\leq \underset{\{X_p\}_{p=1}^{N_{\rm in}},\{Y_p\}_{p=1}^{N_{b}}}{\mathbb{E}}\bigg[\sup_{u_{\mathfrak{m}, \boldsymbol{\theta}}\in\mathcal{PNN}}\big|\mathcal{L}(u_{\mathfrak{m}, \boldsymbol{\theta}})-\widehat{\mathcal{L}}(u_{\mathfrak{m}, \boldsymbol{\theta}})\big|\bigg] + \tau \\
    & ~~~~~~~~~~\leq C_1(\Omega, B_0, d, W, L)\cdot M^2  B_{\boldsymbol{\theta}}^{2L} N_{s}^{-\frac{1}{2}}  \sqrt{\log(B_{\boldsymbol{\theta}} WLN_{s})} + \tau
\end{align*}
with probability as least 
\begin{align*}
    1 - \exp \bigg\{-\frac{N_s\tau^{2}}{32d^2(|\partial \Omega|^2+|\Omega|^2)(B_0+1)^2W^{4L}B_{\boldsymbol{\theta}}^{4L}}\bigg\} \, .
\end{align*}
This implies that with probability at least $1-\xi$, we have
\begin{align*}
    &\sup_{u_{\mathfrak{m}, \boldsymbol{\theta}}\in\mathcal{PNN}}\big|\mathcal{L}(u_{\mathfrak{m}, \boldsymbol{\theta}})-\widehat{\mathcal{L}}(u_{\mathfrak{m}, \boldsymbol{\theta}})\big| \\
    &~~~~~~~~~\leq C_2(\Omega, B_0, d, W, L)\cdot M^2  B_{\boldsymbol{\theta}}^{2L} N_{s}^{-\frac{1}{2}}  \big(\sqrt{\log(B_{\boldsymbol{\theta}} WLN_{s})}+\sqrt{\log\xi^{-1}}\big) \, .
\end{align*}
\end{proof}

\subsection{Detailed optimization error analysis}
\label{proof of opt}
\subsubsection{Analysis of the iteration error} \label{iter}
We first introduce the following results  for the PGD algorithm.
\begin{lem} 

\label{opt_lem}
Let $d_1, d_2 \in \mathbb{N}$, let $U, V \geq 0$, let $\mathrm{X} \subset \mathbb{R}^{d_1}$ and $\mathrm{Y} \subseteq \mathbb{R}^{d_2}$ be closed and convex, and let $F: \mathbb{R}^{d_1} \times \mathbb{R}^{d_2} \rightarrow \mathbb{R}_{+}$be a function such that $F(\boldsymbol{x}, \boldsymbol{y})$ is differentiable while
$\boldsymbol{y} \mapsto F(\boldsymbol{x}, \boldsymbol{y})$ is  convex for all $\boldsymbol{x} \in \mathbb{R}^{d_1}$. Meanwhile, assume that
\begin{equation}
\label{opteq1}
\|(\nabla_{\boldsymbol{y}} F)(\boldsymbol{x}, \boldsymbol{y})\|_2 \leq V \, ,
\end{equation}
\begin{equation}
\label{opteq2}
\|\nabla F(\boldsymbol{x}_1, \boldsymbol{y}_1) - \nabla F(\boldsymbol{x}_2, \boldsymbol{y}_2) \|_2 \leq K \|(\boldsymbol{x}_1, \boldsymbol{y}_1) - (\boldsymbol{x}_2, \boldsymbol{y}_2)\|_2 \, .
\end{equation}
for all $(\boldsymbol{x}, \boldsymbol{y}), (\boldsymbol{x}_1, \boldsymbol{y}_1)$ and $  (\boldsymbol{x}_2, \boldsymbol{y}_2) \in \mathrm{X} \times \mathrm{Y}$. Choose $(\boldsymbol{x}_0, \boldsymbol{y}_0) \in \mathrm{X} \times \mathrm{Y}$ and set
\begin{equation}
\label{opteq3}
({\boldsymbol{x}}_{t+1}, {\boldsymbol{y}}_{t+1})=\operatorname{Proj}_{\mkern 3mu \mathrm{X} \times \mathrm{Y}} \! \big\{(\boldsymbol{x}_t, \boldsymbol{y}_t)-\lambda \cdot \nabla F(\boldsymbol{x}_t, \boldsymbol{y}_t)\big\} \, ,
\end{equation}
for $t = 0, 1, \ldots, T$, where
$$
\lambda=\frac{1}{T} \wedge \frac{2}{K} \, .
$$
Let $\boldsymbol{y}^* \in Y$ and assume
\begin{equation}
\label{opteq4}
|F(\boldsymbol{x}_t, \boldsymbol{y}^*)-F(\boldsymbol{x}_0, \boldsymbol{y}^*)| \leq U \cdot\|\boldsymbol{y}^*\|_2 \cdot\|\boldsymbol{x}_t-\boldsymbol{x}_0\|_2 \, .
\end{equation}
for all $t=1, \ldots, T$. Then it holds:
\begin{align} \label{eq:opt}
F(\boldsymbol{x}_{\scriptscriptstyle T}, \boldsymbol{y}_{\scriptscriptstyle T}) -F(\boldsymbol{x}_0, \boldsymbol{y}^*) \leq U \cdot\|\boldsymbol{y}^*\|_2 \cdot \operatorname{diam}({\rm X})+\frac{\|\boldsymbol{y}^*-\boldsymbol{y}_0\|_2^2}{2}+\frac{V^2}{2 \mm T} \, .
\end{align}
\end{lem}

If we could choose $F$ to be $\widehat{F}$, 
$(\boldsymbol{x}_t, \boldsymbol{y}_t)$ to be $(\boldsymbol{\theta}_{\rm in}^{\mathfrak{m}}, \boldsymbol{\theta}_{\rm out}^{\mathfrak{m}})^{\scriptscriptstyle [t]}$ , and $\boldsymbol{y}^{*}$ to be $\boldsymbol{\theta}_{\rm out}^{\mathfrak{m, *}}$ in Lemma \ref{opt_lem}, then according to (\ref{eq:opt}), we can obtain an estimate of $\widehat{F}((\boldsymbol{\theta}^{\mathfrak{m}}_{ \rm in})^{\scriptscriptstyle [T]}, (\boldsymbol{\theta}^{\mathfrak{m}}_{ \rm out})^{\scriptscriptstyle [T]})-\widehat{F}((\boldsymbol{\theta}^{\mathfrak{m}}_{ \rm in})^{\scriptscriptstyle [0]}, \boldsymbol{\theta}^{\mathfrak{m}, *}_{ \rm out})$, which is exactly the iteration error in (\ref{eq:opt decomp}). However, this necessitates satisfying the requirements outlined in (\ref{opteq1}), (\ref{opteq2}) and (\ref{opteq4}). Therefore, we propose the following lemmas to meticulously characterize the properties of $\widehat{F}$. Note that proofs for all the lemmas in this section can be found in the Appendix \ref{proof of opt}.

We first obtain an upper bound of $\nabla_{\boldsymbol{\theta}^{\mathfrak{m}}_{\rm out}} \widehat{F}(\boldsymbol{\theta}^{\mathfrak{m}}_{ \rm in}, \boldsymbol{\theta}^{\mathfrak{m}}_{ \rm out})$ to meet the condition in (\ref{opteq1}).
\begin{lem}
\label{bound_lem}
For $u_{\mathfrak{m}, \boldsymbol{\theta}}\in \mathcal{PNN}(\mathfrak{m}, M, \{W, L, B_{\boldsymbol{\theta}}\})$, we denote the empirical risk $\widehat{\mathcal{L}}(u_{\mathfrak{m}, \boldsymbol{\theta}})$ in (\ref{eq12}) as $\widehat{F}(\boldsymbol{\theta}_{\rm total}^{\mathfrak{m}})=\widehat{F}(\boldsymbol{\theta}_{\rm in}^{\mathfrak{m}}, \boldsymbol{\theta}_{\rm out}^{\mathfrak{m}})$ to omit the dependence on sample points. Then, we have
\begin{align}
\big \|\nabla_{\boldsymbol{\theta}^{\mathfrak{m}}_{\rm out}} \widehat{F}\big(\boldsymbol{\theta}^{\mathfrak{m}}_{ \rm in}, \boldsymbol{\theta}^{\mathfrak{m}}_{ \rm out}\big)\big\|_2^{2} \leq C_1 \cdot \mathfrak{m} \cdot M^2 \cdot B^{4L}_{\boldsymbol{\theta}} 
\, ,
\end{align}
where $C_1$ is a universal constant which only depends on $\Omega, W, L, d$ and $B_0$, $\mathfrak{m}$ is the total number of sub-networks in $u_{\mathfrak{m}, \boldsymbol{\theta}}$, $\|\boldsymbol{\theta}_{\rm out}^{\mathfrak{m}}\|_1 \leq M$, and each sub-network weight in $\boldsymbol{\theta}_{\rm in}^{\mathfrak{m}}$ belongs to $[-B_{\boldsymbol{\theta}}, B_{\boldsymbol{\theta}}]$.
\end{lem}
Next, in order to satisfy the condition in (\ref{opteq2}), we will estimate the Lipschitz constant of $\nabla \widehat{F}(\boldsymbol{\theta}^{\mathfrak{m}}_{\rm total})$. For this endeavor, we first associate the Lipschitz property of the gradient $\nabla f$ with the norm of Hessian matrix $\nabla^{2}f$.

\begin{lem}
\label{hessian_lem}
For $f(\boldsymbol{x})$ which is convex and twice differentiable, it holds that 
$$ \|\nabla^2 f(\boldsymbol{x})\|_{2, 2} \leq \|\nabla^2 f(\boldsymbol{x})\|_{\rm F} \le K \Longrightarrow \|\nabla f(\boldsymbol{x})-\nabla f(\boldsymbol{y})\|_2 \le K\|\boldsymbol{x}-\boldsymbol{y}\|_2 \, ,$$
where $\|\cdot\|_{2, 2}$ represents the spectral norm of the matrix, while $\|\cdot\|_{\rm F}$ represents the Frobenius norm of the matrix.
\end{lem}
Thus, it suffices to estimate the Frobenius norm of the Hessian matrix of $\widehat{F}(\boldsymbol{\theta}^{\mathfrak{m}}_{\rm total})$.
\begin{lem}
\label{hessian}
With notations and symbols consistent with those in Lemma \ref{bound_lem}, we have
$$
\big\|\nabla^2_{\boldsymbol{\theta}^{\mathfrak{m}}_{\rm total}} \widehat{F}\big(\boldsymbol{\theta}^{\mathfrak{m}}_{\rm total}\big)\big\|_{\rm F} \le C_2 \cdot \mathfrak{m} \cdot M^2 \cdot B^{4L}_{\boldsymbol{\theta}} \, ,
$$
where $C_2$ is a universal constant which only depends on $\Omega, W, L, d$ and $B_0$. Utilizing Lemma \ref{hessian_lem}, $\nabla \widehat{F}(\boldsymbol{\theta}^{\mathfrak{m}}_{\rm total})$ is then equipped with Lipschitz continuity and Lipschitz constant $C_2 \cdot \mathfrak{m} \cdot M^2 \cdot B^{4L}_{\boldsymbol{\theta}}$.
\end{lem}
Finally, we assures that $\widehat{F}$ satisfies the condition in (\ref{opteq4}).
\begin{lem}
With notations and symbols consistent with those in Lemma \ref{bound_lem}, we have
\label{lip_lem}
\begin{align*}
&\big|\widehat{F}\big(\boldsymbol{\theta}^{\mathfrak{m}, 1}_{ \rm in}, \boldsymbol{\theta}^{\mathfrak{m}}_{ \rm out}\big)-\widehat{F}\big(\boldsymbol{\theta}^{\mathfrak{m}, 2}_{ \rm in}, \boldsymbol{\theta}^{\mathfrak{m}}_{ \rm out}\big)\big| \leq C_3 \cdot M\cdot B^{3L}_{\boldsymbol{\theta}} \cdot \big\|\boldsymbol{\theta}^{\mathfrak{m}}_{ \rm out}\big\|_2 \cdot \big\|\boldsymbol{\theta}_{\rm in}^{\mathfrak{m}, 1} - \boldsymbol{\theta}_{\rm in}^{\mathfrak{m}, 2}\big\|_2 \, ,
\end{align*}
where $\boldsymbol{\theta}_{\rm in}^{\mathfrak{m}, 1}$ and $\boldsymbol{\theta}_{\rm in}^{\mathfrak{m}, 2}$ denote two different sub-network weight vectors. $C_3$ is a universal constant which only depends on $\Omega, W, L, d$ and $B_0$.
\end{lem}
Based on these, we return to our specific setting. Let $\zeta=\bar{M}$ in (\ref{eq:M}) to ensure that $ \boldsymbol{\theta}^{\mathfrak{m}, *}_{\rm out} \in B_{\zeta}=B_{\bar{M}}$, while we naturally have $(\boldsymbol{\theta}^{\mathfrak{m}}_{ \rm in})^{\scriptscriptstyle [0]} \in A_{\eta} $. Now, combining Lemmas \ref{opt_lem}, \ref{bound_lem}, \ref{hessian} and \ref{lip_lem}, together with (\ref{eq:norm*}), we achieve the following corollary to bound the iteration error.
\begin{cor} \label{cor:iteration}
    Let $\zeta=\bar{M}$ in (\ref{eq:M}). Then, we get $u_{\mathcal{A}}$, the output of the PGD algorithm in (\ref{eq:pgd}), belonging to $\mathcal{PNN}(\mathfrak{m}, \bar{M}, \{\bar{W}, \bar{L}, B_{\bar{\boldsymbol{\theta}}} + \eta\})$. Also, we run the algorithm with step size $\lambda$ satisfying
    \begin{align*}
        \lambda = T^{-1} \wedge 2\mm C_2^{-1} \m \mm \mathfrak{m}^{-1} \mm \bar{M}^{-2} \mm (B_{\bar{\boldsymbol{\theta}}}+\eta)^{-4\bar{L}} \, ,
    \end{align*}
    where $T$ is the total number of iterations, and $\eta$ is the projection radius of sub-network weights in (\ref{eq:eta}). Then, with $u^{*}_{\mathfrak{m}} \in \mathcal{PNN}(\mathfrak{m}, \bar{M}, \{\bar{W}, \bar{L}, B_{\bar{\boldsymbol{\theta}}}+\eta \})$ defined in (\ref{eq:u*}), the iteration error in (\ref{eq:opt decomp}) is bounded by
\begin{align*} \widehat{\mathcal{L}}(u_{\mathcal{A}})-\widehat{\mathcal{L}}(u_{\mathfrak{m}}^{*}) \leq \frac{C_3 \cdot \bar{M}^2\cdot (B_{\bar{\boldsymbol{\theta}}}+\eta)^{3\bar{L}} \cdot \eta}{\sqrt{R}}+\frac{\bar{M}^2}{2 R}+\frac{C_1 \cdot \mathfrak{m} \cdot \bar{M}^2 \cdot (B_{\bar{\boldsymbol{\theta}}}+\eta)^{4\bar{L}}}{2\mm T} \, .
\end{align*}
Here, $C_1$ is from Lemma \ref{bound_lem}, $C_2$ is from Lemma \ref{hessian} and $C_3$ is from Lemma \ref{lip_lem}.
\end{cor}

\subsubsection{Analysis of the initialization error} \label{init}
We first propose the following lemma to estimate the probability of $G_{\mathfrak{m}, \bar{\mathfrak{m}}, R, \delta}$.
\begin{lem} \label{lem:prob of G}
    For $u_{\bar{\mathfrak{m}}, \bar{\boldsymbol{\theta}}} \in \mathcal{PNN}(\bar{\mathfrak{m}}, \bar{M}, \{\bar{W}, \bar{L}, B_{\bar{\boldsymbol{\theta}}}\})$ in Corollary \ref{cor4.1} with sub-network parameters $(\bar{\boldsymbol{\theta}}_1, \ldots, \bar{\boldsymbol{\theta}}_{\bar{\mathfrak{m}}})$, if we have $\mathfrak{m} \in \mathbb{N}$ satisfying 
    $$
    \mathfrak{m} = \bar{\mathfrak{m}} \cdot R \cdot Q\, , \quad R, Q \in \mathbb{N} \,\, \text{and} \,\, Q  \text{ is sufficiently large} \, .
    $$
    Then, it holds that
    $$
    \mathbb{P}\big(G_{\mathfrak{m}, \bar{\mathfrak{m}}, R, \delta}\big) \ge 1 - \bar{\mathfrak{m}}R\Big[1-\delta^{\bar{W}(\bar{W}+1)\bar{L}}(2B_{\bar{\boldsymbol{\theta}}})^{-\bar{W}(\bar{W}+1)\bar{L}}\Big]^{Q} \, .
    $$
\end{lem}

 Intuitively, the initialization error would be well-controlled if there exists only a slight perturbation between the target weights $\bar{\boldsymbol{\theta}}_k$ and random initialization $(\boldsymbol{\theta}_{s_{k, v}})^{\scriptscriptstyle[0]}$. We introduce the following lemma to formalize such intuition.
\begin{lem}
\label{lip_lem2}
Choose $u_{\bar{\mathfrak{m}}, \bar{\boldsymbol{\theta}}} \in \mathcal{PNN}(\bar{\mathfrak{m}}, \bar{M}, \{\bar{W}, \bar{L}, B_{\bar{\boldsymbol{\theta}}}\})$ in Corollary \ref{cor4.1}. Then, it holds that
\begin{align*}
 \widehat{\mathcal{L}}(u_{\mathfrak{m}}^{*}) -\widehat{\mathcal{L}}(u_{\bar{\mathfrak{m}}, \bar{\boldsymbol{\theta}}}) \leq C_4 \cdot \bar{M}^2 \cdot B^ {\raisebox{0.2ex}{$\scriptscriptstyle 3\bar{L}$}}_{\bar{\boldsymbol{\theta}}} \cdot \max_{k = 1, \ldots, \bar{\mathfrak{m}}} \max_{v = 1, \ldots, R} \big\|(\boldsymbol{\theta}_{s_{k,v}})^{\scriptscriptstyle [0]} -\bar{\boldsymbol{\theta}}_k\big\|_{\infty} \, 
\end{align*}
with $u^{*}_{\mathfrak{m}} \in \mathcal{PNN}(\mathfrak{m}, \bar{M}, \{\bar{W}, \bar{L}, B_{\bar{\boldsymbol{\theta}}}+\eta\})$ defined in (\ref{eq:u*}). Here the random indices $s_{k, v}$ are from (\ref{eq:skv}).  $C_4$ is a universal constant which only depends on $\Omega, \bar{W}, \bar{L}, d$ and $B_0$.
\end{lem}
Combining Lemmas \ref{lem:prob of G} and \ref{lip_lem2}, by properly selecting the sub-network size $\mathfrak{m}$ of $u_{\mathfrak{m}, \boldsymbol{\theta}}$, we can bound the initialization error with arbitrary high probability and precision.
\begin{prop} \label{lem:pertur}
    Choose $u_{\bar{\mathfrak{m}}, \bar{\boldsymbol{\theta}}} \in \mathcal{PNN}(\bar{\mathfrak{m}}, \bar{M}, \{\bar{W}, \bar{L}, B_{\bar{\boldsymbol{\theta}}}\})$ in Corollary \ref{cor4.1}. Let $\delta>0$, $R, Q\in\mathbb{N}$ while $Q$ is sufficiently large. If we set the number of sub-networks  $\mathfrak{m} = \bar{\mathfrak{m}} \cdot R \cdot Q$, then with probability at least 
    \begin{align*}
        1 - \bar{\mathfrak{m}}R\Big[1-\delta^{\bar{W}(\bar{W}+1)\bar{L}}(2B_{\bar{\boldsymbol{\theta}}})^{-\bar{W}(\bar{W}+1)\bar{L}}\Big]^{Q} \, ,
    \end{align*}
    the initialization error in (\ref{eq:opt decomp}) is bounded by
    \begin{align*}
        \widehat{\mathcal{L}}(u_{\mathfrak{m}}^{*}) -\widehat{\mathcal{L}}(u_{\bar{\mathfrak{m}}, \bar{\boldsymbol{\theta}}}) \leq C_4 \cdot \bar{M}^2 \cdot B^ {\raisebox{0.2ex}{$\scriptscriptstyle 3\bar{L}$}}_{\bar{\boldsymbol{\theta}}} \cdot \delta \, ,
    \end{align*}
    with $u^{*}_{\mathfrak{m}} \in  \mathcal{PNN}({\mathfrak{m}}, \bar{M}, \{\bar{W}, \bar{L}, B_{\bar{\boldsymbol{\theta}}}+\eta\})$ defined in (\ref{eq:u*}). $C_4$ is from Lemma \ref{lip_lem2}.
\end{prop}

\subsubsection{Proof of Lemma \ref{opt_lem}}
    In the first step of the proof we show
\begin{align*}
\frac{1}{T} \sum_{t=0}^{T-1} F(\boldsymbol{x}_t, \boldsymbol{y}_t) \leq \frac{1}{T} \sum_{t=0}^{T-1} F(\boldsymbol{x}_t, \boldsymbol{y}^*)+\frac{\|\boldsymbol{y}^*-\boldsymbol{y}_0\|^2}{2}+\frac{1}{2 \mm T^2} \sum_{t=0}^{T-1}\|(\nabla_{\boldsymbol{y}} F)(\boldsymbol{x}_t, \boldsymbol{y}_t)\|^2\, .
\end{align*}
By convexity of $\boldsymbol{y} \mapsto F(\boldsymbol{x}_t, \boldsymbol{y})$ and because of $\boldsymbol{y}^* \in \mathrm{Y}$ we have
\begin{align*}
& F(\boldsymbol{x}_t, \boldsymbol{y}_t)-F(\boldsymbol{x}_t, \boldsymbol{y}^*) \leq 
\big \langle(\nabla_{\boldsymbol{y}} F)(\boldsymbol{x}_t, \boldsymbol{y}_t), \boldsymbol{y}_t-\boldsymbol{y}^* \big \rangle \\
&~~~~=\frac{1}{2 \mkern 2mu \lambda} \cdot 2 \cdot \big \langle \lambda \mkern 2mu (\nabla_{\boldsymbol{y}} F)(\boldsymbol{x}_t, \boldsymbol{y}_t), \boldsymbol{y}_t-\boldsymbol{y}^*\big \rangle \\
& ~~~~ =\frac{1}{2 \mkern 2mu \lambda}  \Big[ -\big\|\boldsymbol{y}_t-\boldsymbol{y}^*-\lambda \mm (\nabla_{\boldsymbol{y}} F)(\boldsymbol{x}_t, \boldsymbol{y}_t)\big\|_2^2+\|\boldsymbol{y}_t-\boldsymbol{y}^*\|_2^2+ \big \|\lambda \mm (\nabla_{\boldsymbol{y}} F)(\boldsymbol{x}_t, \boldsymbol{y}_t) \big \|_2^2 \Big] \\
& ~~~~ \leq \frac{1}{2 \mkern 2mu \lambda}  \Big[ -\big \|\operatorname{Proj}_{\, \rm Y}\big\{\boldsymbol{y}_t-\lambda \mm (\nabla_{\boldsymbol{y}} F)(\boldsymbol{x}_t, \boldsymbol{y}_t)\big\}-\boldsymbol{y}^*\big \|_2^2+\|\boldsymbol{y}_t-\boldsymbol{y}^*\|_2^2+\lambda^2  \big\|(\nabla_{\boldsymbol{y}} F)(\boldsymbol{x}_t, \boldsymbol{y}_t)\big\|_2^2 \Big] \\
& ~~~~ =\frac{1}{2 \mkern 2mu \lambda}  \Big[ \|\boldsymbol{y}_t-\boldsymbol{y}^*\|^2-\|\boldsymbol{y}_{t+1}-\boldsymbol{y}^*\|_2^2+\lambda^2  \big\|(\nabla_{\boldsymbol{y}} F)(\boldsymbol{x}_t, \boldsymbol{y}_t)\big\|_2^2 \Big] \, .
\end{align*}
This implies
\begin{align*}
& \frac{1}{T} \sum_{t=0}^{T-1} F(\boldsymbol{x}_t, \boldsymbol{y}_t)-\frac{1}{T} \sum_{t=0}^{T-1} F(\boldsymbol{x}_t, \boldsymbol{y}^*) =\frac{1}{T} \sum_{t=0}^{T-1}\big[F(\boldsymbol{x}_t, \boldsymbol{y}_t)-F(\boldsymbol{x}_t, \boldsymbol{y}^*)\big] \\
& ~~~~\leq \frac{1}{T} \sum_{t=0}^{T-1} \frac{1}{2 \mm \lambda}  \big(\|\boldsymbol{y}_t-\boldsymbol{y}^*\|_2^2-\|\boldsymbol{y}_{t+1}-\boldsymbol{y}^*\|_2^2 \big)+\frac{1}{T} \sum_{t=0}^{T-1} \frac{\lambda}{2}  \big \|(\nabla_{\boldsymbol{y}} F)(\boldsymbol{x}_t, \boldsymbol{y}_t) \big \|_2^2 \\
& ~~~~ =\frac{1}{2}  \sum_{t=0}^{T-1}\big(\|\boldsymbol{y}_t-\boldsymbol{y}^*\|_2^2-\|\boldsymbol{y}_{t+1}-\boldsymbol{y}^*\|_2^2 \big)+\frac{1}{2 \mm T^2} \sum_{t=0}^{T-1} \big \|(\nabla_{\boldsymbol{y}} F)(\boldsymbol{x}_t, \boldsymbol{y}_t)\big\|_2^2 \\
& ~~~~ \leq \frac{\|\boldsymbol{y}_0-\boldsymbol{y}^*\|^2}{2}+\frac{1}{2 \mm T^2} \sum_{t=0}^{T-1} \big \|(\nabla_{\boldsymbol{y}} F)(\boldsymbol{x}_t, \boldsymbol{y}_t)\big\|_2^2 \, .
\end{align*}
In the second step of the proof we show the assertion. Using the result of the first step we get

\begin{align*}
&\min _{t=0, \ldots, T} F(\boldsymbol{x}_t, \boldsymbol{y}_t)  \leq \frac{1}{T} \sum_{t=0}^{T-1} F(\boldsymbol{x}_t, \boldsymbol{y}_t)\\
&~~~~~~\leq  \frac{1}{T} \sum_{t=0}^{T-1} F(\boldsymbol{x}_t, \boldsymbol{y}^*)+\frac{\|\boldsymbol{y}^*-\boldsymbol{y}_0\|^2}{2}+\frac{1}{2 \mm T^2} \sum_{t=0}^{T-1} \big \|(\nabla_{\boldsymbol{y}} F)(\boldsymbol{x}_t, \boldsymbol{y}_t)\big \|_2^2 \\
&~~~~~~\leq  F(\boldsymbol{x}_0, \boldsymbol{y}^*)+\frac{1}{T} \sum_{t=0}^{T-1}|F(\boldsymbol{x}_t, \boldsymbol{y}^*)-F(\boldsymbol{x}_0, \boldsymbol{y}^*)|+\frac{\|\boldsymbol{y}^*-\boldsymbol{y}_0\|_2^2}{2} \\
&~~~~~~~~~  +\frac{1}{2 \mm T^2} \sum_{t=0}^{T-1}\big\|(\nabla_{\boldsymbol{y}} F)(\boldsymbol{x}_t, \boldsymbol{y}_t) \big\|_2^2 \, .
\end{align*}
By (\ref{opteq4}) we get
\begin{align*}
\frac{1}{T} \sum_{t=0}^{T-1}|F(\boldsymbol{x}_t, \boldsymbol{y}^*)-F(\boldsymbol{x}_0, \boldsymbol{y}^*)| & \leq \frac{1}{T} \sum_{t=0}^{T-1} U \cdot \|\boldsymbol{y}^*\|_2 \cdot\|\boldsymbol{x}_t-\boldsymbol{x}_0\|_2 \\
& \leq U \cdot\|\boldsymbol{y}^*\|_2 \cdot \operatorname{diam}(\mathrm{X}) \, .
\end{align*}
And by (\ref{opteq1}) we get
$$
\frac{1}{2 \mm T^2} \sum_{t=0}^{T-1} \big \|(\nabla_{\boldsymbol{y}} F)(\boldsymbol{x}_t, \boldsymbol{y}_t) \big\|_2^2 \leq \frac{1}{2 \mm T^2} \sum_{t=0}^{T-1} V^2=\frac{V^2}{2 \mm T} \, .
$$
Then,
$$
\min _{t=0, \ldots, T} F(\boldsymbol{x}_t, \boldsymbol{y}_t) \leq  F(\boldsymbol{x}_0, \boldsymbol{y}^*)+U \cdot\|\boldsymbol{y}^*\|_2 \cdot \operatorname{diam}({\rm X})+\frac{\|\boldsymbol{y}^*-\boldsymbol{y}_0\|_2^2}{2}+\frac{V^2}{2 \mm T} \, .
$$
Denote by $\boldsymbol{z} = (\boldsymbol{x}, \boldsymbol{y})$. Since ${\rm X} \times {\rm Y}$ is a convex set, we have
\begin{align*}
F(\boldsymbol{z}_{t+1}) &= F(\boldsymbol{z}_{t}) + \int_{0}^{1} \frac{\partial F\big(\boldsymbol{z}_{t} + w(\boldsymbol{z}_{t+1}-\boldsymbol{z}_{t})\big)}{\partial w} \, {\rm d}w \\
&=F(\boldsymbol{z}_{t}) + \int_{0}^{1} \nabla F\big(\boldsymbol{z}_{t} + w(\boldsymbol{z}_{t+1}-\boldsymbol{z}_{t})\big)^{\rm T}(\boldsymbol{z}_{t+1}-\boldsymbol{z}_{t}) \, {\rm d}w \, .
\end{align*}
Since $\nabla F(\boldsymbol{x}, \boldsymbol{y})$ satisfies (\ref{opteq2}), it holds that
\begin{align*}
    &\int_{0}^{1} \big[\nabla F\big(\boldsymbol{z}_{t} + w(\boldsymbol{z}_{t+1}-\boldsymbol{z}_{t})\big) - \nabla F(\boldsymbol{z}_{t}) \big]^{\rm T}(\boldsymbol{z}_{t+1}-\boldsymbol{z}_{t}) \, {\rm d}w \\
    &~~~~~~\leq \int_{0}^{1} K\|w(\boldsymbol{z}_{t+1}- \boldsymbol{z}_{t})\|_2 \,\|\boldsymbol{z}_{t+1}- \boldsymbol{z}_{t}\|_2\, {\rm d}w = \frac{K}{2} \|\boldsymbol{z}_{t+1}- \boldsymbol{z}_{t}\|^2 \, .
\end{align*}
Recall the algorithm update in  (\ref{opteq3}) 
\begin{equation*}
{\boldsymbol{z}}_{t+1}=\operatorname{Proj}_{\mkern 3mu \mathrm{X} \times \mathrm{Y}} \! \big(\boldsymbol{z}_t -\lambda \cdot \nabla F(\boldsymbol{z}_t)\big) \, ,
\end{equation*}
we know that ${\boldsymbol{z}}_{t+1}$ is the projection of ${\boldsymbol{z}}_{t}-\lambda\cdot \nabla F(\boldsymbol{z}_t)$ onto ${\rm X} \times {\rm Y}$. Then
by the optimality  condition of  projection, we have
$$
\big(({\boldsymbol{z}}_{t}-\lambda\cdot \nabla F(\boldsymbol{z}_t)) - {\boldsymbol{z}}_{t+1}\big)\big(\boldsymbol{u} - {\boldsymbol{z}}_{t+1}\big) \leq 0\, \qquad \forall \, \boldsymbol{u}\in {\rm X} \times {\rm Y}\, .
$$
Let $\boldsymbol{u} = {\boldsymbol{z}}_{t}$, then $\nabla F(\boldsymbol{z}_{t})^{\rm T}(\boldsymbol{z}_{t+1} - \boldsymbol{z}_{t}) \leq - \frac{1}{\lambda}\|\boldsymbol{z}_{t+1}- \boldsymbol{z}_{t}\|^2$. 
Thus, we get
\begin{align*}
F(z_{t+1})&\leq F(\boldsymbol{z}_{t}) + \nabla F(\boldsymbol{z}_{t})^{\rm T}(\boldsymbol{z}_{t+1} - \boldsymbol{z}_{t}) + \frac{K}{2} \|\boldsymbol{z}_{t+1}- \boldsymbol{z}_{t}\|^2\\
&\leq F(\boldsymbol{z}_t) - \frac{1}{\lambda}\|\boldsymbol{z}_{t+1}- \boldsymbol{z}_{t}\|^2 + \frac{K}{2} \|\boldsymbol{z}_{t+1}- \boldsymbol{z}_{t}\|^2\\
&\leq F(\boldsymbol{z}_{t})- \bigg(\frac{1}{\lambda}-\frac{K}{2}\bigg) \|\boldsymbol{z}_{t+1}- \boldsymbol{z}_{t}\|^2 \, .
\end{align*}
Therefore, when $\lambda= T^{-1} \wedge 2K^{-1}$, 
\begin{align*}
F(\boldsymbol{x}_{\scriptscriptstyle T}, \boldsymbol{y}_{\scriptscriptstyle T}) &= \min _{t=0, \ldots, T} F(\boldsymbol{x}_t, \boldsymbol{y}_t)\\
&\leq F(\boldsymbol{x}_0, \boldsymbol{y}^*)+U \cdot\|\boldsymbol{y}^*\| \cdot \operatorname{diam}({\rm X})+\frac{\|\boldsymbol{y}^*-\boldsymbol{y}_0\|^2}{2}+\frac{V^2}{2 \mm T} \, .
\end{align*}

\subsubsection{Proof of Lemma \ref{bound_lem}}

By (\ref{eq12}) and  $u_{\mathfrak{m}, \boldsymbol{\theta}}(\boldsymbol{x})=\sum_{k=1}^{\mathfrak{m}} c_{k} \phi_{\boldsymbol{\theta}}^{k}(\boldsymbol{x})$, it holds that
\begin{align*}
& \big \|\nabla_{\boldsymbol{\theta}^{\mathfrak{m}}_{\rm out}} \widehat{F}\big(\boldsymbol{\theta}^{\mathfrak{m}}_{ \rm in}, \boldsymbol{\theta}^{\mathfrak{m}}_{ \rm out}\big)\big\|_2^{2} = \bigg \|\nabla_{(c_k)_{k=1}^{\mathfrak{m}}} \widehat{\mathcal{L}}\bigg(\sum_{k=1}^{\mathfrak{m}} c_{k} \phi_{\boldsymbol{\theta}}^{k}\bigg) \bigg\|_2^{2}\\
&~~= \sum_{j=1}^{\mathfrak{m}} \bigg\{\frac{|\Omega|}{N_{\rm in}} \sum_{p=1}^{N_{\rm in}} \bigg[ \sum_{k=1}^{\mathfrak{m}} c_{k} \m \nabla_{\! \boldsymbol{x}} \m \phi_{\boldsymbol{\theta}}^{k}(X_p) \nabla_{\! \boldsymbol{x}} {\phi_{\boldsymbol{\theta}}^{j}(X_p)}^{\rm T} + \omega(X_p) \sum_{k=1}^{\mathfrak{m}} c_{k} \mm \phi_{\boldsymbol{\theta}}^{k}(X_p) \mm \phi_{\boldsymbol{\theta}}^{j}(X_p) \\
&~~~~~~~~~~~~ -\phi_{\boldsymbol{\theta}}^{j}(X_p)h(X_p)\bigg]  - \frac{|\partial \Omega|}{N_b} \sum_{p=1}^{N_b}\big[\phi_{\boldsymbol{\theta}}^{j}(Y_p)g(Y_p)\big]\bigg\}^2\\
&~~ \leq \sum_{j=1}^{\mathfrak{m}} 4\bigg\{\bigg\{\frac{|\Omega|}{N_{\rm in}} \sum_{p=1}^{N_{\rm in}} \bigg[ \sum_{k=1}^{\mathfrak{m}}c_{k}   \nabla \phi_{\boldsymbol{\theta}}^{k}(X_p) {\nabla \phi_{\boldsymbol{\theta}}^{j}(X_p)}^{\rm T}\bigg]\bigg\}^2 \\
&~~~~~~~~~~~+ \bigg\{\frac{|\Omega|}{N_{\rm in}} \sum_{p=1}^{N_{\rm in}} \bigg[ \omega(X_p) \sum_{k=1}^{\mathfrak{m}} c_{k}   \phi_{\boldsymbol{\theta}}^{k}(X_p) \phi_{\boldsymbol{\theta}}^{j}(X_p)\bigg] \bigg\}^2 \\
&~~~~~~~~~~~+ \bigg\{\frac{|\Omega|}{N_{\rm in}} \sum_{p=1}^{N_{\rm in}} \big[ \phi_{\boldsymbol{\theta}}^{j}(X_p)h(X_p)\big]\bigg\}^2 + \bigg\{\frac{|\partial \Omega|}{N_b} \sum_{p=1}^{N_b}[\phi_{\boldsymbol{\theta}}^{j}(Y_p)g(Y_p)]\bigg\}^2\bigg\}\\
& \leq \sum_{j=1}^{\mathfrak{m}} 4\bigg\{\bigg\{|\Omega|^2 \frac{1}{N_{\rm in}} \sum_{p=1}^{N_{\rm in}} \Big\|\sum_{k=1}^{\mathfrak{m}} c_{k}  \nabla \phi_{\boldsymbol{\theta}}^{k}(X_p)\Big\|_2^2 \mm \mm \frac{1}{N_{\rm in}} \sum_{p=1}^{N_{\rm in}} \big\|\nabla \phi_{\boldsymbol{\theta}}^{j}(X_p)\big\|_2^2 \bigg\}\\
&~~~~~~~~~~~ + \bigg\{|\Omega|^2 \frac{1}{N_{\rm in}} \sum_{p=1}^{N_{\rm in}} \Big| \omega(X_p) \sum_{k=1}^{\mathfrak{m}} c_{k} \phi_{\boldsymbol{\theta}}^{k}(X_p)\Big|^2 \cdot \frac{1}{N_{\rm in}} \sum_{p=1}^{N_{\rm in}} \big|\phi_{\boldsymbol{\theta}}^{j}(X_p)\big|^2 \bigg\}\\
&~~~~~~~~~~~ + \bigg\{|\Omega|^2 \frac{1}{N_{\rm in}} \sum_{p=1}^{N_{\rm in}} \big| \phi_{\boldsymbol{\theta}}^{j}(X_p)\big|^2 \cdot \frac{1}{N_{\rm in}} \sum_{p=1}^{N_{\rm in}} \big|h(X_p)\big|^2\bigg\} \\
&~~~~~~~~~~~ + \bigg\{|\partial \Omega|^2 \frac{1}{N_b} \sum_{p=1}^{N_b}\big|\phi_{\boldsymbol{\theta}}^{j}(Y_p)\big|^2\cdot \frac{1}{N_b} \sum_{p=1}^{N_b} \big|g(Y_p)\big|^2\bigg\}\bigg\}\\
&\leq \sum_{j=1}^{\mathfrak{m}} 4\bigg\{\bigg\{|\Omega|^2 \frac{1}{N_{\rm in}} \sum_{p=1}^{N_{\rm in}} \Big|\sum_{k=1}^{\mathfrak{m}} c_{k}\Big|^2 \cdot \max_{k} \big\|\nabla \phi_{\boldsymbol{\theta}}^{k}(X_p)\big\|_2^2 \cdot \frac{1}{N_{\rm in}} \sum_{p=1}^{N_{\rm in}} \big\|\nabla \phi_{\boldsymbol{\theta}}^{j}(X_p)\big\|_2^2 \bigg\}\\
&~~~~~~~~~~~+ \bigg\{|\Omega|^2 \frac{1}{N_{\rm in}} \sum_{p=1}^{N_{\rm in}} \big| \omega(X_p) \big| \cdot \Big|\sum_{k=1}^{\mathfrak{m}} c_{k}\Big|^2 \cdot \max_{k}\big|\phi_{\boldsymbol{\theta}}^{k}(X_p)\big|^2 \cdot \frac{1}{N_{\rm in}} \sum_{p=1}^{N_{\rm in}} \big|\phi_{\boldsymbol{\theta}}^{j}(X_p)\big|^2 \bigg\}\\
&~~~~~~~~~~~ + \bigg\{|\Omega|^2 \frac{1}{N_{\rm in}} \sum_{p=1}^{N_{\rm in}} \big| \phi_{\boldsymbol{\theta}}^{j}(X_p)\big|^2 \cdot \frac{1}{N_{\rm in}} \sum_{p=1}^{N_{\rm in}} \big|h(X_p)\big|^2\bigg\} \\
&~~~~~~~~~~~+ \bigg\{|\partial \Omega|^2 \frac{1}{N_b} \sum_{p=1}^{N_b}\big|\phi_{\boldsymbol{\theta}}^{j}(Y_p)\big|^2\cdot \frac{1}{N_b} \sum_{p=1}^{N_b} \big|g(Y_p)\big|^2\bigg\}\bigg\}\, ,
\end{align*}
where the Cauchy-Schwarz inequality is used in the second inequality. By Lemma \ref{Lip of Fi}, it holds that
$$
|\phi_{\boldsymbol{\theta}}(\boldsymbol{x})|\leq (W+1)B_{\boldsymbol{\theta}}\, , \qquad |\partial_{x_m}\phi_{\boldsymbol{\theta}}(\boldsymbol{x})|\leq W^{L-1}B_{\boldsymbol{\theta}}^L \, .
$$
Then, we have
\begin{align*}
\max_{k}\|\nabla \phi_{\boldsymbol{\theta}}^{k}(X_p)\|_2^2 = \max_{k} \sum_{m=1}^{d} |\partial_{x_m} \phi_{\boldsymbol{\theta}}^{k}|^2\leq d\big(W^{L-1}B_{\boldsymbol{\theta}}^{L}\big)^2\, .
\end{align*}
Therefore, we get
\begin{align*}
&\big\|\nabla_{\boldsymbol{\theta}^{\mathfrak{m}}_{\rm out}} \widehat{F}\big(\boldsymbol{\theta}^{\mathfrak{m}}_{ \rm in}, \boldsymbol{\theta}^{\mathfrak{m}}_{ \rm out}\big)\big\|_2^{2}\\
&~~~~~~~~~~\leq 4 \mathfrak{m} \cdot \max \{|\Omega|,|\partial\Omega|\}\cdot 
\big(d^2 M^2(W^{L-1}B_{\boldsymbol{\theta}}^L)^4 + B_0^2 M^2(W+1)^4B_{\boldsymbol{\theta}}^4+ 2B_0^2(W+1)^2B_{\boldsymbol{\theta}}^2\big) \, .
\end{align*}

\subsubsection{Proof of Lemma \ref{hessian_lem}}
Write $\boldsymbol{x}=\sum_{j=1}^n a_j \boldsymbol{e}_j$, for coefficients $c_1, \ldots, c_n$. Suppose that $\|\boldsymbol{x}\|_2=1$, i.e. $\sum_j|a_j|^2=1$. Then
\begin{align*}
\|\boldsymbol{A}\boldsymbol{x}\|_2^2 & =\|\sum_{j=1}^{n} a_j \boldsymbol{A} \boldsymbol{e}_j\|_2^2 \leq\big(\sum_j|a_j|\|\boldsymbol{A} \boldsymbol{e}_j\|_2\big)^2 \\
& \leq\big(\sum_{j=1}^{n} |a_j|^2\big) \sum_{j=1}^{n}\|A \boldsymbol{e}_j\|_2^2=\sum_{j=1}^{n}\|\boldsymbol{A} \boldsymbol{e}_j\|_2^2=\|\boldsymbol{A}\|_{\rm F}^2\, ,
\end{align*}
where the triangle inequality is used in the first inequality and Cauchy-Schwarz in the second.
As $\boldsymbol{x}$ was arbitrary, we get $\|\boldsymbol{A}\|_2 \leq\|\boldsymbol{A}\|_{\rm F}$. Thus, if we have
$$ 
\|\nabla^2 f(\boldsymbol{x})\|_{2, 2} \leq \|\nabla^2 f(\boldsymbol{x})\|_{\rm F} \le K \, ,
$$
it further holds that
\begin{align*}
\|\nabla f(\boldsymbol{x})-\nabla f(\boldsymbol{y})\|_2 \leq \|\nabla^{2}f(\boldsymbol{\xi})\|_{2,2}\|\boldsymbol{x}-\boldsymbol{y}\|_2\leq K\|\boldsymbol{x}-\boldsymbol{y}\|_2 \, ,
\end{align*}
where $\xi = t\boldsymbol{x} + (1-t)\boldsymbol{y}$, $t \in (0, 1)$.



\subsubsection{Proof of Lemma \ref{hessian}}

The goal is to estimate the Frobenius norm of the Hessian matrix of $\widehat{F}(\boldsymbol{\theta}^{\mathfrak{m}}_{\rm total})$, which is equivalent to estimate the second derivative of $\widehat{\mathcal{L}}(u_{\mathfrak{m}, \boldsymbol{\theta}})$ w.r.t. the weight parameter. For simplicity, we will demonstrate the proof process using the derivative w.r.t. $a^{\scriptscriptstyle (0)}_{\scriptscriptstyle 1,1,1}$, the innermost weight of the first sub-network in $u_{\mathfrak{m}, \boldsymbol{\theta}}$. Now, we have
\begin{align*}
&\Big|\partial_{a_{\scalebox{0.52}{$1,1,1$}}^{\scalebox{0.5}{$(0)$}}}\partial_{a_{\scalebox{0.52}{$1,1,1$}}^{\scalebox{0.5}{$(0)$}}} \widehat{\mathcal{L}}(u_{\mathfrak{m}, \boldsymbol{\theta}}) \Big|\\
&~~~= \bigg|\partial_{a_{\scalebox{0.52}{$1,1,1$}}^{\scalebox{0.5}{$(0)$}}}\partial_{a_{\scalebox{0.52}{$1,1,1$}}^{\scalebox{0.5}{$(0)$}}}\bigg[\frac{1}{2} \sum_{i=1}^{d} |\partial_{x_i} u_{\mathfrak{m}, \boldsymbol{\theta}}(X)|^2 + \frac{1}{2} \omega(X) u_{\mathfrak{m}, \boldsymbol{\theta}}^2(X) - u_{\mathfrak{m}, \boldsymbol{\theta}}(X)h(X) - u_{\mathfrak{m}, \boldsymbol{\theta}}(Y)g(Y)\bigg]\bigg|\\
&~~~\leq \bigg|\sum_{k=1}^{d} \big[{\partial_{a_{\scalebox{0.52}{$1,1,1$}}^{\scalebox{0.5}{$(0)$}}}} {\partial_{x_k}} u_{\mathfrak{m}, \boldsymbol{\theta}}(X) \cdot {\partial_{a_{\scalebox{0.52}{$1,1,1$}}^{\scalebox{0.5}{$(0)$}}}}{\partial_{x_k}} u_{\mathfrak{m}, \boldsymbol{\theta}}(X) + {\partial_{a_{\scalebox{0.52}{$1,1,1$}}^{\scalebox{0.5}{$(0)$}}} \partial_{a_{\scalebox{0.52}{$1,1,1$}}^{\scalebox{0.5}{$(0)$}}}} {\partial_{x_k}} u_{\mathfrak{m}, \boldsymbol{\theta}}(X) \cdot {\partial_{x_k}} u_{\mathfrak{m}, \boldsymbol{\theta}}(X)\big]\bigg|\\
&~~~~~~ + \big|\omega(X) \big[\partial_{a_{\scalebox{0.52}{$1,1,1$}}^{\scalebox{0.5}{$(0)$}}} u_{\mathfrak{m}, \boldsymbol{\theta}}(X) \cdot \partial_{a_{\scalebox{0.52}{$1,1,1$}}^{\scalebox{0.5}{$(0)$}}} u_{\mathfrak{m}, \boldsymbol{\theta}}(X) + u_{\mathfrak{m}, \boldsymbol{\theta}}(X) \partial_{a_{\scalebox{0.52}{$1,1,1$}}^{\scalebox{0.5}{$(0)$}}} \partial_{a_{\scalebox{0.52}{$1,1,1$}}^{\scalebox{0.5}{$(0)$}}} u_{\mathfrak{m}, \boldsymbol{\theta}}(X) \big]\big|\\
&~~~~~~+ \big|2B_0 \mm \partial_{ a_{\scalebox{0.52}{$1,1,1$}}^{\scalebox{0.5}{$(0)$}}} \partial_{ a_{\scalebox{0.52}{$1,1,1$}}^{\scalebox{0.5}{$(0)$}}} u_{\mathfrak{m}, \boldsymbol{\theta}}(X)\big| \, ,
\end{align*}
where we have utilized $\|\nabla_{\boldsymbol{x}} u_{\mathfrak{m}, \boldsymbol{\theta}}(X)\|^2 = \sum_{i=1}^{d} |{\partial_{x_i}} u_{\mathfrak{m}, \boldsymbol{\theta}}(X)|^2$. Therefore, the task reduces to estimating the following partial derivatives:
\begin{itemize}
    \item \textbf{First Order Derivatives: } $\partial_{x_1} u_{\mathfrak{m}, \boldsymbol{\theta}}(\boldsymbol{x}) \, , \ \partial_{ a_{\scalebox{0.52}{$1,1,1$}}^{\scalebox{0.5}{$(0)$}}} u_{\mathfrak{m}, \boldsymbol{\theta}}(\boldsymbol{x}) \, .$

    \item \textbf{Second Order Derivatives: } $\partial_{ a_{\scalebox{0.52}{$1,1,1$}}^{\scalebox{0.5}{$(0)$}}} \partial_{x_1} u_{\mathfrak{m}, \boldsymbol{\theta}}(\boldsymbol{x}) \, , \ \partial_{ a_{\scalebox{0.52}{$1,1,1$}}^{\scalebox{0.5}{$(0)$}}} \partial_{ a_{\scalebox{0.52}{$1,1,1$}}^{\scalebox{0.5}{$(0)$}}} u_{\mathfrak{m}, \boldsymbol{\theta}}(\boldsymbol{x}) \, .$

     \item \textbf{Third Order Derivative: } $\partial_{ a_{\scalebox{0.52}{$1,1,1$}}^{\scalebox{0.5}{$(0)$}}} \partial_{ a_{\scalebox{0.52}{$1,1,1$}}^{\scalebox{0.5}{$(0)$}}} \partial_{x_1} u_{\mathfrak{m}, \boldsymbol{\theta}}(\boldsymbol{x}) \, .$
\end{itemize}
Note that we will use \(\phi_{\scriptscriptstyle k,j}^{\scriptscriptstyle (\ell)}(\boldsymbol{x})\) to denote the \(j\)-th output of the \(k\)-th sub-network at layer \(\ell\) in the rest of this subsection. Also, we will assume all neural network weights are non-negative when estimating the upper bounds of above partial derivatives (if not, the usual triangle inequality would lead to the same conclusion).

We first focus on the first order derivatives. It holds that
$${\partial_{x_1}} u_{\mathfrak{m}, \boldsymbol{\theta}}(\boldsymbol{x}) = \sum_{k=1}^{\mathfrak{m} } c_{k} \cdot {\partial_{x_1} \phi_{\boldsymbol{\theta}}^{k}(\boldsymbol{x})} \, .$$
Also, we have
\begin{align*}
&{\partial_{x_1} \phi_{\boldsymbol{\theta}}^{k}(\boldsymbol{x})}  ={\partial_{x_1} \phi_{ k, 1}^{\scalebox{0.65}{$(L)$}}(\boldsymbol{x})}  = \sum_{s_{L-1} = 1}^{\scriptscriptstyle N_{L-1}} \!\! a_{k, \m 1, \m s_{\raisebox{-0.1ex}{$\scalebox{0.6}{$L\mmn-\mmn1$}$}}}^{\scalebox{0.65}{$(L\mn -\mn 1\m)$}} {\partial_{x_1} \phi_{k, \m s_{\raisebox{-0.1ex}{$\scalebox{0.6}{$L\mmn-\mmn1$}$}}}^{\scalebox{0.65}{$(L\mn -\mn 1\m)$}}(\boldsymbol{x})} \\
& = \sum_{s_{L-1} = 1}^{\scriptscriptstyle N_{L-1}} \!\! a_{k, \m 1, \m s_{\raisebox{-0.1ex}{$\scalebox{0.6}{$L\mmn-\mmn1$}$}}}^{\scalebox{0.65}{$(L\mn -\mn 1\m)$}} \rho^{\prime}\bigg(\sum_{s_{L-2} = 1}^{\scriptscriptstyle N_{L-2}} \!\! a_{k, \m 1, \m s_{\raisebox{-0.1ex}{$\scalebox{0.6}{$L\mmn-\mmn2$}$}}}^{\scalebox{0.65}{$(L\mn -\mn 2\m)$}}\cdot \phi_{k, \m s_{\raisebox{-0.1ex}{$\scalebox{0.6}{$L\mmn-\mmn2$}$}}}^{\scalebox{0.65}{$(L\mn -\mn 2\m)$}}(\boldsymbol{x})+ b_{k, \m s_{\raisebox{-0.1ex}{$\scalebox{0.6}{$L\mmn-\mmn1$}$}}}^{\scalebox{0.65}{$(L\mn -\mn 2\m)$}}\bigg) \cdot \bigg[\sum_{s_{L-2} = 1}^{\scriptscriptstyle N_{L-2}} \!\! a_{k, \m s_{\raisebox{-0.1ex}{$\scalebox{0.6}{$L\mmn-\mmn1$}$}}, \m s_{\raisebox{-0.1ex}{$\scalebox{0.6}{$L\mmn-\mmn2$}$}}}^{\scalebox{0.65}{$(L\mn -\mn 2\m)$}} \cdot {\partial_{x_1} \phi_{k, \m s_{\raisebox{-0.1ex}{$\scalebox{0.6}{$L\mmn-\mmn2$}$}}}^{\scalebox{0.65}{$(L\mn -\mn 2\m)$}}(\boldsymbol{x})}\bigg]\\
&\leq \sum_{s_{L-1} = 1}^{\scriptscriptstyle N_{L-1}} \!\! a_{k, \m 1, \m s_{\raisebox{-0.1ex}{$\scalebox{0.6}{$L\mmn-\mmn1$}$}}}^{\scalebox{0.65}{$(L\mn -\mn 1\m)$}} \cdot 1 \cdot \bigg[\sum_{s_{L-2} = 1}^{\scriptscriptstyle N_{L-2}} \!\! a_{k, \m s_{\raisebox{-0.1ex}{$\scalebox{0.6}{$L\mmn-\mmn1$}$}}, \m s_{\raisebox{-0.1ex}{$\scalebox{0.6}{$L\mmn-\mmn2$}$}}}^{\scalebox{0.65}{$(L\mn -\mn 2\m)$}} \cdot 1 \cdot \bigg[\sum_{s_{L-3} = 1}^{\scriptscriptstyle N_{L-3}} \!\! a_{k, \m s_{\raisebox{-0.1ex}{$\scalebox{0.6}{$L\mmn-\mmn2$}$}}, \m s_{\raisebox{-0.1ex}{$\scalebox{0.6}{$L\mmn-\mmn3$}$}}}^{\scalebox{0.65}{$(L\mn -\mn 3\m)$}} \cdots \sum_{s_{1} = 1}^{\scriptscriptstyle N_{1}}a_{k, s_{2}, s_{1}}^{\scalebox{0.6}{$(1\m)$}}\cdot a_{k, s_{2}, 1}^{\scalebox{0.6}{$(0\m)$}}\bigg]\bigg]\\
&\leq \sum_{s_{L-1} = 1}^{\scriptscriptstyle N_{L-1}}   \sum_{s_{L-2} = 1}^{\scriptscriptstyle N_{L-2}} \!\!  \cdots \sum_{s_{1} = 1}^{\scriptscriptstyle N_{1}} a_{k, \m 1, \m s_{\raisebox{-0.1ex}{$\scalebox{0.6}{$L\mmn-\mmn1$}$}}}^{\scalebox{0.65}{$(L\mn -\mn 1\m)$}} a_{k, \m s_{\raisebox{-0.1ex}{$\scalebox{0.6}{$L\mmn-\mmn1$}$}}, \m s_{\raisebox{-0.1ex}{$\scalebox{0.6}{$L\mmn-\mmn2$}$}}}^{\scalebox{0.65}{$(L\mn -\mn 2\m)$}} \cdots a_{k, s_{2}, s_{1}}^{\scalebox{0.6}{$(1\m)$}}\cdot a_{k, s_{2}, 1}^{\scalebox{0.6}{$(0\m)$}} \leq \bigg(\prod_{i = 1}^{L-1} N_i\bigg) (B_{\boldsymbol{\theta}})^{L}\, .
\end{align*}
Therefore, it holds that $$\partial_{x_1} u_{\mathfrak{m}, \boldsymbol{\theta}}(\boldsymbol{x}) \leq M \bigg(\prod_{i = 1}^{L-1} N_i\bigg) (B_{\boldsymbol{\theta}})^{L}\, .$$
As for the partial derivative w.r.t. $a^{\scriptscriptstyle (0)}_{\scriptscriptstyle 1,1,1}$, we have
\begin{align} \label{eq:par_a111}
\partial_{a_{\scalebox{0.52}{$1,1,1$}}^{\scalebox{0.5}{$(0)$}}}  u_{\mathfrak{m}, \boldsymbol{\theta}}(\boldsymbol{x}) = c_1 \cdot {\partial_{a_{\scalebox{0.52}{$1,1,1$}}^{\scalebox{0.5}{$(0)$}}} \phi_{\boldsymbol{\theta}}^{1}(\boldsymbol{x})} \, .
\end{align}
Also, it holds that
\begin{align} \label{eq:c_par_a111}
&~~~~{\partial_{a_{\scalebox{0.52}{$1,1,1$}}^{\scalebox{0.5}{$(0)$}}} \phi_{\boldsymbol{\theta}}^{1}(\boldsymbol{x})}   ={\partial_{a_{\scalebox{0.52}{$1,1,1$}}^{\scalebox{0.5}{$(0)$}}} \phi_{\scriptscriptstyle k, 1}^{\scalebox{0.65}{$(L)$}}(\boldsymbol{x})} = \sum_{s_{L-1} = 1}^{\scriptscriptstyle N_{L-1}} \!\! a_{1, \m 1, \m s_{\raisebox{-0.1ex}{$\scalebox{0.6}{$L\mmn-\mmn1$}$}}}^{\scalebox{0.65}{$(L\mn -\mn 1\m)$}} {\partial_{a_{\scalebox{0.52}{$1,1,1$}}^{\scalebox{0.5}{$(0)$}}} \phi_{1, \m s_{\raisebox{-0.1ex}{$\scalebox{0.6}{$L\mmn-\mmn1$}$}}}^{\scalebox{0.65}{$(L\mn -\mn 1\m)$}}(\boldsymbol{x})} \notag \\
&~~~~ = \sum_{s_{L-1} = 1}^{\scriptscriptstyle N_{L-1}} \!\! a_{1, \m 1, \m s_{\raisebox{-0.1ex}{$\scalebox{0.6}{$L\mmn-\mmn1$}$}}}^{\scalebox{0.65}{$(L\mn -\mn 1\m)$}} \rho^{\prime}\bigg(\sum_{s_{L-2} = 1}^{\scriptscriptstyle N_{L-2}} \!\! a_{1, \m s_{\raisebox{-0.1ex}{$\scalebox{0.6}{$L\mmn-\mmn1$}$}}, \m s_{\raisebox{-0.1ex}{$\scalebox{0.6}{$L\mmn-\mmn2$}$}}}^{\scalebox{0.65}{$(L\mn -\mn 2\m)$}}\cdot \phi_{1, \m s_{\raisebox{-0.1ex}{$\scalebox{0.6}{$L\mmn-\mmn2$}$}}}^{\scalebox{0.65}{$(L\mn -\mn 2\m)$}}(\boldsymbol{x})+ b_{1, \m s_{\raisebox{-0.1ex}{$\scalebox{0.6}{$L\mmn-\mmn1$}$}}}^{\scalebox{0.65}{$(L\mn -\mn 2\m)$}}\bigg)  \cdot \notag \\
&~~~~~~~~~~~~~~~~~~~~\bigg[\sum_{s_{L-2} = 1}^{\scriptscriptstyle N_{L-2}} \!\! a_{1, \m s_{\raisebox{-0.1ex}{$\scalebox{0.6}{$L\mmn-\mmn1$}$}}, \m s_{\raisebox{-0.1ex}{$\scalebox{0.6}{$L\mmn-\mmn2$}$}}}^{\scalebox{0.65}{$(L\mn -\mn 2\m)$}} \mmn \cdot {\partial_{a_{\scalebox{0.52}{$1,1,1$}}^{\scalebox{0.5}{$(0)$}}} \phi_{1, \m s_{\raisebox{-0.1ex}{$\scalebox{0.6}{$L\mmn-\mmn2$}$}}}^{\scalebox{0.65}{$(L\mn -\mn 2\m)$}}(\boldsymbol{x})}\bigg] \notag \\
&~~~~ = \sum_{s_{L-1} = 1}^{\scriptscriptstyle N_{L-1}} \!\! a_{1, \m 1, \m s_{\raisebox{-0.1ex}{$\scalebox{0.6}{$L\mmn-\mmn1$}$}}}^{\scalebox{0.65}{$(L\mn -\mn 1\m)$}} \rho^{\prime}\bigg(\sum_{s_{L-2} = 1}^{\scriptscriptstyle N_{L-2}} \!\! a_{1, \m s_{\raisebox{-0.1ex}{$\scalebox{0.6}{$L\mmn-\mmn1$}$}}, \m s_{\raisebox{-0.1ex}{$\scalebox{0.6}{$L\mmn-\mmn2$}$}}}^{\scalebox{0.65}{$(L\mn -\mn 2\m)$}}\cdot \phi_{1, \m s_{\raisebox{-0.1ex}{$\scalebox{0.6}{$L\mmn-\mmn2$}$}}}^{\scalebox{0.65}{$(L\mn -\mn 2\m)$}}(\boldsymbol{x})+ b_{1, \m s_{\raisebox{-0.1ex}{$\scalebox{0.6}{$L\mmn-\mmn1$}$}}}^{\scalebox{0.65}{$(L\mn -\mn 2\m)$}}\bigg)  \cdot \notag \\
&~~~~~~~~~~~~~~~~\bigg[\sum_{s_{L-2} = 1}^{\scriptscriptstyle N_{L-2}} \!\! a_{1, \m s_{\raisebox{-0.1ex}{$\scalebox{0.6}{$L\mmn-\mmn1$}$}}, \m s_{\raisebox{-0.1ex}{$\scalebox{0.6}{$L\mmn-\mmn2$}$}}}^{\scalebox{0.65}{$(L\mn -\mn 2\m)$}} \cdot  \bigg[\cdots \sum_{s_1 = 1}^{\scriptscriptstyle N_{1}} a_{1, s_2, s_1}^{\scalebox{0.6}{$(1\m)$}} \cdot  x_1 \rho^{\prime}\bigg(\sum_{s_{0} = 1}^{d}a_{1, s_1, s_0}^{\scalebox{0.6}{$(0\m)$}}\cdot x_{s_0} + b_{1, s_1}^{\scalebox{0.6}{$(0\m)$}}\bigg)\bigg]\bigg] \notag \\
&~~~~= \sum_{s_{L-1} = 1}^{\scriptscriptstyle N_{L-1}} \!\! a_{1, \m 1, \m s_{\raisebox{-0.1ex}{$\scalebox{0.6}{$L\mmn-\mmn1$}$}}}^{\scalebox{0.65}{$(L\mn -\mn 1\m)$}} \! \cdot \! \bigg[\sum_{s_{L-2} = 1}^{\scriptscriptstyle N_{L-2}} \!\! a_{1, \m s_{\raisebox{-0.1ex}{$\scalebox{0.6}{$L\mmn-\mmn1$}$}}, \m s_{\raisebox{-0.1ex}{$\scalebox{0.6}{$L\mmn-\mmn2$}$}}}^{\scalebox{0.65}{$(L\mn -\mn 2\m)$}} \! \cdot \! \bigg[\sum_{s_{L-3} = 1}^{\scriptscriptstyle N_{L-3}} \!\! a_{1, \m s_{\raisebox{-0.1ex}{$\scalebox{0.6}{$L\mmn-\mmn2$}$}}, \m s_{\raisebox{-0.1ex}{$\scalebox{0.6}{$L\mmn-\mmn3$}$}}}^{\scalebox{0.65}{$(L\mn -\mn 3\m)$}} \cdots \sum_{s_{1} = 1}^{\scriptscriptstyle N_{1}}a_{1, s_2, s_1}^{\scalebox{0.6}{$(1\m)$}}\cdot x_1\bigg]\bigg] \cdot \notag \\
&~~~~~~~~~~~~~~~~\rho^{\prime}\bigg(\sum_{s_{L-2} = 1}^{\scriptscriptstyle N_{L-2}} \!\! a_{1, \m s_{\raisebox{-0.1ex}{$\scalebox{0.6}{$L\mmn-\mmn1$}$}}, \m s_{\raisebox{-0.1ex}{$\scalebox{0.6}{$L\mmn-\mmn2$}$}}}^{\scalebox{0.65}{$(L\mn -\mn 2\m)$}}\cdot \phi_{1, \m s_{\raisebox{-0.1ex}{$\scalebox{0.6}{$L\mmn-\mmn2$}$}}}^{\scalebox{0.65}{$(L\mn -\mn 2\m)$}}(\boldsymbol{x})+ b_{1, \m s_{\raisebox{-0.1ex}{$\scalebox{0.6}{$L\mmn-\mmn1$}$}}}^{\scalebox{0.65}{$(L\mn -\mn 2\m)$}}\bigg)\cdots \rho^{\prime}\bigg(\sum_{s_{1} = 1}^{\scriptscriptstyle N_{1}}a_{1, s_2, s_1}^{\scalebox{0.6}{$(1\m)$}}\cdot \phi_{1, s_1}^{\scalebox{0.6}{$(1\m)$}}(\boldsymbol{x}) + b_{1, s_{2}}^{\scalebox{0.6}{$(1\m)$}}\bigg) \cdot \notag \\
&~~~~~~~~~~~~~~~~\rho^{\prime}\bigg(\sum_{s_{0} = 1}^{d}a_{1, s_1, s_0}^{\scalebox{0.6}{$(0\m)$}}\cdot x_{s_0} + b_{1, s_1}^{\scalebox{0.6}{$(0\m)$}}\bigg)\leq \bigg(\prod_{i = 1}^{L -1} N_i\bigg) (B_{\boldsymbol{\theta}})^{L-1} \, .
\end{align}
 Therefore, we have $$\partial_{a_{\scalebox{0.52}{$1,1,1$}}^{\scalebox{0.5}{$(0)$}}} u_{\mathfrak{m}, \boldsymbol{\theta}}(\boldsymbol{x}) \leq M \bigg(\prod_{i = 1}^{L -1} N_i\bigg) (B_{\boldsymbol{\theta}})^{L-1} \, .$$

 Next, we estimate the second order partial derivatives. By (\ref{eq:par_a111}), we have
 \begin{align} \label{eq:par_x1_a111}
\partial_{a_{\scalebox{0.52}{$1,1,1$}}^{\scalebox{0.5}{$(0)$}}} \partial_{x_1} u_{\mathfrak{m}, \boldsymbol{\theta}}(\boldsymbol{x}) = c_1 \cdot \partial_{x_1} {\partial_{a_{\scalebox{0.52}{$1,1,1$}}^{\scalebox{0.5}{$(0)$}}} \phi_{\boldsymbol{\theta}}^{1}(\boldsymbol{x})} \, .
\end{align}
Meanwhile, by (\ref{eq:c_par_a111}), it holds that
\begin{align} \label{eq:c_par_x1_a111}
&\partial_{x_1}{\partial_{a_{\scalebox{0.52}{$1,1,1$}}^{\scalebox{0.5}{$(0)$}}} \phi_{\boldsymbol{\theta}}^{1}(\boldsymbol{x})} \notag \\
&~~~ = \bigg[\sum_{s_{L-1} = 1}^{\scriptscriptstyle N_{L-1}} \!\! a_{1, \m 1, \m s_{\raisebox{-0.1ex}{$\scalebox{0.6}{$L\mmn-\mmn1$}$}}}^{\scalebox{0.65}{$(L\mn -\mn 1\m)$}} \cdot \Big[\sum_{s_{L-2} = 1}^{\scriptscriptstyle N_{L-2}} \!\! a_{1, \m s_{\raisebox{-0.1ex}{$\scalebox{0.6}{$L\mmn-\mmn1$}$}}, \m s_{\raisebox{-0.1ex}{$\scalebox{0.6}{$L\mmn-\mmn2$}$}}}^{\scalebox{0.65}{$(L\mn -\mn 2\m)$}} \cdot \bigg[\sum_{s_{L-3} = 1}^{\scriptscriptstyle N_{L-3}} \!\! a_{1, \m s_{\raisebox{-0.1ex}{$\scalebox{0.6}{$L\mmn-\mmn2$}$}}, \m s_{\raisebox{-0.1ex}{$\scalebox{0.6}{$L\mmn-\mmn3$}$}}}^{\scalebox{0.65}{$(L\mn -\mn 3\m)$}} \cdots \sum_{s_{1} = 1}^{\scriptscriptstyle N_{1}}a_{1, s_2, s_1}^{\scalebox{0.6}{$(1\m)$}}\bigg]\bigg]\bigg] \cdot \notag \\
& ~~~~~~~ \rho^{\prime}\bigg(\sum_{s_{L-2} = 1}^{\scriptscriptstyle N_{L-2}} \!\! a_{1, \m s_{\raisebox{-0.1ex}{$\scalebox{0.6}{$L\mmn-\mmn1$}$}}, \m s_{\raisebox{-0.1ex}{$\scalebox{0.6}{$L\mmn-\mmn2$}$}}}^{\scalebox{0.65}{$(L\mn -\mn 2\m)$}}\cdot \phi_{1, \m s_{\raisebox{-0.1ex}{$\scalebox{0.6}{$L\mmn-\mmn2$}$}}}^{\scalebox{0.65}{$(L\mn -\mn 2\m)$}}(\boldsymbol{x})+ b_{1, \m s_{\raisebox{-0.1ex}{$\scalebox{0.6}{$L\mmn-\mmn1$}$}}}^{\scalebox{0.65}{$(L\mn -\mn 2\m)$}}\bigg)\cdots \notag \\
& ~~~~~~~ \rho^{\prime}\bigg(\sum_{s_{1} = 1}^{\scriptscriptstyle N_{1}}a_{1, s_2, s_1}^{\scalebox{0.6}{$(1\m)$}}\cdot \phi_{1, s_1}^{\scalebox{0.6}{$(1\m)$}}(\boldsymbol{x}) + b_{1, s_{2}}^{\scalebox{0.6}{$(1\m)$}}\bigg) \cdot \rho^{\prime}\bigg(\sum_{s_{0} = 1}^{d}a_{1, s_1, s_0}^{\scalebox{0.6}{$(0\m)$}}\cdot x_{s_0} + b_{1, s_1}^{\scalebox{0.6}{$(0\m)$}}\bigg)  \notag \\
& ~~~~~~~+ \bigg[\sum_{s_{L-1} = 1}^{\scriptscriptstyle N_{L-1}} \!\! a_{1, \m 1, \m s_{\raisebox{-0.1ex}{$\scalebox{0.6}{$L\mmn-\mmn1$}$}}}^{\scalebox{0.65}{$(L\mn -\mn 1\m)$}} \cdot \bigg[\sum_{s_{L-2} = 1}^{\scriptscriptstyle N_{L-2}} \!\! a_{1, \m s_{\raisebox{-0.1ex}{$\scalebox{0.6}{$L\mmn-\mmn1$}$}}, \m s_{\raisebox{-0.1ex}{$\scalebox{0.6}{$L\mmn-\mmn2$}$}}}^{\scalebox{0.65}{$(L\mn -\mn 2\m)$}} \cdot \bigg[\sum_{s_{L-3} = 1}^{\scriptscriptstyle N_{L-3}} \!\! a_{1, \m s_{\raisebox{-0.1ex}{$\scalebox{0.6}{$L\mmn-\mmn2$}$}}, \m s_{\raisebox{-0.1ex}{$\scalebox{0.6}{$L\mmn-\mmn3$}$}}}^{\scalebox{0.65}{$(L\mn -\mn 3\m)$}} \cdots \sum_{s_{1} = 1}^{\scriptscriptstyle N_{1}}a_{1, s_2, s_1}^{\scalebox{0.6}{$(1\m)$}}\cdot x_1\bigg]\bigg]\bigg] \cdot \notag \\
& ~~~~~~~ \partial_{x_1}\bigg[ \rho^{\prime}\bigg(\sum_{s_{L-2} = 1}^{\scriptscriptstyle N_{L-2}} \!\! a_{1, \m s_{\raisebox{-0.1ex}{$\scalebox{0.6}{$L\mmn-\mmn1$}$}}, \m s_{\raisebox{-0.1ex}{$\scalebox{0.6}{$L\mmn-\mmn2$}$}}}^{\scalebox{0.65}{$(L\mn -\mn 2\m)$}}\cdot \phi_{1, \m s_{\raisebox{-0.1ex}{$\scalebox{0.6}{$L\mmn-\mmn2$}$}}}^{\scalebox{0.65}{$(L\mn -\mn 2\m)$}}(\boldsymbol{x})+ b_{1, \m s_{\raisebox{-0.1ex}{$\scalebox{0.6}{$L\mmn-\mmn1$}$}}}^{\scalebox{0.65}{$(L\mn -\mn 2\m)$}}\bigg) \cdot  \notag \\
& ~~~~~~~ \cdots \rho^{\prime}\bigg(\sum_{s_{1} = 1}^{\scriptscriptstyle N_{1}}a_{1, s_2, s_1}^{\scalebox{0.6}{$(1\m)$}}\cdot \phi_{1, s_1}^{\scalebox{0.6}{$(1\m)$}}(\boldsymbol{x}) + b_{1, s_{2}}^{\scalebox{0.6}{$(1\m)$}}\bigg) \cdot \rho^{\prime}\bigg(\sum_{s_{0} = 1}^{d}a_{1, s_1, s_0}^{\scalebox{0.6}{$(0\m)$}}\cdot x_{s_0} + b_{1, s_1}^{\scalebox{0.6}{$(0\m)$}}\bigg)\bigg] \notag \\
& ~~~\leq \bigg(\prod_{i = 1}^{L-1} N_i\bigg) (B_{\boldsymbol{\theta}})^{L}  \bigg[\rho^{\prime}\bigg(\sum_{s_{L-2} = 1}^{\scriptscriptstyle N_{L-2}} \!\! a_{1, \m s_{\raisebox{-0.1ex}{$\scalebox{0.6}{$L\mmn-\mmn1$}$}}, \m s_{\raisebox{-0.1ex}{$\scalebox{0.6}{$L\mmn-\mmn2$}$}}}^{\scalebox{0.65}{$(L\mn -\mn 2\m)$}}\cdot \phi_{1, \m s_{\raisebox{-0.1ex}{$\scalebox{0.6}{$L\mmn-\mmn2$}$}}}^{\scalebox{0.65}{$(L\mn -\mn 2\m)$}}(\boldsymbol{x})+ b_{1, \m s_{\raisebox{-0.1ex}{$\scalebox{0.6}{$L\mmn-\mmn1$}$}}}^{\scalebox{0.65}{$(L\mn -\mn 2\m)$}}\bigg)\cdots \notag \\
& ~~~~~~~ \rho^{\prime}\bigg(\sum_{s_{1} = 1}^{\scriptscriptstyle N_{1}}a_{1, s_2, s_1}^{\scalebox{0.6}{$(1\m)$}}\cdot \phi_{1, s_1}^{\scalebox{0.6}{$(1\m)$}}(\boldsymbol{x}) + b_{1, s_{2}}^{\scalebox{0.6}{$(1\m)$}}\bigg) \cdot \rho^{\prime}\bigg(\sum_{s_{0} = 1}^{d}a_{1, s_1, s_0}^{\scalebox{0.6}{$(0\m)$}}\cdot x_{s_0} + b_{1, s_1}^{\scalebox{0.6}{$(0\m)$}}\bigg)  \notag \\
& ~~~~~~~ + x_1  \partial_{x_1}\rho^{\prime}\bigg(\sum_{s_{L-2} = 1}^{\scriptscriptstyle N_{L-2}} \!\! a_{1, \m s_{\raisebox{-0.1ex}{$\scalebox{0.6}{$L\mmn-\mmn1$}$}}, \m s_{\raisebox{-0.1ex}{$\scalebox{0.6}{$L\mmn-\mmn2$}$}}}^{\scalebox{0.65}{$(L\mn -\mn 2\m)$}}\cdot \phi_{1, \m s_{\raisebox{-0.1ex}{$\scalebox{0.6}{$L\mmn-\mmn2$}$}}}^{\scalebox{0.65}{$(L\mn -\mn 2\m)$}}(\boldsymbol{x})+ b_{1, \m s_{\raisebox{-0.1ex}{$\scalebox{0.6}{$L\mmn-\mmn1$}$}}}^{\scalebox{0.65}{$(L\mn -\mn 2\m)$}}\bigg) \notag \\
& ~~~~~~~ \cdots \rho^{\prime}\bigg(\sum_{s_{1} = 1}^{\scriptscriptstyle N_{1}}a_{1, s_2, s_1}^{\scalebox{0.6}{$(1\m)$}}\cdot \phi_{1, s_1}^{\scalebox{0.6}{$(1\m)$}}(\boldsymbol{x}) + b_{1, s_{2}}^{\scalebox{0.6}{$(1\m)$}}\bigg) \cdot \rho^{\prime}\bigg(\sum_{s_{0} = 1}^{d}a_{1, s_1, s_0}^{\scalebox{0.6}{$(0\m)$}}\cdot x_{s_0} + b_{1, s_1}^{\scalebox{0.6}{$(0\m)$}}\bigg)\cdot  \notag \\
& ~~~~~~~ +  \cdots  + x_1  \rho^{\prime}\bigg(\sum_{s_{L-2} = 1}^{\scriptscriptstyle N_{L-2}} \!\! a_{1, \m s_{\raisebox{-0.1ex}{$\scalebox{0.6}{$L\mmn-\mmn1$}$}}, \m s_{\raisebox{-0.1ex}{$\scalebox{0.6}{$L\mmn-\mmn2$}$}}}^{\scalebox{0.65}{$(L\mn -\mn 2\m)$}}\cdot \phi_{1, \m s_{\raisebox{-0.1ex}{$\scalebox{0.6}{$L\mmn-\mmn2$}$}}}^{\scalebox{0.65}{$(L\mn -\mn 2\m)$}}(\boldsymbol{x})+ b_{1, \m s_{\raisebox{-0.1ex}{$\scalebox{0.6}{$L\mmn-\mmn1$}$}}}^{\scalebox{0.65}{$(L\mn -\mn 2\m)$}}\bigg) \cdot \notag \\
& ~~~~~~~ \cdots \rho^{\prime}\bigg(\sum_{s_{1} = 1}^{\scriptscriptstyle N_{1}}a_{1, s_2, s_1}^{\scalebox{0.6}{$(1\m)$}}\cdot \phi_{1, s_1}^{\scalebox{0.6}{$(1\m)$}}(\boldsymbol{x}) + b_{1, s_{2}}^{\scalebox{0.6}{$(1\m)$}}\bigg) \cdot \partial_{x_1}\rho^{\prime}\bigg(\sum_{s_{0} = 1}^{d}a_{1, s_1, s_0}^{\scalebox{0.6}{$(0\m)$}}\cdot x_{s_0} + b_{1, s_1}^{\scalebox{0.6}{$(0\m)$}}\bigg)\bigg] \, .
\end{align}
Since we have 
\begin{align*}
&\partial_{x_1}\rho^{\prime}\bigg(\sum_{s_{L-2} = 1}^{\scriptscriptstyle N_{L-2}} \!\! a_{1, \m s_{\raisebox{-0.1ex}{$\scalebox{0.6}{$L\mmn-\mmn1$}$}}, \m s_{\raisebox{-0.1ex}{$\scalebox{0.6}{$L\mmn-\mmn2$}$}}}^{\scalebox{0.65}{$(L\mn -\mn 2\m)$}}\cdot \phi_{1, \m s_{\raisebox{-0.1ex}{$\scalebox{0.6}{$L\mmn-\mmn2$}$}}}^{\scalebox{0.65}{$(L\mn -\mn 2\m)$}}(\boldsymbol{x})+ b_{1, \m s_{\raisebox{-0.1ex}{$\scalebox{0.6}{$L\mmn-\mmn1$}$}}}^{\scalebox{0.65}{$(L\mn -\mn 2\m)$}}\bigg)  \\
&~~~~~~= \rho^{\prime \prime}\bigg(\sum_{s_{L-2} = 1}^{\scriptscriptstyle N_{L-2}} \!\! a_{1, \m 1, \m s_{\raisebox{-0.1ex}{$\scalebox{0.6}{$L\mmn-\mmn2$}$}}}^{\scalebox{0.65}{$(L\mn -\mn 2\m)$}}\cdot \phi_{1, \m s_{\raisebox{-0.1ex}{$\scalebox{0.6}{$L\mmn-\mmn2$}$}}}^{\scalebox{0.65}{$(L\mn -\mn 2\m)$}}(\boldsymbol{x})+ b_{1, \m s_{\raisebox{-0.1ex}{$\scalebox{0.6}{$L\mmn-\mmn1$}$}}}^{\scalebox{0.65}{$(L\mn -\mn 2\m)$}}\bigg) \sum_{s_{L-2} = 1}^{\scriptscriptstyle N_{L-2}} \!\! a_{1, \m s_{\raisebox{-0.1ex}{$\scalebox{0.6}{$L\mmn-\mmn1$}$}}, \m s_{\raisebox{-0.1ex}{$\scalebox{0.6}{$L\mmn-\mmn2$}$}}}^{\scalebox{0.65}{$(L\mn -\mn 2\m)$}}\cdot \partial_{x_1}\phi_{1, \m s_{\raisebox{-0.1ex}{$\scalebox{0.6}{$L\mmn-\mmn2$}$}}}^{\scalebox{0.65}{$(L\mn -\mn 2\m)$}}(\boldsymbol{x})\\
&~~~~~~ \leq \sum_{s_{L-2} = 1}^{\scriptscriptstyle N_{L-2}} \!\! a_{1, \m s_{\raisebox{-0.1ex}{$\scalebox{0.6}{$L\mmn-\mmn1$}$}}, \m s_{\raisebox{-0.1ex}{$\scalebox{0.6}{$L\mmn-\mmn2$}$}}}^{\scalebox{0.65}{$(L\mn -\mn 2\m)$}}\cdot \bigg(\prod_{i = 1}^{L-3} N_i\bigg) (B_{\boldsymbol{\theta}})^{L - 2} \leq \bigg(\prod_{i = 1}^{L-2} N_i\bigg) (B_{\boldsymbol{\theta}})^{L-1} \, .
\end{align*}
Then, above estimations lead to
\begin{align*}
\partial_{a_{\scalebox{0.52}{$1,1,1$}}^{\scalebox{0.5}{$(0)$}}}{\partial_{x_1}  \phi_{\boldsymbol{\theta}}^{1}(\boldsymbol{x})} \leq \bigg(\prod_{i = 1}^{L-1} N_i\bigg) (B_{\boldsymbol{\theta}})^{L} \cdot L \cdot \bigg(\prod_{i = 1}^{L-1} N_i\bigg) (B_{\boldsymbol{\theta}})^{L} = L \cdot  \bigg(\prod_{i = 1}^{L-1} N_i\bigg)^2 \cdot (B_{\boldsymbol{\theta}})^{2L} \, .
\end{align*}
Therefore, it holds that
\begin{align*}
\partial_{a_{\scalebox{0.52}{$1,1,1$}}^{\scalebox{0.5}{$(0)$}}}\partial_{x_1}  u_{\mathfrak{m}, \boldsymbol{\theta}}(\boldsymbol{x}) \leq  M\cdot L \cdot  \bigg(\prod_{i = 1}^{L-1} N_i\bigg)^2 \cdot (B_{\boldsymbol{\theta}})^{2L}  \, .
\end{align*}
Besides, w.r.t $a_{\scriptscriptstyle 1,1,1}^{\scriptscriptstyle (0)}$, we have
\begin{align*}
&\partial_{a_{\scalebox{0.52}{$1,1,1$}}^{\scalebox{0.5}{$(0)$}}}\rho^{\prime}\bigg(\sum_{s_{L-2} = 1}^{\scriptscriptstyle N_{L-2}} \!\! a_{1, \m s_{\raisebox{-0.1ex}{$\scalebox{0.6}{$L\mmn-\mmn1$}$}}, \m s_{\raisebox{-0.1ex}{$\scalebox{0.6}{$L\mmn-\mmn2$}$}}}^{\scalebox{0.65}{$(L\mn -\mn 2\m)$}}\cdot \phi_{1, \m s_{\raisebox{-0.1ex}{$\scalebox{0.6}{$L\mmn-\mmn2$}$}}}^{\scalebox{0.65}{$(L\mn -\mn 2\m)$}}(\boldsymbol{x})+ b_{1, \m s_{\raisebox{-0.1ex}{$\scalebox{0.6}{$L\mmn-\mmn1$}$}}}^{\scalebox{0.65}{$(L\mn -\mn 2\m)$}}\bigg)  \\
&= \rho^{\prime \prime}\bigg(\sum_{s_{L-2} = 1}^{\scriptscriptstyle N_{L-2}} \!\! a_{1, \m s_{\raisebox{-0.1ex}{$\scalebox{0.6}{$L\mmn-\mmn1$}$}}, \m s_{\raisebox{-0.1ex}{$\scalebox{0.6}{$L\mmn-\mmn2$}$}}}^{\scalebox{0.65}{$(L\mn -\mn 2\m)$}}\cdot \phi_{1, \m s_{\raisebox{-0.1ex}{$\scalebox{0.6}{$L\mmn-\mmn2$}$}}}^{\scalebox{0.65}{$(L\mn -\mn 2\m)$}}(\boldsymbol{x})+ b_{1, \m s_{\raisebox{-0.1ex}{$\scalebox{0.6}{$L\mmn-\mmn1$}$}}}^{\scalebox{0.65}{$(L\mn -\mn 1\m)$}}\bigg) \sum_{s_{L-2} = 1}^{\scriptscriptstyle N_{L-2}} \!\! a_{1, \m s_{\raisebox{-0.1ex}{$\scalebox{0.6}{$L\mmn-\mmn1$}$}}, \m s_{\raisebox{-0.1ex}{$\scalebox{0.6}{$L\mmn-\mmn2$}$}}}^{\scalebox{0.65}{$(L\mn -\mn 2\m)$}}\cdot \partial_{a_{\scalebox{0.52}{$1,1,1$}}^{\scalebox{0.5}{$(0)$}}}(\phi_{1, \m s_{\raisebox{-0.1ex}{$\scalebox{0.6}{$L\mmn-\mmn2$}$}}}^{\scalebox{0.65}{$(L\mn -\mn 2\m)$}}(\boldsymbol{x}))\\
& \leq \sum_{s_{L-2} = 1}^{\scriptscriptstyle N_{L-2}} \!\! a_{1, \m s_{\raisebox{-0.1ex}{$\scalebox{0.6}{$L\mmn-\mmn1$}$}}, \m s_{\raisebox{-0.1ex}{$\scalebox{0.6}{$L\mmn-\mmn2$}$}}}^{\scalebox{0.65}{$(L\mn -\mn 2\m)$}}\cdot \bigg(\prod_{i = 1}^{L-3} N_i\bigg) (B_{\boldsymbol{\theta}})^{L - 3} \leq \bigg(\prod_{i = 1}^{L-2} N_i\bigg) (B_{\boldsymbol{\theta}})^{L-2} \, .
\end{align*}
Then, it holds that
\begin{align*}
\partial_{a_{\scalebox{0.52}{$1,1,1$}}^{\scalebox{0.5}{$(0)$}}}{\partial_{a_{\scalebox{0.52}{$1,1,1$}}^{\scalebox{0.5}{$(0)$}}}  u_{\mathfrak{m}, \boldsymbol{\theta}}(\boldsymbol{x})} \leq M \cdot L \cdot  \bigg(\prod_{i = 1}^{L-1} N_i\bigg)^2 \cdot (B_{\boldsymbol{\theta}})^{2L-1} \, .
\end{align*}

Finally, we turn to the third order derivative. Initially, by (\ref{eq:par_x1_a111}), we have
$$\partial_{a_{\scalebox{0.52}{$1,1,1$}}^{\scalebox{0.5}{$(0)$}}}\partial_{a_{\scalebox{0.52}{$1,1,1$}}^{\scalebox{0.5}{$(0)$}}}\partial_{x_1}  u_{\mathfrak{m}, \boldsymbol{\theta}}(\boldsymbol{x}) = c_1 \cdot \partial_{a_{\scalebox{0.52}{$1,1,1$}}^{\scalebox{0.5}{$(0)$}}}  \partial_{x_1}{\partial_{a_{\scalebox{0.52}{$1,1,1$}}^{\scalebox{0.5}{$(0)$}}} \phi_{\boldsymbol{\theta}}^{1}(\boldsymbol{x})} \, .$$
Then, by fully expanding (\ref{eq:c_par_x1_a111}), it holds that
\begin{align*}
\partial_{x_1}&{\partial_{a_{\scalebox{0.52}{$1,1,1$}}^{\scalebox{0.5}{$(0)$}}} \phi_{\boldsymbol{\theta}}^{1}(\boldsymbol{x})} \\
& = \bigg[\sum_{s_{L-1} = 1}^{\scriptscriptstyle N_{L-1}} \!\! a_{1, \m 1, \m s_{\raisebox{-0.1ex}{$\scalebox{0.6}{$L\mmn-\mmn1$}$}}}^{\scalebox{0.65}{$(L\mn -\mn 1\m)$}} \cdot \bigg[\sum_{s_{L-2} = 1}^{\scriptscriptstyle N_{L-2}} \!\! a_{1, \m s_{\raisebox{-0.1ex}{$\scalebox{0.6}{$L\mmn-\mmn1$}$}}, \m s_{\raisebox{-0.1ex}{$\scalebox{0.6}{$L\mmn-\mmn2$}$}}}^{\scalebox{0.65}{$(L\mn -\mn 2\m)$}} \cdot \bigg[\sum_{s_{L-3} = 1}^{\scriptscriptstyle N_{L-3}} \!\! a_{1, \m s_{\raisebox{-0.1ex}{$\scalebox{0.6}{$L\mmn-\mmn2$}$}}, \m s_{\raisebox{-0.1ex}{$\scalebox{0.6}{$L\mmn-\mmn3$}$}}}^{\scalebox{0.65}{$(L\mn -\mn 3\m)$}} \cdots \sum_{s_{1} = 1}^{\scriptscriptstyle N_{1}}a_{1, s_2, s_1}^{\scalebox{0.6}{$(1\m)$}}\bigg]\bigg]\bigg] \cdot\\
& \qquad \bigg[\rho^{\prime}\bigg(\sum_{s_{L-2} = 1}^{\scriptscriptstyle N_{L-2}} \!\! a_{1, \m s_{\raisebox{-0.1ex}{$\scalebox{0.6}{$L\mmn-\mmn1$}$}}, \m s_{\raisebox{-0.1ex}{$\scalebox{0.6}{$L\mmn-\mmn2$}$}}}^{\scalebox{0.65}{$(L\mn -\mn 2\m)$}}\cdot \phi_{1, \m s_{\raisebox{-0.1ex}{$\scalebox{0.6}{$L\mmn-\mmn2$}$}}}^{\scalebox{0.65}{$(L\mn -\mn 2\m)$}}(\boldsymbol{x})+ b_{1, \m s_{\raisebox{-0.1ex}{$\scalebox{0.6}{$L\mmn-\mmn1$}$}}}^{\scalebox{0.65}{$(L\mn -\mn 2\m)$}}\bigg)\cdots\\
& \qquad \rho^{\prime}\bigg(\sum_{s_{1} = 1}^{\scriptscriptstyle N_{1}}a_{1, s_2, s_1}^{\scalebox{0.6}{$(1\m)$}}\cdot \phi_{1, s_1}^{\scalebox{0.6}{$(1\m)$}}(\boldsymbol{x}) + b_{1, s_{2}}^{\scalebox{0.6}{$(1\m)$}}\bigg) \cdot \rho^{\prime}\bigg(\sum_{s_{0} = 1}^{d}a_{1, s_1, s_0}^{\scalebox{0.6}{$(0\m)$}}\cdot x_{s_0} + b_{1, s_1}^{\scalebox{0.6}{$(0\m)$}}\bigg) \\
&\qquad + x_1  \partial_{x_1}\rho^{\prime}\bigg(\sum_{s_{L-2} = 1}^{\scriptscriptstyle N_{L-2}} \!\! a_{1, \m s_{\raisebox{-0.1ex}{$\scalebox{0.6}{$L\mmn-\mmn1$}$}}, \m s_{\raisebox{-0.1ex}{$\scalebox{0.6}{$L\mmn-\mmn2$}$}}}^{\scalebox{0.65}{$(L\mn -\mn 2\m)$}}\cdot \phi_{1, \m s_{\raisebox{-0.1ex}{$\scalebox{0.6}{$L\mmn-\mmn2$}$}}}^{\scalebox{0.65}{$(L\mn -\mn 2\m)$}}(\boldsymbol{x})+ b_{1, \m s_{\raisebox{-0.1ex}{$\scalebox{0.6}{$L\mmn-\mmn1$}$}}}^{\scalebox{0.65}{$(L\mn -\mn 2\m)$}}\bigg)\cdot\\
& \qquad \cdots \rho^{\prime}\bigg(\sum_{s_{1} = 1}^{\scriptscriptstyle N_{1}}a_{1, s_2, s_1}^{\scalebox{0.6}{$(1\m)$}}\cdot \phi_{1, s_1}^{\scalebox{0.6}{$(1\m)$}}(\boldsymbol{x}) + b_{1, s_{2}}^{\scalebox{0.6}{$(1\m)$}}\bigg) \cdot \rho^{\prime}\bigg(\sum_{s_{0} = 1}^{d}a_{1, s_1, s_0}^{\scalebox{0.6}{$(0\m)$}}\cdot x_{s_0} + b_{1, s_1}^{\scalebox{0.6}{$(0\m)$}}\bigg)\\
& \qquad +  \cdots  + x_1 \rho^{\prime}\bigg(\sum_{s_{L-2} = 1}^{\scriptscriptstyle N_{L-2}} \!\! a_{1, \m s_{\raisebox{-0.1ex}{$\scalebox{0.6}{$L\mmn-\mmn1$}$}}, \m s_{\raisebox{-0.1ex}{$\scalebox{0.6}{$L\mmn-\mmn2$}$}}}^{\scalebox{0.65}{$(L\mn -\mn 2\m)$}}\cdot \phi_{1, \m s_{\raisebox{-0.1ex}{$\scalebox{0.6}{$L\mmn-\mmn2$}$}}}^{\scalebox{0.65}{$(L\mn -\mn 2\m)$}}(\boldsymbol{x})+ b_{1, \m s_{\raisebox{-0.1ex}{$\scalebox{0.6}{$L\mmn-\mmn1$}$}}}^{\scalebox{0.65}{$(L\mn -\mn 2\m)$}}\bigg) \cdot\\
& \qquad \cdots \rho^{\prime}\bigg(\sum_{s_{1} = 1}^{\scriptscriptstyle N_{1}}a_{1, s_2, s_1}^{\scalebox{0.6}{$(1\m)$}}\cdot \phi_{1, s_1}^{\scalebox{0.6}{$(1\m)$}}(\boldsymbol{x}) + b_{1, s_{2}}^{\scalebox{0.6}{$(1\m)$}}\bigg) \cdot \partial_{x_1}\rho^{\prime}\bigg(\sum_{s_{0} = 1}^{d}a_{1, s_1, s_0}^{\scalebox{0.6}{$(0\m)$}}\cdot x_{s_0} + b_{1, s_1}^{\scalebox{0.6}{$(0\m)$}}\bigg)\bigg]\\
& = \bigg[\sum_{s_{L-1} = 1}^{\scriptscriptstyle N_{L-1}} \!\! a_{1, \m 1, \m s_{\raisebox{-0.1ex}{$\scalebox{0.6}{$L\mmn-\mmn1$}$}}}^{\scalebox{0.65}{$(L\mn -\mn 1\m)$}} \cdot \bigg[\sum_{s_{L-2} = 1}^{\scriptscriptstyle N_{L-2}} \!\! a_{1, \m s_{\raisebox{-0.1ex}{$\scalebox{0.6}{$L\mmn-\mmn1$}$}}, \m s_{\raisebox{-0.1ex}{$\scalebox{0.6}{$L\mmn-\mmn2$}$}}}^{\scalebox{0.65}{$(L\mn -\mn 2\m)$}} \cdot \bigg[\sum_{s_{L-3} = 1}^{\scriptscriptstyle N_{L-3}} \!\! a_{1, \m s_{\raisebox{-0.1ex}{$\scalebox{0.6}{$L\mmn-\mmn2$}$}}, \m s_{\raisebox{-0.1ex}{$\scalebox{0.6}{$L\mmn-\mmn3$}$}}}^{\scalebox{0.65}{$(L\mn -\mn 3\m)$}} \cdots \sum_{s_{1} = 1}^{\scriptscriptstyle N_{1}}a_{1, s_2, s_1}^{\scalebox{0.6}{$(1\m)$}}\bigg]\bigg]\bigg] \cdot \\
& \qquad \bigg[\rho^{\prime}\bigg(\sum_{s_{L-2} = 1}^{\scriptscriptstyle N_{L-2}} \!\! a_{1, \m s_{\raisebox{-0.1ex}{$\scalebox{0.6}{$L\mmn-\mmn1$}$}}, \m s_{\raisebox{-0.1ex}{$\scalebox{0.6}{$L\mmn-\mmn2$}$}}}^{\scalebox{0.65}{$(L\mn -\mn 2\m)$}}\cdot \phi_{1, \m s_{\raisebox{-0.1ex}{$\scalebox{0.6}{$L\mmn-\mmn2$}$}}}^{\scalebox{0.65}{$(L\mn -\mn 2\m)$}}(\boldsymbol{x})+ b_{1, \m s_{\raisebox{-0.1ex}{$\scalebox{0.6}{$L\mmn-\mmn1$}$}}}^{\scalebox{0.65}{$(L\mn -\mn 2\m)$}}\bigg)\cdots\\
& \qquad \rho^{\prime}\bigg(\sum_{s_{1} = 1}^{\scriptscriptstyle N_{1}}a_{1, s_2, s_1}^{\scalebox{0.6}{$(1\m)$}}\cdot \phi_{1, s_1}^{\scalebox{0.6}{$(1\m)$}}(\boldsymbol{x}) + b_{1, s_{2}}^{\scalebox{0.6}{$(1\m)$}}\bigg) \cdot \rho^{\prime}\bigg(\sum_{s_{0} = 1}^{d}a_{1, s_1, s_0}^{\scalebox{0.6}{$(0\m)$}}\cdot x_{s_0} + b_{1, s_1}^{\scalebox{0.6}{$(0\m)$}}\bigg) \\
&\qquad + x_1  \rho^{\prime \prime}\bigg(\sum_{s_{L-2} = 1}^{\scriptscriptstyle N_{L-2}} \!\! a_{1, \m s_{\raisebox{-0.1ex}{$\scalebox{0.6}{$L\mmn-\mmn1$}$}}, \m s_{\raisebox{-0.1ex}{$\scalebox{0.6}{$L\mmn-\mmn2$}$}}}^{\scalebox{0.65}{$(L\mn -\mn 2\m)$}}\cdot \phi_{1, \m s_{\raisebox{-0.1ex}{$\scalebox{0.6}{$L\mmn-\mmn2$}$}}}^{\scalebox{0.65}{$(L\mn -\mn 2\m)$}}(\boldsymbol{x})+ b_{1, \m s_{\raisebox{-0.1ex}{$\scalebox{0.6}{$L\mmn-\mmn1$}$}}}^{\scalebox{0.65}{$(L\mn -\mn 2\m)$}}\bigg) \sum_{s_{L-2} = 1}^{\scriptscriptstyle N_{L-2}} \!\! a_{1, \m s_{\raisebox{-0.1ex}{$\scalebox{0.6}{$L\mmn-\mmn1$}$}}, \m s_{\raisebox{-0.1ex}{$\scalebox{0.6}{$L\mmn-\mmn2$}$}}}^{\scalebox{0.65}{$(L\mn -\mn 2\m)$}}\cdot \partial_{x_1}\phi_{1, \m s_{\raisebox{-0.1ex}{$\scalebox{0.6}{$L\mmn-\mmn2$}$}}}^{\scalebox{0.65}{$(L\mn -\mn 2\m)$}}(\boldsymbol{x}) \\
&\qquad \cdot \rho^{\prime}\bigg(\sum_{s_{L-3} = 1}^{\scriptscriptstyle N_{L-3}} \!\! a_{1, \m s_{\raisebox{-0.1ex}{$\scalebox{0.6}{$L\mmn-\mmn2$}$}}, \m s_{\raisebox{-0.1ex}{$\scalebox{0.6}{$L\mmn-\mmn3$}$}}}^{\scalebox{0.65}{$(L\mn -\mn 3\m)$}}\cdot \phi_{1, \m s_{\raisebox{-0.1ex}{$\scalebox{0.6}{$L\mmn-\mmn3$}$}}}^{\scalebox{0.65}{$(L\mn -\mn 3\m)$}}(\boldsymbol{x})+ b_{1, \m s_{\raisebox{-0.1ex}{$\scalebox{0.6}{$L\mmn-\mmn2$}$}}}^{\scalebox{0.65}{$(L\mn -\mn 3\m)$}}\bigg) \cdots \rho^{\prime}\bigg(\sum_{s_{0} = 1}^{d}a_{1, s_1, s_0}^{\scalebox{0.6}{$(0\m)$}}\cdot x_{s_0} + b_{1, s_1}^{\scalebox{0.6}{$(0\m)$}}\bigg)\\
& \qquad +  \cdots + x_1  \rho^{\prime}\bigg(\sum_{s_{L-2} = 1}^{\scriptscriptstyle N_{L-2}} \!\! a_{1, \m s_{\raisebox{-0.1ex}{$\scalebox{0.6}{$L\mmn-\mmn1$}$}}, \m s_{\raisebox{-0.1ex}{$\scalebox{0.6}{$L\mmn-\mmn2$}$}}}^{\scalebox{0.65}{$(L\mn -\mn 2\m)$}}\cdot \phi_{1, \m s_{\raisebox{-0.1ex}{$\scalebox{0.6}{$L\mmn-\mmn2$}$}}}^{\scalebox{0.65}{$(L\mn -\mn 2\m)$}}(\boldsymbol{x})+ b_{1, \m s_{\raisebox{-0.1ex}{$\scalebox{0.6}{$L\mmn-\mmn1$}$}}}^{\scalebox{0.65}{$(L\mn -\mn 2\m)$}}\bigg) \cdot \\
& \qquad \cdots \rho^{\prime}\bigg(\sum_{s_{1} = 1}^{\scriptscriptstyle N_{1}}a_{1, s_2, s_1}^{\scalebox{0.6}{$(1\m)$}}\cdot \phi_{1, s_1}^{\scalebox{0.6}{$(1\m)$}}(\boldsymbol{x}) + b_{1, s_{2}}^{\scalebox{0.6}{$(1\m)$}}\bigg) \cdot \partial_{x_1}\rho^{\prime}\bigg(\sum_{s_{0} = 1}^{d}a_{1, s_1, s_0}^{\scalebox{0.6}{$(0\m)$}}\cdot x_{s_0} + b_{1, s_1}^{\scalebox{0.6}{$(0\m)$}}\bigg)\bigg] \, .
\end{align*}
Further taking derivative w.r.t. $a_{\scriptscriptstyle 1,1,1}^{\scriptscriptstyle(0)}$, we have
\begin{align*}
&\partial_{a_{\scalebox{0.52}{$1,1,1$}}^{\scalebox{0.5}{$(0)$}}}   \partial_{x_1}{\partial_{a_{\scalebox{0.52}{$1,1,1$}}^{\scalebox{0.5}{$(0)$}}} \phi_{\boldsymbol{\theta}}^{1}(\boldsymbol{x})} \\
&~~ = \bigg[\sum_{s_{L-1} = 1}^{\scriptscriptstyle N_{L-1}} \!\! a_{1, \m 1, \m s_{\raisebox{-0.1ex}{$\scalebox{0.6}{$L\mmn-\mmn1$}$}}}^{\scalebox{0.65}{$(L\mn -\mn 1\m)$}} \cdot \bigg[\sum_{s_{L-2} = 1}^{\scriptscriptstyle N_{L-2}} \!\! a_{1, \m s_{\raisebox{-0.1ex}{$\scalebox{0.6}{$L\mmn-\mmn1$}$}}, \m s_{\raisebox{-0.1ex}{$\scalebox{0.6}{$L\mmn-\mmn2$}$}}}^{\scalebox{0.65}{$(L\mn -\mn 2\m)$}} \cdot \bigg[\sum_{s_{L-3} = 1}^{\scriptscriptstyle N_{L-3}} \!\! a_{1, \m s_{\raisebox{-0.1ex}{$\scalebox{0.6}{$L\mmn-\mmn2$}$}}, \m s_{\raisebox{-0.1ex}{$\scalebox{0.6}{$L\mmn-\mmn3$}$}}}^{\scalebox{0.65}{$(L\mn -\mn 3\m)$}} \cdots \sum_{s_{1} = 1}^{\scriptscriptstyle N_{1}}a_{1, s_2, s_1}^{\scalebox{0.6}{$(1\m)$}}\bigg]\bigg]\bigg] \cdot \\
& \qquad \bigg[\partial_{a_{\scalebox{0.52}{$1,1,1$}}^{\scalebox{0.5}{$(0)$}}}\bigg[\rho^{\prime}\bigg(\sum_{s_{L-2} = 1}^{\scriptscriptstyle N_{L-2}} \!\! a_{1, \m s_{\raisebox{-0.1ex}{$\scalebox{0.6}{$L\mmn-\mmn1$}$}}, \m s_{\raisebox{-0.1ex}{$\scalebox{0.6}{$L\mmn-\mmn2$}$}}}^{\scalebox{0.65}{$(L\mn -\mn 2\m)$}}\cdot \phi_{1, \m s_{\raisebox{-0.1ex}{$\scalebox{0.6}{$L\mmn-\mmn2$}$}}}^{\scalebox{0.65}{$(L\mn -\mn 2\m)$}}(\boldsymbol{x})+ b_{1, \m s_{\raisebox{-0.1ex}{$\scalebox{0.6}{$L\mmn-\mmn1$}$}}}^{\scalebox{0.65}{$(L\mn -\mn 2\m)$}}\bigg)\cdots\\
& \qquad \rho^{\prime}\bigg(\sum_{s_{1} = 1}^{\scriptscriptstyle N_{1}}a_{1, s_2, s_1}^{\scalebox{0.6}{$(1\m)$}}\cdot \phi_{1, s_1}^{\scalebox{0.6}{$(1\m)$}}(\boldsymbol{x}) + b_{1, s_{2}}^{\scalebox{0.6}{$(1\m)$}}\bigg) \cdot \rho^{\prime}\bigg(\sum_{s_{0} = 1}^{d}a_{1, s_1, s_0}^{\scalebox{0.6}{$(0\m)$}}\cdot x_{s_0} + b_{1, s_1}^{\scalebox{0.6}{$(0\m)$}}\bigg)\bigg] \\
&\qquad + x_1 \partial_{a_{\scalebox{0.52}{$1,1,1$}}^{\scalebox{0.5}{$(0)$}}}\bigg[ \rho^{\prime \prime}\bigg(\sum_{s_{L-2} = 1}^{\scriptscriptstyle N_{L-2}} \!\! a_{1, \m s_{\raisebox{-0.1ex}{$\scalebox{0.6}{$L\mmn-\mmn1$}$}}, \m s_{\raisebox{-0.1ex}{$\scalebox{0.6}{$L\mmn-\mmn2$}$}}}^{\scalebox{0.65}{$(L\mn -\mn 2\m)$}}\cdot \phi_{1, \m s_{\raisebox{-0.1ex}{$\scalebox{0.6}{$L\mmn-\mmn2$}$}}}^{\scalebox{0.65}{$(L\mn -\mn 2\m)$}}(\boldsymbol{x})+ b_{1, \m s_{\raisebox{-0.1ex}{$\scalebox{0.6}{$L\mmn-\mmn1$}$}}}^{\scalebox{0.65}{$(L\mn -\mn 2\m)$}}\bigg) \! \cdot \!\! \sum_{s_{L-2} = 1}^{\scriptscriptstyle N_{L-2}} \!\! a_{1, \m s_{\raisebox{-0.1ex}{$\scalebox{0.6}{$L\mmn-\mmn1$}$}}, \m s_{\raisebox{-0.1ex}{$\scalebox{0.6}{$L\mmn-\mmn2$}$}}}^{\scalebox{0.65}{$(L\mn -\mn 2\m)$}}\cdot \partial_{x_1}\phi_{1, \m s_{\raisebox{-0.1ex}{$\scalebox{0.6}{$L\mmn-\mmn2$}$}}}^{\scalebox{0.65}{$(L\mn -\mn 2\m)$}}(\boldsymbol{x}) \\
&\qquad \cdot \rho^{\prime}\bigg(\sum_{s_{L-3} = 1}^{\scriptscriptstyle N_{L-3}} \!\! a_{1, \m s_{\raisebox{-0.1ex}{$\scalebox{0.6}{$L\mmn-\mmn2$}$}}, \m s_{\raisebox{-0.1ex}{$\scalebox{0.6}{$L\mmn-\mmn3$}$}}}^{\scalebox{0.65}{$(L\mn -\mn 3\m)$}}\cdot \phi_{1, \m s_{\raisebox{-0.1ex}{$\scalebox{0.6}{$L\mmn-\mmn3$}$}}}^{\scalebox{0.65}{$(L\mn -\mn 3\m)$}}(\boldsymbol{x})+ b_{1, \m s_{\raisebox{-0.1ex}{$\scalebox{0.6}{$L\mmn-\mmn2$}$}}}^{\scalebox{0.65}{$(L\mn -\mn 3\m)$}}\bigg) \cdots \rho^{\prime}\bigg(\sum_{s_{0} = 1}^{d}a_{1, s_1, s_0}^{\scalebox{0.6}{$(0\m)$}}\cdot x_{s_0} + b_{1, s_1}^{\scalebox{0.6}{$(0\m)$}}\bigg)\bigg]\\
& \qquad +  \cdots + x_1  \partial_{a_{\scalebox{0.52}{$1,1,1$}}^{\scalebox{0.5}{$(0)$}}} \bigg[\rho^{\prime}\bigg(\sum_{s_{L-2} = 1}^{\scriptscriptstyle N_{L-2}} \!\! a_{1, \m s_{\raisebox{-0.1ex}{$\scalebox{0.6}{$L\mmn-\mmn1$}$}}, \m s_{\raisebox{-0.1ex}{$\scalebox{0.6}{$L\mmn-\mmn2$}$}}}^{\scalebox{0.65}{$(L\mn -\mn 2\m)$}}\cdot \phi_{1, \m s_{\raisebox{-0.1ex}{$\scalebox{0.6}{$L\mmn-\mmn2$}$}}}^{\scalebox{0.65}{$(L\mn -\mn 2\m)$}}(\boldsymbol{x})+ b_{1, \m s_{\raisebox{-0.1ex}{$\scalebox{0.6}{$L\mmn-\mmn1$}$}}}^{\scalebox{0.65}{$(L\mn -\mn 2\m)$}}\bigg)\\
& \qquad \cdots \rho^{\prime}\bigg(\sum_{s_{1} = 1}^{\scriptscriptstyle N_{1}}a_{1, s_2, s_1}^{\scalebox{0.6}{$(1\m)$}}\cdot \phi_{1, s_1}^{\scalebox{0.6}{$(1\m)$}}(\boldsymbol{x}) + b_{1, s_{2}}^{\scalebox{0.6}{$(1\m)$}}\bigg) \cdot \partial_{x_1}\rho^{\prime}\bigg(\sum_{s_{0} = 1}^{d}a_{1, s_1, s_0}^{\scalebox{0.6}{$(0\m)$}}\cdot x_{s_0} + b_{1, s_1}^{\scalebox{0.6}{$(0\m)$}}\bigg)\bigg]\bigg] \, .
& 
\end{align*}
The calculation of the typical item tells us
\begin{align*}
&\partial_{a_{\scalebox{0.52}{$1,1,1$}}^{\scalebox{0.5}{$(0)$}}}\bigg[\rho^{\prime \prime}\bigg(\sum_{s_{L-2} = 1}^{\scriptscriptstyle N_{L-2}} \!\! a_{1, \m s_{\raisebox{-0.1ex}{$\scalebox{0.6}{$L\mmn-\mmn1$}$}}, \m s_{\raisebox{-0.1ex}{$\scalebox{0.6}{$L\mmn-\mmn2$}$}}}^{\scalebox{0.65}{$(L\mn -\mn 2\m)$}}\cdot \phi_{1, \m s_{\raisebox{-0.1ex}{$\scalebox{0.6}{$L\mmn-\mmn2$}$}}}^{\scalebox{0.65}{$(L\mn -\mn 2\m)$}}(\boldsymbol{x})+ b_{1, \m s_{\raisebox{-0.1ex}{$\scalebox{0.6}{$L\mmn-\mmn1$}$}}}^{\scalebox{0.65}{$(L\mn -\mn 2\m)$}}\bigg) \sum_{s_{L-2} = 1}^{\scriptscriptstyle N_{L-2}} \!\! a_{1, \m s_{\raisebox{-0.1ex}{$\scalebox{0.6}{$L\mmn-\mmn1$}$}}, \m s_{\raisebox{-0.1ex}{$\scalebox{0.6}{$L\mmn-\mmn2$}$}}}^{\scalebox{0.65}{$(L\mn -\mn 2\m)$}}\cdot \partial_{x_1}\phi_{1, \m s_{\raisebox{-0.1ex}{$\scalebox{0.6}{$L\mmn-\mmn2$}$}}}^{\scalebox{0.65}{$(L\mn -\mn 2\m)$}}(\boldsymbol{x}) \\
&\qquad \quad \cdot \rho^{\prime}\bigg(\sum_{s_{L-3} = 1}^{\scriptscriptstyle N_{L-3}} \!\! a_{1, \m s_{\raisebox{-0.1ex}{$\scalebox{0.6}{$L\mmn-\mmn2$}$}}, \m s_{\raisebox{-0.1ex}{$\scalebox{0.6}{$L\mmn-\mmn3$}$}}}^{\scriptscriptstyle (L - 3)}\cdot \phi_{1, \m s_{\raisebox{-0.1ex}{$\scalebox{0.6}{$L\mmn-\mmn3$}$}}}^{\scalebox{0.65}{$(L\mn -\mn 3\m)$}}(\boldsymbol{x})+ b_{1, \m s_{\raisebox{-0.1ex}{$\scalebox{0.6}{$L\mmn-\mmn2$}$}}}^{\scalebox{0.65}{$(L\mn -\mn 3\m)$}}\bigg) \cdots \rho^{\prime}\bigg(\sum_{s_{0} = 1}^{d}a_{1, s_1, s_0}^{\scalebox{0.6}{$(0\m)$}}\cdot x_{s_0} + b_{1, s_1}^{\scalebox{0.6}{$(0\m)$}}\bigg)\bigg]\\
& = \rho^{\prime \prime \prime}\bigg(\sum_{s_{L-2} = 1}^{\scriptscriptstyle N_{L-2}} \!\! a_{1, \m s_{\raisebox{-0.1ex}{$\scalebox{0.6}{$L\mmn-\mmn1$}$}}, \m s_{\raisebox{-0.1ex}{$\scalebox{0.6}{$L\mmn-\mmn2$}$}}}^{\scalebox{0.65}{$(L\mn -\mn 2\m)$}}\cdot \phi_{1, \m s_{\raisebox{-0.1ex}{$\scalebox{0.6}{$L\mmn-\mmn2$}$}}}^{\scalebox{0.65}{$(L\mn -\mn 2\m)$}}(\boldsymbol{x})+ b_{1, \m s_{\raisebox{-0.1ex}{$\scalebox{0.6}{$L\mmn-\mmn1$}$}}}^{\scalebox{0.65}{$(L\mn -\mn 2\m)$}}\bigg)\\
&\quad \cdot\sum_{s_{L-2} = 1}^{\scriptscriptstyle N_{L-2}} \!\! a_{1, \m s_{\raisebox{-0.1ex}{$\scalebox{0.6}{$L\mmn-\mmn1$}$}}, \m s_{\raisebox{-0.1ex}{$\scalebox{0.6}{$L\mmn-\mmn2$}$}}}^{\scalebox{0.65}{$(L\mn -\mn 2\m)$}}\cdot \partial_{a_{\scalebox{0.52}{$1,1,1$}}^{\scalebox{0.5}{$(0)$}}}\phi_{1, \m s_{\raisebox{-0.1ex}{$\scalebox{0.6}{$L\mmn-\mmn2$}$}}}^{\scalebox{0.65}{$(L\mn -\mn 2\m)$}}(\boldsymbol{x}) \sum_{s_{L-2} = 1}^{\scriptscriptstyle N_{L-2}} \!\! a_{1, \m s_{\raisebox{-0.1ex}{$\scalebox{0.6}{$L\mmn-\mmn1$}$}}, \m s_{\raisebox{-0.1ex}{$\scalebox{0.6}{$L\mmn-\mmn2$}$}}}^{\scalebox{0.65}{$(L\mn -\mn 2\m)$}}\cdot \partial_{x_1}\phi_{1, \m s_{\raisebox{-0.1ex}{$\scalebox{0.6}{$L\mmn-\mmn2$}$}}}^{\scalebox{0.65}{$(L\mn -\mn 2\m)$}}(\boldsymbol{x})\\
&\quad\cdot \rho^{\prime}\bigg(\sum_{s_{L-3} = 1}^{\scriptscriptstyle N_{L-3}} \!\! a_{1, \m s_{\raisebox{-0.1ex}{$\scalebox{0.6}{$L\mmn-\mmn2$}$}}, \m s_{\raisebox{-0.1ex}{$\scalebox{0.6}{$L\mmn-\mmn3$}$}}}^{\scalebox{0.65}{$(L\mn -\mn 3\m)$}}\cdot \phi_{1, \m s_{\raisebox{-0.1ex}{$\scalebox{0.6}{$L\mmn-\mmn3$}$}}}^{\scalebox{0.65}{$(L\mn -\mn 3\m)$}}(\boldsymbol{x})+ b_{1, \m s_{\raisebox{-0.1ex}{$\scalebox{0.6}{$L\mmn-\mmn2$}$}}}^{\scalebox{0.65}{$(L\mn -\mn 3\m)$}}\bigg) \cdots \rho^{\prime}\bigg(\sum_{s_{0} = 1}^{d}a_{1, s_1, s_0}^{\scalebox{0.6}{$(0\m)$}}\cdot x_{s_0} + b_{1, s_1}^{\scalebox{0.6}{$(0\m)$}}\bigg)\\
&\quad + \rho^{\prime \prime}\bigg(\sum_{s_{L-2} = 1}^{\scriptscriptstyle N_{L-2}} \!\! a_{1, \m s_{\raisebox{-0.1ex}{$\scalebox{0.6}{$L\mmn-\mmn1$}$}}, \m s_{\raisebox{-0.1ex}{$\scalebox{0.6}{$L\mmn-\mmn2$}$}}}^{\scalebox{0.65}{$(L\mn -\mn 2\m)$}}\cdot \phi_{1, \m s_{\raisebox{-0.1ex}{$\scalebox{0.6}{$L\mmn-\mmn2$}$}}}^{\scalebox{0.65}{$(L\mn -\mn 2\m)$}}(\boldsymbol{x})+ b_{1, \m s_{\raisebox{-0.1ex}{$\scalebox{0.6}{$L\mmn-\mmn1$}$}}}^{\scalebox{0.65}{$(L\mn -\mn 2\m)$}}\bigg) \sum_{s_{L-2} = 1}^{\scriptscriptstyle N_{L-2}} \!\! a_{1, \m s_{\raisebox{-0.1ex}{$\scalebox{0.6}{$L\mmn-\mmn1$}$}}, \m s_{\raisebox{-0.1ex}{$\scalebox{0.6}{$L\mmn-\mmn2$}$}}}^{\scalebox{0.65}{$(L\mn -\mn 2\m)$}}\cdot \partial_{a_{\scalebox{0.52}{$1,1,1$}}^{\scalebox{0.5}{$(0)$}}}\partial_{x_1}\phi_{1, \m s_{\raisebox{-0.1ex}{$\scalebox{0.6}{$L\mmn-\mmn2$}$}}}^{\scalebox{0.65}{$(L\mn -\mn 2\m)$}}(\boldsymbol{x}) \\
&\quad \cdot \rho^{\prime}\bigg(\sum_{s_{L-3} = 1}^{\scriptscriptstyle N_{L-3}} \!\! a_{1, \m s_{\raisebox{-0.1ex}{$\scalebox{0.6}{$L\mmn-\mmn2$}$}}, \m s_{\raisebox{-0.1ex}{$\scalebox{0.6}{$L\mmn-\mmn3$}$}}}^{\scalebox{0.65}{$(L\mn -\mn 3\m)$}}\cdot \phi_{1, \m s_{\raisebox{-0.1ex}{$\scalebox{0.6}{$L\mmn-\mmn3$}$}}}^{\scalebox{0.65}{$(L\mn -\mn 3\m)$}}(\boldsymbol{x})+ b_{1, \m s_{\raisebox{-0.1ex}{$\scalebox{0.6}{$L\mmn-\mmn2$}$}}}^{\scalebox{0.65}{$(L\mn -\mn 3\m)$}}\bigg) \cdots \rho^{\prime}\bigg(\sum_{s_{0} = 1}^{d}a_{1, s_1, s_0}^{\scalebox{0.6}{$(0\m)$}}\cdot x_{s_0} + b_{1, s_1}^{\scalebox{0.6}{$(0\m)$}}\bigg)\\
& \quad +  \cdots  + \rho^{\prime \prime}\bigg(\sum_{s_{L-2} = 1}^{\scriptscriptstyle N_{L-2}} \!\! a_{1, \m s_{\raisebox{-0.1ex}{$\scalebox{0.6}{$L\mmn-\mmn1$}$}}, \m s_{\raisebox{-0.1ex}{$\scalebox{0.6}{$L\mmn-\mmn2$}$}}}^{\scalebox{0.65}{$(L\mn -\mn 2\m)$}}\cdot \phi_{1, \m s_{\raisebox{-0.1ex}{$\scalebox{0.6}{$L\mmn-\mmn2$}$}}}^{\scalebox{0.65}{$(L\mn -\mn 2\m)$}}(\boldsymbol{x})+ b_{1, \m s_{\raisebox{-0.1ex}{$\scalebox{0.6}{$L\mmn-\mmn1$}$}}}^{\scalebox{0.65}{$(L\mn -\mn 2\m)$}}\bigg) \sum_{s_{L-2} = 1}^{\scriptscriptstyle N_{L-2}} \!\! a_{1, \m s_{\raisebox{-0.1ex}{$\scalebox{0.6}{$L\mmn-\mmn1$}$}}, \m s_{\raisebox{-0.1ex}{$\scalebox{0.6}{$L\mmn-\mmn2$}$}}}^{\scalebox{0.65}{$(L\mn -\mn 2\m)$}}\cdot \partial_{x_1}\phi_{1, \m s_{\raisebox{-0.1ex}{$\scalebox{0.6}{$L\mmn-\mmn2$}$}}}^{\scalebox{0.65}{$(L\mn -\mn 2\m)$}}(\boldsymbol{x}) \\
&\quad \cdot \rho^{\prime}\bigg(\sum_{s_{L-3} = 1}^{\scriptscriptstyle N_{L-3}} \!\! a_{1, \m s_{\raisebox{-0.1ex}{$\scalebox{0.6}{$L\mmn-\mmn2$}$}}, \m s_{\raisebox{-0.1ex}{$\scalebox{0.6}{$L\mmn-\mmn3$}$}}}^{\scalebox{0.65}{$(L\mn -\mn 3\m)$}}\cdot \phi_{1, \m s_{\raisebox{-0.1ex}{$\scalebox{0.6}{$L\mmn-\mmn3$}$}}}^{\scalebox{0.65}{$(L\mn -\mn 3\m)$}}(\boldsymbol{x})+ b_{1, \m s_{\raisebox{-0.1ex}{$\scalebox{0.6}{$L\mmn-\mmn2$}$}}}^{\scalebox{0.65}{$(L\mn -\mn 3\m)$}}\bigg) \cdots \rho^{\prime\prime}\bigg(\sum_{s_{0} = 1}^{d}a_{1, s_1, s_0}^{\scalebox{0.6}{$(0\m)$}}\cdot x_{s_0} + b_{1, s_1}^{\scalebox{0.6}{$(0\m)$}}\bigg)\cdot x_1\\
&\leq (2L +2)\bigg(\prod_{i = 1}^{L-1} N_i\bigg)^2 (B_{\boldsymbol{\theta}})^{2L} \, .
\end{align*}
Therefore, we have
\begin{align*}
&\partial_{a_{\scalebox{0.52}{$1,1,1$}}^{\scalebox{0.5}{$(0)$}}}\partial_{a_{\scalebox{0.52}{$1,1,1$}}^{\scalebox{0.5}{$(0)$}}}\partial_{x_1}  u_{\mathfrak{m}, \boldsymbol{\theta}}(\boldsymbol{x}) \\
&~~\leq c_1 \cdot \sum_{s_{L-1} = 1}^{\scriptscriptstyle N_{L-1}}  \sum_{s_{L-2} = 1}^{\scriptscriptstyle N_{L-2}} \!\!  \cdots \sum_{s_{1} = 1}^{\scriptscriptstyle N_{1}} a_{k, \m 1, \m s_{\raisebox{-0.1ex}{$\scalebox{0.6}{$L\mmn-\mmn1$}$}}}^{\scalebox{0.65}{$(L\mn -\mn 1\m)$}} a_{k, \m s_{\raisebox{-0.1ex}{$\scalebox{0.6}{$L\mmn-\mmn1$}$}}, \m s_{\raisebox{-0.1ex}{$\scalebox{0.6}{$L\mmn-\mmn2$}$}}}^{\scalebox{0.65}{$(L\mn -\mn 2\m)$}} \cdots a_{k, s_{2}, s_{1}}^{\scalebox{0.6}{$(1\m)$}}\cdot L \cdot (2L +2)\bigg(\prod_{i = 1}^{L-1} N_i\bigg)^2 (B_{\boldsymbol{\theta}})^{2L}\\
&~~\leq M\cdot L \cdot (2L +2) \bigg(\prod_{i = 1}^{L-1} N_i\bigg)^3 (B_{\boldsymbol{\theta}})^{3L-1} \, .
\end{align*}
Combining above estimations together, we obtain the total upper bound
\begin{align*}
&\Big|\partial_{a_{\scalebox{0.52}{$1,1,1$}}^{\scalebox{0.5}{$(0)$}}}\partial_{a_{\scalebox{0.52}{$1,1,1$}}^{\scalebox{0.5}{$(0)$}}} \widehat{\mathcal{L}}(u_{\mathfrak{m}, \boldsymbol{\theta}}) \Big| \\
&~~~= \bigg|\partial_{a_{\scalebox{0.52}{$1,1,1$}}^{\scalebox{0.5}{$(0)$}}}\partial_{a_{\scalebox{0.52}{$1,1,1$}}^{\scalebox{0.5}{$(0)$}}}\bigg[\frac{1}{2} \sum_{i=1}^{d} |{\partial_{x_i}} u_{\mathfrak{m}, \boldsymbol{\theta}}(X)|^2 + \frac{1}{2} \omega(X) u_{\mathfrak{m}, \boldsymbol{\theta}}^2(X) - u_{\mathfrak{m}, \boldsymbol{\theta}}(X)h(X) - u_{\mathfrak{m}, \boldsymbol{\theta}}(Y)g(Y)\bigg]\bigg|\\
&~~~  \leq C(B_0, L)\cdot M^2 \cdot \bigg(\prod_{i = 1}^{L-1} N_i\bigg)^4 (B_{\boldsymbol{\theta}})^{4L}\, .
\end{align*}

It can be readily observed that the upper bound obtained by differentiating the neural network output w.r.t. $a^{\scriptscriptstyle (0)}_{ 1,1,1}$ also controls the upper bounds obtained by differentiating the output w.r.t. general $a^{\scriptscriptstyle (\ell)}_{ k,i,j}$. Since there are at most $(\mathfrak{n}_L + 1)\cdot \mathfrak{m}$ weights in $u_{\mathfrak{m}, \boldsymbol{\theta}}$, we have
$$\big \|\nabla^2_{\boldsymbol{\theta}^{\mathfrak{m}}_{\rm total}} \widehat{\mathcal{L}}(u_{\mathfrak{m}, \boldsymbol{\theta}})\big \|_{\rm F} \leq C(B_0, L)\cdot (\mathfrak{n}_L + 1)\cdot \mathfrak{m} \cdot M^2 \cdot \bigg(\prod_{i = 1}^{L-1} N_i\bigg)^4 (B_{\boldsymbol{\theta}})^{4L} \, .
$$

\subsubsection{Proof of Lemma \ref{lip_lem}}

For $i=1, 2$, when $u_{\mathfrak{m}, \boldsymbol{\theta}}$ is parameterized with $(\boldsymbol{\theta}^{\mathfrak{m}, i}_{ \rm in}, \boldsymbol{\theta}^{\mathfrak{m}}_{ \rm out})$, we denote it as $u_{\mathfrak{m}, i}=\sum_{k=1}^{\mathfrak{m}}c_k \cdot \phi^{k}_{\boldsymbol{\theta}, i}$. Then, it holds that
\begin{align*}
&\big|\widehat{F}\big(\boldsymbol{\theta}^{\mathfrak{m}, 1}_{ \rm in}, \boldsymbol{\theta}^{\mathfrak{m}}_{ \rm out})\big)-\widehat{F}\big(\boldsymbol{\theta}^{\mathfrak{m}, 2}_{ \rm in}, \boldsymbol{\theta}^{\mathfrak{m}}_{ \rm out})\big)\big| \\
&~~\leq C(\Omega)\bigg\{\Big|\frac{1}{N_{\rm in}}\sum_{p=1}^{N_{\rm in}}\frac{\|\nabla u_{\mathfrak{m}, 1}(X_p)\|_2^2-\|\nabla u_{\mathfrak{m}, 2}(X_p)\|_2^2}{2}\Big|+ \Big|\frac{1}{N_{\rm in}}\sum_{p=1}^{N_{\rm in}}\frac{ \omega(X_p) \big[u^2_{\mathfrak{m}, 1}(X_p)- u^2_{\mathfrak{m}, 2}(X_p)\big]}{2}\Big|\\
&~~~~~~+ \Big|\frac{1}{N_{\rm in}}\sum_{p=1}^{N_{\rm in}}[u_{\mathfrak{m}, 1}(X_p)h(X_p)- u_{\mathfrak{m}, 2}(X_p)h(X_p)]\Big|+ \Big|\frac{1}{N_{b}}\sum_{p=1}^{N_{b}}[u_{\mathfrak{m}, 1}(Y_p)g(Y_p)- u_{\mathfrak{m}, 2}(Y_p)g(Y_p)]\Big|\bigg\} \, ,
\end{align*}
where $C(\Omega)=\max\{|\Omega|, |\partial \Omega|\}$.
By Lemma \ref{Lip of Fi}, we have 
$$
|\phi_{\boldsymbol{\theta}}(\boldsymbol{x})|\leq (W+1)B_{\boldsymbol{\theta}} \, , \qquad
|\partial_{x_m}\phi_{\boldsymbol{\theta}}(\boldsymbol{x})|\leq W^{L-1}B_{\boldsymbol{\theta}}^L \, .
$$
Also, it holds that
\begin{align*}
	&~~~~~\big|\phi_{\boldsymbol{\theta}}(\boldsymbol{x})-\phi_{\tilde{\boldsymbol{\theta}}}(\boldsymbol{x})\big|
	\leq 2W^{L}\sqrt{L}\cdot B_{\boldsymbol{\theta}}^{L-1} \big\|\boldsymbol{\theta}-\tilde{\boldsymbol{\theta}}\big\|_2\, ,\quad \forall \boldsymbol{x}\in\Omega\,,  \\
	&\big|\partial_{x_m}\phi_{\boldsymbol{\theta}}(\boldsymbol{x})-\partial_{x_m}\phi_{\tilde{\boldsymbol{\theta}}}(\boldsymbol{x})\big|
	\leq 2W^{2L-1}\sqrt{L}(L+1)\cdot B_{\boldsymbol{\theta}}^{2L}\big\|\boldsymbol{\theta}-\tilde{\boldsymbol{\theta}}\big\|_2\, ,\quad \forall \boldsymbol{x}\in\Omega\, .
\end{align*}
Then, we have the following estimations. Firstly, it holds that

\begin{align*}
&\bigg|\frac{1}{N_{\rm in}}\sum_{p=1}^{N_{\rm in}}\frac{\|\nabla u_{\mathfrak{m}, 1}(X_p)\|_2^2-\|\nabla u_{\mathfrak{m}, 2}(X_p)\|_2^2}{2}\bigg| \\
&~~~= \frac{1}{2N_{\rm in}}\sum_{p=1}^{N_{\rm in}}\bigg(\Big\|\sum_{k=1}^{\mathfrak{m}} c_{k} \cdot \nabla \phi_{\boldsymbol{\theta}, 1}^{k} (X_p)\Big\|_{2}^2-\Big\|\sum_{k=1}^{\mathfrak{m}} c_{k}  \cdot \nabla \phi_{\boldsymbol{\theta}, 2}^{k}(X_p)\Big\|_{2}^2\bigg)\\
&~~~ \leq \frac{1}{2N_{\rm in}}\sum_{p=1}^{N_{\rm in}}\bigg\{\bigg|\sum_{m=1}^d \Big[\sum_{k=1}^{\mathfrak{m}} |c_{k}|  \cdot \Big(\partial_{x_m} \phi_{\boldsymbol{\theta}, 1}^{k} (X_p)+ \partial_{x_m} \phi_{\boldsymbol{\theta}, 2}^{k}(X_p) \Big)\Big]\\
&~~~~~~~~~~~~~~~~~~~~~  \Big[\sum_{k=1}^{\mathfrak{m}} |c_{k}|  \cdot \big(\partial_{x_m} \phi_{\boldsymbol{\theta}, 1}^{k} (X_p)-\partial_{x_m} \phi_{\boldsymbol{\theta}, 2}^{k}(X_p) \big)\Big]\bigg|\bigg\}\\
& ~~~\leq d M W^{L-1}B_{\boldsymbol{\theta}}^L \cdot \sum_{k=1}^{\mathfrak{m}} | c_{k} | \cdot \big| \partial_{x_m} \phi_{\boldsymbol{\theta}, 1}^{k} (X_p)-\partial_{x_m} \phi_{\boldsymbol{\theta}, 2}^{k}(X_p) \big|\\
& ~~~ \leq d M W^{L-1}B_{\boldsymbol{\theta}}^{L} \cdot \sqrt{\sum_{k=1}^{\mathfrak{m}} | c_{k} |^2} \cdot \sqrt{\sum_{k=1}^{\mathfrak{m}} \big| \partial_{x_m} \phi_{\boldsymbol{\theta}, 1}^{k} (X_p)-\partial_{x_m} \phi_{\boldsymbol{\theta}, 2}^{k}(X_p) \big|^2}\\
& ~~~ \leq  2d \mm W^{3L-2} \sqrt{L}(L+1) M B_{\boldsymbol{\theta}}^{3L} \cdot \|\boldsymbol{\theta}^{\mathfrak{m}}_{\rm out}\|_2 \cdot \big\|\boldsymbol{\theta}^{\mathfrak{m}, 1}_{ \rm in} - \boldsymbol{\theta}^{\mathfrak{m}, 2}_{ \rm in}\big\|_2 \, ,
\end{align*}
 where the third inequality utilizes the Cauchy-Schwarz inequality. Also, we have
\begin{align*}
&\bigg|\frac{1}{N_{\rm in}}\sum_{p=1}^{N_{\rm in}}\frac{ \omega(X_p)\big[u^2_{\mathfrak{m}, 1}(X_p)-u^2_{\mathfrak{m}, 2}(X_p)\big]}{2}\bigg|\\
&~~~= \bigg|\frac{1}{2N_{\rm in}}\sum_{p=1}^{N_{\rm in}}\omega(X_p)\Big[\Big(\sum_{k=1}^{\mathfrak{m} } c_{k} \cdot \phi_{\boldsymbol{\theta}, 1}^{k} (X_p)\Big)^2-\Big(\sum_{k=1}^{\mathfrak{m} } c_{k}  \cdot \phi_{\boldsymbol{\theta}, 2}^{k}(X_p)\Big)^2\Big]\bigg|\\
&~~~ \leq \frac{B_0}{2N_{\rm in}}\sum_{p=1}^{N_{\rm in}}\bigg\{\bigg|\Big[\sum_{k=1}^{\mathfrak{m} } |c_{k}|  \cdot \big(\phi_{\boldsymbol{\theta}, 1}^{k} (X_p)+ \phi_{\boldsymbol{\theta}, 2}^{k}(X_p) \big)\Big] \Big[\sum_{k=1}^{\mathfrak{m} } |c_{k}|  \cdot \big(\phi_{\boldsymbol{\theta}, 1}^{k} (X_p)-\phi_{\boldsymbol{\theta}, 2}^{k}(X_p) \big)\Big]\bigg|\bigg\}\\
&~~~\leq B_0 M (W+1)B_{\boldsymbol{\theta}} \bigg|\Big[\sum_{k=1}^{\mathfrak{m} } |c_{k}|  \cdot \big(\phi_{\boldsymbol{\theta}, 1}^{k} (X_p)-\phi_{\boldsymbol{\theta}, 2}^{k}(X_p) \big)\Big]\bigg|\\
&~~~\leq B_0 M(W+1)B_{\boldsymbol{\theta}} \sqrt{\sum_{k=1}^{\mathfrak{m} } | c_{k} |^2}\cdot \sqrt{\sum_{k=1}^{\mathfrak{m} } \big|\phi_{\boldsymbol{\theta}, 1}^{k} (X_p)- \phi_{\boldsymbol{\theta}, 2}^{k}(X_p) \big|^2}\\
&~~~\leq 2B_0 W^{L}(W+1)\sqrt{L}\cdot MB_{\boldsymbol{\theta}}^{L} \cdot \|\boldsymbol{\theta}_{\rm out}^{\mathfrak{m}}\|_2 \cdot \big\|\boldsymbol{\theta}^{\mathfrak{m}, 1}_{ \rm in} - \boldsymbol{\theta}^{\mathfrak{m}, 2}_{ \rm in}\big\|_2 \, .
\end{align*}
Finally, it holds that
\begin{align*}
&\Big|\frac{1}{N_{\rm in}}\sum_{p=1}^{N_{\rm in}}\Big[\big(u_{\mathfrak{m}, 1}(X_p)- u_{\mathfrak{m}, 2}(X_p)\big)h(X_p)\Big]\Big|+ \Big|\frac{1}{N_{b}}\sum_{p=1}^{N_{b}}\Big[\big(u_{\mathfrak{m}, 1}(Y_p)- u_{\mathfrak{m}, 2}(Y_p)\big)g(Y_p)\Big]\Big|\\
&~~~\leq 2B_0 \bigg|\Big[\sum_{k=1}^{\mathfrak{m}} c_{k}  \cdot \big(\phi_{\boldsymbol{\theta}, 1}^{k} (X_p)-\phi_{\boldsymbol{\theta}, 2}^{k}(X_p) \big)\Big]\bigg| \leq 2B_0 \sqrt{\sum_{k=1}^{\mathfrak{m}} | c_{k} |^2}\cdot \sqrt{\sum_{k=1}^{\mathfrak{m}} \big|\phi_{\boldsymbol{\theta}, 1}^{k} (X_p)- \phi_{\boldsymbol{\theta}, 2}^{k}(X_p) \big|^2}\\
&~~~\leq 4B_0 W^{L}\sqrt{L}\cdot B_{\boldsymbol{\theta}}^{L-1}  \cdot \|\boldsymbol{\theta}_{\rm out}^{\mathfrak{m}}\|_2 \cdot \big\|\boldsymbol{\theta}^{\mathfrak{m}, 1}_{ \rm in} - \boldsymbol{\theta}^{\mathfrak{m}, 2}_{ \rm in}\big\|_2 \, .
\end{align*}
Combining above estimations, we complete the proof of this lemma.

\subsubsection{Proof of Lemma \ref{lem:prob of G}}
 In the following, we will focus on the probability of $G_{\mathfrak{m}, \bar{\mathfrak{m}}, R, \delta}^{\mkern 2mu \raisebox{0.3ex}{$\scriptstyle \mathsf{c} $}}$. The proof will be divided into two steps: {\rm (i)} the case where $\bar{\mathfrak{m}}=1$ and $R=1$, i.e., $G_{\mathfrak{m},1,1,\delta}^{\mkern 2mu \raisebox{0.2ex}{$\scriptstyle \mathsf{c} $}}$;
 {\rm (ii)} the general case $G_{\mathfrak{m}, \bar{\mathfrak{m}}, R, \delta}^{\mkern 2mu \raisebox{0.2ex}{$\scriptstyle \mathsf{c} $}}$ where $\mathfrak{m}, R \in \mathbb{N}$.
 
 {\bf Step 1.} For $\bar{\mathfrak{m}}=R=1$, when $\mathfrak{m}= Q$, $G_{\mathfrak{m},1,1,\delta}^{\mkern 2mu \raisebox{0.2ex}{$\scriptstyle \mathsf{c} $}}$ denotes the event where, each sub-network weight vector among $(\boldsymbol{\theta}_1)^{\scriptscriptstyle [0]}$ through $(\boldsymbol{\theta}_{\scriptscriptstyle Q})^{\scriptscriptstyle [0]}$ satisfies $\|(\boldsymbol{\theta}_{\cdot})^{\scriptscriptstyle [0]}-\bar{\boldsymbol{\theta}}_1\|_{\infty} > \delta$. Since $\bar{\boldsymbol{\theta}}_1$ can be treated as a fixed vector in $[-B_{\bar{\boldsymbol{\theta}}}, B_{\bar{\boldsymbol{\theta}}}]^{\mm \mathfrak{D}(\bar{W}, \bar{L}, d)}$, where $\mathfrak{D}({W}, {L}, d)$ is defined in (\ref{eq:align_dim}), and we initialize all the sub-network parameters independently from $U[-B_{\bar{\boldsymbol{\theta}}}, B_{\bar{\boldsymbol{\theta}}}]$, it then holds that
$$
\mathbb{P}\Big[\big\|(\boldsymbol{\theta}_i)^{\scriptscriptstyle[0]}-\bar{\boldsymbol{\theta}}_1\big\|_{\infty} \leq \delta\Big] \ge \bigg(\frac{\delta}{2B_{\bar{\boldsymbol{\theta}}}}\bigg)^{\mathfrak{D}(\bar{W}, \bar{L}, d)} \ge \bigg(\frac{\delta}{2B_{\bar{\boldsymbol{\theta}}}}\bigg)^{\bar{W}(\bar{W}+1)\bar{L}}
$$
for any $i \in \{1,\ldots, Q\}$. Thus, we have
$$
\mathbb{P}\big(G_{\mathfrak{m},1,1,\delta}^{\mkern 2mu \raisebox{0.2ex}{$\scriptstyle \mathsf{c} $}}\big) = \mathbb{P}\Big[ \mkern 1mu \forall i \in \{1,\ldots, Q\} : \big\|(\boldsymbol{\theta}_i)^{\scriptscriptstyle[0]}-\bar{\boldsymbol{\theta}}_1\big\|_{\infty} > \delta \Big] \leq \Big[1-\delta^{\bar{W}(\bar{W}+1)\bar{L}}(2B_{\bar{\boldsymbol{\theta}}})^{-\bar{W}(\bar{W}+1)\bar{L}}\Big]^Q \, .
$$

{\bf Step 2.}  For general $\bar{\mathfrak{m}}, R \in \mathbb{N}$, when $\mathfrak{m}= \bar{\mathfrak{m}} \cdot R \cdot Q$, $G_{\mathfrak{m},\bar{\mathfrak{m}},R,\delta}$ denotes the event where, for each target $\bar{\boldsymbol{\theta}}_k$ in $(\bar{\boldsymbol{\theta}}_1, \ldots, \bar{\boldsymbol{\theta}}_{\bar{\mathfrak{m}}})$, at least $R$ weight vectors among $(\boldsymbol{\theta}_1)^{\scriptscriptstyle [0]}$ through $(\boldsymbol{\theta}_{\scriptscriptstyle \bar{\mathfrak{m}} \cdot R \cdot Q})^{\scriptscriptstyle [0]}$ satisfy $\|(\boldsymbol{\theta}_{\cdot})^{\scriptscriptstyle [0]}-\bar{\boldsymbol{\theta}}_1\|_{\infty} \leq  \delta$. It can be readily observed that
\begin{align*}
G_{\mathfrak{m},\bar{\mathfrak{m}},R,\delta} \supseteq \bigcap_{j=1}^{R} \bigcap_{k=1}^{\bar{\mathfrak{m}}} \Big\{ \exists \mkern 2mu i \in \big \{[(j-1)\bar{\mathfrak{m}}+k-1] \mkern 1mu Q,\ldots,[(j-1)\bar{\mathfrak{m}}+k] \mkern 1mu Q \big \} : \big\|(\boldsymbol{\theta}_i)^{\scriptscriptstyle[0]}-\bar{\boldsymbol{\theta}}_k\big\|_{\infty} \leq \delta  \Big\} \, .
\end{align*}
Thus, it holds that
\begin{align*}
G_{\mathfrak{m},\bar{\mathfrak{m}},R,\delta}^{\mkern 2mu \raisebox{0.2ex}{$\scriptstyle \mathsf{c} $}} \subseteq \bigcup_{j=1}^{R} \bigcup_{k=1}^{\bar{\mathfrak{m}}} \Big\{ \forall \mkern 2mu i \in \big \{[(j-1)\bar{\mathfrak{m}}+k-1] \mkern 1mu Q,\ldots,[(j-1)\bar{\mathfrak{m}}+k] \mkern 1mu Q \big \} : \big\|(\boldsymbol{\theta}_i)^{\scriptscriptstyle[0]}-\bar{\boldsymbol{\theta}}_k\big\|_{\infty} > \delta  \Big\} \, .
\end{align*}
This implies
\begin{align*}
\mathbb{P}\big(G_{\mathfrak{m},\bar{\mathfrak{m}},R,\delta}^{\mkern 2mu \raisebox{0.2ex}{$\scriptstyle \mathsf{c} $}}\big) & \leq \sum_{j=1}^{R} \sum_{k=1}^{\bar{\mathfrak{m}}} \mathbb{P}\Big[ \mkern 1mu \forall \mkern 2mu i \in \big \{[(j-1)\bar{\mathfrak{m}}+k-1] \mkern 1mu Q,\ldots,[(j-1)\bar{\mathfrak{m}}+k] \mkern 1mu Q \big \} : \big\|(\boldsymbol{\theta}_i)^{\scriptscriptstyle[0]}-\bar{\boldsymbol{\theta}}_k\big\|_{\infty} > \delta \Big] \\
&= \bar{\mathfrak{m}} \mkern 2mu R \mkern 2mu \mathbb{P}\big(G_{\mathfrak{m},1,1,\delta}^{\mkern 2mu \raisebox{0.2ex}{$\scriptstyle \mathsf{c} $}}\big) \leq \bar{\mathfrak{m}} \mkern 2mu R \Big[1-\delta^{\bar{W}(\bar{W}+1)\bar{L}}(2B_{\bar{\boldsymbol{\theta}}})^{-\bar{W}(\bar{W}+1)\bar{L}}\Big]^Q \, .
\end{align*}

Therefore, we have 
\begin{align*}
\mathbb{P}\big(G_{\mathfrak{m},\bar{\mathfrak{m}},R,\delta}\big)=1-\mathbb{P}\big(G_{\mathfrak{m},\bar{\mathfrak{m}},R,\delta}^{\mkern 2mu \raisebox{0.2ex}{$\scriptstyle \mathsf{c} $}}\big) \ge 1 - \bar{\mathfrak{m}}R\Big[1-\delta^{\bar{W}(\bar{W}+1)\bar{L}}(2B_{\bar{\boldsymbol{\theta}}})^{-\bar{W}(\bar{W}+1)\bar{L}}\Big]^{Q} \, .
\end{align*}

\subsubsection{Proof of Lemma \ref{lip_lem2}}
Firstly, we have
\begin{align*}
&\big| \widehat{\mathcal{L}}\big(u_{\mathfrak{m}}^{*}\big) -\widehat{\mathcal{L}}(u_{\bar{\mathfrak{m}}, \bar{\boldsymbol{\theta}}})\big| \leq C(\Omega)\bigg\{\bigg|\frac{1}{N_{\rm in}}\sum_{p=1}^{N_{\rm in}}\frac{\|\nabla u_{\mathfrak{m}}^{*}(X_p)\|_2^2-\|\nabla u_{\bar{\mathfrak{m}}, \bar{\boldsymbol{\theta}}}(X_p)\|_2^2}{2}\big|\\
&~~~~~~~~~~~~~~~~~~~~~~~~~~~+ \bigg|\frac{1}{N_{\rm in}}\sum_{p=1}^{N_{\rm in}}\frac{ \omega(X_p)\big[u_{\mathfrak{m}}^{*}(X_p)\big]^2- \omega(X_p)u_{\bar{\mathfrak{m}}, \bar{\boldsymbol{\theta}}}^2(X_p)}{2}\bigg|\\
&~~~~~~~~~~~~~~~~~~~~~~~~~~~+ \bigg|\frac{1}{N_{\rm in}}\sum_{p=1}^{N_{\rm in}}[u_{\mathfrak{m}}^{*}(X_p)h(X_p)- u_{\bar{\mathfrak{m}}, \bar{\boldsymbol{\theta}}}(X_p)h(X_p)]\bigg|\\
&~~~~~~~~~~~~~~~~~~~~~~~~~~~+ \bigg|\frac{1}{N_{b}}\sum_{p=1}^{N_{b}}[u_{\mathfrak{m}}^{*}(Y_p)g(Y_p)- u_{\bar{\mathfrak{m}}, \bar{\boldsymbol{\theta}}}(Y_p)g(Y_p)]\bigg|\bigg\} \, ,
\end{align*}
where $C(\Omega)=\max\{|\Omega|, |\partial\Omega|\}$. By substituting equation (\ref{eq:u*}) and (\ref{eq:umbar}), we have  that
\begin{align*}
&\bigg|\frac{1}{N_{\rm in}}\sum_{p=1}^{N_{\rm in}}\frac{\|\nabla u_{\mathfrak{m}}^{*}(X_p)\|_2^2-\|\nabla u_{\bar{\mathfrak{m}}, \bar{\boldsymbol{\theta}}}(X_p)\|_2^2}{2}\bigg|  \\
&= \frac{1}{2N_{\rm in}}\sum_{p=1}^{N_{\rm in}}\bigg(\Big\|\sum_{k=1}^{ \bar{\mathfrak{m}}}\sum_{v=1}^{R} \frac{\bar{c}_k}{R}\cdot \nabla \big(\phi_{\boldsymbol{\theta}}^{s_{k,v}}\big)^{\scriptscriptstyle [0]} (X_p)\Big\|_{2}^2-\Big\|\sum_{k=1}^{\bar{\mathfrak{m}}} \bar{c}_{k} \cdot \nabla \phi_{\bar{\boldsymbol{\theta}}}^{k}(X_p)\Big\|_{2}^2\bigg)\\
&= \frac{1}{2N_{\rm in}}\sum_{p=1}^{N_{\rm in}}\bigg(\Big\|\sum_{k=1}^{ \bar{\mathfrak{m}}}\sum_{v=1}^{R} \frac{\bar{c}_k}{R}\cdot \nabla \big(\phi_{\boldsymbol{\theta}}^{s_{k,v}}\big)^{\scriptscriptstyle [0]} (X_p)\Big\|_{2}^2-\Big\|\sum_{k=1}^{\bar{\mathfrak{m}}} \sum_{v=1}^{R}\frac{\bar{c}_{k}}{R}\ \cdot \nabla \phi_{\bar{\boldsymbol{\theta}}}^{k}(X_p)\Big\|_{2}^2\bigg)\\
& \leq \frac{1}{2N_{\rm in}}\sum_{p=1}^{N_{\rm in}}\bigg\{\bigg|\sum_{m=1}^d \Big[\sum_{k=1}^{ \bar{\mathfrak{m}}}\sum_{v=1}^{R} \Big|\frac{\bar{c}_{k}}{R}\Big|  \cdot \big(\partial_{x_m} (\phi_{\boldsymbol{\theta}}^{s_{k,v}})^{\scriptscriptstyle [0]} (X_p)+ \partial_{x_m} \phi_{\bar{\boldsymbol{\theta}}}^{k}(X_p) \big)\Big]\\
&\qquad \qquad \qquad  \Big[\sum_{k=1}^{ \bar{\mathfrak{m}}}\sum_{v=1}^{R} \Big|\frac{\bar{c}_{k}}{R}\Big|  \cdot \big(\partial_{x_m} (\phi_{\boldsymbol{\theta}}^{s_{k,v}})^{\scriptscriptstyle [0]} (X_p)-\partial_{x_m} \phi_{\bar{\boldsymbol{\theta}}}^{k}(X_p) \big)\Big]\bigg|\bigg\}\\
& \leq d \bar{M} \bar{W}^{L-1} B^ {\raisebox{0.2ex}{$\scriptscriptstyle \bar{L}$}}_{\bar{\boldsymbol{\theta}}} \cdot \sum_{k=1}^{ \bar{\mathfrak{m}}}\sum_{v=1}^{R} \Big| \frac{\bar{c}_{k}}{R} \Big| \cdot \big| \partial_{x_m} (\phi_{\boldsymbol{\theta}}^{s_{k,v}})^{\scriptscriptstyle [0]} (X_p)-\partial_{x_m}\phi_{\bar{\boldsymbol{\theta}}}^{k}(X_p) \big|\\
& \leq d \m \bar{W}^{3\bar{L}-2} \sqrt{\bar{L}}(\bar{L}+1) \bar{M}^2  B^ {\raisebox{0.2ex}{$\scriptscriptstyle 3\bar{L}$}}_{\bar{\boldsymbol{\theta}}} \cdot \max_{k = 1, \ldots, \bar{\mathfrak{m}}} \max_{v=1,\ldots,R} \left\|(\boldsymbol{\theta}_{s_{k,v}})^{\scriptscriptstyle [0]} -\bar{\boldsymbol{\theta}}_k\right\|_{2}\\
&\leq  2d \m \bar{W}^{3\bar{L}-2} \sqrt{\bar{L}}(\bar{L}+1) \bar{M}^2  B^ {\raisebox{0.2ex}{$\scriptscriptstyle 3\bar{L}$}}_{\bar{\boldsymbol{\theta}}} \sqrt{\bar{W}(\bar{W}+1)\bar{L}} \cdot \max_{k = 1, \ldots, \bar{\mathfrak{m}}} \max_{v=1,\ldots,R}\big\|(\boldsymbol{\theta}_{s_{k,v}})^{\scriptscriptstyle [0]} -\bar{\boldsymbol{\theta}}_k\big\|_{\infty} \, .
\end{align*}
where the third inequality utilizes the property from Lemma \ref{df-lip}. Also, it holds that
\begin{align*}
&\bigg|\frac{1}{N_{\rm in}}\sum_{p=1}^{N_{\rm in}}\frac{ \omega(X_p)\big[u_{\mathfrak{m}}^{*}(X_p)\big]^2- \omega(X_p)u_{\bar{\mathfrak{m}}, \bar{\boldsymbol{\theta}}}^2(X_p)}{2}\bigg| \\
&= \bigg|\frac{1}{2N_{\rm in}}\sum_{p=1}^{N_{\rm in}}\omega(X_p)\bigg[\Big(\sum_{k=1}^{ \bar{\mathfrak{m}}}\sum_{v=1}^{R} \frac{\bar{c}_k}{R}\cdot(\phi_{\boldsymbol{\theta}}^{s_{k,v}})^{\scriptscriptstyle [0]} (X_p)\Big)^2-\Big(\sum_{k=1}^{\bar{\mathfrak{m}}} \bar{c}_{k} \cdot \phi_{\bar{\boldsymbol{\theta}}}^{k}(X_p)\Big)^2\bigg]\bigg|\\
&= \bigg|\frac{1}{2N_{\rm in}}\sum_{p=1}^{N_{\rm in}}\omega(X_p)\bigg[\Big(\sum_{k=1}^{ \bar{\mathfrak{m}}}\sum_{v=1}^{R} \frac{\bar{c}_k}{R}\cdot(\phi_{\boldsymbol{\theta}}^{s_{k,v}})^{\scriptscriptstyle [0]} (X_p)\Big)^2-\Big(\sum_{k=1}^{\bar{\mathfrak{m}}} \sum_{v=1}^{R} \frac{\bar{c}_k}{R} \cdot \phi_{\bar{\boldsymbol{\theta}}}^{k}(X_p)\Big)^2\bigg]\bigg|\\
& \leq \frac{B_0}{2N_{\rm in}}\sum_{p=1}^{N_{\rm in}}\bigg\{\bigg|\bigg[\sum_{k=1}^{ \bar{\mathfrak{m}}}\sum_{v=1}^{R} \Big|\frac{\bar{c}_{k}}{R}\Big|  \cdot \big((\phi_{\boldsymbol{\theta}}^{s_{k,v}})^{\scriptscriptstyle [0]}  (X_p)+ \phi_{\bar{\boldsymbol{\theta}}}^{k}(X_p) \big)\bigg]\\
&\qquad \qquad \qquad  \bigg[\sum_{k=1}^{ \bar{\mathfrak{m}}}\sum_{v=1}^{R}\Big|\frac{\bar{c}_{k}}{R}\Big|  \cdot \big((\phi_{\boldsymbol{\theta}}^{s_{k,v}})^{\scriptscriptstyle [0]}  (X_p)-\phi_{\bar{\boldsymbol{\theta}}}^{k}(X_p) \big)\bigg]\bigg|\bigg\}\\
&\leq B_0 \bar{M} (\bar{W}+1)B_{\bar{\boldsymbol{\theta}}} \bigg|\bigg[\sum_{k=1}^{ \bar{\mathfrak{m}}}\sum_{v=1}^{R} \Big|\frac{\bar{c}_{k}}{R}\Big|  \cdot \big((\phi_{\boldsymbol{\theta}}^{s_{k,v}})^{\scriptscriptstyle [0]} (X_p)-\phi_{\bar{\boldsymbol{\theta}}}^{k}(X_p) \big)\bigg]\bigg|\\
&\leq 2B_0 \bar{W}^{\bar{L}}(\bar{W}+1)\sqrt{\bar{L}}\cdot \bar{M}^2 B^ {\raisebox{0.2ex}{$\scriptscriptstyle \bar{L}$}}_{\bar{\boldsymbol{\theta}}} \cdot \sqrt{\bar{W}(\bar{W}+1)\bar{L}} \cdot \max_{k = 1, \ldots, \bar{\mathfrak{m}}} \max_{v=1,\ldots,R}\big\|(\boldsymbol{\theta}_{s_{k,v}})^{\scriptscriptstyle [0]} -\bar{\boldsymbol{\theta}}_k\big\|_{\infty} \, .
\end{align*}
Finally, we have
\begin{align*}
&\Big|\frac{1}{N_{\rm in}}\sum_{p=1}^{N_{\rm in}}[u_{\mathfrak{m}}^{*}(X_p)h(X_p)- u_{\bar{\mathfrak{m}}, \bar{\boldsymbol{\theta}}}(X_p)h(X_p)]\Big|  + \Big|\frac{1}{N_{b}}\sum_{p=1}^{N_{b}}[u_{\mathfrak{m}}^{*}(Y_p)g(Y_p)- u_{\bar{\mathfrak{m}}, \bar{\boldsymbol{\theta}}}(Y_p)g(Y_p)]\Big| \\
&\leq B_0 \bigg|\bigg[\sum_{k=1}^{ \bar{\mathfrak{m}}}\sum_{v=1}^{R} \Big|\frac{\bar{c}_{k}}{R}\Big|  \cdot \big((\phi_{\boldsymbol{\theta}}^{s_{k,v}})^{\scriptscriptstyle [0]} (X_p)-\phi_{\bar{\boldsymbol{\theta}}}^{k}(X_p) \big)\bigg]\bigg| +  B_0 \bigg|\bigg[\sum_{k=1}^{ \bar{\mathfrak{m}}}\sum_{v=1}^{R} \Big|\frac{\bar{c}_{k}}{R}\Big|  \cdot \big((\phi_{\boldsymbol{\theta}}^{s_{k,v}})^{\scriptscriptstyle [0]} (Y_p)-\phi_{\bar{\boldsymbol{\theta}}}^{k}(Y_p) \big)\bigg]\bigg|\\
&\leq 4B_0 \bar{W}^{\bar{L}}\sqrt{\bar{L}}\cdot B^ {\raisebox{0.2ex}{$\scriptscriptstyle \bar{L}-1$}}_{\bar{\boldsymbol{\theta}}} \cdot \bar{M} \cdot \sqrt{\bar{W}(\bar{W}+1)\bar{L}} \cdot \max_{k = 1, \ldots, \bar{\mathfrak{m}}} \big\|(\boldsymbol{\theta}_{s_{k,v}})^{\scriptscriptstyle [0]} -\bar{\boldsymbol{\theta}}_k\big\|_{\infty} \, .
\end{align*}
Combining above estimations, we can complete the proof of this lemma.

\section*{Acknowledgments}

This work is supported by the National Key Research and Development Program of China (No. 2020YFA0714200), by the National Nature Science Foundation of China (No. 12125103, No. 12071362, No. 12371441), and by the Fundamental Research Funds for the Central Universities.


\begin{thebibliography}{xx}

\bibitem[AAAR19]{Cosmin2019Artificial}
Cosmin Anitescu, Elena Atroshchenko, Naif Alajlan, and Timon Rabczuk.
\newblock Artificial neural network methods for the solution of second order
  boundary value problems.
\newblock {\em Computers, Materials and Continua}, 59(1):345--359, 2019.

\bibitem[ADN59]{agmon1959estimates}
Shmuel Agmon, Avron Douglis, and Louis Nirenberg.
\newblock Estimates near the boundary for solutions of elliptic partial
  differential equations satisfying general boundary conditions. i.
\newblock {\em Communications on pure and applied mathematics}, 12(4):623--727,
  1959.

\bibitem[AZLS19]{allen2019convergence}
Zeyuan Allen-Zhu, Yuanzhi Li, and Zhao Song.
\newblock A convergence theory for deep learning via over-parameterization.
\newblock In {\em International Conference on Machine Learning}, pages
  242--252. PMLR, 2019.

\bibitem[Bac23]{bach2023learning}
Francis Bach.
\newblock Learning theory from first principles, 2023.

\bibitem[BBM05]{bartlett2005local}
Peter~L Bartlett, Olivier Bousquet, and Shahar Mendelson.
\newblock Local rademacher complexities.
\newblock {\em The Annals of Statistics}, 33(4):1497--1537, 2005.

\bibitem[BDG20]{Berner2020Numerically}
Julius Berner, Markus Dablander, and Philipp Grohs.
\newblock Numerically solving parametric families of high-dimensional
  kolmogorov partial differential equations via deep learning.
\newblock In {\em Advances in Neural Information Processing Systems},
  volume~33, pages 16615--16627. Curran Associates, Inc., 2020.

\bibitem[Bel21]{belkin2021fit}
Mikhail Belkin.
\newblock Fit without fear: remarkable mathematical phenomena of deep learning
  through the prism of interpolation.
\newblock {\em Acta Numerica}, 30:203--248, 2021.

\bibitem[BFT17]{bartlettspectrally}
Peter~L Bartlett, Dylan~J Foster, and Matus~J Telgarsky.
\newblock Spectrally-normalized margin bounds for neural networks.
\newblock {\em Advances in neural information processing systems}, 30, 2017.

\bibitem[BGKP22]{berner2022modern}
Julius Berner, Philipp Grohs, Gitta Kutyniok, and Philipp Petersen.
\newblock The modern mathematics of deep learning.
\newblock {\em Mathematical Aspects of Deep Learning}, page~1, 2022.

\bibitem[BHLM19]{bartlett2019nearly}
Peter~L Bartlett, Nick Harvey, Christopher Liaw, and Abbas Mehrabian.
\newblock Nearly-tight vc-dimension and pseudodimension bounds for piecewise
  linear neural networks.
\newblock {\em The Journal of Machine Learning Research}, 20(1):2285--2301,
  2019.

\bibitem[BHMM19]{belkin2019reconciling}
Mikhail Belkin, Daniel Hsu, Siyuan Ma, and Soumik Mandal.
\newblock Reconciling modern machine-learning practice and the classical
  bias--variance trade-off.
\newblock {\em Proceedings of the National Academy of Sciences},
  116(32):15849--15854, 2019.

\bibitem[BJK22]{beck2022full}
Christian Beck, Arnulf Jentzen, and Benno Kuckuck.
\newblock Full error analysis for the training of deep neural networks.
\newblock {\em Infinite Dimensional Analysis, Quantum Probability and Related
  Topics}, 25(02):2150020, 2022.

\bibitem[BK19]{bauer2019deep}
Benedikt Bauer and Michael Kohler.
\newblock On deep learning as a remedy for the curse of dimensionality in
  nonparametric regression.
\newblock {\em The Annals of Statistics}, 47(4):2261--2285, 2019.

\bibitem[BLLT20]{bartlett2020benign}
Peter~L Bartlett, Philip~M Long, G{\'a}bor Lugosi, and Alexander Tsigler.
\newblock Benign overfitting in linear regression.
\newblock {\em Proceedings of the National Academy of Sciences},
  117(48):30063--30070, 2020.

\bibitem[BMM18]{belkin2018understand}
Mikhail Belkin, Siyuan Ma, and Soumik Mandal.
\newblock To understand deep learning we need to understand kernel learning.
\newblock In {\em International Conference on Machine Learning}, pages
  541--549. PMLR, 2018.

\bibitem[BMR21]{bartlett2021deep}
Peter~L Bartlett, Andrea Montanari, and Alexander Rakhlin.
\newblock Deep learning: a statistical viewpoint.
\newblock {\em Acta numerica}, 30:87--201, 2021.

\bibitem[BRT19]{belkin2019does}
Mikhail Belkin, Alexander Rakhlin, and Alexandre~B Tsybakov.
\newblock Does data interpolation contradict statistical optimality?
\newblock In {\em The 22nd International Conference on Artificial Intelligence
  and Statistics}, pages 1611--1619. PMLR, 2019.

\bibitem[BS07]{brenner2007mathematical}
Susanne Brenner and Ridgway Scott.
\newblock {\em The mathematical theory of finite element methods}, volume~15.
\newblock Springer Science \& Business Media, 2007.

\bibitem[CDJ+24]{jiao2024pinns}
Mo~Chen, Zhao Ding, Yuling Jiao, Xiliang Lu, Peiying Wu, and Jerry~Zhijian
  Yang.
\newblock Convergence analysis of pinns with over-parameterization.
\newblock {\em Communications in Computational Physics}, 2024.
\newblock in press.

\bibitem[Cia02]{ciarlet2002finite}
Philippe~G Ciarlet.
\newblock {\em The finite element method for elliptic problems}.
\newblock SIAM, 2002.

\bibitem[COB19]{chizat2019lazy}
Lenaic Chizat, Edouard Oyallon, and Francis Bach.
\newblock On lazy training in differentiable programming.
\newblock {\em Advances in Neural Information Processing Systems}, 32, 2019.

\bibitem[CS02]{cucker2002mathematical}
Felipe Cucker and Steve Smale.
\newblock On the mathematical foundations of learning.
\newblock {\em Bulletin of the American mathematical society}, 39(1):1--49,
  2002.

\bibitem[Cyb89]{cybenko1989approximation}
George Cybenko.
\newblock Approximation by superpositions of a sigmoidal function.
\newblock {\em Mathematics of control, signals and systems}, 2(4):303--314,
  1989.

\bibitem[DHP21]{devore2021neural}
Ronald DeVore, Boris Hanin, and Guergana Petrova.
\newblock Neural network approximation.
\newblock {\em Acta Numerica}, 30:327--444, 2021.

\bibitem[DJL+22]{duan2021convergence}
Chenguang Duan, Yuling Jiao, Yanming Lai, Dingwei Li, Jerry~Zhijian Yang,
  et~al.
\newblock Convergence rate analysis for deep ritz method.
\newblock {\em Communications in Computational Physics}, 31(4):1020--1048,
  2022.

\bibitem[DJSZ23]{dai2023solving}
Yongcheng Dai, Bangti Jin, Ramesh Sau, and Zhi Zhou.
\newblock Solving elliptic optimal control problems via neural networks and
  optimality system.
\newblock {\em arXiv e-prints}, pages arXiv--2308, 2023.

\bibitem[DK23]{drews2023analysis}
Selina Drews and Michael Kohler.
\newblock Analysis of the expected $l_2$ error of an over-parametrized deep
  neural network estimate learned by gradient descent without regularization.
\newblock {\em arXiv preprint arXiv:2311.14609}, 2023.

\bibitem[DLL+19]{du2019gradient}
Simon Du, Jason Lee, Haochuan Li, Liwei Wang, and Xiyu Zhai.
\newblock Gradient descent finds global minima of deep neural networks.
\newblock In {\em International conference on machine learning}, pages
  1675--1685. PMLR, 2019.

\bibitem[DSSSC08]{duchi2008efficient}
John Duchi, Shai Shalev-Shwartz, Yoram Singer, and Tushar Chandra.
\newblock Efficient projections onto the $l_1$-ball for learning in high
  dimensions.
\newblock In {\em Proceedings of the 25th international conference on Machine
  learning}, pages 272--279, 2008.

\bibitem[FLM21]{farrell2021deep}
Max~H Farrell, Tengyuan Liang, and Sanjog Misra.
\newblock Deep neural networks for estimation and inference.
\newblock {\em Econometrica}, 89(1):181--213, 2021.

\bibitem[GK22]{grohs2022mathematical}
Philipp Grohs and Gitta Kutyniok.
\newblock {\em Mathematical aspects of deep learning}.
\newblock Cambridge University Press, 2022.

\bibitem[GKP20]{guhring2020error}
Ingo G{\"u}hring, Gitta Kutyniok, and Philipp Petersen.
\newblock Error bounds for approximations with deep relu neural networks in
  \(w_{s,p}\) norms.
\newblock {\em Analysis and Applications}, 18(05):803--859, 2020.

\bibitem[GN21]{gine2021mathematical}
Evarist Gin{\'e} and Richard Nickl.
\newblock {\em Mathematical foundations of infinite-dimensional statistical
  models}.
\newblock Cambridge university press, 2021.

\bibitem[GR21]{guhring2021approximation}
Ingo G{\"u}hring and Mones Raslan.
\newblock Approximation rates for neural networks with encodable weights in
  smoothness spaces.
\newblock {\em Neural Networks}, 134:107--130, 2021.

\bibitem[GRS18]{golowich2018size}
Noah Golowich, Alexander Rakhlin, and Ohad Shamir.
\newblock Size-independent sample complexity of neural networks.
\newblock In {\em Conference On Learning Theory}, pages 297--299. PMLR, 2018.

\bibitem[GTGT77]{gilbarg1977elliptic}
David Gilbarg, Neil~S Trudinger, David Gilbarg, and NS~Trudinger.
\newblock {\em Elliptic partial differential equations of second order}, volume
  224.
\newblock Springer, 1977.

\bibitem[HJK+20]{hutzenthaler2020overcoming}
Martin Hutzenthaler, Arnulf Jentzen, Thomas Kruse, Tuan Anh~Nguyen, and
  Philippe von Wurstemberger.
\newblock Overcoming the curse of dimensionality in the numerical approximation
  of semilinear parabolic partial differential equations.
\newblock {\em Proceedings of the Royal Society A}, 476(2244):20190630, 2020.

\bibitem[HJW18]{han2018solving}
Jiequn Han, Arnulf Jentzen, and E~Weinan.
\newblock Solving high-dimensional partial differential equations using deep
  learning.
\newblock {\em Proceedings of the National Academy of Sciences},
  115(34):8505--8510, 2018.

\bibitem[HJZ23a]{hu2023solving1}
Tianhao Hu, Bangti Jin, and Zhi Zhou.
\newblock Solving elliptic problems with singular sources using singularity
  splitting deep ritz method.
\newblock {\em SIAM Journal on Scientific Computing}, 45(4):A2043--A2074, 2023.

\bibitem[HJZ23b]{hu2023solving2}
Tianhao Hu, Bangti Jin, and Zhi Zhou.
\newblock Solving poisson problems in polygonal domains with singularity
  enriched physics informed neural networks.
\newblock {\em arXiv preprint arXiv:2308.16429}, 2023.

\bibitem[HMRT22]{hastie2022surprises}
Trevor Hastie, Andrea Montanari, Saharon Rosset, and Ryan~J Tibshirani.
\newblock Surprises in high-dimensional ridgeless least squares interpolation.
\newblock {\em The Annals of Statistics}, 50(2):949--986, 2022.

\bibitem[Hor91]{hornik1991approximation}
Kurt Hornik.
\newblock Approximation capabilities of multilayer feedforward networks.
\newblock {\em Neural networks}, 4(2):251--257, 1991.

\bibitem[HSW89]{hornik1989multilayer}
Kurt Hornik, Maxwell Stinchcombe, and Halbert White.
\newblock Multilayer feedforward networks are universal approximators.
\newblock {\em Neural networks}, 2(5):359--366, 1989.

\bibitem[HSX21]{hong2021rademacher}
Qingguo Hong, Jonathan~W Siegel, and Jinchao Xu.
\newblock Rademacher complexity and numerical quadrature analysis of stable
  neural networks with applications to numerical pdes.
\newblock {\em arXiv preprint arXiv:2104.02903}, 2021.

\bibitem[JGH18]{jacot2018neural}
Arthur Jacot, Franck Gabriel, and Cl{\'e}ment Hongler.
\newblock Neural tangent kernel: Convergence and generalization in neural
  networks.
\newblock {\em Advances in neural information processing systems}, 31, 2018.

\bibitem[JJL+24]{ji2024deep}
Xia Ji, Yuling Jiao, Xiliang Lu, Pengcheng Song, and Fengru Wang.
\newblock Deep ritz method for elliptical multiple eigenvalue problems.
\newblock {\em Journal of Scientific Computing}, 98(2):48, 2024.

\bibitem[JLL+22]{jiao2022rate}
Yuling Jiao, Yanming Lai, Dingwei Li, Xiliang Lu, Fengru Wang, Jerry~Zhijian
  Yang, et~al.
\newblock A rate of convergence of physics informed neural networks for the
  linear second order elliptic pdes.
\newblock {\em Communications in Computational Physics}, 31(4):1272--1295,
  2022.

\bibitem[JLL+23a]{jiao2023error}
Yuling Jiao, Yanming Lai, Yisu Lo, Yang Wang, and Yunfei Yang.
\newblock Error analysis of deep ritz methods for elliptic equations.
\newblock {\em Analysis and Applications}, 2023.

\bibitem[JLL+23b]{jiao2023deep}
Yuling Jiao, Yanming Lai, Xiliang Lu, Fengru Wang, Jerry~Zhijian Yang, and
  Yuanyuan Yang.
\newblock Deep neural networks with relu-sine-exponential activations break
  curse of dimensionality in approximation on h{\"o}lder class.
\newblock {\em SIAM Journal on Mathematical Analysis}, 55(4):3635--3649, 2023.

\bibitem[JLW24]{jiao2024error}
Yuling Jiao, Yanming Lai, and Yang Wang.
\newblock Error analysis of three-layer neural network trained with pgd for
  deep ritz method.
\newblock {\em arXiv preprint arXiv:2405.11451}, 2024.

\bibitem[JLWY24]{jiao2024over}
Yuling Jiao, Xiliang Lu, Peiying Wu, and Jerry~Zhijian Yang.
\newblock Convergence analysis for over-parameterized deep learning.
\newblock {\em Communications in Computational Physics}, 2024.
\newblock in press.

\bibitem[JLY+23]{jiao2023improved}
Yuling Jiao, Xiliang Lu, Jerry~Zhijian Yang, Cheng Yuan, and Pingwen Zhang.
\newblock Improved analysis of pinns: Alleviate the cod for compositional
  solutions.
\newblock {\em Ann. Appl. Math.}, 39:239--263, 2023.

\bibitem[JSLH23]{jiao2021deep}
Yuling Jiao, Guohao Shen, Yuanyuan Lin, and Jian Huang.
\newblock Deep nonparametric regression on approximate manifolds: Nonasymptotic
  error bounds with polynomial prefactors.
\newblock {\em The Annals of Statistics}, 51(2):691--716, 2023.

\bibitem[JW23]{jentzen2023overall}
Arnulf Jentzen and Timo Welti.
\newblock Overall error analysis for the training of deep neural networks via
  stochastic gradient descent with random initialisation.
\newblock {\em Applied Mathematics and Computation}, 455:127907, 2023.

\bibitem[JWY23]{jiao2023approximation}
Yuling Jiao, Yang Wang, and Yunfei Yang.
\newblock Approximation bounds for norm constrained neural networks with
  applications to regression and gans.
\newblock {\em Applied and Computational Harmonic Analysis}, 65:249--278, 2023.

\bibitem[JYZ+23]{jiao2023rate}
Yuling Jiao, Jerry~Zhijian Yang, Junyu Zhou, et~al.
\newblock A rate of convergence of weak adversarial neural networks for the
  second order parabolic pdes.
\newblock {\em Communications in Computational Physics}, 34(3):813--836, 2023.

\bibitem[KK21]{kohler2021over}
Michael Kohler and Adam Krzyzak.
\newblock Over-parametrized deep neural networks minimizing the empirical risk
  do not generalize well.
\newblock {\em Bernoulli}, 27(4):2564--2597, 2021.

\bibitem[KK23a]{kohler2023opt}
Michael Kohler and Adam Krzyzak.
\newblock On the rate of convergence of an over-parametrized deep neural
  network regression estimate with relu activation function learned by gradient
  descent.
\newblock preprint, 2023.

\bibitem[KK23b]{kohler2023rate}
Michael Kohler and Adam Krzyzak.
\newblock On the rate of convergence of an over-parametrized transformer
  classifier learned by gradient descent.
\newblock {\em arXiv preprint arXiv:2312.17007}, 2023.

\bibitem[KL21]{kohler2021rate}
Michael Kohler and Sophie Langer.
\newblock On the rate of convergence of fully connected deep neural network
  regression estimates.
\newblock {\em The Annals of Statistics}, 49(4):2231--2249, 2021.

\bibitem[Kol06]{koltchinskii2006local}
Vladimir Koltchinskii.
\newblock Local rademacher complexities and oracle inequalities in risk
  minimization.
\newblock {\em The Annals of Statistics}, 34(6):2593--2656, 2006.

\bibitem[KPRS22]{kutyniok2022theoretical}
Gitta Kutyniok, Philipp Petersen, Mones Raslan, and Reinhold Schneider.
\newblock A theoretical analysis of deep neural networks and parametric pdes.
\newblock {\em Constructive Approximation}, 55(1):73--125, 2022.

\bibitem[KRV22]{kanade2022exponential}
Varun Kanade, Patrick Rebeschini, and Tomas Vaskevicius.
\newblock Exponential tail local rademacher complexity risk bounds without the
  bernstein condition.
\newblock {\em arXiv preprint arXiv:2202.11461}, 2022.

\bibitem[LCL+21]{lu2021machine}
Yiping Lu, Haoxuan Chen, Jianfeng Lu, Lexing Ying, and Jose Blanchet.
\newblock Machine learning for elliptic pdes: Fast rate generalization bound,
  neural scaling law and minimax optimality.
\newblock {\em ICLR}, 2021.

\bibitem[LLW21]{lu2021priori}
Yulong Lu, Jianfeng Lu, and Min Wang.
\newblock A priori generalization analysis of the deep ritz method for solving
  high dimensional elliptic partial differential equations.
\newblock In {\em Conference on learning theory}, pages 3196--3241. PMLR, 2021.

\bibitem[LMK22]{lanthaler2022error}
Samuel Lanthaler, Siddhartha Mishra, and George~E Karniadakis.
\newblock Error estimates for deeponets: A deep learning framework in infinite
  dimensions.
\newblock {\em Transactions of Mathematics and Its Applications}, 6(1):tnac001,
  2022.

\bibitem[LML+20]{lu2020mean}
Yiping Lu, Chao Ma, Yulong Lu, Jianfeng Lu, and Lexing Ying.
\newblock A mean field analysis of deep resnet and beyond: Towards provably
  optimization via overparameterization from depth.
\newblock In {\em International Conference on Machine Learning}, pages
  6426--6436. PMLR, 2020.

\bibitem[LMMK21]{DeepXDE}
Lu~Lu, Xuhui Meng, Zhiping Mao, and George~Em Karniadakis.
\newblock Deepxde: A deep learning library for solving differential equations.
\newblock {\em SIAM Review}, 63(1):208--228, 2021.

\bibitem[LR20]{liang2020just}
Tengyuan Liang and Alexander Rakhlin.
\newblock Just interpolate: Kernel “ridgeless” regression can generalize.
\newblock {\em The Annals of Statistics}, 48(3):1329--1347, 2020.

\bibitem[LSYZ21]{lu2021deep}
Jianfeng Lu, Zuowei Shen, Haizhao Yang, and Shijun Zhang.
\newblock Deep network approximation for smooth functions.
\newblock {\em SIAM Journal on Mathematical Analysis}, 53(5):5465--5506, 2021.

\bibitem[LZB22]{liu2022loss}
Chaoyue Liu, Libin Zhu, and Mikhail Belkin.
\newblock Loss landscapes and optimization in over-parameterized non-linear
  systems and neural networks.
\newblock {\em Applied and Computational Harmonic Analysis}, 2022.

\bibitem[M+89]{mcdiarmid1989method}
Colin McDiarmid et~al.
\newblock On the method of bounded differences.
\newblock {\em Surveys in combinatorics}, 141(1):148--188, 1989.

\bibitem[Men18]{mendelson2018learning}
Shahar Mendelson.
\newblock Learning without concentration for general loss functions.
\newblock {\em Probability Theory and Related Fields}, 171(1):459--502, 2018.

\bibitem[MM22]{mishra2022estimates}
Siddhartha Mishra and Roberto Molinaro.
\newblock Estimates on the generalization error of physics-informed neural
  networks for approximating a class of inverse problems for pdes.
\newblock {\em IMA Journal of Numerical Analysis}, 42(2):981--1022, 2022.

\bibitem[MR21]{mishra2021enhancing}
Siddhartha Mishra and T~Konstantin Rusch.
\newblock Enhancing accuracy of deep learning algorithms by training with
  low-discrepancy sequences.
\newblock {\em SIAM Journal on Numerical Analysis}, 59(3):1811--1834, 2021.

\bibitem[MZ21]{muller2021error}
Johannes M{\"u}ller and Marius Zeinhofer.
\newblock Error estimates for the variational training of neural networks with
  boundary penalty.
\newblock {\em arXiv preprint arXiv:2103.01007}, 2021.

\bibitem[MZD+24]{mahankali2024beyond}
Arvind Mahankali, Haochen Zhang, Kefan Dong, Margalit Glasgow, and Tengyu Ma.
\newblock Beyond ntk with vanilla gradient descent: A mean-field analysis of
  neural networks with polynomial width, samples, and time.
\newblock {\em Advances in Neural Information Processing Systems}, 36, 2024.

\bibitem[Ngu21]{nguyen2021proof}
Quynh Nguyen.
\newblock On the proof of global convergence of gradient descent for deep relu
  networks with linear widths.
\newblock In {\em International Conference on Machine Learning}, pages
  8056--8062. PMLR, 2021.

\bibitem[NI20]{nakada2020adaptive}
Ryumei Nakada and Masaaki Imaizumi.
\newblock Adaptive approximation and generalization of deep neural network with
  intrinsic dimensionality.
\newblock {\em J. Mach. Learn. Res.}, 21:174--1, 2020.

\bibitem[NVKM20]{nakkiran2020optimal}
Preetum Nakkiran, Prayaag Venkat, Sham~M Kakade, and Tengyu Ma.
\newblock Optimal regularization can mitigate double descent.
\newblock In {\em International Conference on Learning Representations}, 2020.

\bibitem[Pet20]{petersen2020neural}
Philipp~Christian Petersen.
\newblock Neural network theory.
\newblock {\em University of Vienna}, 2020.

\bibitem[Pin99]{pinkus1999approximation}
Allan Pinkus.
\newblock Approximation theory of the mlp model.
\newblock {\em Acta Numerica 1999: Volume 8}, 8:143--195, 1999.

\bibitem[PV18]{petersen2018optimal}
Philipp Petersen and Felix Voigtlaender.
\newblock Optimal approximation of piecewise smooth functions using deep relu
  neural networks.
\newblock {\em Neural Networks}, 108:296--330, 2018.

\bibitem[RPK19]{raissi2019physics}
Maziar Raissi, Paris Perdikaris, and George~E Karniadakis.
\newblock Physics-informed neural networks: A deep learning framework for
  solving forward and inverse problems involving nonlinear partial differential
  equations.
\newblock {\em Journal of Computational Physics}, 378:686--707, 2019.

\bibitem[SH+20]{schmidt2020nonparametric}
Johannes Schmidt-Hieber et~al.
\newblock Nonparametric regression using deep neural networks with relu
  activation function.
\newblock {\em Annals of Statistics}, 48(4):1875--1897, 2020.

\bibitem[Shi20]{shin2020convergence}
Yeonjong Shin.
\newblock On the convergence of physics informed neural networks for linear
  second-order elliptic and parabolic type pdes.
\newblock {\em Communications in Computational Physics}, 28(5):2042--2074,
  2020.

\bibitem[SJHH21]{son2021sobolev}
Hwijae Son, Jin~Woo Jang, Woo~Jin Han, and Hyung~Ju Hwang.
\newblock Sobolev training for the neural network solutions of pdes.
\newblock {\em arXiv preprint arXiv:2101.08932}, 2021.

\bibitem[SN21]{suzuki2021deep}
Taiji Suzuki and Atsushi Nitanda.
\newblock Deep learning is adaptive to intrinsic dimensionality of model
  smoothness in anisotropic besov space.
\newblock {\em Advances in Neural Information Processing Systems}, 34, 2021.

\bibitem[SS18]{Justin2018DGM}
Justin~A. Sirignano and K.~Spiliopoulos.
\newblock Dgm: A deep learning algorithm for solving partial differential
  equations.
\newblock {\em Journal of Computational Physics}, 375:1339--1364, 2018.

\bibitem[Suz18]{suzuki2018adaptivity}
Taiji Suzuki.
\newblock Adaptivity of deep relu network for learning in besov and mixed
  smooth besov spaces: optimal rate and curse of dimensionality.
\newblock In {\em International Conference on Learning Representations}, 2018.

\bibitem[SX20]{siegel2020approximation}
Jonathan~W Siegel and Jinchao Xu.
\newblock Approximation rates for neural networks with general activation
  functions.
\newblock {\em Neural Networks}, 128:313--321, 2020.

\bibitem[SYZ20]{shen2020deep}
Zuowei Shen, Haizhao Yang, and Shijun Zhang.
\newblock Deep network approximation characterized by number of neurons.
\newblock {\em Communications in Computational Physics}, 28(5):1768--1811,
  2020.

\bibitem[SYZ21a]{shen2021deepnn}
Zuowei Shen, Haizhao Yang, and Shijun Zhang.
\newblock Deep network with approximation error being reciprocal of width to
  power of square root of depth.
\newblock {\em Neural Computation}, 33(4):1005--1036, 2021.

\bibitem[SYZ21b]{shen2021neural}
Zuowei Shen, Haizhao Yang, and Shijun Zhang.
\newblock Neural network approximation: Three hidden layers are enough.
\newblock {\em Neural Networks}, 141:160--173, 2021.

\bibitem[SYZ22]{shen2022optimal}
Zuowei Shen, Haizhao Yang, and Shijun Zhang.
\newblock Optimal approximation rate of relu networks in terms of width and
  depth.
\newblock {\em Journal de Math{\'e}matiques Pures et Appliqu{\'e}es},
  157:101--135, 2022.

\bibitem[TB23]{tsigler2023benign}
Alexander Tsigler and Peter~L Bartlett.
\newblock Benign overfitting in ridge regression.
\newblock {\em Journal of Machine Learning Research}, 24(123):1--76, 2023.

\bibitem[Tel21]{telgarsky2021deep}
Matus Telgarsky.
\newblock Deep learning theory lecture notes, 2021.

\bibitem[vdG00]{van2000empirical}
Sara~A. van~de Geer.
\newblock {\em Empirical processes in M-estimation}, volume~6.
\newblock Cambridge university press, 2000.

\bibitem[VDVW96]{VanJhon}
Aad Van Der~Vaart and Jon Wellner.
\newblock {\em Weak convergence}.
\newblock Springer, 1996.

\bibitem[Ver20]{vershynin2020memory}
Roman Vershynin.
\newblock Memory capacity of neural networks with threshold and rectified
  linear unit activations.
\newblock {\em SIAM Journal on Mathematics of Data Science}, 2(4):1004--1033,
  2020.

\bibitem[VYS21]{vardi2021optimal}
Gal Vardi, Gilad Yehudai, and Ohad Shamir.
\newblock On the optimal memorization power of relu neural networks.
\newblock In {\em International Conference on Learning Representations}, 2021.

\bibitem[Wei20]{weinan2020machine}
E~Weinan.
\newblock Machine learning and computational mathematics.
\newblock {\em Communications in Computational Physics}, 28(5):1639--1670,
  2020.

\bibitem[WMW19]{weinan2019priori}
E~Weinan, Chao Ma, and Lei Wu.
\newblock A priori estimates of the population risk for two-layer neural
  networks.
\newblock {\em Communications in Mathematical Sciences}, 17(5):1407--1425,
  2019.

\bibitem[WMW20]{weinan2020comparative}
E~Weinan, Chao Ma, and Lei Wu.
\newblock A comparative analysis of optimization and generalization properties
  of two-layer neural network and random feature models under gradient descent
  dynamics.
\newblock {\em Science China Mathematics}, pages 1--24, 2020.

\bibitem[WY17]{Weinan2017The}
E.~Weinan and Ting Yu.
\newblock The deep ritz method: A deep learning-based numerical algorithm for
  solving variational problems.
\newblock {\em Communications in Mathematics and Statistics}, 6(1):1--12, 2017.

\bibitem[WYP22]{wang2022and}
Sifan Wang, Xinling Yu, and Paris Perdikaris.
\newblock When and why pinns fail to train: A neural tangent kernel
  perspective.
\newblock {\em Journal of Computational Physics}, 449:110768, 2022.

\bibitem[XZ21]{xu2020towards}
Yunbei Xu and Assaf Zeevi.
\newblock Towards optimal problem dependent generalization error bounds in
  statistical learning theory.
\newblock In {\em NeurIPS}, 2021.

\bibitem[Yar17]{yarotsky2017error}
Dmitry Yarotsky.
\newblock Error bounds for approximations with deep relu networks.
\newblock {\em Neural Networks}, 94:103--114, 2017.

\bibitem[Yar18]{yarotsky2018optimal}
Dmitry Yarotsky.
\newblock Optimal approximation of continuous functions by very deep relu
  networks.
\newblock In {\em Conference on Learning Theory}, pages 639--649. PMLR, 2018.

\bibitem[Yar21]{yarotsky2021elementary}
Dmitry Yarotsky.
\newblock Elementary superexpressive activations.
\newblock In {\em International Conference on Machine Learning}, pages
  11932--11940. PMLR, 2021.

\bibitem[YH24]{yang2024deeper}
Yahong Yang and Juncai He.
\newblock Deeper or wider: A perspective from optimal generalization error with
  sobolev loss.
\newblock {\em arXiv preprint arXiv:2402.00152}, 2024.

\bibitem[YZ20]{yarotsky2020phase}
Dmitry Yarotsky and Anton Zhevnerchuk.
\newblock The phase diagram of approximation rates for deep neural networks.
\newblock {\em Advances in neural information processing systems},
  33:13005--13015, 2020.

\bibitem[YZ24]{yang2024optimal}
Yunfei Yang and Ding-Xuan Zhou.
\newblock Optimal rates of approximation by shallow relu k neural networks and
  applications to nonparametric regression.
\newblock {\em Constructive Approximation}, pages 1--32, 2024.

\bibitem[ZBYZ20]{Yaohua2020weak}
Yaohua Zang, Gang Bao, Xiaojing Ye, and Haomin Zhou.
\newblock Weak adversarial networks for high-dimensional partial differential
  equations.
\newblock {\em Journal of Computational Physics}, 411:109409, 2020.

\bibitem[ZG19]{zou2019improved}
Difan Zou and Quanquan Gu.
\newblock An improved analysis of training over-parameterized deep neural
  networks.
\newblock {\em Advances in neural information processing systems}, 32, 2019.

\bibitem[Zho20]{zhou2020universality}
Ding-Xuan Zhou.
\newblock Universality of deep convolutional neural networks.
\newblock {\em Appl. Comput. Harmon. Anal.}, 48(2):787--794, 2020.

\end{thebibliography}
\end{document}